\documentclass[11pt,reqno]{amsproc}
\usepackage{amsmath,amsfonts,amssymb,amsthm}
\usepackage[abbrev,lite,nobysame]{amsrefs}
\usepackage{mathrsfs}
\usepackage{graphics,graphicx}
\usepackage[dvipsnames]{xcolor}
\usepackage[margin=1in]{geometry}

\usepackage{enumitem}
\usepackage[colorlinks=true, pdfstartview=FitV, linkcolor=BrickRed,citecolor=black, urlcolor=black]{hyperref}

\usepackage[centercolon]{mathtools}
\usepackage{dsfont}
\usepackage{soul}

\DeclareMathAlphabet{\mathup}{OT1}{\familydefault}{m}{n}
\newcommand{\dd}[1]{\mathop{}\!\mathup{d} #1}

\renewcommand{\div}[1]{\mathop{}\!\mathup{div} #1}
\renewcommand{\arg}[1]{\mathop{}\!\mathup{Arg} #1}

\newcommand{\N}{\mathbb{N}}
\newcommand{\Z}{\mathbb{Z}}

\newcommand{\R}{\mathbb{R}}
\newcommand{\Co}{\mathbb{C}}
\newcommand{\T}{\mathbb{T}}

\newcommand{\cM}{{\mathcal M}}

\newcommand{\err}{\mathrm{err}}
\newcommand{\vspan}{\mathrm{span}}

\newcommand{\Lapp}{\mathcal{L}_{\mathrm{app}}}

\newcommand{\Int}{\mathrm{Int}}

\newcommand{\Id}{\mathrm{Id}}

\newcommand\e{{\rm e}}

\newcommand{\dsOne}{\mathds{1}}

\def\loc{{\mathrm{loc}}}

\newcommand{\D}{\ensuremath{\mathcal{D}}}

\renewcommand{\Re}{\mathrm{Re}}
\renewcommand{\Im}{\mathrm{Im}}

\newcommand{\eps}{\varepsilon}

\newtheorem{proposition}{Proposition}[section]
\newtheorem{theorem}{Theorem}[section]
\newtheorem{lemma}{Lemma}[section]

\newtheorem{example}{Example}[section]

\theoremstyle{definition}
\newtheorem{definition}{Definition}[section]

\theoremstyle{remark}
\newtheorem{remark}{Remark}[section]

\newcommand{\G}{\mathcal{G}}
\newcommand{\F}{\mathcal{F}}

\renewcommand{\L}{\mathcal{L}}

\newcommand{\BigO}{O}
\newcommand{\M}{\mathcal{M}}
\newcommand{\cI}{\mathcal{I}}

\newcommand{\bd}{\mathrm{bd}}
\newcommand{\main}{\mathrm{main}}
\newcommand{\error}{\mathrm{error}}

\DeclareMathOperator{\I}{I}
\DeclareMathOperator{\II}{II}
\DeclareMathOperator{\III}{III}
\DeclareMathOperator{\IV}{IV}

\DeclareMathOperator{\A}{A}

\DeclareMathOperator{\Bi}{Bi}
\DeclareMathOperator{\Ai}{Ai}

\newcommand{\TC}{\mathrm{TC}}
\newcommand{\LV}{\mathrm{LV}}

\DeclareMathAlphabet{\mathup}{OT1}{\familydefault}{m}{n}

\newcommand{\supp}[1]{\mathop{}\!\mathup{supp} #1}

\newcommand{\dyn}{\mathcal{L}^{\mathrm{dyn}}}
\numberwithin{equation}{section}

\makeatletter
\def\namedlabel#1#2{\begingroup
    #2%
    \def\@currentlabel{#2}%
    \phantomsection\label{#1}\endgroup
}
\makeatother

\makeatletter

\renewcommand\subsubsection{\@startsection{subsubsection}{3}%
\normalparindent{.5\linespacing\@plus.7\linespacing}{-.5em}
{\normalfont\bfseries}}

\def\@tocline#1#2#3#4#5#6#7{\relax
  \ifnum #1>\c@tocdepth 
  \else
    \par \addpenalty\@secpenalty\addvspace{#2}%
    \begingroup \hyphenpenalty\@M
    \@ifempty{#4}{%
      \@tempdima\csname r@tocindent\number#1\endcsname\relax
    }{%
      \@tempdima#4\relax
    }%
    \parindent\z@ \leftskip#3\relax \advance\leftskip\@tempdima\relax
    \rightskip\@pnumwidth plus4em \parfillskip-\@pnumwidth
    #5\leavevmode\hskip-\@tempdima
      \ifcase #1
       \or\or \hskip 1em \or \hskip 2em \else \hskip 3em \fi%
      #6\nobreak\relax
    \dotfill\hbox to\@pnumwidth{\@tocpagenum{#7}}\par
    \nobreak
    \endgroup
  \fi}
\makeatother

\begin{document}

\title[Spectral instability in the smooth Ponomarenko dynamo]{Spectral instability in the smooth Ponomarenko dynamo}

\author[V. Navarro-Fernández]{Víctor Navarro-Fernández}
\address{(VNF) Department of Mathematics, Imperial College London, London, SW7 2AZ, UK}
\email{v.navarro-fernandez@imperial.ac.uk}

\author[D. Villringer]{David Villringer}
\address{(DV) Department of Mathematics, Imperial College London, London, SW7 2AZ, UK}
\email{d.villringer22@imperial.ac.uk}



\date{\today}

\begin{abstract}
We consider the kinematic dynamo equations for a passive vector in $\mathcal{M} \times \mathbb{T} \subseteq \mathbb{R}^2 \times \mathbb{T}$ describing the evolution of a magnetic field with resistivity $\eps > 0$, that is transported by a given velocity field. For a broad class of $C^3$ velocity fields with helical geometry, we establish the existence of solutions that exhibit exponential growth over time. We construct an unstable eigenmode via detailed resolvent estimates of the corresponding linear operator, which we carry out by introducing suitable Green’s functions that accurately approximate the local behaviour of the true system. This approach yields an explicit asymptotic expression for the growing mode, providing a sharp description of the instability mechanism. We first derive the results with $\M=\R^2$ for a large class of velocity fields that includes finite energy examples. We then consider the case of domains with boundary, where $\M\times\T$ denotes a periodic cylinder, annular cylinder, or the exterior of a cylinder, with the boundary conditions of perfectly conducting walls. Our results offer a rigorous and sharp mathematical justification for the physically conjectured process by which helical flows can sustain magnetic field generation in the Ponomarenko dynamo, with growth rate of order $\eps^{1/3}$.
\end{abstract}

\maketitle

\setcounter{tocdepth}{2}
\tableofcontents

\section{Introduction}\label{s:intro}

The kinematic dynamo equations describe the evolution of a magnetic field $B^\eps\in\R^3$ under the influence of a prescribed, divergence-free velocity field $u\in\R^3$, and play a central role in understanding magnetic field generation in astrophysical and geophysical contexts, see e.g.\ the original work by Larmor \cite{Larmor19}. Despite their linearity, the equations exhibit rich and subtle dynamics as a result of the interplay between advection, stretching and diffusion. We consider the Cauchy problem
\begin{equation}\label{eq:dynamo}
\left\lbrace
    \begin{array}{rcl}
        \partial_t B^\eps + (u\cdot \nabla)B^\eps - (B^\eps\cdot \nabla)u & = & \eps\Delta B^\eps, \\
        \div (B^\eps) & = & 0, \\
        B^\eps(0,\cdot) & = & B^\eps_{\mathrm{in}},
    \end{array}
\right.
\end{equation}
with $t\in (0,\infty)$ and $x\in \M\times\T$. Here $\M \subseteq \R^2$, $\eps>0$ denotes the magnetic diffusivity, and $B^\eps_\mathrm{in}$ represents an initial configuration for the magnetic field. When $\partial \M \neq \emptyset$, we further impose the \emph{perfectly conducting} boundary conditions
\begin{equation}\label{eq:BC}
B^\eps \cdot \hat n|_{\partial(\M\times \T)} =0, \quad (\nabla \times B^\eps) \times \hat n |_{\partial(\M\times \T)} =0,
\end{equation}
where $\hat n$ denotes the outwards pointing unit normal vector to the boundary of $\M\times\T$. This choice of perfectly conducting walls as boundary conditions is standard for the magnetohydrodynamics (MHD) equations, see e.g.\ \cites{SermangeTemam1983,roberts1967}.

A central question in the framework of equation \eqref{eq:dynamo} is whether one can find examples of vector fields $u$ for which there is a sustained exponential growth of the total magnetic energy---typically represented by the $L^2$ norm of $B^\eps$. 
Indeed, constructing examples of velocity fields that lead to exponential magnetic energy growth has been a major focus of research among physicists and mathematicians in the last decades, see the monographs \cites{ChildressGilbert,AK98,GhilChildress1987}.

In this regard, we say that a vector field $u$ is a \emph{kinematic dynamo} if there exist an initial datum $B_\mathrm{in}^\eps \in L^2(\M\times\T)$ satisfying appropriate boundary conditions, a coefficient $\gamma_\eps>0$, and a constant $c_\eps>0$---all possibly depending on the magnetic diffusivity $\eps$, such that solutions to equations \eqref{eq:dynamo} satisfy
\begin{equation}\label{eq:gamma}
    \|B^\eps(t)\|_{L^2} \geq c_\eps\|B^\eps_{\mathrm{in}}\|_{L^2}\e^{\gamma_\eps t},
\end{equation}
see for instance \cite{AK98}*{Chapter V} for further details about this definition.

The dynamo problem can also be addressed for the more physical case of the nonlinear MHD equations, which are composed of \eqref{eq:dynamo} coupled with the Navier--Stokes equations for fluids. In this situation, conservation of energy prevents the exponential growth from occurring indefinitely in time, therefore the problem is formulated as a transfer from kinetic to magnetic energy in finite time. This question has been extensively studied numerically, e.g.\ \cites{BrandenburgSubramanian2005,Moffatt1978,RuedigerHollerbach2006,ChildressGilbert}, but only a few rigorous results are available in the literature \cites{GerardVaret05,GerardVaretRousset07,FriedlanderVishik91,Vishik86}. Crucially, since equations \eqref{eq:dynamo} are nothing but the linearised MHD equations around the (possible) steady state $(u,0)$, a good understanding of the kinematic dynamo equations may allow for a more precise analysis of nonlinear instability for the MHD equations.

For the linear problem \eqref{eq:dynamo}, we have the following classification of kinematic dynamos depending on how the growth rate $\gamma_\eps$ \eqref{eq:gamma} varies with the magnetic diffusivity $\eps$. The simplest scenario corresponds to $\eps=0$ in \eqref{eq:dynamo}---i.e.\ the non-resistive case:
\begin{itemize}
    \item if $\gamma_0>0$, $u$ defines an \emph{ideal dynamo}.
\end{itemize} 
In this case, the problem can be reduced to proving positivity of the top Lyapunov exponent for the Lagrangian trajectories; with well-prepared initial data, the magnetic energy will grow exponentially fast. The more challenging problem of finding \emph{universal} ideal dynamos---where the growth occurs for all divergence free $B_\mathrm{in}\in L^2$---has been studied with random velocity fields in \cites{BaxendaleRozovskii93,CotiZelatiNavarroFernandez}.

If $\eps>0$, a distinction depending on the behaviour of $\gamma_\eps$ as $\eps\to 0$ was introduced in \cite{VainshteinZeldovich1972}, giving rise to two different types of dynamo: 
\begin{itemize}
    \item if $\liminf_{\eps\to 0} \gamma_\eps > 0$, $u$ defines a \emph{fast dynamo},
    \item if $\limsup_{\eps \to 0}\gamma_{\eps} \leq 0$, $u$ defines a \emph{slow dynamo}.
\end{itemize}
The fast dynamo problem has been rigorously formulated in Arnold's problems book \cite{ArnoldsProblems} for compact manifolds in $\R^3$. As stated, this is still an outstanding open problem, however some relevant examples of fast dynamo action have been found in the full space $\R^3$. For instance, in \cite{ZRMS1984} unbounded vector fields of the form $u(t,x) = C(t)x$, for some random traceless matrix $C(t)$, are proved to be fast dynamos; or in \cite{gilbert1988} an example of a discontinuous axisymmetric flow has been shown to exhibit fast dynamo action---we will elaborate on this specific flow later on. Only recently, bounded and Lipschitz examples of fast dynamos in the full space have been found in \cite{CZSV25} by exploiting the notion of the \emph{alpha-effect}, see also \cites{Childress1979,Moffatt1978,Roberts1970}. Moreover, it has been proven that fast dynamos in the full space not only exist but are, in some sense, generic \cites{CZSV25,Roberts1970}. The existence of fast dynamos in bounded subdomains of $\R^3$ is still unknown, but some progress has been made very recently in \cite{Rowan25}, where a \emph{subsequential} fast dynamo on the three-dimensional torus has been proved to exist, namely an example with the property
\[
\limsup_{\eps\to 0} \limsup_{t\to\infty} \frac{1}{t}\log\|B^\eps(t)\|_{L^2} >0.
\]

In this paper, we shall focus on the case of slow dynamos, where the growth rate decays to zero as $\eps \to 0$. This phenomenon is intricately linked to concepts in boundary layer theory---indeed, whilst velocity fields defining fast dynamos must necessarily also be ideal dynamos (see \cite{Vishik89}), this is not the case for slow dynamos. Instead, somewhat counter-intuitively, the introduction of the dissipative term $\eps \Delta$ into the equation may actually have a \emph{destabilizing} effect. Analogous phenomena are well documented for the Navier--Stokes equations, where profiles that may be stable for the Euler equations can become linearly unstable once viscosity is introduced, see for instance \cites{GGN16a,GGN16b}. 

Despite the fact that slow dynamos should in principle be less complicated to study than fast dynamos, there are remarkably few rigorous results in the literature. In large part, this lack of results (for both fast and slow dynamos) can be attributed to the existence of so-called \emph{anti-dynamo theorems}, results which show that dynamo action may only be maintained by ``sufficiently complicated'' velocity fields. Without any aim of completeness, we mention Zeldovich’s theorem \cites{zeldovich1980magnetic,Zeldovich_1992}, which states that a velocity field with zero vertical component cannot sustain a dynamo, and Cowling’s theorem \cite{Cowling33}, which rules out dynamos for which the magnetic field $B$ is axisymmetric. These results highlight that a dynamo requires a genuinely three-dimensional magnetic field \cite{AK98}.

Despite these theoretical restrictions, in \cite{Ponomarenko73} Ponomarenko gives an example of a relatively simple velocity field under which the analysis of the kinematic dynamo equations reduces to a system of coupled ordinary differential equations, but which still may yield dynamo action. Indeed, in \cite{Ponomarenko73}, velocities of the of the form 
\begin{equation}
\label{eq:ponomarenko velocity field}
u = r\Omega(r)\hat\theta+ U(r)\hat z,
\end{equation}
are considered, where $(r,\theta,z)\in \R\times\T\times\T$ denotes the standard cylindrical coordinate system, and
\begin{equation}
\label{eq:fast ponomarenko}
\Omega(r)=\begin{cases}
\Omega & r \leq r_0,\\
0 & r>r_0,
\end{cases}
\quad 
U(r)=
\begin{cases}
U & r \leq r_0,\\
0 & r>r_0,
\end{cases}
\end{equation}
for some constant $\Omega, U\neq 0$. In fact, in \cite{gilbert1988} Gilbert shows that this choice of axisymmetric and discontinuous velocity field actually defines a fast dynamo, for which the growth is asymptotically localised near the discontinuity at $r=r_0$ as $\eps \to 0$. 

To rectify the somewhat unphysical nature of a discontinuous velocity field, in \cite{gilbert1988} the author further considers divergence-free and radial velocity fields of the form \eqref{eq:ponomarenko velocity field}, where $\Omega, U$ are arbitrary, smooth functions of $r$. Via a formal asymptotic WKB expansion, Gilbert claims that the maximum growth rate of dynamos driven by velocity fields of the form \eqref{eq:ponomarenko velocity field} is given by 
$\gamma_\eps \sim \eps^{1/3}$.
Henceforth, we shall refer to dynamo action by such velocity fields as ``smooth Ponomarenko dynamos''. In particular, $\gamma_\eps\to 0$ as $\eps\to 0$, and so the dynamo action is that of a \emph{slow dynamo}. This statement has been supported since with strong numerical \cites{Wynne-thesis,Peyrot_Gilbert_Plunian_2009,Normand03} and experimental evidence \cites{Gailitis2000,verhille2010}. The only rigorous result on the smooth Ponomarenko dynamo we are aware of is from \cite{GerardVaretRousset07}, where the authors prove finite time instability for the forced MHD equations with a Ponomarenko-type velocity profile. It is worth remarking that both the lower and upper bounds from \cite{GerardVaretRousset07} are in good agreement with the growth rate $\gamma_\eps\sim \eps^{1/3}$, further supporting the idea that the results obtained in \cite{gilbert1988} are sharp.

In this work, we shall provide a rigorous proof of dynamo action for a large class of functions $\Omega, U$ in the Ponomarenko dynamo \eqref{eq:ponomarenko velocity field}, for $\M \times \mathbb{T}$ being either the full space $\R^2\times\T$, a cylinder, an annular cylinder, or the exterior of a cylinder, equipped with perfectly conducting boundary conditions. Furthermore, we shall give a precise description of the growing magnetic field profile in the limit $\eps \to 0$.

\subsection{Preliminaries and main results}

At this point, it is worth outlining the ways in which the choice of Ponomarenko velocity field \eqref{eq:ponomarenko velocity field} simplifies the analysis of the kinematic dynamo equations. 
Writing the magnetic vector in cylindrical components $B = B_r\hat r + B_\theta\hat \theta + B_z\hat z$, we find that the kinematic dynamo equations \eqref{eq:dynamo} become
\begin{align*} 
    \partial_t B_r +\Omega(r) \partial_\theta B_r+U(r)\partial_zB_r & =\eps \left(\Delta B_r-\frac{B_r}{r^2}-\frac{2}{r^2}\partial_\theta B_\theta \right), \\
    \partial_t B_\theta+\Omega(r) \partial_\theta B_\theta+U(r)\partial_zB_\theta
     & = \eps\left(\Delta B_\theta-\frac{B_\theta}{r^2}+\frac{2}{r^2}\partial_\theta B_r \right) + r\Omega'(r)B_r, \\
    \partial_t B_z +\Omega(r)\partial_\theta B_z+U(r)\partial_zB_z & =\eps \Delta B_z + U'(r)B_r,
\end{align*}
where $\Delta = \partial_r^2 + r^{-1}\partial_r + r^{-2}\partial_\theta^2 + \partial_z^2$ represents the scalar Laplacian operator in cylindrical coordinates. Observe that the first two equations---for $B_r$ and $B_\theta$---are coupled through the stretching term $r\Omega'(r)B_r$ appearing in the second equation, as well as through the Laplacian coupling, which is a by-product of the cylindrical geometry. Crucially however, the equations corresponding to $B_r, B_\theta$ do not contain $B_z$, and so the PDE admits a Jordan block-like structure. Indeed, our strategy will be to solve the coupled system for the radial and azimuthal components $(B_r,B_\theta)$, and then, only once a solution is found, do we solve for the $z$ component, which satisfies a standard advection-diffusion equation with a forcing of the form $U'(r)B_r$.

In this coordinate system, we consider the components $(r,\theta)\in\M=\cI\times\T$, where $\cI\subseteq[0,\infty)$ is an interval in the positive real line, namely $\M\subseteq\R^2$ consists of a circle, an annulus, or the exterior of a circle centred at the origin. We can further make the modal form Ansatz 
\[
B(t,r,\theta,z) = b(r) \e^{\lambda t+ i(m\theta + kz)},
\]
for $b:\mathcal{I}\to \mathbb{C}^3$, $\lambda\in\Co$ and $m,k\in \Z$. The components of the vector $b = b_r\hat r+b_\theta\hat\theta + b_z\hat z$ satisfy the following system of ordinary differential equations,
\begin{align}
    \lambda b_r + i(\Omega(r)m +U(r)k)b_r & = \eps\left( \partial_r^2b_r + \frac{1}{r}\partial_rb_r - \frac{m^2}{r^2}b_r - k^2b_r -\frac{1}{r^2}b_r - \frac{2im}{r^2}b_\theta\right), \label{eq:modal-eq-r}\\
    \lambda b_\theta + i(\Omega(r)m +U(r)k)b_\theta & = r\Omega'(r)b_r + \eps\left( \partial_r^2b_\theta + \frac{1}{r}\partial_rb_\theta - \frac{m^2}{r^2}b_\theta - k^2b_\theta -\frac{1}{r^2}b_\theta + \frac{2im}{r^2}b_r\right), \label{eq:modal-eq-theta}\\
    \lambda b_z + i(\Omega(r)m +U(r)k)b_z & = U'(r)b_r + \eps\left( \partial_r^2b_z + \frac{1}{r}\partial_rb_z - \frac{m^2}{r^2}b_z - k^2b_z\right). \label{eq:modal-eq-z}
\end{align}
The perfectly conducting boundary conditions \eqref{eq:BC} translate into 
\begin{equation}\label{eq:BC-radial}
    b_r|_{\partial \mathcal{I}\setminus\{0\}}=(rb_\theta)'|_{\partial \mathcal{I}\setminus\{0\}}=b_z'|_{\partial \mathcal{I}\setminus\{0\}}=0
\end{equation}
under this modal Ansatz. If $\cI\subseteq[0,\infty)$ is of the form $\cI=[0,q)$ for some $0<q\leq\infty$, it is worth noting that $\{0\}\in\partial\cI$ is not an actual element of $\partial(\M\times\T)$ and just a fabrication of the coordinate system. 
Therefore, we can appreciate that finding an exponentially growing solution to the kinematic dynamo equations may be achieved by finding an \emph{unstable eigenvalue} of the kinematic dynamo operator, that is, finding a finite energy magnetic field $b(r)$, and $\lambda \in \mathbb{C}$ with $\Re(\lambda)>0$ so that equations \eqref{eq:modal-eq-r}--\eqref{eq:modal-eq-z} are represented by 
\begin{equation}\label{eq:Ponomarenko-operator}
\L^{\mathrm{dyn}}b=\lambda b,
\end{equation}
where the differential operator $\L^{\mathrm{dyn}}$ will be referred to as the \emph{Ponomarenko dynamo operator}. Observe that an eigenvalue $\lambda$ in \eqref{eq:Ponomarenko-operator} with positive real part gives the growth condition \eqref{eq:gamma}. 
Indeed, the bulk of this paper shall be devoted to studying the spectrum of the operator $\L^{\mathrm{dyn}}$, and we shall say that any (smooth) velocity field of the form \eqref{eq:ponomarenko velocity field} for which there is an unstable eigenfunction of $\L^{\mathrm{dyn}}$ is a (smooth) Ponomarenko dynamo.

With this spectral point of view established, we now move on to providing a summary of the main results that we prove in this paper.

\subsubsection{Main results}

The main aim of this work is to find conditions on the functions $\Omega, U$ which guarantee the existence of an exponentially growing mode for the Ponomarenko dynamo operator defined by \eqref{eq:modal-eq-r}--\eqref{eq:modal-eq-z}. Of course the exact conditions depend on the domain $\M$ in question, but there are a few commonalities. In fact, we will show that the existence of a growing mode can be guaranteed by the local properties of $\Omega, U$ around some critical radius $r_0$ determined by the Fourier modes $m,k$, as long as the behaviours of $\Omega(r), U(r)$ as $r \to \infty$ are not too degenerate (if the domain we are interested in includes this endpoint). Since the precise form of the degeneracies we need to rule out is somewhat technical, it is instructive to give an informal sketch of the assumptions we shall place on the functions $\Omega, U$. Indeed, we will assume conditions on $\Omega, U$ that formally amount to:
$\Omega(r),U(r)$ are smooth functions on $[0,\infty)$, which furthermore satisfy
\begin{enumerate}
    \item $\Omega, U$ grow at most polynomially as $r\to\infty$,
    \item $\Omega, U$ do not oscillate infinitely often around a fixed value as $ r\to \infty$,
    \item As $r \to \infty$, $|\Omega(r)+U(r)| \gg |r\Omega'(r)|$.
\end{enumerate}

We refer to Assumptions \ref{H0}--\ref{H3} in Section \ref{s:admissible-fields} for a more rigorous statement of these properties. We draw special attention to the last of these assumptions, since (at least heuristically) it gives some insight as to why we can always find a growing \emph{eigenfunction} for \eqref{eq:dynamo}, even when we pose the equation on a non-compact domain. Indeed, on an intuitive level, this assumption amounts to saying that for large $r$, the stretching term $r\Omega'(r)$ is ``weaker'' than the transport term $u\cdot \nabla$, and therefore it cannot overcome the decay caused by the interplay of the Laplacian and mixing as a result of transport. Therefore, all growth in the equation must happen in the vicinity of the origin, where the Laplacian admits a compact resolvent, and so growing solutions are indeed attained by eigenfunctions. More formally, we can say that the equation \eqref{eq:dynamo} can be treated as an advection diffusion equation, perturbed by the \emph{relatively compact} term $r\Omega'(r)B_r$. Since such a perturbation does not disturb the essential spectrum, which for the advection diffusion equation will be contained in the left half plane, any spectral value in the right half plane will immediately be an eigenvalue for the dynamo equation. Indeed, in Section \ref{s:injectivity} we shall make this intuition rigorous.

Our main results, stated rigorously below, then show that under the aforementioned conditions, slow dynamo action with rate $\gamma_\eps \sim \eps^{1/3}$, can be deduced from a purely local, geometric condition at some critical radius $r_0>0$. Moreover, we give a precise asymptotic expansion of the associated eigenfunction $B^\eps$ as $ \eps \to 0$.

Our construction of the dynamo profiles proceeds via the analysis of the ODE system defined by the operator $\dyn$, which is a meaningful problem even when its connection to the PDE \eqref{eq:dynamo} is less relevant (for instance when $m, k \notin \mathbb{Z}$, and so do not represent ``true'' Fourier variables). We therefore split our main result into two theorems, one which focuses exclusively on the ODE, whereas the other synthesises the ODE result into an instability result for our PDE. 

\begin{theorem}\label{thm:dynamo}
    Let $\cI =[p,q]$ or $\cI = [p,\infty)$, where $p \in [0,\infty)$, $q \in (p,\infty)$, and let the velocity field $u=r\Omega(r)\hat\theta + U(r)\hat z$ satisfy Assumptions \ref{H0}--\ref{H2}. Pick any $r_0 \in \mathrm{int}(\cI)$ so that
    \[
    r_0\left|\frac{\dd}{\dd r}\log\left|\frac{\Omega'(r_0)}{U'(r_0)}\right|\right|<4,
    \]
    and any $|M|$ small.
    Then, for all $\eps>0$ sufficiently small, there exist a constant $\mu\in\Co$ with $\Re(\mu)>0$, and a $C^2$ vector field $b^\eps: \cI \to \Co^3$ with $b^\eps \in (L^1\cap L^\infty)(\cI,r\dd r)$, so that $\L^{\mathrm{dyn}}b^\eps=\eps^{1/3}\mu b^\eps$. Furthermore, if $\cI \neq [0,\infty)$, $b^\eps$ satisfies the perfectly conducting boundary conditions 
    \[
    b^\eps_r|_{\partial \cI \setminus \{0\}}=0, \quad (rb^\eps_\theta)'|_{\partial \cI \setminus \{0\}}=0, \quad (b^\eps_z)'|_{\partial \cI \setminus \{0\}}=0.
    \]    
    In addition, the growth rate satisfies the asymptotic expansion $\mu = \mu_\star + o_{\eps \to 0}(1)$, where
   \begin{equation*}
    \Re(\mu_\star) = \left[ |M|^{1/2}\left(\frac{|\Omega'(r_0)|^{1/2}}{r_0^{1/2}} - \frac{1}{2}\left| \Omega''(r_0)-\frac{\Omega'(r_0)}{U'(r_0)}U''(r_0) \right|^{1/2}    \right) - |M|^2\left(\frac{1}{r_0^2}+\frac{|\Omega'(r_0)|^2}{|U'(r_0)|^2}\right)\right ]>0.
   \end{equation*}
    If either $\cI$ is compact, or Assumption \ref{H3} is satisfied, there further exists $a>0$ such that for any $N\in\N$ and all $\eps>0$ small enough depending on $N$, the $(r,\theta)$ components of $b^\eps$ admit the asymptotic expansion,
    \begin{align*}
    \begin{pmatrix}
    b^\eps_{r}\\
    b^\eps_{\theta}
    \end{pmatrix}
     = 
    \begin{pmatrix}
    \eps^{1/3}\sqrt{\frac{-2iM}{r_0^3\Omega'(r_0)}}\\
    1
    \end{pmatrix}\left (\e^{-\frac{1}{2}\eps^{-2/3}c_2^{1/2}(r-r_0)^2}+\psi_{\mathrm{err}}^\eps \right)
    \end{align*}
   for some explicitly computable constant $c_2 \in \mathbb{C}$ with $\Re(c_2^{1/2})>0$, and
   \[
   \sup_{r \in \mathcal{I}}\left|\left(1+(|r-r_0|\eps^{-1/3})^N\right)\psi_{\mathrm{err}}^\eps\right| \lesssim \eps^a.
   \]
\end{theorem}

Once the result for the ODE is established, under an additional technical assumption, we can immediately deduce the existence of a slow dynamo, as stated in the next theorem.

\begin{theorem}\label{thm:full dynamo}
If in addition to the assumptions of Theorem \ref{thm:dynamo}, the compact set $\mathcal{R}_0$ from \ref{H2} contains an interval around $r_0$, then, for all $\eps>0$ sufficiently small, there exists $m,k \in \mathbb{Z}$ so that 
\[
B^\eps(t,r,\theta,z) = \e^{\eps^{1/3}\mu t}\e^{i(m\theta + kz)} b^\eps(r) 
\]
is a divergence-free, finite energy solution to the kinematic dynamo equations \eqref{eq:dynamo}, where $b^\eps$ and $\mu$ are given by Theorem \ref{thm:dynamo}.
\end{theorem}

\begin{remark}
Whilst the existence of a growing mode for the dynamo equation may be deduced already under Assumptions \ref{H0}--\ref{H2}, Assumption \ref{H3} is needed to ensure that the growing mode is an \emph{isolated eigenvalue} for the kinematic dynamo operator on a suitable space. This is the subject of Section \ref{s:injectivity}. Once this is achieved, our asymptotic expansion is deduced using tools from classical spectral theory.
\end{remark}
\begin{remark}
Although throughout the paper we make the standing assumption that $z \in \mathbb{T}$, our results readily imply a corresponding dynamo result for the case when $z \in \mathbb{R}$. Indeed, under the assumptions of Theorem \ref{thm:full dynamo}, we immediately deduce the existence of an open set $\mathcal{K} \subset \mathbb{R}$, with eigenmode $b^\eps(r,k)$, so that as $k$ varies throughout $\mathcal{K}$, $\Re(\mu(k))>0$. Then, setting
\begin{equation*}
B(t,r,\theta,z)=\int_{\mathcal{K}}\e^{\eps^{1/3}\mu(k)t} \e^{i(m\theta+kz)} b^\eps(r,k) \dd k,
\end{equation*}
we obtain a slow dynamo in $L^2(\mathcal{M} \times \mathbb{R})$.
\end{remark}
Our results show that, for a large class of velocity fields, the Ponomarenko dynamo \eqref{eq:ponomarenko velocity field} exhibits $L^p$ norm exponential growth of the magnetic field for any $1\leq p\leq \infty$. In particular, it displays slow dynamo action with growth rate \eqref{eq:gamma} of the form
\[
\gamma_\eps = \eps^{1/3}(\Re(\mu_\star) + o_{\eps\to 0}(1))>0.
\]

At this point, it is worth providing some relatively simple, explicit examples of velocity fields which satisfy the aforementioned Assumptions \ref{H0}--\ref{H3}. Indeed, we have the following result.

\begin{example}
In each of the following cases, the conclusions of Theorems \ref{thm:dynamo}, \ref{thm:full dynamo} hold true on $\mathcal{I}=[p,q]$ or $\cI = [p,\infty)$ for any $p \in [0,\infty)$, $q \in (p,\infty)$.
\begin{enumerate}
    \item Simplified Ponomarenko model: $\Omega(r)=1-r$, $U(r)=1-r^2$
    \item \label{eq:gaussian} Gaussian velocity fields: $\Omega(r)=e^{-ar^2}$, $U(r)=e^{-br^2}$, with $a,b>0$ and $a \neq b$.
    \item \label{eq:compact support} Compactly supported velocity fields: $\Omega(r)=(1-r)\chi(r)$, $U(r)=(1-r^2)\chi(r)$, where 
    \[
    \chi(r) = \begin{cases}
        1 & \text{if } r\leq 2, \\
        0 & \text{if } r\geq 3,
    \end{cases}
    \]
    and $\chi(r)$ is smooth.
\end{enumerate}
Furthermore, the conclusions of Theorems \ref{thm:dynamo}, \ref{thm:full dynamo} hold true on the interval $\mathcal{I}=[p,q]$, for $p>0$, $q \in (p,\infty)$ for the Taylor--Couette flow
\begin{equation*}
u(r)=\left(\frac{a_1}{r}+a_3 r\right)\hat \theta+(a_2 \log(r)+a_4) \hat z,
\end{equation*}
which is an explicit stationary solution to the three-dimensional Navier--Stokes equations.
\end{example}

In particular, examples (\ref{eq:gaussian}) and (\ref{eq:compact support}) are, to the best of our knowledge, the first rigorous proofs of slow dynamo action driven respectively by finite energy and by compactly supported velocity fields, whereas the Taylor--Couette flow is the first rigorous proof of dynamo action driven by an explicit solution to the Navier--Stokes equations. 

Once we have shown that the Taylor--Couette flow generates a dynamo, we can employ the framework of \cite{Friedlander_Pavlović_Shvydkoy_2006} to deduce nonlinear instability of the 3D MHD equations around the steady state $(u_\TC,0)$, where $u_\TC$ denotes the Taylor--Couette flow. The details of this construction diverge from the main objective of this paper, and will be presented in the forthcoming work \cite{VNF-DV-Nonlinear}.

\addtocontents{toc}{\protect\setcounter{tocdepth}{1}}
\subsection*{Notation}
Due to the length of the paper, we collect here certain notational conventions that will be frequently used. 
\begin{itemize}
    \item We denote by $B(t,r,\theta,z)$ magnetic fields that solve the full kinematic dynamo PDE \eqref{eq:dynamo}, whereas $b(r)$ denotes a solution to the modal form ODE \eqref{eq:modal-eq-r}-\eqref{eq:modal-eq-z}. 
    
    \item Throughout the paper, $\mathcal{I} \subseteq [0,\infty)$ denotes an interval of the form $\mathcal{I}=[p,q]$, with $p \in [0,\infty)$, $q \in (p,\infty]$ (where in the case $q=+\infty$, we take $\mathcal{I}=[p,\infty)$). Since the interval arises from a change of variables into cylindrical coordinates, the point $0$ (should it be an endpoint of $\mathcal{I}$) is not in fact a physical boundary. Thus, we make the abuse of notation and denote the boundary $\partial \mathcal{I}=\{q\}$ in the case where $p=0$. Next, the letters $m,k$ denote Fourier variables for $\theta$, $z$ respectively.

    \item We will employ the notation $a\lesssim b$ to denote that there exists a constant $C>0$, independent of all the relevant quantities (in particular always independent of $\eps$), such that $a\leq Cb$. Moreover, in many different estimates we will use a generic constant $C>0$ that might change from line to line.

    \item We use the notation ``big $O$" and ``little $o$" for the limit $\eps\to 0$ as it is standard. Namely we say that an object is $O_{\eps\to 0}(1)$ if it is bounded independently of $\eps$ for all $\eps>0$ small enough, and we say that a function $f$ is $o_{\eps\to 0}(1)$ if $\lim_{\eps \to 0}f=0$.

    \item In keeping with the notation of \cite{Peyrot_Gilbert_Plunian_2009}, we set $m=M\eps^{-1/3}$, $k=K\eps^{-1/3}$, for some real coefficients $M,K=O_{\eps\to 0}(1)$.

    \item Given a function space $F$, if all the components of a vector valued function 
    \[
    f = (f_1,\hdots,f_n):\cM\times\T\to\R^n
    \]
    satisfy $f_i \in F(\cM\times\T)$, then we will abuse the notation and write as well $f\in F(\cM\times\T)$.
\end{itemize}

\addtocontents{toc}{\protect\setcounter{tocdepth}{2}}
\section{Strategy}\label{s:strategy}

In this section we give an overview on the method employed to establish the main results, i.e.\ Theorems \ref{thm:dynamo} and \ref{thm:full dynamo}. Let us recall the main ideas presented thus far:
\begin{enumerate}
    \item We consider helical vector fields---after Ponomarenko---of the form 
    \[
    u(r) = r\Omega(r)\hat\theta+U(r)\hat z.
    \]
    In Section \ref{s:analysis-equations} we define the class of admissible vector fields.
    
    \item We look for growing modes: solutions to equations \eqref{eq:dynamo} in modal form
    \[
    B(t,r,\theta,z) = b(r)\e^{\lambda t}\e^{i(m\theta+kz)},
    \]
    for some $\lambda\in\Co$ with $\Re(\lambda)>0$, and $m,k\in\Z$. This Ansatz transforms a PDE problem \eqref{eq:dynamo} into an ODE problem \eqref{eq:modal-eq-r}--\eqref{eq:modal-eq-z}.

    \item Since the equations for $(b_r,b_\theta)$ in \eqref{eq:modal-eq-r}--\eqref{eq:modal-eq-theta} are decoupled from $b_z$, we put our efforts in solving them first. After a suitable change of variables introduced in Section \ref{s:analysis-equations}, we denote the operator associated to such ODE by $\L-\lambda$.
    
    \item In Section \ref{s:injectivity} we present analytical and spectral properties of the operator $\L-\lambda$, and we make sure that all admissible solutions satisfy the divergence-free condition.
    
    \item The main goal and bulk of this paper is to find $V$ in some suitable functional space, and $\lambda\in\Co$ with $\Re(\lambda)>0$, such that 
    \[
    (\L-\lambda)V = 0.
    \]
    This is the main purpose of Sections \ref{s:proof}, \ref{s:approx-Green's} and \ref{s:boundaries}.

    \item Once solutions $(b_r,b_\theta)$ are obtained, we construct the vertical component $b_z$ in Section \ref{s:z-component}.
\end{enumerate}

Without further ado, let us present some details on the strategy to find a growing mode of the kinematic dynamo operator.

\subsection{Resolvent formalism and Green's functions}
\label{s:heuristics}

We begin by discussing the strategy for constructing a growing mode in the case where $\mathcal{M}=\mathbb{R}^2$, in other words, $r \in [0,\infty)$. Hence, we shall study the operator $\L:\mathcal{D}(\L) \subset X\to X$, where $X$ denotes the Banach space
\[
X = L^\infty ((0,\infty),\max\{1,r^2\}w_\eps(r-r_0)\dd r),
\]
with the weight function
\begin{equation}\label{eq:weight}
w_\eps(s) = 1+(\eps^{-1/3}|s|)^N,
\end{equation}
for some fixed $N>0$. In particular, notice that $X\subset (L^1\cap L^\infty)(0,\infty)$. 

The main challenge of the spectral analysis will be to bridge the gap between the formal, local analysis of Gilbert \cite{gilbert1988}---outlined in Section \ref{s:gilbert-scaling}---and the existence of a genuine, global eigenfunction of the dynamo operator. On one hand, at the level of generality we pose the problem, there is no real hope of finding an explicit solution to the system of ODEs \eqref{eq:modal-eq-r}--\eqref{eq:modal-eq-theta}. On the other hand, there is no immediate property, such as for instance normality, of the operator $\L$, that would facilitate spectral analysis. Nevertheless, we shall use the fact that our equation involves a small parameter $\eps$ to study its spectral properties using tools from perturbation theory. In particular, we will be interested in studying the \emph{resolvent} $(\L-\lambda)^{-1}$ of our operator, in the formal limit when $\eps \to 0$. This of course carries with it its own difficulties, since the ``straightforward'' approach of constructing the inverse of a complicated operator by viewing it as perturbation of a known operator will simply not work for our case. Indeed, whilst for certain subregions of $[0,\infty)$, the operator may be well-approximated by a simpler ODE (as seen for instance in Gilbert's analysis in Section \ref{s:gilbert-scaling}), the approximations are not global. This issue is well known in the community, and in the case of asymptotic expansions, is typically dealt with using the technique of ``matching inner and outer expansions'', i.e.\ constructing asymptotic expansions which are valid in overlapping spatial regions. In fact, this technique has been rigorously employed to prove nonlinear instability for instance in the seminal works \cites{Grenier_2000, Desjardins_Grenier_2003}, and in particular in the context of the MHD equations, in the works \cites{GerardVaret05, GerardVaretRousset07}. Inspired by this intuition, we employ a technique of ``matched local inverses'', namely constructing approximate inverses for our equation in disjoint parts of the domain, which can be glued together to produce a global inverse. Due to the importance of this technique in our work, we now describe it in more detail.

We split the interval $[0,\infty)$ into distinct regions, in each of which our ODE is well-approximated (in the strong sense) by a simpler ODE, which we know how to solve. In particular, given an interval, say $[a,b] \subset [0,\infty)$, we seek an approximate operator $\Lapp$, so that $\Lapp$ is invertible on $L^\infty(-\infty,\infty)$, and $(\Lapp)^{-1}$ restricted to $L^\infty([a,b])$ satisfies $\|(\L-\Lapp)(\Lapp)^{-1}\|=o_{\eps \to 0}(1)$. Namely, uniformly for all $f$ such that $\supp(f) \subset [a,b]$, there holds 
\begin{equation*}
\|(\L-\Lapp)(\Lapp)^{-1}f\|_{L^\infty([a,b])}\leq o_{\eps \to 0}(1)\|f\|_{L^\infty([a,b])}.
\end{equation*}
Additionally, we shall require that, for any compactly supported $f$, $(\Lapp)^{-1}f$ \emph{decays sufficiently quickly} (typically this means exponentially with an $\eps$ dependent exponent) outside of the support of $f$. Indeed, this property turns out to be rather generic for elliptic equations on $(-\infty,\infty)$, so long as one is sufficiently far away from the spectrum of the ODE. Supposing further that $\L-\Lapp$ is solely a multiplication operator by a function with at most polynomial growth at infinity---call it $R(r)$, we can now hope to deduce uniformly for any $f$ with $\supp(f) \subset [a,b]$
\begin{align*}
\|(\L-\Lapp)(\Lapp)^{-1}f&\|_{L^\infty([0,\infty))} \leq \|(\L-\Lapp)(\Lapp)^{-1}f\|_{L^\infty([a,b])}\\
& \quad + \sup_{r\leq a}e^{-\eps^{-\theta}|r-a|}|R(r)|\|f\|_{L^\infty([a,b])}+\sup_{r \geq b}e^{-\eps^{-\theta}|r-b|}|R(r)|\|f\|_{L^\infty([a,b])}\\
&\lesssim o_{\eps \to 0}(1)\|f\|_{L^\infty([a,b])},
\end{align*}
for some particular $\theta>0$, provided $R(r)$ is small near $a$ and $b$. This means that $\Lapp$ is a \emph{global} approximation to the ODE, as long as $f$ is supported in $[a,b]$. 

From here, we construct a partition 
\begin{equation*}
[0,\infty)=\bigcup_{i\in I}[p_i,p_{i+1}),
\end{equation*}
where $I$ denotes a finite set of indices,
and to each interval $[p_i,p_{i+1})$ we associate an operator $\Lapp^i$ that satisfies 
\begin{equation*}
|(\L-\Lapp^i)(\Lapp^i)^{-1}f\|_{L^\infty([0,\infty))}\lesssim o_{\eps \to 0}(1)\|f\|_{L^\infty([p_i,p_{i+1}])},
\end{equation*}
uniformly for $\supp(f) \subset [p_i,p_{i+1})$. Then we can in fact compute the true inverse of $\L$ on $[0,\infty)$ via a Neumann series type construction.
Set 
\[
\G_{\mathrm{app}}f=\sum_{i \in I} (\Lapp^i)^{-1}(\mathds{1}_{[p_i,p_{i+1})}f),
\]
and then define
\begin{equation*}
\L^{-1}=\G_{\mathrm{app}}\sum_{n \geq 0}(-1)^n(\L \G_{\mathrm{app}}-\Id)^n.
\end{equation*}
Applying $\L$ to this operator truncated at degree $N$ yields 
\begin{equation*}
\sum_{n =0}^N(-1)^n(\L \G_{\mathrm{app}}-\Id)^{n+1}+\sum_{n =0}^N(-1)^n(\L \G_{\mathrm{app}}-\Id)^{n}=(-1)^N (\L \G_{\mathrm{app}}-\Id)^{N+1}+\Id.
\end{equation*}
Hence, as long as $\|\L \G_{\mathrm{app}}-\Id\|<1$, we indeed have an inverse. But note that 
\begin{align*}
\|(\L \G_{\mathrm{app}}-\Id)f\|_{L^\infty([0,\infty))} & = \left\lVert\L\sum_{i \in I} (\Lapp^i)^{-1}(\mathds{1}_{[p_i,p_{i+1})}f)-\sum_{i \in I}\mathds{1}_{[p_i,p_{i+1})}f\right\rVert_{L^\infty([0,\infty))}\\
& =\left\lVert\sum_{i \in I} (\L-\Lapp^i)(\Lapp^i)^{-1}(\mathds{1}_{[p_i,p_{i+1})}f)\right\rVert_{L^\infty([0,\infty))} \\
& \lesssim o_{\eps \to 0}(1)\|f\|_{L^\infty([0,\infty))},
\end{align*}
and so for all $\eps$ small enough, our series converges and defines an inverse.

When applying the above heuristic to the operator $\L$, we encounter the issue that $\L$ contains terms of order $r^{-2}$, which ensure that $|(\L-\Lapp)(\Lapp)^{-1}f|$ is always large near the origin, even for compactly supported $f$. To rectify this, we introduce the auxiliary space
\[
Y = L^\infty ((0,\infty),r^2w_\eps(r-r_0)\dd r),
\]
which allows for quadratic divergences as $r\to 0$, and will be used to measure the size of the ``error''. 
In particular, we will show that $\|\L \G_{\mathrm{app}}-\Id\|_{Y \to Y}=o_{\eps \to 0}(1)$. In order for this inverse to be well defined, it must map $X$ to itself, so we further show that $\G_{\mathrm{app}}$ actually defines a bounded map $\G_{\mathrm{app}}:Y \to X$. 

With this in mind, we can now sketch a high level overview of the proof of Theorem \ref{thm:dynamo}.

\begin{enumerate}
    \item We begin by fixing a curve 
     \[
    \Gamma\subset \{z\in\Co\mid \Re(z)>0\}.
    \] 
For $\lambda \in \Gamma$, we then construct a ``glued approximate operator'' $ \G^\lambda$, satisfying 
$$\|(\L-\lambda) \G^\lambda -\Id\|_{Y \to Y} \lesssim o_{\eps \to 0}(1),$$
uniformly for $\lambda \in \Gamma$. Crucially, $\G^\lambda$ is of the form 
    \[
    \G^\lambda f=\G_2^\lambda (f\mathds{1}_{[r_0-\eps^\gamma, r_0+\eps^\gamma]})+\G_{\mathrm{rem}}(f\mathds{1}_{\mathbb{R}\setminus \mathds{1}_[r_0-\eps^\gamma, r_0+\eps^\gamma] }),
    \]
    where $\G_2^\lambda$ is the Green's function of Gilbert's approximate operator \eqref{eq:Gilbert equation}, and $\G_{\mathrm{rem}}$ \emph{does not depend on $\lambda \in \Gamma$}. 
    \item Using $\G^\lambda$, we show that inverse of the operator $\L-\lambda$ can be written as the infinite series
    \[
    (\L-\lambda)^{-1}f = \G^\lambda \sum_{n=0}^\infty (-1)^n\left((\L-\lambda)\G^\lambda - \Id\right)^n f,
    \]
    uniformly for all $\lambda\in\Gamma$. 

    \item We define the Riesz projector 
    \[
    Pf = \frac{1}{2\pi i}\int_\Gamma (\lambda-\L)^{-1}f\dd \lambda,
    \]
    and compute $P\phi$, where $\phi$ is the approximate growing mode from Section \ref{s:gilbert-scaling}. We show that 
    \[
    P\phi=-\frac{1}{2\pi i}\int_{\Gamma}\G^\lambda \phi d\lambda +\mathrm{error}
    \]
    where the error term is of order $o_{\eps \to 0}(1)$.
    \item 
    Recalling the definition of $\G^\lambda$, it holds by Cauchy's theorem
    \[
    -\frac{1}{2\pi i}\int_{\Gamma}\G^\lambda \phi d\lambda =-\frac{1}{2 \pi i}\int_{\Gamma} \G_2^\lambda(\phi \mathds{1}_{[r_0-\eps^\gamma, r_0+\eps^\gamma]}) \dd \lambda ,
    \]
    and since $\phi$ is \emph{by its very definition} an eigenfunction of Gilbert's operator \eqref{eq:Gilbert equation}, we conclude that $P\phi\neq 0$. Therefore, there exists $\lambda_0\in \Int(\Gamma)$, and a function $\psi\in X$ such that $\psi\in\ker(\L-\lambda_0)$, and so we have constructed a growing mode.
\end{enumerate}

The majority of the paper will be devoted to the careful construction of the operator $\G^\lambda$, since we shall need precise, quantitative bounds on it, in order to rigorously carry out the above programme.

\subsubsection{Extension to bounded domains}

The strategy for proving the existence of a growing mode we have described thus far inherently relies on the problem being posed on an unbounded domain. In particular, the property that $(\L_{\mathrm{app}})^{-1}f$ decays exponentially outside of the support of $f$ is evidently false in general if $(\L_{\mathrm{app}})^{-1}$ is constructed to furthermore satisfy boundary conditions on some $\mathcal{I} \subset [0,\infty)$. Thus, our previous strategy needs to be amended in order to cover domains with boundary. 

In particular, we begin by defining a new Banach space, denoted by $X_{\mathcal{I}}$, which consists of 
\begin{equation}\label{def:X_I}
X_{\mathcal{I}}= L^\infty(\mathcal{I},\max\{1,r^2\}w_\eps(r-r_0)\dd r),
\end{equation}
and new domain 
\begin{equation*}
\mathcal{D}(\dyn_\mathcal{I})=\{b \in X_\mathcal{I} \cap H^2_{\mathrm{loc}}(\mathcal{I}):\dyn b \in X, \quad b_r|_{\partial \mathcal{I}}=(rb_\theta)'|_{\partial \mathcal{I}}=0\},
\end{equation*}
where we introduce the perfectly conducting walls boundary conditions \eqref{eq:BC-radial}, and we abuse notation setting $\partial \mathcal{I}=\{q\}$ in the case when $\mathcal{I}=[0,q]$.

The intuition we follow is that the growing mode we constructed is highly concentrated around the critical layer $r_0$, i.e.\ it decays rapidly as the boundaries of the domain are approached. Thus, on an intuitive level, the boundary conditions are ``almost'' satisfied, and adding a small corrective term may in fact be enough to obtain a growing mode that satisfies the perfectly conducting boundary conditions. Guided by this heuristic, we prove that the inverse operator $(\L-\lambda)^{-1}_{\mathrm{pc}}$ with perfectly conducting boundary conditions can be written as the full inverse $(\L-\lambda)^{-1}|_{\mathcal{I}}$ on $[0,\infty)$, restricted to the interval $\mathcal{I}$, plus corrective terms, which consist of appropriately scaled homogeneous solutions of the equation $(\L-\lambda)V=0$ on $\mathcal{I}$, ensuring that the boundary conditions are satisfied. In this way, we write 
\begin{equation*}
(\L-\lambda)^{-1}_{\mathrm{pc}}f=\G^\lambda\sum_{n \geq 0} (-1)^n \left ( (\L-\lambda)\G^\lambda -\Id\right )^nf+\mathrm{boundary \ terms},
\end{equation*}
where $\G^\lambda$ is precisely the approximate inverse we constructed for the case $[0,\infty)$, and the boundary terms depend \emph{solely} on the values of $(\L-\lambda)^{-1}f$ at the boundary of the domain. Recalling that the growing mode we constructed for $[0,\infty)$ decays rapidly away from $r_0$, when $f=\phi$ is the approximate growing mode from Section \ref{s:gilbert-scaling} we deduce that the boundary terms may indeed be treated as an error. Thus, we once again write 
\begin{equation*}
P\phi =\frac{1}{2 \pi i} \int_{\Gamma}(\lambda-\L)^{-1}\phi \dd \lambda =-\frac{1}{2 \pi i}\int_{\Gamma} \G^\lambda \phi \dd \lambda +\mathrm{error} \neq 0,
\end{equation*}
and deduce the existence of a growing mode, satisfying the perfectly conducting boundary conditions \eqref{eq:BC-radial}.

\subsection{Design of the approximate operators}

A key point of our proof is the design of the approximated operator $\Lapp$. Whilst of course there are some case-by-case subtleties, as a general rule of thumb, we construct the approximate operator as follows. 

We determine the ``dominant'' terms of the operator $\mathcal{L}$ in any spatial region by keeping track of two parameters: the magnetic diffusivity $\eps$, and the radius $r$. Combined, one may associate to every term in the equation defining $\L$ a fairly crude ``relative size'' as $\eps \to 0$, depending on the spatial region. For instance, in the expression $\eps^{1/3}r^{-2}-\eps^{-1/3} r^2$, the dominant term for $r \ll \eps^{1/6}$ is $\eps^{1/3}r^{-2}$, whereas for $r \gg \eps^{1/6}$ it is $-\eps^{-1/3}r^2$. Whilst this approach is crude, due to the perturbative nature of the problem it actually yields remarkable accurate approximate operators throughout overlapping spatial domains, which have relatively simple inverse operators. In fact, our hypotheses \ref{H0}-\ref{H2} will ensure that there is only one subset of $[0,\infty)$ (in the example $\eps^{1/3}r^{-2}-\eps^{-1/3} r^2$ one might think of the ``transition point'' $r=\eps^{1/6}$), where this naive ``order of magnitude analysis'' does not give a simple approximate operator , and this will be precisely where we expect the growing mode to localise. Fortunately, in this critical regime, we can make use of Gilbert's approximation scheme (see Section \ref{s:gilbert-scaling}) in order to derive careful local inverses, for values of the spectral parameter very close to a spectral point of the approximate operator. In our previous notation, this yields the operator $\G^\lambda_2$, and it will allow us to close our argument as explained at in Section \ref{s:heuristics}. 

We start our analysis in Section \ref{s:approx-Green's} by considering the case $\cI=[0,\infty)$ and a simplified scenario for the transport function $T(r)$, that is first assumed to have one single zero at $r=r_0$. In this setting, we show that there exists $\gamma>0$ such that the domain $[0,\infty)$ can be split as follows.

\begin{itemize}
    \item \emph{Towards infinity}: $\supp(f)\subset[r_0+\eps^\gamma,\infty)$. In this region, the operator $\mathcal{L}$ is well approximated by a semiclassical Schr\"odinger operator with complex, non-vanishing potential. In particular, it will allow us to be very rough in our construction of the approximate inverse, denoted by $\G_3$, which is obtained by freezing the value of the potential $T(r)$ on the boundaries of a very fine partition of $[r_0+\eps^\gamma, \infty)$, and gluing together the resulting approximate inverses. The resulting map $\mathcal{G}_3$ is independent of $\lambda$. Further details about the precise definition and the estimates can be found in Section \ref{sec:tow-inf}.

     \item \emph{Towards zero}: $\supp(f)\subset[\eps^\gamma,r_0 - \eps^\gamma)$. Once again, in this region there is a clear ``leading order operator'', which again is a semiclassical Schr\"odinger operator with complex, non-vanishing potential. Thus, we proceed entirely as in the case towards infinity, taking care this time of the asymptotics as we approach $r=0$. We denote the Green's function associated to this regime by $\G_1$, and it is again defined independently of $\lambda$. Details about this regime can be found in Section \ref{s:tow-zero}.
     
    \item \emph{Near zero}: $\supp(f)\subset[0,\eps^\gamma)$. This regime presents the first part of the domain where additional care must be taken, since our approximate operator will involve terms that diverge quadratically as $r \to 0$. It is however of crucial importance, since the resulting Green's function will turn out to have a \emph{regularising effect}, mapping functions that diverge quadratically at $r=0$ to bounded functions at $r=0$. We denote the Green's function associated to this regime by $\G_0$, and we make it such that it does not depend on $\lambda$. The approximation procedure yields an ODE that will be solved using modified Bessel functions. The analysis of this part is covered in Section \ref{sec:near-0}.

    \item \emph{Around the critical radius}: $\supp(f)\subset[r_0-\eps^\gamma, r_0+\eps^\gamma)$. This construction is the crux of the argument, and requires detailed analysis of parabolic cylinder functions near spectral values of the Hermite equation. We denote the Green's function in this interval by $\G_2^\lambda$, making this the only regime where the approximated operator depends on $\lambda$. The approximation procedure will follow Gilbert's ideas presented in Section \ref{s:gilbert-scaling}. Further details on the definition and precise estimates on $\G_2^\lambda$ can be found in Section \ref{s:around-r0}.

\end{itemize}

The form of the approximate Green's function $\G^\lambda$ in the situation of Section \ref{s:approx-Green's} will be given by
\[
\begin{split}
    \G^\lambda f & = \G_0 (f\dsOne_{[0,\eps^\gamma)}) +  \G_1 (f\dsOne_{[\eps^\gamma,r_0-\eps^\gamma)}) + \G_2^\lambda (f\dsOne_{[r_0-\eps^\gamma,r_0+\eps^\gamma)}) +  \G_3 (f\dsOne_{[r_0+\eps^\gamma,\infty)}),
\end{split}
\]
where $\dsOne_A$ denotes the indicator function of the set $A$: $\dsOne_A(r)=1$ if $r\in A$, and $\dsOne_A(r)=0$ otherwise.

In Section \ref{s:linear-zeroes}, we extend the arguments from the previous Sections \ref{s:around-r0}, \ref{sec:tow-inf}, \ref{s:tow-zero}, \ref{sec:near-0} to the case where the transport function vanishes linearly at a finite set of points, denoted by $s_j>0$. We study the problem locally around each $s_j$, specifically in the interval $[s_j-\eps^\omega,s_j+\eps^\omega)$ for some parameter $\omega>0$. We denote each corresponding approximated Green's function by $\G^\LV_j$, and we solve the problem using Airy functions. All in all, we obtain an approximate inverse operator that can be iterated to obtain a well--defined inverse for $\L-\lambda$ on $[0,\infty)$. 

In particular, this inverse will also be key in allowing us to construct homogeneous solutions to $(\L-\lambda)V=0$ on intervals $\mathcal{I} \subset [0,\infty)$, by setting (for instance) $V=(\L-\lambda)^{-1}(\mathds{1}_{[0,\infty)\setminus \mathcal{I}})$. By definition, it holds $(\L-\lambda)V=0$ on $\mathcal{I}$, and via a careful analysis of the iterative procedure defining $V$, it can be shown that $V$ is not identically zero on $\mathcal{I}$. 

When constructing resolvents for our problem on bounded domains with perfectly conducting boundary conditions, this property will be crucial. Indeed, from the point of view of the ODE, it is clear that if one may find $n$ linearly independent homogeneous solutions to an ODE with $n$ boundary conditions, that are furthermore ``linearly independent at the boundary'' (in the sense that the matrix consisting of their values at the boundary is invertible), the ODE must be invertible, simply by adding suitable scaled homogeneous solutions to any particular solution to the ODE in order to ``fix'' the boundary conditions. In the case of the Ponomarenko dynamo, when $\mathcal{I}=[p,q]$ there will generically be four boundary conditions that need to be satisfied. Thus, we endeavour to construct four linearly independent solutions to $(\mathcal{L}-\lambda)v_i=0$ on $[p,q]$, $i=1,2,3,4$. The idea will be to set for instance
$$
v_1=(\mathcal{L}-\lambda)^{-1}\begin{pmatrix}
\mathds{1}_{[p-\eps^\delta,p]}\\
0
\end{pmatrix}
$$
for $\delta$ small. Certainly $(\mathcal{L}-\lambda)v_1=0$ in $(p,q)$, and by continuity, we can hope to show that for $r \in [p,q]$, $r$ close to $p$, $v_1(r) \neq 0$, so it is not the trivial solution. Furthermore, we can hope that the inverse operator $(\L-\lambda)^{-1}$ we construct inherits the exponential decay property of its building blocks, namely that $(\L-\lambda)^{-1}v(r)$ decays exponentially away from the support of $v$.--- In particular, this would imply that $|v_1(q)|\ll 1$. Setting thus 
$$
v_2(r)=(\mathcal{L}-\lambda)^{-1}\begin{pmatrix}
\mathds{1}_{[q,q+\eps^\delta]}\\
0
\end{pmatrix}
$$
and repeating the argument, we deduce $v_2(q) =O(1)$, whereas $v_2(p) \sim 0$. Hence, we deduce that $v_1,v_2$ are linearly independent. Similarly, we can define $v_3, v_4$ by localising the initial condition in the second component, i.e. 
$$
v_3=(\mathcal{L}-\lambda)^{-1}\begin{pmatrix}
0\\
\mathds{1}_{[p-\eps^\delta,p]},
\end{pmatrix}
$$
and once again showing that the corresponding solutions are small (respectively large) at the endpoints.
In fact, in Section \ref{s:boundaries} we undertake a detailed dynamical analysis of the iterative procedure defining $(\mathcal{L}-\lambda)^{-1}$, studying the propagation of mass by each term of the series defining $(\L-\lambda)^{-1}$, and conclude that not only do $v_i$ $i=1,2,3,4$ define linearly independent solutions to $(\mathcal{L}-\lambda)v_i=0$, but they also satisfy suitable estimates as $\eps \to 0$, allowing them to be treated perturbatively as $\eps \to 0$ in the construction of the inverse problem $(\L-\lambda)b=v$ on $[p,q]$, subject to perfectly conducting boundary conditions.

\subsection{Comparison with existing literature}

The (discontinuous) Ponomarenko dynamo was first introduced in \cite{Ponomarenko73}, and it was subsequently shown to be a fast dynamo in \cite{gilbert1988}. In that paper, Gilbert goes through the argument sketched in Section \ref{s:gilbert-scaling} in order to deduce the existence of a growing mode for the smooth Ponomarenko dynamo. This is further expanded upon in the more recent paper \cite{Peyrot_Gilbert_Plunian_2009}. As mentioned in the introduction, the only rigorous result on the Ponomarenko dynamo we are aware of is \cite{GerardVaretRousset07}, which proves nonlinear instability for the full MHD equations perturbed around Ponomarenko-like velocity fields. The authors consider a sequence of velocity fields $u^\nu$ that formally converge to the discontinuous Ponomarenko dynamo \eqref{eq:fast ponomarenko}, and employ the method of matched asymptotic expansions to show that initial data of order $\nu^p$ for $p>0$ and $\nu\ll 1$ can grow to size $O(1)$ in time $O(|\log(\nu)|)$. This result demonstrates the sort of nonlinear instability that is conjectured to occur as a by-product of the dynamo effect. However, since nonlinear instability results are inherently concerned with short time growth, to the best of the authors' knowledge, it is not possible to directly deduce linear spectral instability from either the results or the methods of \cite{GerardVaretRousset07}. Indeed, the method of matched asymptotic expansions yields an exponentially growing first order term, which is dominant for times of order $|\log(\nu)|$, but for times larger than that, higher order terms may dominate, whose behaviours are quite difficult to discern.

Deducing linear spectral instability for ODEs is itself a rich field, but we shall draw special attention to the paper \cite{GGN16a}, which served as an inspiration for some of our techniques. In this paper, the authors prove spectral instability for the linearised Navier--Stokes equations around a wide class of shear flows, via an analysis of the Orr--Sommerfeld equations. Such analysis carries with it the inherent challenge that the Orr--Sommerfeld instability problem is not a spectral problem in the traditional sense---the spectral parameter enters multiplicatively in the equation. Thus, there are considerably fewer abstract tools available to study such a problem. In \cite{GGN16a} the authors circumvent this issue via an ingenious approach, involving the construction of two homogeneous solutions to the equation using carefully constructed approximate inverses for the operator. By means of studying the dispersion relation of their homogenous solutions, they are able to construct a growing mode, which satisfies the desired boundary conditions, whilst giving detailed estimates on the spectral parameter in the process. Such an approach is particularly well suited for a boundary layer problem, in which one expects the instability to arise \emph{as a result} of the boundary. Constructing an approximate inverse to the equation which is finite both at the origin and at infinity requires that the ODE is invertible on $[0,\infty)$ provided one does not impose any boundary conditions at $0$.

It is precisely this property that fails for our construction---homogenous solutions to the Ponomarenko dynamo equations are either finite or blow up near zero, and finding a solution that is well behaved both at the origin and at infinity is the entire point of our endeavour. Therefore, the strategy of constructing a homogenous solution which is well-behaved at infinity, and then using continuity arguments to show that this solution must also be well-behaved at zero for some value of the spectral parameter does not seem to work in our situation. Crucially however, since our problem is indeed a traditional spectral problem, we can fall back to powerful results from classical spectral theory, and deduce the existence of a growing mode simply by finding approximate inverse operators on a circle enclosing the predicted location of the eigenvalue, and then estimating the resulting construction with care. In particular, we believe our methodology may be applicable in some generality for deducing rigorous spectral instability results from formal asymptotic expansions, provided the problem is posed as an eigenvalue problem. It may be possible to adapt our methods to study further dynamo problems, such as for instance the ``almost-fast'' dynamo on $\T^3$ constructed in \cite{Soward_1987}.

\section{Analysis of the equations and admissible vector fields}\label{s:analysis-equations}

In this section we explain how the equations under consideration are derived. We begin with a simplified model to motivate the scaling in $\eps$ from Theorem \ref{thm:dynamo} for $\lambda$, $m$ and $k$. We then give an overview on Gilbert’s asymptotic analysis, which leads to a change that is crucial the rest of the analysis presented in this paper.

\subsection{A simplified model}
\label{s:heuristics2}
Before delving into the Ponomarenko dynamo equations in more detail, it is useful to get some feel for how the construction of the dynamo occurs from a heuristic point of view. So far it is unclear how the critical radius $r_0>0$ enters into the equations, as well as what scaling of the Fourier modes $m,k\in\Z$ yields the largest growth rate. We will try to answer these questions in this section, by means of a formal order of magnitude analysis on a simplified model of the dynamo equations, which discards the terms of lesser relevance to the dynamics.

Indeed, consider the system
\begin{equation}
\label{eq:toy model}
\begin{split}
&\partial_t b_r=\eps\left(\partial_r^2 b_r-\frac{m^2}{r^{2}}b_r -\frac{2im}{r^2}b_\theta\right)-i(m\Omega(r)+kU(r))b_r\\
&\partial_t b_\theta=\eps\left(\partial_r^2b_\theta-\frac{m^2}{r^{2}}b_\theta+\frac{2im}{r^2}b_r\right)-i(m\Omega(r)+kU(r)b_\theta+r\Omega'(r)b_r.
\end{split}
\end{equation}
We split the equation into an ``unstable'' stretching term, given by 
\begin{equation*}
\partial_t \begin{pmatrix}
b_r\\
b_\theta
\end{pmatrix}
=-\eps mr^{-2}\begin{pmatrix}
b_r\\
b_\theta
\end{pmatrix}+\eps \begin{pmatrix}
-2imr^{-2} b_\theta\\
2imr^{-2}b_r
\end{pmatrix}+\begin{pmatrix}
0\\
r\Omega'(r)b_r
\end{pmatrix},
\end{equation*}
and a stable term, which is simply the diagonal advection-diffusion equation with velocity field $m\Omega(r)+kU(r)$. Since the advection-diffusion component is formally negative, any growth that occurs must be a result of the unstable term. In addition, since the unstable term does not involve any derivatives in $r$, we can diagonalize the equation. A computation yields that for $r\gtrsim 1$, the associated diagonal system is well described by the matrix 
\begin{equation*}
\partial_t \begin{pmatrix}
b_r\\
b_\theta
\end{pmatrix}=\begin{pmatrix}
-\eps m^2 +i\sqrt{-\eps m(\eps m-i)} & 0\\
0 & -\eps m^2 -i\sqrt{-\eps m(\eps m-i)}
\end{pmatrix}\begin{pmatrix}
b_r\\
b_\theta
\end{pmatrix}.
\end{equation*}
In other words, the growth rate resulting from the stretching term is of order 
\begin{equation}\label{eq:growth-rate-heuristics}
-\eps m^2 +\sqrt{\eps m} \sim \sqrt{m\eps} \quad \text{as } \eps \to 0.
\end{equation}
It thus remains to show that this growth rate formally dominates the decay from the advection diffusion equation. From classical results in enhanced dissipation, see e.g.\ \cites{BedrossianCZ, AlbrittonRajNovack, DongyiWei, DVEnhancedDissipation, Gardner_Liss_Mattingly_2024}, we recall that so long as $m\Omega'(r_0)+kU'(r_0) \neq 0$ $m, k =O_{\eps \to 0}(1)$, the equation
\begin{equation*}
\partial_r b_r=\eps\partial_r^2 b_r+i(m\Omega(r)+kU(r))b_r
\end{equation*}
experiences \emph{enhanced dissipation} of order $\eps^{1/3}$ for functions localised near $r=r_0$. Hence, near such a point, we can expect a competition between two mechanisms in the dynamo equations: a local decay of order $\eps^{1/3}$ from the advection-diffusion equations, and growth of order $\eps^{1/2}$ from the unstable terms. Since $\eps^{1/3} \gg \eps^{1/2}$ for $\eps $ small, the decay will be the dominating force near $r_0$ so long as $m\Omega'(r_0)+kU'(r_0) \neq 0$, and thus no dynamo action can take place. Hence, it will be crucial to designate a so called ``critical radius'' $r_0>0$, determined by the Fourier modes $m,k$, at which the azimuthal and vertical shears are aligned in such a way that $m\Omega'(r_0)+kU'(r_0)=0$. In this case, the heuristic enhanced dissipation argument yields local decay of order $\eps^{1/2}$, which is precisely the same order as the stretching from the unstable term. In fact, recall from \eqref{eq:growth-rate-heuristics} that the heuristic growth rate from the stretching term is $-\eps m^2+\sqrt{\eps m}$, which is maximized at order $\eps^{1/3}$ when $m=O(\eps^{-1/3})$. Therefore, the same argument yields that, so long as $m,k=O(\eps^{-1/3})$, the stretching and diffusion terms are balanced at order $\eps^{1/3}$ as $\eps \to 0$, and therefore there is potential for exponential growth at rate $\eps^{1/3}$ to occur near such a critical radius $r_0$. The precise relative signs of the stretching and diffusion terms will then be determined by the local geometry of the flow, for which a more subtle analysis is needed.

\subsection{Gilbert's asymptotic approach}\label{s:gilbert-scaling}
Having provided a heuristic towards the existence of a critical radius $r_0$ where growth takes place, as well as towards the fact that the optimal growth rate of order $\eps^{1/3}$ is achieved when $m,k \sim \eps^{-1/3}$, we now recall the main ideas from Gilbert’s asymptotic analysis \cite{gilbert1988}, that give conditions for the stretching term to be locally stronger than the diffusion term. These ideas, suitably adapted, will guide the construction of the desired growing mode for the full problem. In particular, they yield a convenient change of variables, which we examine in detail in Section \ref{s:equations}.

From the exposition in the previous section, we assume that the growing mode will asymptotically concentrate as $\eps\to 0$ around a critical radius $r_0>0$, that will be formally treated as a \emph{boundary layer}. We take the Fourier transform $\theta \mapsto m=\eps^{-1/3}M$, $z \mapsto k=\eps^{-1/3}K$, where $M,K =O_{\eps\to 0}(1)$, and we further make the modal Ansatz $B(t,r,\theta,z)=\e^{\eps^{1/3}\mu t}\e^{i\eps^{-1/3}(M\theta+Kz)}$. Thus, the $(r,\theta)$ components of the kinematic dynamo equations become
\begin{align}
\begin{split}\label{eq:gilbert expansion1}
    \eps^{1/3}\mu b_r & + i\eps^{-1/3}(M\Omega(r) +KU(r))b_r \\
    &=  \eps\left( \partial_r^2b_r + \frac{1}{r}\partial_rb_r - \eps^{2/3}\frac{M^2}{r^2}b_r - \eps^{2/3}K^2b_r -\frac{1}{r^2}b_r - \eps^{1/3}\frac{2iM}{r^2}b_\theta\right),
\end{split}\\
\begin{split}\label{eq:gilbert expansion2}
        \eps^{1/3}\mu b_\theta & + i\eps^{-1/3}(M\Omega(r) +KU(r))b_\theta \\
        & = \eps\left( \partial_r^2b_\theta + \frac{1}{r}\partial_rb_\theta - \eps^{2/3}\frac{M^2}{r^2}b_\theta - \eps^{2/3}K^2b_\theta -\frac{1}{r^2}b_\theta + \eps^{1/3}\frac{2iM}{r^2}b_r\right)+ r\Omega'(r)b_r.
    \end{split}
\end{align}
Since the dynamo growth rate is entirely determined by $\Re(\lambda) = \Re(\eps^{1/3}\mu)$, we may replace 
\[
\lambda \mapsto \lambda-\eps^{-1/3}i(M\Omega(r_0)+KU(r_0))
\]
without altering the growth rate. Now, in order for growth to occur near $r_0$, by Section \ref{s:heuristics2} we need to impose that the transport term has vanishing derivative at $r_0$. Therefore, we assume that
\begin{equation}\label{eq:Gilbert2}
    M\Omega'(r_0) + KU'(r_0) = 0.
\end{equation}
We now Taylor expand the equation around $r_0$, making the change of variables $s=\eps^{-1/3}(r-r_0)$. Making the perturbation series Ansatz 
\begin{equation*}
b_r=\sum_{n \geq 0}\eps^{\frac{n}{3}}b_r^{(n)}, \quad \sum_{ n \geq 0}\eps^{\frac{n}{3}}b_\theta^{(n)},
\end{equation*}
it follows that the only term of order $1$ in \eqref{eq:gilbert expansion1}, \eqref{eq:gilbert expansion2} is the term from the $\theta$ equation given by 
\begin{equation*}
r_0\Omega'(r_0)b_r^{(0)}=0,
\end{equation*}
from which we deduce $b_r^{(0)}=0$. Next, the $O(\eps^{1/3})$ terms are
\begin{equation*}
(\partial_s^2 -M^2r_0^{-2}-K^2)b_\theta^{(0)}+r_0\Omega'(r_0)b_r^{(1)}=\lambda b_\theta^{(0)}+ic_2 s^2b_\theta^{(0)},
\end{equation*}
where 
\begin{equation}\label{eq:P2}
c_2=\frac{i}{2}\left(M\Omega''(r_0) + KU''(r_0)\right).    
\end{equation}
We also find that the $O(\eps^{2/3})$ terms from \eqref{eq:gilbert expansion1} are given by 
\begin{equation*}
\partial_s^2b_r^{(1)} -M^2 r_0^{-2}b_r^{(1)}-K^2b_r^{(1)}-2iMr_0^{-2}b_\theta^{(0)}=\lambda b_r^{(1)} +ic_2 s^2 b_r^{(1)}.
\end{equation*}
We therefore arrive at the closed system of ODEs 
\begin{align*}
\partial_s^2b_r^{(1)} -M^2 r_0^{-2}b_r^{(1)}-K^2b_r^{(1)}-2iMr_0^{-2}b_\theta^{(0)} & =\lambda b_r^{(1)} +ic_2 s^2 b_r^{(1)}\\
(\partial_s^2 -M^2r_0^{-2}-K^2)b_\theta^{(0)}+r_0\Omega'(r_0)b_r^{(1)} & =\lambda b_\theta^{(0)}+ic_2 s^2b_\theta^{(0)}.
\end{align*}
Now, we make the change of variables 
\begin{equation}\label{eq:change-variables-go}
b_r = \alpha\eps^{1/3}(V_2-V_1), \quad b_\theta = V_1 + V_2,
\end{equation}
with 
\begin{equation}\label{eq:alpha}
\alpha^2 = -\frac{2iM}{r_0^3\Omega'(r_0)}.
\end{equation}
Furthermore, we set $\lambda=\eps^{1/3}\mu$, with $\mu=O(1)$ as $\eps \to 0$.
We thus see that $V_1,V_2$ satisfy the ordinary differential equations
\begin{equation}
\label{eq:Gilbert equation}
\begin{split}
\varepsilon\partial_r^2V_1 - c_2\varepsilon^{-1/3}(r-r_0)^2V_1 - \varepsilon^{1/3}\left(  M^2\left[ \frac{1}{r_0^2} + \left(\frac{\Omega'(r_0)}{U'(r_0)}\right)^2 \right] +\mu -\sqrt{\frac{-2iM\Omega'(r_0)}{r_0}} \right)V_1 = 0,\\
\varepsilon\partial_r^2V_2 - c_2\varepsilon^{-1/3}(r-r_0)^2V_2 - \varepsilon^{1/3}\left(  M^2\left[ \frac{1}{r_0^2} + \left(\frac{\Omega'(r_0)}{U'(r_0)}\right)^2 \right] +\mu +\sqrt{\frac{-2iM\Omega'(r_0)}{r_0}} \right)V_2 =0.
\end{split}
\end{equation}
We make the change of variables $z=\zeta s$, where recall $s=\eps^{-1/3}(r-r_0)$, with $\zeta =\sqrt{2}c_2^{1/4}$ so that
\begin{equation}\label{eq:gilbert parabolic cylinder}
V_1''(z) + \left(-\frac{1}{2}c_2^{-1/2}q(\mu) -\frac{1}{4}z^2\right)V_1(z) = 0,
\end{equation}
and where we set 
\[
q(\mu) = M^2\left[ \frac{1}{r_0^2} + \left(\frac{\Omega'(r_0)}{U'(r_0)}\right)^2 \right] +\mu - \sqrt{\frac{-2iM\Omega'(r_0)}{r_0}}.
\]
Analogously, we define $q(\mu)$ with a $+$ sign in front of the last addend in the equation for $V_2$, and we choose the root that satisfies $\Re(\sqrt{-iM\Omega'(r_0)})>0$. In this way we see that the corresponding eigenvalue problem \eqref{eq:gilbert parabolic cylinder} for the second component produces an eigenvalue with a real part strictly larger than for the first component, and therefore we can focus on the analysis of the first component. 

Solutions to equations of the form \eqref{eq:gilbert parabolic cylinder} are called \emph{parabolic cylinder functions}---see Appendix \ref{s:appendix-parabolic} for more details, and they have finite energy on $(-\infty, \infty)$ if and only if 
\begin{equation}\label{eq:gilbert condition parabolic}
\frac{1}{2}c_2^{-1/2}q(\mu)=-j-\frac{1}{2}, \quad j=0,1,2,\dots.
\end{equation} 
Therefore, solving for $\mu$ so that \eqref{eq:gilbert condition parabolic} holds, we see indeed that
\[
\Re(\mu) = \sqrt{\frac{|M||\Omega'(r_0)|}{r_0}} - M^2\left(\frac{1}{r_0^2}+\frac{|\Omega'(r_0)|^2}{|U'(r_0)|^2}\right) - \sqrt{2|c_2|}\left(j + \frac{1}{2}\right).
\]
To maximize the growth rate, we pick $j=0$. Moreover, using \eqref{eq:P2} and making the appropriate choices of complex roots we obtain,
\[
\Re(\mu) = |M|^{1/2}\left(\frac{|\Omega'(r_0)|^{1/2}}{r_0^{1/2}} - \frac{1}{2}\left| \Omega''(r_0)-\frac{\Omega'(r_0)}{U'(r_0)}U''(r_0) \right|^{1/2}    \right) - |M|^2\left(\frac{1}{r_0^2}+\frac{|\Omega'(r_0)|^2}{|U'(r_0)|^2}\right).
\]
Hence via \eqref{eq:Gilbert2} and rewriting $\lambda = \mu\eps^{1/3}$, $m = M\eps^{-1/3}$, $k=K\eps^{-1/3}$, which is precisely the growth rate of our Theorem \ref{thm:dynamo}. Moreover, this shows that upon letting the absolute value of $M\in\R$ be sufficiently small, we find that $\Re(\mu)>0$ provided that
\begin{equation}\label{eq:gilbert3}
r_0\left |\frac{\dd}{\dd r} \log \left |\frac{\Omega'(r_0)}{U'(r_0)}\right |\right |<4.
\end{equation}
Hence, we find that locally near $r_0$, the stretching term is stronger than the enhanced dissipation, precisely if \eqref{eq:gilbert3} is satisfied.
\begin{remark}
From our presentation, it may seem that the condition \eqref{eq:Gilbert2} determines the Fourier modes $m,k$ \emph{given} the critical radius $r_0$, making the value $r_0$ appear as a free parameter. However, from the point of view of the full dynamo PDEs, it is more apt to view the condition \eqref{eq:Gilbert2} as the Fourier modes \emph{determining} the critical radius $r_0$---In other words, upon picking Fourier modes $m,k$, the condition \eqref{eq:Gilbert2} determines the localisation of the growing mode for the Ponomarenko dynamo equations.
\end{remark}

\subsection{Reformulation of the kinematic dynamo equations}\label{s:equations}

The main goal of this paper is to bridge the gap from this formal analysis of to a fully rigorous description of the growing modes of the Ponomarenko dynamo. Whilst the analysis of \cite{gilbert1988} is very precise in a small neighbourhood of $r_0$, for the existence of a genuine eigenfunction of the kinematic dynamo equations, we need to extend the analysis to a global one. To do so, we shall begin by bringing the equations \eqref{eq:modal-eq-r}, \eqref{eq:modal-eq-theta} into a form more ammenable for analysis.

We fix a positive radius $r_0>0$, and write $\Omega_0=\Omega(r_0)$ and $U_0 = U(r_0)$. Furthermore, we write $\Omega_0'$, $\Omega_0''$ to denote $\Omega'(r_0)$, $\Omega''(r_0)$, and analogously with $U$. Following Gilbert's analysis, we make the following Ansatz:
\begin{itemize}
    \item $\lambda = \mu \eps^{1/3}$, where $\mu\in\Co$, and $\eps>0$ denotes the magnetic diffusivity from \eqref{eq:dynamo}.
    \item $m = M\eps^{-1/3}$, and $k=K\eps^{-1/3}$, with $M,K\in \R$. Moreover, we choose the frequencies such that Gilbert's formula \eqref{eq:Gilbert2} is satisfied,
    \[
    M\Omega_0' + KU_0' = 0.
    \]
\end{itemize}
Under this choice of parameters $\lambda$, $m$ and $k$, we can rewrite the equations corresponding to the radial and azimuthal components of the magnetic field \eqref{eq:modal-eq-r}--\eqref{eq:modal-eq-theta}, which take the form,
\begin{align*}
\mu\varepsilon^{1/3}b_r & + i\varepsilon^{-1/3}M\left( \Omega-\frac{\Omega_0'}{U_0'}U\right)b_r \\
& = \varepsilon\left( \partial_r^2 b_r +\frac{1}{r}\partial_rb_r \right) - \varepsilon^{1/3}M^2\left( \frac{1}{r^2}+\left( \frac{\Omega_0'}{U_0'} \right)^2 \right)b_r -\frac{\varepsilon}{r^2}b_r -\frac{2i\varepsilon^{2/3} M}{r^2}b_\theta, \\
\mu\varepsilon^{1/3}b_\theta & + i\varepsilon^{-1/3}M\left( \Omega-\frac{\Omega_0'}{U_0'}U\right)b_\theta \\
& = \varepsilon\left( \partial_r^2 b_\theta +\frac{1}{r}\partial_rb_\theta \right) - \varepsilon^{1/3}M^2\left( \frac{1}{r^2}+\left( \frac{\Omega_0'}{U_0'} \right)^2 \right)b_\theta -\frac{\varepsilon}{r^2}b_\theta +\frac{2i\varepsilon^{2/3} M}{r^2}b_r + r\Omega'b_r.
\end{align*}
As mentioned in the previous sections, the $z$ component is decoupled from the $r$ and $\theta$ components. The bulk of our efforts will go towards finding a solution to this coupled system of ODEs, where $\mu\in\Co$ satisfies $\Re(\mu)>0$. After we solve this problem, in Section \ref{s:z-component} we will address the behaviour of the vertical component.

We want to write the equation in a more convenient form, following the ideas exposed in Section \ref{s:gilbert-scaling}. In order to do so, we subtract
\[
i\eps^{-1/3}M\left( \Omega_0-\frac{\Omega_0'}{U_0'}U_0\right)b_r \quad \text{and} \quad i\eps^{-1/3}M\left( \Omega_0-\frac{\Omega_0'}{U_0'}U_0\right)b_\theta
\]
from the left hand side in the first and second equations respectively. In this manner, we only modify the imaginary part of the eigenvalue $\mu\eps^{1/3}$, leaving the growth rate invariant. Taking this into account, we find the following system of equations,
\begin{equation}\label{eq:radial}
\begin{split}
\mu\varepsilon^{1/3}b_r & + i\varepsilon^{-1/3}M\left( (\Omega-\Omega_0) -\frac{\Omega_0'}{U_0'}(U-U_0) \right)b_r \\
& = \varepsilon\left( \partial_r^2 b_r +\frac{1}{r}\partial_rb_r \right) - \varepsilon^{1/3}M^2\left( \frac{1}{r^2}+\left( \frac{\Omega_0'}{U_0'} \right)^2 \right)b_r -\frac{\varepsilon}{r^2}b_r -\frac{2i\varepsilon^{2/3} M}{r^2}b_\theta,
\end{split}
\end{equation}
\begin{equation}\label{eq:azimutal}
\begin{split}
\mu\varepsilon^{1/3}b_\theta & + i\varepsilon^{-1/3}M\left( (\Omega-\Omega_0) -\frac{\Omega_0'}{U_0'}(U-U_0) \right)b_\theta \\
& = \varepsilon\left( \partial_r^2 b_\theta +\frac{1}{r}\partial_rb_\theta \right) - \varepsilon^{1/3}M^2\left( \frac{1}{r^2}+\left( \frac{\Omega_0'}{U_0'} \right)^2 \right)b_\theta -\frac{\varepsilon}{r^2}b_\theta +\frac{2i\varepsilon^{2/3} M}{r^2}b_r + r\Omega'b_r.
\end{split}
\end{equation}
The addition of these purely imaginary terms yields a ``transport term'' of the form
\[
T(r) = \left((\Omega(r)-\Omega_0) - \frac{\Omega_0'}{U_0'}(U(r)-U_0)\right),
\]
that crucially vanishes up to second order as $r\to r_0$. In order to make the notation cleaner, we point out that equations \eqref{eq:radial}--\eqref{eq:azimutal} can be rewritten in contracted form as
\begin{align*}
    \eps\left( \partial_r^2b_r + \frac{1}{r}\partial_rb_r\right) & = \lambda b_r + A_\eps b_r + B_\eps b_\theta, \\
    \eps\left( \partial_r^2b_\theta + \frac{1}{r}\partial_rb_\theta\right) & = \lambda b_\theta + A_\eps b_\theta + C_\eps b_r,
\end{align*}
where $\lambda = \mu\eps^{1/3}$, and with
\[
A_\eps = \frac{\eps}{r^2} + \eps^{1/3}M^2\left( \frac{1}{r^2}+\left( \frac{\Omega_0'}{U_0'} \right)^2 \right) + i\eps^{-1/3}MT(r),
\]
\[
B_\eps = \frac{2iM\eps^{2/3}}{r^2}, \quad C_\eps = -r\Omega' - \frac{2iM\eps^{2/3}}{r^2}.
\]
Before analysing the structure of solutions to this system of ODEs, we perform one final change of variables, following the ideas outlined in Section \ref{s:gilbert-scaling}. This transformation---originally introduced implicitly in Gilbert’s paper \cite{gilbert1988}---has the advantage of yielding a symmetric scaling in $\eps$ across both components of the equations. As in \eqref{eq:change-variables-go} we set $(b_r,b_\theta)\mapsto(V_1,V_2)$, defined by
\begin{equation}
b_r = \alpha\eps^{1/3}(V_2-V_1), \quad b_\theta = V_1 + V_2,
\end{equation}
or equivalently
\[
V_1 = \frac{1}{2}\left(b_\theta - \frac{1}{\alpha} \eps^{-1/3}b_r\right), \quad  V_2 = \frac{1}{2}\left(b_\theta + \frac{1}{\alpha} \eps^{-1/3}b_r\right),
\]
where $\alpha\in\Co$ is given by \eqref{eq:alpha}. After the change of variables, the equations that $V_1$ and $V_2$ satisfy have the form
\begin{equation*}
\begin{split}
    \eps\left( \partial_r^2V_1 + \frac{1}{r}\partial_rV_1\right) & = \lambda V_1 + A_\eps V_1 - \frac{1}{2}\left(\frac{1}{\alpha}\eps^{-1/3}B_\eps + \alpha\eps^{1/3}C_\eps\right)V_1 - \frac{1}{2}\left(\frac{1}{\alpha}\eps^{-1/3}B_\eps - \alpha\eps^{1/3}C_\eps\right)V_2,  \\
    \eps\left( \partial_r^2V_2 + \frac{1}{r}\partial_rV_2\right) & = \lambda V_2 + A_\eps V_2 + \frac{1}{2}\left(\frac{1}{\alpha}\eps^{-1/3}B_\eps + \alpha\eps^{1/3}C_\eps\right)V_2 + \frac{1}{2}\left(\frac{1}{\alpha}\eps^{-1/3}B_\eps - \alpha\eps^{1/3}C_\eps\right)V_1. 
\end{split}
\end{equation*}
One can see from here that the cross terms have the same scaling in $\eps$ in both components. The above ODE defines a linear operator, which we denote (up to the inclusion of the term $\mu\eps^{1/3}$) by $\L$. In other words, the above ODEs may be written as
\[
(\L-\mu\varepsilon^{1/3})\begin{pmatrix}
    V_1\\
    V_2
\end{pmatrix} = 0.
\]
Analysing the precise form of the factors in $\L$ more carefully, we find a useful structure in the cross factors. Notice that we can compute,
\[
\begin{split}
\frac{1}{\alpha}\eps^{-1/3}B_\eps - \alpha\eps^{1/3}C_\eps & = \frac{2iM\varepsilon^{1/3}}{r^2}\left(\frac{1}{\alpha} + \alpha\varepsilon^{2/3}\right) + \alpha\varepsilon^{1/3}r\Omega' \\
& = \frac{2iM\alpha\varepsilon}{r^2} + \frac{2iM\varepsilon^{1/3}}{\alpha}\left(\frac{1}{r^2}-\frac{1}{r_0^2}\right) + \frac{2iM\varepsilon^{1/3}}{\alpha r_0^2} \\
& \qquad + \alpha\varepsilon^{1/3}(r\Omega'-r_0\Omega_0') + \alpha\varepsilon^{1/3}r_0\Omega_0'.
\end{split}
\]
Hence, making the choice of $\alpha$ as in \eqref{eq:alpha}, we find that
\[
\frac{1}{2}\left(\frac{1}{\alpha}\eps^{-1/3}B_\eps - \alpha\eps^{1/3}C_\eps\right) = \frac{iM\alpha\varepsilon}{r^2} + \frac{iM\varepsilon^{1/3}}{\alpha}\left(\frac{1}{r^2}-\frac{1}{r_0^2}\right) + \frac{\alpha\varepsilon^{1/3}}{2}(r\Omega'-r_0\Omega_0'),
\]
and therefore, we observe that the cross terms will vanish for $r\to r_0$ and $\eps\to 0$. This will be precisely the key idea at the core of our approach to the problem: analysing the behaviour of $\L$ for values very close to $r_0$, where, according to Gilbert's idea \cite{gilbert1988}, the exponential growth is expected to occur. Following a similar argument we find that this choice of $\alpha\in\Co$ yields for the other factor
\[
\frac{1}{2}\left(\frac{1}{\alpha}\eps^{-1/3}B_\eps + \alpha\eps^{1/3}C_\eps\right) = -\frac{iM\alpha\varepsilon}{r^2} + \frac{2iM\varepsilon^{1/3}}{\alpha r_0^2} + \frac{iM\varepsilon^{1/3}}{\alpha}\left(\frac{1}{r^2}-\frac{1}{r_0^2}\right) - \frac{\alpha\varepsilon^{1/3}}{2}(r\Omega'-r_0\Omega_0').
\]

We can now combine all this information and rewrite \eqref{eq:radial}--\eqref{eq:azimutal} in the desired form for the analysis we will carry out. Before doing so, however, let us introduce some more shorthand notation for a more condensed presentation of the equations. We define
\[
D_\eps^+ = \frac{iM\varepsilon^{1/3}}{\alpha}\left(\frac{1}{r^2}-\frac{1}{r_0^2}\right) + \frac{\alpha\varepsilon^{1/3}}{2}(r\Omega'-r_0\Omega_0'),
\]
\[
D_\eps^- = \frac{iM\varepsilon^{1/3}}{\alpha}\left(\frac{1}{r^2}-\frac{1}{r_0^2}\right) - \frac{\alpha\varepsilon^{1/3}}{2}(r\Omega'-r_0\Omega_0'),
\]
and we take without loss of generality the positive sign in the square root when defining $\alpha$ in \eqref{eq:alpha}. Thus, $\L-\mu\eps^{1/3}$ will be the operator associated to the equations
\begin{equation}\label{eq:1}
\eps\left( \partial_r^2V_1 + \frac{1}{r}\partial_rV_1\right) = \left(\lambda + A_\eps +\frac{iM\alpha\varepsilon}{r^2} - \sqrt{\frac{-2iM\Omega_0'}{r_0}}\varepsilon^{1/3} - D_\varepsilon^-\right)V_1 -\left(\frac{iM\alpha\varepsilon}{r^2} + D_\eps^+\right)V_2,
\end{equation}
\begin{equation}\label{eq:2}
\eps\left( \partial_r^2V_2 + \frac{1}{r}\partial_rV_2\right) = \left(\lambda + A_\eps -\frac{iM\alpha\varepsilon}{r^2} + \sqrt{\frac{-2iM\Omega_0'}{r_0}}\varepsilon^{1/3} + D_\varepsilon^-\right)V_2 +\left(\frac{iM\alpha\varepsilon}{r^2} + D_\eps^+\right)V_1.
\end{equation}
From here on, unless otherwise specified, any vector will be assumed to be written in the basis induced by this change of variables,
\begin{equation}\label{eq:change-variable}
V_r\hat r+V_\theta\hat\theta+V_z\hat z \mapsto \frac{1}{2}\left(V_\theta-\frac{1}{\alpha}\eps^{-1/3}V_r\right)\hat v_1 + \frac{1}{2}\left(V_\theta+\frac{1}{\alpha}\eps^{-1/3}V_r\right)\hat v_2 +  V_z \hat v_3.
\end{equation}
which is well-defined for any $\eps>0$.
\begin{remark}
From the change of variables, we obtain further evidence of the heuristic from Section \ref{s:heuristics2} that the growth rate of order $\eps^{1/3}$ in our main result Theorem \ref{thm:dynamo} is sharp. Indeed, the only ``unstable'' terms (i.e. those that are non-negative when undertaking energy estimates) are of order $O(\eps^{1/3})$, aligning precisely with our result.
\end{remark}

\subsection{Admissible vector fields and examples}\label{s:admissible-fields}
Thus far we have avoided giving an explicit description of the necessary conditions needed for Theorem \ref{thm:dynamo} to hold. In this section we shall outline them, and in particular we will eventually endeavour to provide simple, checkable conditions which imply our hypotheses, as well as some explicit examples of particular interest. For any fixed $r_0>0$, we begin by defining the function
\begin{equation}\label{eq:transport-function}
T(r)=\Omega(r)-\Omega(r_0)+\frac{\Omega'(r_0)}{U'(r_0)}(U(r)-U(r_0)),
\end{equation}
which we will henceforth refer to as the \emph{transport function}. Without further ado, we now list the general hypotheses under which we can deduce the existence of a growing mode for the Ponomarenko dynamo on the domain $\mathcal{I} \subseteq [0,\infty)$, with perfectly conducting boundary conditions.

\emph{General Hypotheses}: Fix a set $\mathcal{J} \subseteq [0,\infty)$ that is open in the topology on $[0,\infty)$, and so that furthermore $\overline{\mathcal{I}} \subseteq \mathcal{J}$. We assume the following.
\begin{itemize}
    \item[\namedlabel{H0}{H0}] The velocity field is at least $\Omega, U\in C^3(\overline{\mathcal{J}})$.
    
    \item[\namedlabel{H1}{H1}] The transport function vanishes up to second order only at $r_0>0$, namely 
    \[
    T(r_0)=T'(r_0)=0, \quad T''(r_0)\neq 0.
    \]
    There exists a (possibly empty) finite set of points 
    \[
    \mathcal{F}_0=\{s_1,\hdots,s_q\}\subset \mathcal{J} \setminus \{0\},
    \]
    where the transport function satisfies $T(s) = 0$, $T'(s)\neq 0$, for all $s\in\mathcal{F}_0$. Moreover, $T(r)\neq 0$ for all $r\in \mathcal{J}\setminus (r_0\cup\mathcal{F}_0)$.
    
    \item[\namedlabel{H2}{H2}] 
    There exist numbers $N_3,N_4\in\N$, a constant $\Xi >0$, and a compact set $\mathcal{R}_0\subset\mathcal{J}$ so that uniformly for $r_0 \in \mathcal{R}_0$ it holds for all $r,r' \geq \Xi$
\begin{equation*}
\frac{|T(r)-T(r')|}{|T(r')|}\leq C \sum_{n=1}^{N_3}|r-r'|^n, \quad
\frac{|r\Omega'(r)|}{|T(r')|}\leq C \sum_{n=1}^{N_4}|r-r'|^n.
\end{equation*}
Furthermore, we assume that $\inf_{r_0 \in \mathcal{R}_0}\inf_{r \geq \Xi}|T(r)|>0$.
\end{itemize}
Notice that, in case that $0 \in \overline{\mathcal{J}}$, Assumptions \ref{H0} and \ref{H1} together imply that there exists $\tau\in\R\setminus\{0\}$ such that
\[
\lim_{r\to 0}T(r) = \tau.
\]

These \emph{General Hypotheses} \ref{H0}--\ref{H2} are required for our construction of the growing mode to hold true. From the proof of Theorem \ref{theorem:mainResult} one can appreciate that any vector field satisfying \ref{H0}--\ref{H2} can be proved to be a slow dynamo with growth rate $\Re(\lambda) = \eps^{1/3}\Re(\mu)>0$. If we desire to obtain further information on the eigenfunction, i.e.\ an asymptotic expansion as stated in Theorem \ref{thm:dynamo}, we need to impose the following additional assumption on the velocity field. Indeed, provided that $\mathcal{I}$ is an unbounded domain, we have the following further assumption.

\begin{itemize}
    \item[\namedlabel{H3}{H3}] The velocity field satisfies either
    \begin{itemize}
        \item $|r\Omega'(r)|$ is bounded and $|r\Omega'(r)|\to 0$ as $r\to\infty$;
    \end{itemize}
    or
    \begin{itemize}
        \item $T(r)$ is bounded from above or below, $|T(r)|\to\infty$ as $r\to\infty$ and there exists $C>0$, $\Lambda \in \mathbb{R}$ such that
        \[
        \sup_{r\in \overline{\mathcal{I}}} \frac{|r\Omega'(r)|^2}{|T(r)+\Lambda|} \leq C.
        \]
    \end{itemize}
\end{itemize}
As mentioned earlier, Assumption \ref{H3} should be thought of as a collection of checkable conditions that imply that the kinematic dynamo equations may be treated as a relatively compact perturbation of the advection diffusion equation, thus implying that all growing modes must be isolated eigenvalues. This intuition is made precise in Section \ref{s:injectivity}.

\begin{remark}
Upon introducing the modal form Ansatz $B(t,r,\theta,z)=b(r)\e^{\lambda t + i(m\theta+kz)}$, the kinematic dynamo equations are transformed into a system of ODEs, where the physical meaning of the coefficients $m,k$ is no longer of direct importance. Indeed, it is entirely possible to find a growing mode for the system of ODEs \eqref{eq:modal-eq-r}--\eqref{eq:modal-eq-z} for $m, k \notin \mathbb{Z}$, which, whilst of independent interest from the point of view of the ODE, is rather irrelevant in the context of the kinematic dynamo PDE \eqref{eq:dynamo}. Recalling further that according to Gilbert's scaling, we pick Fourier modes 
\begin{equation}\label{eq:m&k}
m=M\eps^{-1/3}, \quad k=-m\frac{\Omega'(r_0)}{U'(r_0)},
\end{equation}
as $\eps \to 0$, these will generically not be integers. Thus, simply fixing a single $r_0 >0$, $M \in \mathbb{Z}$ and showing that the associated system of ODEs has a growing mode for all $\eps>0$ small enough is not enough to deduce the slow dynamo claim: our estimates need to be robust enough to allow for some ``wiggle room'' in both $M$ and $r_0$. Indeed, for $\eps>0$ small enough, one can always perturb $M$ by a term of order $\eps^{1/3}$ to ensure that $M\eps^{-1/3}$ is an integer. The issue then becomes the coupling between $m$ and $k$ via \eqref{eq:m&k}. Here, we once again make use of the fact that $M\eps^{-1/3}$ is very large for $\eps$ small enough, and thus it is generically not too difficult to find \emph{some} $r_0>0$ so that 
\[
-\eps^{-1/3}M\frac{\Omega'(r_0)}{U'(r_0)} \in\Z.
\]
Thus, for any $\eps>0$, to make sure that the ODE result is meaningful in the context of the PDE \eqref{eq:dynamo}, we first make a perturbation of order $\eps^{1/3}$ to ensure that $M\eps^{-1/3}$ is an integer, and then make a slight perturbation to $r_0$ so that $k$ as defined in \eqref{eq:m&k} is also an integer. The point of the compact set $\mathcal{R}_0$ in \ref{H2} is then to guarantee that estimates used in the construction of growing modes for \eqref{eq:modal-eq-r}--\eqref{eq:modal-eq-z} holds \emph{uniformly} for the perturbations needed on $M$, $r_0$ as $\eps\to 0$, so that we indeed find a growing mode for the PDE \eqref{eq:dynamo} \emph{for all $\eps \to 0$}. That being said, the uniformity in $r_0$ is not entirely needed for all applications---for instance, if we instead let the variable $z$ live in $\mathbb{R}$, then the assumption $k \in \mathbb{Z}$ can be dispensed with. In keeping with this, assumption \ref{H2} could in principle be satisfied for the compact set $\mathcal{R}_0=\{r_0\}$. In Theorems \ref{thm:dynamo}, \ref{thm:full dynamo} we thus differentiate between this case and the case where $\mathcal{R}_0$ contains a non-empty open interval.
\end{remark}

Having outlined our assumptions, we now endeavour to ``translate'' them into the main properties of the transport function $T(r)$ needed for our proofs. This is the subject of the following lemma.

\begin{lemma}[Consequences of \ref{H0}--\ref{H2}]
\label{lemma:consequences of assumptions}
Assume \ref{H0}--\ref{H2} hold true, and pick any $\gamma , \delta >0$ so that 
\[
\frac{2}{9} < \gamma < \frac{1}{3}, \quad \gamma<\delta, \quad \gamma+\delta<\frac{2}{3}.
\]
Then, there exist a compact interval $\mathcal{R} \subset \mathcal{I}$ with $r_0\in\mathcal{R}$, a finite set $\mathcal{F}_0(r_0)$ depending on $r_0 \in \mathcal{R}$ of constant cardinality, and a constant $C>0$ independent of $r_0 \in \mathcal{R}$, so that for any $\eps>0$ small enough the following holds:
\begin{itemize}
    \item[\namedlabel{P1}{P1}]  $T(s_j)=0$ for all $s_j \in \mathcal{F}_0(r_0)$, and $|T'(s_j)|$ is bounded below uniformly in $r_0 \in \mathcal{R}$, $s_j \in \mathcal{F}_0(r_0)$. Also, it holds 
    \begin{equation*}
    \sup_{r_0 \in \mathcal{R}} \sup_{s_j \in \mathcal{F}_0(r_0)}\sup_{|r-s_j|\leq 1}\frac{|T(r)-T'(s_j)(r-s_j)|}{|r-s_j|^2}<\infty.
    \end{equation*}
    \item[\namedlabel{P2}{P2}] $|T''(r_0)|$ is bounded below uniformly in $r_0 \in \mathcal{R}$, and there holds 
    \begin{equation*}
    \sup_{r_0 \in \mathcal{R}}\sup_{|r-r_0|<1}\frac{|T(r)-\frac{1}{2}T''(r_0)(r-r_0)^2|}{|r-r_0|^3}<\infty.
    \end{equation*}
    \item[\namedlabel{P3}{P3}] There exist a positive number $N_1\in\N$ and a collection of possibly $\eps-$dependent coefficients $B_n(\eps)\geq 0$, $1\leq n\leq N_1$, such that for any $\eps>0$ small, and any two points $r, r' \in \mathcal{J}$, $r'\geq \eps^\gamma$ with $|r'-R|\geq \eps^\gamma$ for all $R\in\{r_0\} \cup \mathcal{F}_0$, there holds that
    \[
    \frac{|T(r)-T(r')|}{|T(r')|} \leq \sum_{n=1}^{N_1} B_n(\eps)|r-r'|^n.
    \]
    Additionally, there exists $\beta>0$ such that the coefficients $B_n(\eps)\geq 0$ satisfy
    \[
    \sum_{n=1}^{N_1} B_n(\eps)\eps^{n\delta} \leq C\eps^\beta,
    \]
    for some $\beta$ independent of $r_0 \in \mathcal{R}$.
    \item[\namedlabel{P4}{P4}] There exist a number $N_2\in\N$ and a collection of coefficients $D_0,D_n\geq 0$, with $1\leq n\leq N_2$, such that for any $\eps>0$ small, and any two points $r, r' \in \mathcal{J}$, $r'\geq \eps^\gamma$ with $|r'-R|\geq \eps^\gamma$ for all $R \in \{r_0\} \cup \mathcal{F}_0(r_0)$, there holds that
    \[
    \frac{|r\Omega'(r)|}{|T(r')|}\leq \eps^{-2\gamma}\left( D_0 + \sum_{n=1}^{N_2} D_n|r-r'|^n\right).
    \]
    \item[\namedlabel{P5}{P5}] If $0 \in \overline{\mathcal{I}}$, then $|\lim_{r \to 0}T(r)|$ is bounded below uniformly in $r_0 \in \mathcal{R}$, and 
    $$\sup_{r_0 \in \mathcal{R}}\sup_{r\leq 1}\frac{|T(r)-T(0)|}{r}<\infty.
    $$
    \item[\namedlabel{P6}{P6}]
    Let $s_1,s_2 \in \mathcal{F}_0(r_0) \cup\{r_0\}$ be neighbouring points, i.e.\ $(s_1,s_2) \cap (\mathcal{F}_0(r_0) \cup\{r_0\}) =\emptyset$. Then, 
    \begin{itemize}
        \item if $s_1,s_2 \neq r_0$,
        \[
        \inf_{r \in [s_1+\eps^\gamma, s_2-\eps^\gamma]}|T(r)|\geq C\eps^{\gamma},
        \]
        for some constant $C>0$ independent of $r_0 \in \mathcal{R}$;
        \item if either of $s_1,s_2$ are equal to $r_0$,
        \[
        \inf_{r \in [s_1+\eps^\gamma, s_2-\eps^\gamma]}|T(r)|\geq C\eps^{2\gamma}
        \]
        for some constant $C>0$ independent of $r_0 \in \mathcal{R}$.
    \end{itemize}
\end{itemize}
\end{lemma}
\begin{proof}
Let $p\geq 0$ and $q>0$ that can be taken to be $q=\infty$, define $\mathcal{I}=[p,q]$, and $\mathcal{J}=[p_1,q_1]$ with $p_1 \in [0,p]$, $q_1 \in [q,\infty]$.
Assume that $\mathcal{F}_0$ is nonempty, let $p_2>\Xi$ and define $\mathcal{K}=[p_1,p_2]$ be a compact interval so that $\mathcal{F}_0 \subset \mathrm{int}(\mathcal{K})$. 
By the implicit function theorem, for any $s_j \in \mathcal{F}_0$, there exist an open neighbourhood of $(s_j,r_0)$, and a differentiable function $\phi_j$ so that the unique zero of $T(s,r_0)$ is given by $(\phi_j(r_0),r_0)$. Hence, since $\mathcal{F}_0$ is finite, we may find a  compact interval $\mathcal{R}_0\supset\mathcal{R} \ni r_0$, and a number $\omega_0>0$ so that for $r_0 \in \mathcal{R}$, $\mathrm{dist}(r,\mathcal{F}_0\cup\{r_0\})< \omega_0$, $T(r,r_0)=0$ if and only if $r=\phi_j(r_0)$, and $|r-s_j|<\omega_0$, for some $j$.
Next, by Taylor's remainder theorem, 
\begin{equation*}
\begin{split}
\left|T(r)-\frac{1}{2}T''(r_0)(r-r_0)^2\right| & \leq \sup_{s \in [r_0,r]}|T^{(3)}(s)||r-r_0|^3,\\
|T(r)-T'(s_j)(r-s_j)| & \leq \sup_{\xi \in [s_j,r]}|T''(\xi)||r-s_j|^2.
\end{split}
\end{equation*}
and $\sup_{\xi \in [s_j,r_0]}|T''(\xi)|, \sup_{s \in [r_0,r]}|T^{(3)}(s)|$ are continuous in $r_0$, and thus bounded on the compact interval $\mathcal{R}$. Similarly, since $s_j=\phi_j(r_0)$, and $|T''(r_0)|$, $|T'(s_j)|$, are positive functions that are continuous in $r_0$, upon possibly shrinking the size of $\mathcal{R}$, they may be bounded below uniformly for $r_0 \in \mathcal{R}$. Hence, points \ref{P1}, \ref{P2} have been proven. Additionally, \ref{P5} is deduced in entirely the same way as \ref{P1}, \ref{P2}.

Next, upon possibly further shrinking $\mathcal{R}$, there exists some $\omega_1$ so that 
\begin{equation*}
\begin{split}
&|T(r)|\geq \frac{1}{4}|T''(r_0)||r-r_0|^2\\
&|T(r)| \geq \frac{1}{2}|T'(\phi_j(r_0))||r-\phi_j(r_0)|,
\end{split}
\end{equation*}
respectively for $|r-r_0|<\omega_1$, $|r-\phi_j(r_0)|<\omega_1$, and $|r-\phi_j(r_0)|<\omega_1$, which implies $|r-s_j|<\omega_0$. Consider now the set $\{(r,r_0)\in \mathcal{R} \times \mathcal{K}:\ |r-\phi_j(r_0)| \geq \omega_1 \ \forall j, \ |r-r_0|\geq \omega_1\}$. This is a compact set, and $T(r,r_0)$ is a non-vanishing continuous function from it to the reals. Thus, $|T(r,r_0)|$ is uniformly bounded below on this set, which combined with \ref{P1}, \ref{P2} immediately yields \ref{P6}.

There is only \ref{P3} and \ref{P4} left to be studied, so let us start with the former. We may write,
\begin{equation*}
\frac{|T(r)-T(r')|}{|T(r')|} \leq \frac{|T'(r')|}{|T(r')|}|r-r'|+\frac{\sup_{s \in \mathcal{I}}|T''(s)|}{|T(r')|}|r-r'|^2.
\end{equation*}
But now, note that $|T'(r')|\leq \sup_{\xi \in \mathcal{K}}|T''(\xi)||r'-r_0|$, and as soon as $|r'-r_0|<\omega_1$, we can write
\begin{equation*}
\frac{|T(r)-T(r')|}{|T(r')|}\leq C_1|r'-r_0|^{-1}|r-r'|+C_2|r'-r_0|^{-2}|r-r'|^2.
\end{equation*}
Thus, for $\omega_1\geq|r'-r_0| \geq \eps^\gamma$, this is bounded uniformly for $r_0 \in \mathcal{R}$ by 
\begin{equation*}
\frac{|T(r)-T(r')|}{|T(r')|} \leq C \left (\eps^{-\gamma}|r-r'|+ \eps^{-2\gamma}|r-r'|^2 \right ).
\end{equation*}
Similarly, we can show that as soon as $\eps^\gamma \leq |r'-\phi_j(r_0)|<\omega_1$, we have 
\begin{equation*}
\frac{|T(r)-T(r')|}{|T(r')|} \leq C \eps^{-\gamma}|r-r'|.
\end{equation*}
Furthermore, since away from these two sets, we have shown that $|T(r')|$ is bounded below uniformly for $r_0 \in \mathcal{R}$, we have 
\begin{equation*}
\frac{|T(r)-T(r')|}{|T(r')|} \leq C|r-r'|.
\end{equation*}
An identical argument shows that 
\begin{equation*}
\frac{|r\Omega(r)|}{|T(r')|}\leq C\eps^{-2\gamma}|r-r'|
\end{equation*}
uniformly for $r,r' \in \mathcal{K}$, $|r-\phi_j(r_0)|\geq \eps^\gamma$, $|r-r_0|\geq \eps^\gamma$.
Finally, suppose that $r \geq \Xi$, $r' \leq \Xi$. If both $r,r' \in \mathcal{K}$, then certainly \ref{P3} is satisfied, so assume that $r \notin \mathcal{K}$, so that $|r-r'| \gtrsim 1$. Hence, observing further that by \ref{H2} we have the bound
$$
|T(r)| \leq C \sum_{n=0}^{N_3+1}|r-\Xi|^n, 
$$
we write 
\begin{equation*}
\frac{|T(r)-T(r')|}{|T(r')|}\leq C\eps^{-2\gamma} \sum_{n=0}^{N_3+1}|r-\Xi|^n \leq C \eps^{-2\gamma} \sum_{n=0}^{N_3+1}|r-r'|^n,
\end{equation*}
where we have used that $|r-\Xi|^n \leq C(n)(|r-r'|^n+|r'-\Xi|^n)$. But now, note that $|r'-\Xi|^n \leq \Xi^n$, and that $1 \lesssim |r-r'|$. Thus, we may easily bound 
\begin{equation*}
\frac{|T(r)-T(r')|}{|T(r')|}\leq C \eps^{-2\gamma}\sum_{n=2}^{N_3+1}|r-r'|^n.
\end{equation*}
Similarly, if $r' \notin \mathcal{K}$, $r\leq \Xi$, we simply bound 
\begin{equation*}
\frac{|T(r)-T(r')|}{|T(r')|}\lesssim 1+\frac{|T(\Xi)-T(r')|}{|T(r')|} \lesssim C\sum_{n=0}^{N_3}|r'-\Xi|^n.
\end{equation*}
Proceeding entirely as before, using that $|r-r'|\gtrsim 1$, this is bounded by 
\begin{equation*}
C\sum_{n=1}^{N_3}|r'-r|^n,
\end{equation*}
and so in any case we may write
\begin{equation*}
\frac{|T(r)-T(r')|}{|T(r')|}\leq C \left ( \eps^{-\gamma}|r-r'|+\eps^{-2\gamma}|r-r'|^2+\eps^{-2\gamma}\sum_{n=3}^{N_3+1}|r-r'|^n\right ).
\end{equation*}
Hence, 
\begin{equation*}
\sum_{n=1}^{N_3+1}B_n(\eps)\eps^{n\delta} \leq C(\eps^{-\gamma+\delta}+\eps^{-2\gamma+2\delta}+2\eps^{-2\gamma+3\delta}).
\end{equation*}
The argument for $|r\Omega'(r)||T(r')|^{-1}$
is entirely analogous. Hence, this shows that \ref{P3} and \ref{P4} hold true, and so the proof of the lemma is completed.
\end{proof}

\subsubsection{Sufficient conditions for Assumption \ref{H2}.}
While conditions \ref{H0} and \ref{H1} are rather easily checkable, the condition \ref{H2} is slightly less intuitive. We show that a simple computation on the derivatives of $T(r)$ is sufficient to verify \ref{H2} in many cases.

\begin{lemma}
\label{lemma:checkable hypotheses}
Let $\Omega,U \in C^\infty([0,\infty))$, and set $T(r)$ as defined in \eqref{eq:transport-function}. Suppose that there exists $\Xi>0$, and a compact set $\mathcal{R}_0 \ni r_0$ so that the following hold:
\begin{enumerate}
    \item \label{simple 1}There exist $N_1 \geq 2$, $C_i<\infty$ so that for all $0 \leq i \leq N_1-1$, $r_0 \in\mathcal{R}_0$ it holds
    \begin{equation*}
    \sup_{r\geq \Xi}\frac{|T^{(i)}(r)|}{|T(r)|}+\sup_{r \geq \Xi}\frac{|(r\Omega'(r))^{(i)}|}{|T(r)|} \leq C_i.
    \end{equation*}
    \item \label{simple 2} There holds 
    \begin{equation*}
    \sup_{r_0 \in \mathcal{R}_0}\left (\sup_{r \geq \Xi}|T^{(N_1)}(r)|+\sup_{r \geq \Xi}|(r\Omega'(r))^{(N_2)}|\right )<\infty.
    \end{equation*}
\end{enumerate}
If further $\inf_{r_0 \in \mathcal{R}_0} \inf_{r \geq \Xi}|T(r)|>0$, then, condition \ref{H2} is satisfied for the interval $\mathcal{R}_0$.
\end{lemma}

\begin{proof}
By a Taylor expansion there holds 
\begin{equation*}
|T(r)-T(r')| \leq \sup_{\xi_1 \in [r,r']}|T'(\xi_1)||r-r'| \leq |r-r'|\left(|T'(r')|+|r-r'|\sup_{\xi_2\in [r,r']}|T''(\xi_2)|\right).
\end{equation*}
Inductively, we deduce that 
\begin{equation*}
|T(r)-T(r')|\leq |r-r'|\sum_{1 \leq i \leq N_1-1}|T^{(i)}(r')||r-r'|^{i-1}+|r-r'|^{N_1}\sup_{\xi \geq 0}|T^{(N_1)}(\xi)|.
\end{equation*}
Therefore, dividing through by $|T(r')|$, we deduce that for any $r_0 \in \mathcal{R}_0$ there holds 
\begin{equation*}
\frac{|T(r)-T(r')|}{|T(r')|} \leq |r-r'|\sum_{i \leq N_1} B_i|r-r'|^{i-1},
\end{equation*}
where we have set 
\[
B_i=\sup_{r_0 \in \mathcal{R}_0}\sup_{r' \geq \Xi}\frac{|T^{(i)}(r')|}{|T(r')|}
\]
for $i \leq N_1-1$, and 
\[
B_{N_1}=\sup_{r_0 \in \mathcal{R}_0}\sup_{r \geq 0}|T^{(N_1)}(r)|\sup_{r'\geq \Xi}|T(r')|^{-1}.
\]
An identical argument yields the claim for 
$|r\Omega'(r)||T(r')|^{-1}$.
\end{proof}

\subsubsection{Relevant examples}

In this section, we present illustrative examples that fall within the scope of Assumptions \ref{H0}--\ref{H3} and are of particular interest due to their historical significance or physical relevance. Additionally, we will make sure that these examples also satisfy the condition
\begin{equation}\label{eq:log-derivative}
r_0\left|\frac{\dd}{\dd r}\log\left|\frac{\Omega'(r_0)}{U'(r_0)}\right|\right|<4
\end{equation}
from Theorem \ref{thm:dynamo}.

\begin{lemma}[Simplified model]\label{lemma:simplified-model}
    The smooth vector field
    \[
    \Omega(r) = 1-r, \quad U(r) = 1-r^2,
    \]
    satisfies \ref{H0}-\ref{H3} as well as \eqref{eq:log-derivative} with any choice of compact set $\mathcal{R}_0 \subset (0,\infty)$.
\end{lemma}

This example of a polynomial vector field exhibiting slow dynamo action is relevant due to the simplicity of its transport function $T(r) \sim (r-r_0)^2$. It has been widely used in applied and numerical studies of the smooth Ponomarenko dynamo, see e.g.\ \cites{Peyrot_Gilbert_Plunian_2009,Ruzmaikin_Sokoloff_Shukurov_1988}.

\begin{proof}
With Lemma \ref{lemma:checkable hypotheses} under our belt, this proof becomes elementary. Observe that 
\[
T(r)=\frac{1}{2r_0}(r-r_0)^2
\]
for any fixed $r_0 >0$. It remains to note that $T(r)$ only vanishes at $r=r_0$, moreover
\[
\frac{|T'(r)|}{|T(r)|} = \frac{2}{(r-r_0)^2}, \quad 
\]
and $T^{(i)}=0$ for any $i\geq 2$, hence
\begin{equation*}
\sup_{r \geq 2r_0+1}\frac{|T^{(i)}(r)|}{|T(r)|}<\infty.
\end{equation*}
Similarly, since $r\Omega'(r)=-r$, conditions (\ref{simple 1}) and (\ref{simple 2}) from Lemma \ref{lemma:checkable hypotheses} follow accordingly. Moreover, to see that it also satisfies \eqref{eq:log-derivative}, we compute
\[
r\left|\frac{\dd}{\dd r}\log\left|\frac{\Omega'(r)}{U'(r)}\right|\right| = 1
\]
for all $r\in (0,\infty)$.
\end{proof}

The following relevant examples concern the existence of finite energy---in the sense of finite $L^2$ norm---velocity fields that exhibit slow dynamo action. The vector fields here presented are in particular in $L^1\cap L^\infty$.

\begin{lemma}[Finite energy dynamos]\label{lemma:finite-energy-dynamo}
The following velocity fields satisfy Assumptions \ref{H0}--\ref{H3} and condition \eqref{eq:log-derivative} from Theorem \ref{thm:dynamo}. 
\begin{enumerate}
    \item Gaussian velocity fields
\begin{equation*}
\Omega(r)=\e^{-ar^2}, \quad U(r)=\e^{-br^2}, 
\end{equation*}
with $a,b>0$ and $a\neq b$.

\item Let $\Omega, U$ satisfy the assumptions of Lemma \ref{lemma:checkable hypotheses}, and moreover assume that 
\[
\Omega(r_0)-\frac{\Omega'(r_0)}{U'(r_0)}U(r_0) \neq 0
\]
and \eqref{eq:log-derivative} hold true for all $r_0\in\mathcal{R}_0$, where $\mathcal{R}_0\subset(0,\infty)$ is given by Lemma \ref{lemma:checkable hypotheses}.
Then, there exists a smooth cut-off $\chi\in C^\infty([0,\infty))$, i.e.\ $0\leq\chi\leq 1$ and
\[
\chi(r) = \begin{cases}
    1 & \textrm{if } r\leq R_0, \\
    0 & \textrm{if } r\geq R_1,
\end{cases}
\]
for some $R_1>R_0>0$, so that $\chi(r)\Omega(r)$, $\chi(r)U(r)$ also satisfy \ref{H0}--\ref{H3} and \eqref{eq:log-derivative}.
\end{enumerate}
\end{lemma}

\begin{proof}
We begin by checking the case of Gaussian velocity fields. We have
\begin{equation*}
T(r)= \e^{-ar^2} -\frac{a}{b}\e^{-(a-b)r_0^2}e^{-br^2} +\left ( \frac{a}{b}-1 \right )\e^{-ar_0^2}.
\end{equation*}
A direct computation yields that $T(r)=0$ if and only if
\[
\frac{a}{b}-1 = \frac{a}{b}\e^{-b(r^2-r_0^2)} - \e^{-a(r^2-r_0^2)},
\]
which shows that $r=r_0$ is the unique point where $T(r)$ vanishes---and it does so up to second order---. Finally, since $T^{(i)}(r) \to 0$ as $r \to \infty$ for any $i$, and the same holds for $r\Omega'(r)$, the assumptions of Lemma \ref{lemma:checkable hypotheses} are satisfied. To check that \eqref{eq:log-derivative} holds true, we compute
\[
r\left|\frac{\dd}{\dd r}\log\left|\frac{\Omega'(r)}{U'(r)}\right|\right| = 2r^2|b-a|.
\]
Thus, a smart choice of $a,b>0$ yields that there exists $\mathcal{R}_0\subset (0,\infty)$ compact and with nonempty interior such that \eqref{eq:log-derivative} is satisfied for all $r_0\in\mathcal{R}_0$.

For the case of compactly supported velocity fields, we just need to pick a smooth cut-off so that \ref{H1} is satisfied and $r_0<R_0$, since the relative derivative bounds from \ref{lemma:checkable hypotheses} follow immediately from the condition 
\[
\Omega(r_0)-\frac{\Omega'(r_0)}{U'(r_0)}U(r_0) \neq 0
\] 
and from $\Omega$ and $U$ being smooth. To obtain the corresponding condition \eqref{eq:log-derivative}, simply recall that $\Omega$ and $U$ satisfy it by assumption, hence upon choosing $R_0$ and $R_1$ conveniently, this property readily follows for $\chi\Omega$ and $\chi U$.
\end{proof}

Finally, we show that the Taylor--Couette flow indeed satisfies our assumptions. 
\begin{lemma}[Taylor--Couette Ponomarenko dynamo]
Let 
\[
u(r)=\left(\frac{a_1}{r}+a_3r\right)\hat \theta+(a_2 \log(r)+a_4) \hat z,
\]
i.e.\ $\Omega(r)=a_1r^{-2}+a_3$, $U(r)=a_2\log(r)+a_4$, with $a_1,a_2\neq 0$. Then, the Taylor--Couette flow satisfies Assumptions \ref{H0}--\ref{H3} and \eqref{eq:log-derivative} on either $\cI=[p,q]$ or $\cI=[p,\infty)$ with $0<p<q$, for any $r_0\in\mathrm{int}(\cI)$.
\end{lemma}

\begin{proof}
The transport function is given by 
\begin{equation*}
T(r)=a_1\left (\frac{1}{r^2}-\frac{1}{r_0^2}\right )+a_1\frac{2}{r_0^2}\left (\log(\frac{r}{r_0})\right ).
\end{equation*}
Differentiating this equation in $r$, we obtain 
\begin{equation*}
-\frac{2a_1}{r^3}+\frac{2a_1}{r r_0^2},
\end{equation*}
which is equal to zero only if $r^2=r_0^2$, i.e.\ when $r=r_0$, since we only consider non-negative values of $r$, and thus \ref{H0}, \ref{H1} are satisfied. For \ref{H2}, we apply Lemma \ref{lemma:checkable hypotheses} and deduce easily from our computation of $T(r)$ that $|T'(r)||T(r)|^{-1}$ is bounded uniformly for $r$ large and $r_0$ in any compact interval. Certainly also $|T''(r)|$ is bounded uniformly, and $|T(r)|$ is bounded below uniformly for $r$ large, $r_0$ in a compact interval. Similarly, $r\Omega'(r)=-2a_1r^{-2}$, and so the same considerations can be used in order to conclude \eqref{simple 1}, \eqref{simple 2}, and hence \ref{H2}.
To check \ref{H3} for the case $\mathcal{I}=[p,\infty)$, we note that $|r\Omega'(r)| \to 0$ as $r \to \infty$. Additionally, to obtain condition \eqref{eq:log-derivative} we compute
\[
r\left|\frac{\dd}{\dd r}\log\left|\frac{\Omega'(r)}{U'(r)}\right|\right| = 2,
\]
uniformly for an $r>0$.
\end{proof}

\section{A-priori estimates on the kinematic dynamo operator}\label{s:injectivity}

Before commencing with the proof of the existence of a growing mode, we collect here a couple of facts about the kinematic dynamo operator, which will be important for our analysis. In particular, we shall derive an extremely useful a-priori estimates on the behaviour of eigenfunctions of the kinematic dynamo operator, and discuss some of its consequences. These include being able to show that any growing mode of the dynamo equation must be divergence free, as well as begin able to deduce qualitative spectral properties about the kinematic dynamo operator. In particular, we will be able to give sufficient conditions for all unstable eigenvalues of the kinematic dynamo operator to be isolated, and hence for the asymptotic expansion of the growing mode in Theorem \ref{thm:dynamo} to be valid.

\subsection{The a-priori estimate}\label{s:a-priori}
To start with, we provide an a-priori estimate on all eigenfunctions of the operator $\L$ in $X$, which is a crucial ingredient of the proofs of Lemmas \ref{lemma:injectivity1}, \ref{lemma:injectivity2}, as well as of the proof of Lemma \ref{lemma:divergence free} concerning the divergence free property of the growing mode that we construct. 

\begin{lemma}
\label{lemma:eigenfunctions are smooth}
Let $\mathcal{I}=[0,q]$, for $q \in (0,\infty]$, suppose $\Omega,U\in C^3(\mathcal{I})$, and consider the operator 
\[
\dyn b = \eps\left( \partial^2_r+\frac{1}{r}\partial_r - \frac{m^2}{r^2} - k^2 -\frac{1}{r^2} - i\eps^{-1} (m\Omega(r)+kU(r)) \right)b + \begin{pmatrix}
    -2i\eps mr^{-2} b_\theta \\
    (2i\eps mr^{-2} + r\Omega'(r))b_r
\end{pmatrix}.
\]
If $b=(b_r,b_\theta)$ is an eigenfunction of $\dyn$ with eigenvalue $\lambda\in\Co$, so that further 
\[
b \in \left\lbrace b \in H^2_{\loc}(\mathcal{I}) \cap X:\dyn b \in X_{\mathcal{I}}\right\rbrace,
\]
then the following holds for all $m,N$ large enough depending only on $\Omega$, $U$
\begin{enumerate}
    \item $r^{-2}b \in L^\infty(\mathcal{I},\max\{1,r^2\}\dd r)$,
    \item $r^{-1}\partial_r b\in L^\infty(\mathcal{I},\max\{1,r^2\}\dd r)$,
    \item $\partial_r^2 b \in L^\infty(\mathcal{I},\max\{1,r^2\}\dd r)$.
\end{enumerate}
In fact, if $b_z \in X_{\mathcal{I}}$ satisfies 
\begin{equation}
\eps(\partial_r^2 +\frac{1}{r}\partial_r-\frac{m^2}{r^2}-\frac{1}{r^2}-i\eps^{-1}(m\Omega(r)+kU(r))b_z+U'(r)b_r=\lambda b_z,
\end{equation}
the same bounds also hold for $b_z$.
\end{lemma}

\begin{proof}
Set $W_1=ib_r+b_\theta$ and $W_2=-ib_r+b_\theta$. Then, we may rewrite the equations as 
\begin{equation}\label{eq:operator-apriori}
\begin{split}
    \widehat{\dyn} W & = \eps\left(\partial^2_r+\frac{1}{r}\partial_r-\frac{m^2}{r^2}-k^2-\frac{1}{r^2}-i(m\Omega(r)+kU(r))\right)W \\
    & \quad + \frac{2m}{r^2}\begin{pmatrix}
    -W_1\\
    W_2
\end{pmatrix} + \frac{r\Omega'(r)}{2i}\begin{pmatrix}
    W_1-W_2\\
    W_1-W_2
\end{pmatrix}
\end{split}
\end{equation}
In particular, if $W$ solves the equation $(\widehat \dyn-\lambda)W = 0$, we find that the first component must satisfy
\begin{equation}\label{eq:eq-first-comp-apriori}
    \eps \left(\partial^2_r+\frac{1}{r}\partial_r-\frac{m^2+2m+1}{r^2}\right)W_1 = i(m\Omega(r)+kU(r))W_1-r\Omega'(r)\frac{W_1-W_2}{2i}+(\lambda+\eps k^2) W_1.
\end{equation}
Since by assumption $W\in X=L^\infty(\mathcal{I},\max\{1,r^2\}w_\eps(r-r_0)\dd r)$, with $w_\eps(s) = 1+(\eps^{-1/3}|s|)^N$, we see that as long as the exponent $N>0$ in the weight is sufficiently large, the right hand side of \eqref{eq:eq-first-comp-apriori} is in $L^\infty(\mathcal{I},\max\{1,r^2\}\dd r)$. Indeed, from Assumption \ref{H2} it is clear that $\Omega$, $U$ and $\Omega'$ cannot grow faster than polynomially as $r\to\infty$. In fact, given $\Omega$ and $U$, we may choose $N$ sufficiently large so that the right hand side in \eqref{eq:eq-first-comp-apriori} vanishes at least cubically as $r\to\infty$. With this in mind, we construct a Green's function for this equation. A simple computation shows that the equation 
\[
\left(\partial_r^2+\frac{1}{r}\partial_r-\frac{\alpha^2}{r^2}\right)v=0
\]
has two linearly independent solutions, given by $r^{\alpha}$ and $r^{-\alpha}$. Therefore, upon computing the Wronskian, we can construct a Green's function given by
\begin{equation}\label{eq:Green's_Galpha}
\G_\alpha f=-\frac{1}{2\alpha} \left[r^{-\alpha} \int_0^r s^{\alpha+1}f(s)\dd s+r^{\alpha}\int_r^\infty s^{-\alpha+1}f(s)\dd s \right],
\end{equation}
where without loss of generality we assume $\alpha>0$. In particular, we can rewrite the equation for the first component as
\begin{equation}\label{eq:first-comp-injec}
\left(\partial^2_r+\frac{1}{r}\partial_r-\frac{(m+1)^2}{r^2}\right)W_1 = f,
\end{equation}
with
\[
f=\eps^{-1} \left(i(m\Omega(r)+kU(r))W_1-r\Omega'(r)\frac{W_1-W_2}{2i}+(\lambda+\eps k^2) W_1 \right),
\]
that is clearly in $L^\infty(\mathcal{I})$, and as mentioned before, it decays at least cubically as $r\to\infty$. A direct computation from \eqref{eq:Green's_Galpha} yields the estimate
\begin{equation}\label{eq:Galpha-1}
|\G_\alpha f(r)| \lesssim r^2\|f\|_{L^\infty}
\end{equation}
for any $r>0$, provided that $\alpha>2$. Since $\alpha = m+1$ for $W_1$, and $\alpha = m-1$ for $W_2$, we need to choose $m>3$. In addition, using that $f$ decays at least cubically at infinity, we find that as $r\to\infty$ there holds
\begin{equation}\label{eq:Galpha-2}
|\G_\alpha f(r)|\lesssim \frac{1}{r}.
\end{equation}
Going back to \eqref{eq:first-comp-injec} and setting $\alpha = m+1$, we find that
\[
\left(\partial^2_r+\frac{1}{r}\partial_r-\frac{(m+1)^2}{r^2}\right)(\G_{m+1}f-W_1)=0,
\]
and so $\G_{m+1}f-W_1$ must be in the kernel of  the operator $\partial^2_r+r^{-1}\partial_r-(m+1)^2r^{-2}$, which we have just shown to be precisely $r^{\pm (m+1)}$. $\G_{m+1}f$ is bounded both at zero and, since it lives in $X$ it cannot grow as the right endpoint of $\mathcal{I}$ is approached. Thus, it follows that $\G_{m+1}f=W_1$, and so we find that $W_1$ vanishes quadratically at zero, so that $r^{-2}W_1\in L^\infty(\mathcal{I},\max\{1,r^2\}\dd r)$. (In fact, we may iterate this argument further to show it vanishes at zero at order proportional to $m$). Next, we shall show that $b$ must in fact be differentiable in $r$. To do so, recall that we have proved that we can write
\[
W_1=-\frac{1}{2\alpha}\left[r^{-\alpha} \int_0^r s^{\alpha+1}f(s)ds+r^{\alpha}\int_r^\infty s^{-\alpha+1}f(s)ds\right].
\]
Note that $f \in H^2_{\loc}$ by assumption, and since we are in one dimension, this in fact implies that $f \in C^1$. Therefore, it is clear that $W_1$ is continuously differentiable away from zero, and its derivative is given by 
\[
\partial_r W_1=-\frac{1}{2} \left[- r^{-\alpha-1}\int_0^r s^{\alpha+1}f(s)ds+r^{\alpha-1} \int_r^\infty s^{-\alpha+1}f(s)ds \right].
\]
Hence, arguing as for \eqref{eq:Galpha-1} and \eqref{eq:Galpha-2}, we deduce that
\begin{align*}
    |\partial_r W_1| & \lesssim r\|f\|_{L^\infty} \quad \text{for any } r>0, \\
    |\partial_r W_1| & \lesssim \frac{1}{r^2} \quad \text{as } r\to\infty.
\end{align*}
In particular this implies that (at least) $r^{-1}\partial_r W_1\in L^\infty(\mathcal{I},\max\{1,r^2\}\dd r)$. The same argument, (with the sole exception that now $\alpha=m-1$), carries over for the $W_2$ using the second component in \eqref{eq:operator-apriori}. All in all we can finally combine the facts that 
\begin{itemize}
    \item $\mathcal{L}b \in L^\infty(\mathcal{I},\max\{1,r^2\})$,
    \item $r^{-2}b, \ r^{-1}\partial_r b, \in L^\infty(\mathcal{I},\max\{1,r^2\}\dd r)$, and
    \item $(m\Omega(r)+kU(r))b, \ r\Omega'(r)b \in L^\infty(\mathcal{I},\max\{1,r^2\}\dd r)$ provided that $N$ is large enough,
\end{itemize}
to deduce from 
\[
\eps \partial_r^2 b_r=-\eps \left(\frac{1}{r}\partial_r-\frac{m^2}{r^2}-k^2-\frac{1}{r^{2}}-i(m\Omega(r)+kU(r)) \right)b_r-\frac{2im}{r^2}b_\theta.
\]
that $\partial_r^2 b\in L^\infty(\mathcal{I},\max\{1,r^2\}\dd r)$. In fact, from our arguments it actually follows that $\partial_r^2b$ is continuous. Finally, the same arguments may be repeated for the $z$-component, which then completes the proof. 
\end{proof}

\subsection{Spectral properties of the dynamo operator} \label{s:isolated eigs}

An important application of the a-priori bound from Lemma \ref{lemma:eigenfunctions are smooth} is to the question of whether eigenfunctions of the kinematic dynamo operator are isolated. Indeed, whilst this question is an elementary consequence of the compact embedding of $H^1$ into $L^2$ on bounded domains, on unbounded domains there is a-priori not much of a reason for this property to hold. In particular, the cases where $r \in [0,\infty)$ or $r \in [p,\infty)$ with perfectly conducting boundary conditions require additional care.

Since the dynamo equations can be viewed as a perturbation of the advection diffusion equation by the stretching term $r\Omega'(r)B_r$, it is natural to expect that spectral considerations are most easily dealt with in $L^2$. From this point of view, the role of Lemma \ref{lemma:eigenfunctions are smooth} is then quite simply to show that the set of $L^2$ eigenfunctions of the kinematic dynamo operator contains the set of its eigenfunctions in $X$. Therefore, if the set of $L^2$-eigenvalues is isolated, we conclude a-fortiori that the same holds for the set of $X$-eigenvalues.

The aim this sub-section will therefore precisely be to find conditions so that the set of (unstable) $L^2$-eigenvalues of the kinematic dynamo operator is isolated when the operator acts on functions defined on unbounded domains.
Continuing with the intuition that one can treat the stretching term perturbatively, we will mainly be interested in the behaviour of $r\Omega'(r)$, and in fact we deduce two checkable assumptions on this stretching term that ensure that all unstable eigenvalues are isolated. The first of these addresses the case where the stretching decays at infinity, namely
\[
\lim_{r\to\infty} |r\Omega'(r)| = 0.
\]
This is certainly satisfied in many cases of interest, such as for compactly supported, or finite energy velocity fields. If on the other hand, the velocity field grows at infinity, the second Lemma of this section still yields relatively mild conditions for all eigenfunctions to be isolated. Indeed, we simply require that the transport function $T(r)$ must diverge as $r\to\infty$ at a rate exceeding \emph{the square} of the stretching term, namely $|T(r)|\to\infty$ and
\[
\lim_{r\to\infty} \frac{|T(r)|}{|r\Omega'(r)|^2} = \infty.
\]
We stress once again that these conditions for the velocity fields are required in order to obtain an asymptotic expression for the growing mode eigenfunction. Any examples of vector fields that satisfy Assumptions \ref{H0}--\ref{H2} but that do not satisfy the conditions here introduced, will still have a growing mode of the form stated in Theorem \ref{thm:dynamo}, but we cannot say much about how the eigenfunction looks like.

Without further ado, let us address the main results of this section.
\begin{lemma}
\label{lemma:injectivity1}
Let $\mathcal{I}=[0,\infty)$ or $\mathcal{I}=[p,\infty)$, for some $p >0$. Suppose that 
\[
\lim_{r \to \infty}|r\Omega'(r)|=0.
\]
Then, provided $k^2$ is large enough, there are at most countably many $\lambda \in \mathbb{C}$ with $\Re(\lambda)>0$ for which $\mathcal{L}-\lambda$ is not injective on $X|_{\mathcal{I}}$, and each such $\lambda$ is \emph{isolated}. 
\end{lemma}

\begin{proof}
We shall provide explicit details for this proof in the case where $\mathcal{I}=[0,\infty)$, and mention how to modify the argument for $\mathcal{I}=[p,\infty)$ at the end of the proof.

We consider the operator $S$ defined on $L^2((0,\infty),r\dd r)$ by 
\begin{equation}\label{eq:dissipative-operator}
Sb=\eps \left(\partial_r^2+\frac{1}{r}\partial_r-\frac{m^2}{r^{2}}-k^2-\frac{1}{r^{2}}-i\eps^{-1}(m\Omega(r)+kU(r))\right)b+\begin{pmatrix}
-2\eps imr^{-2}b_\theta\\
2i \eps mr^{-2}b_r
\end{pmatrix},
\end{equation}
where we denote $b=(b_r,b_\theta)$. Furthermore, we endow it with the domain 
\[
\mathcal{D}(S)=H^2((0,\infty),r\dd r)\cap X,
\]
where we denote by $H^2((0,\infty),r\dd r)$ the space of functions which have weak derivatives up to second order in $L^2((0,\infty),r\dd r)$. Note that for any $B\in \mathcal{D}(S)$ we may integrate by parts, and obtain that 
\begin{align*}
\langle Sb,b\rangle_{L^2(r\dd r)}= & -\eps\int_0^\infty \left[r|\partial_r b|^2+\frac{m^2}{r^{2}}|b|^2+k^2|b|^2+\frac{1}{r}|b|^2+4\Re\left(i\frac{mb_\theta\overline{b}_r}{r}\right) \right]\dd r \\
& -i\int_0^\infty r(m\Omega(r)+kU(r))|b|^2 \dd r.
\end{align*}
From this it is clear that $(S,\mathcal{D}(S))$ is a dissipative operator. Furthermore, note that it is in fact densely defined, since $\mathcal{D}(S)$ certainly contains $C_c^\infty(0,\infty)$. Given any $H \in L^2((0,\infty),r\dd r)$, one finds a cut-off $\chi$ so that $H\chi$ has compact support in $(0,\infty)$, and such that $\|H-H\chi\|_{L^2(r\dd r)}\leq \delta/2$. In turn, up to a mollification, one can find a smooth, compactly supported function $\phi$ so that $\|H\chi-\phi\|_{L^2(r\dd r)} <\delta/2$, and hence density follows. Therefore, using \cite{Arendt_Batty_Neubrander_2013a}*{Lemma 3.44}, it follows that $(S,\mathcal{D}(S))$ is closable, and its closure is a dissipative operator. In fact, denoting the closure of $S$ by $(\overline{S},\mathcal{D}(\overline{S}))$, it holds that for any $b \in \mathcal{D}(\overline{S})$ one has 
\[
\Re\langle -\overline{S}b,b\rangle_{L^2(r\dd r)} \geq \eps  \int_0^\infty \left(r|\partial_r b|^2 +r^{-1}|b|^2 \right)\dd r.
\]
Indeed, let $b_n \subset \mathcal{D}(S)$ be so that $b_n \to b, Sb_n \to Sb$. Then
\begin{align*}
\Re\langle -\overline{S}b,b\rangle_{L^2(r\dd r)} & = \lim_{n \to \infty}\Re\langle -\overline{S}b_n,b_n\rangle_{L^2(r\dd r)}\\
&\geq \liminf_{n \to \infty} \eps \int_0^\infty \left(r|\partial_r b_n|^2+r^{-1}|b_n|^2\right) \dd r \geq \eps \int_0^\infty \left(r|\partial_r b|^2 +r^{-1}|b|^2\right)\dd r.
\end{align*}
From now on we abuse notation and identify $S$ with its closure. With this in mind, we will aim to establish the essential spectrum of the operator $S$, and then treat the operator $\mathcal{L}$ as a relatively compact perturbation of $S$. Note that we can write 
\[
\mathcal{L}b=Sb+Kb
\]
where 
\[
K\begin{pmatrix}
b_r\\
b_\theta
\end{pmatrix}=\begin{pmatrix}
0\\
r\Omega'(r)b_r
\end{pmatrix}
\]
in cylindrical coordinates. Furthermore, note that for any $\lambda\in\Co$ with $\Re(\lambda)>0$, there exists $C>0$ such that
\[
\|(S-\lambda)^{-1}f\|_{H^1(r\dd r)} \leq C\|f\|_{L^2(r\dd r)},
\]
where we have abused notation and set $\|b\|_{H^1(r\dd r)}=\int_0^\infty r|\partial_r b|^2+r^{-1}|b|^2 \dd r$.
Therefore, by a standard argument, it holds that $K(S-\lambda)^{-1}$ is a compact operator. We reproduce this here for the convenience of the reader. We begin by extending functions in $L^2(r\dd r)$ to radial functions $L^2(\mathbb{R}^2)$. Then, recalling that the vector gradient in of a radial field $b$ on $R^2$ has modulus given by 
\begin{equation*}
|\nabla b|=|\partial_r b|^2+r^{-2}|b|^2,
\end{equation*}
the $H^1(r\dd r)$ norm of such vector fields corresponds precisely to their $H^1(\mathbb{R}^2)$ norm. Therefore, by Rellich's theorem, if we have a sequence $f_n$ with bounded $L^2(r\dd r)$ norm, there will exist a (non-relabled) subsequence, and a radial $f \in H^1(\mathbb{R}^2)$ so that $(S-\lambda)^{-1}f_n \to f$ in $L^2_{\loc}$. But then, since by assumption $|r \Omega'(r)| \to 0$ as $r \to \infty$, we can pick $R>0$ so that for all $n$ it holds 
\[
\int_{|x|\geq R}|K(S-\lambda)^{-1} f_n|^2 \dd x\leq \frac{\delta}{8}.
\]
By the local convergence, we may further pick $N$ large enough so that for all $m,n \geq N$, 
\[
\int_{|x|<R}|K(S-\lambda)^{-1}f_n-K(S-\lambda)^{-1} f_m|^2 \dd x <\frac{\delta}{2}.
\]
Then, for all $m, n \geq N$, 
\begin{align*}
\int_{\mathbb{R}^2}|K(S-\lambda)^{-1} f_n-K(S-\lambda)^{-1} f_m|^2 \dd x & \leq  \int_{|x|<R}|K(S-\lambda)^{-1} f_n-K(S-\lambda)^{-1} f_m|^2 \dd x\\
& +2\int_{|x| \geq R}|K(S-\lambda)^{-1} f_n|^2 +|K(S-\lambda)^{-1} f_m|^2 \dd x < \delta.
\end{align*}
But note that, since all functions in the above argument are radial, we in fact have shown that there exists a function $f \in L^2((0,\infty),r\dd r)$, so that $K(S-\lambda)^{-1}f_n \to f$ as $n \to \infty$.
Hence, $K(S-\lambda)^{-1}$ is indeed a compact operator. Therefore, we deduce first by \cite{Kato}*{IV-1, Theorem 1.11} that $(S+K,\mathcal{D}(S))$ is closed, and by \cite{Kato}*{IV-5, Theorem 5.35} it follows that the essential spectrum of $\L$ satisfies 
\[
\sigma_{\mathrm{ess}}(\mathcal{L}) =\sigma_{\mathrm{ess}}(S) \subset \{z \in \mathbb{C}: \Re(z) \leq 0\},
\]
where the latter inclusion follows from the fact that $S$ is a dissipative operator. Since furthermore it is clear that $\L-\lambda$ is invertible for $\Re(\lambda)$ sufficiently large, we conclude by \cite{Kato}*{IV-5, Theorem 5.31} (see also the discussion in \cite{Kato}*{IV-6}) that the component of the spectrum of $\mathcal{L}$ contained in the right half-plane consists of at most countably many points, each being isolated eigenvalues with finite multiplicity. But now, note that by Lemma \ref{lemma:eigenfunctions are smooth}, any eigenfunction of the operator $\mathcal{L}$ in $X$ is in $\mathcal{D}(S)$, and so the result follows. 

Finally, we shall briefly mention how this result may be extended to the scenario when $r \in [p,\infty)$. In this case, whilst integration by parts for the $b_r$ component proceeds without issue, the $b_\theta$ component requires some additional care. Indeed, we have 
\begin{equation*}
\int_{p}^\infty \frac{\dd}{\dd r}(r b_\theta') \overline{b_\theta} \dd r= \int_{p}^\infty \frac{\dd}{\dd r}((rb_\theta)'-b_\theta) \overline{b_\theta} \dd r=-\int_{p}^\infty r|b'_\theta|^2 \dd r-\int_{p}^\infty 2 \Re(b_\theta' \overline{b_\theta}) \dd r.
\end{equation*}
Therefore, in totality the energy balance equation reads 
\begin{equation*}
\Re\langle -S b, b\rangle_{L^2(r \dd r)} \geq \eps \int_p^\infty r|\partial_r b|^2 +r^{-1}|b|^2 +k^2|b|^2 +2\Re(b_\theta' \overline{b_\theta}) \dd r.
\end{equation*}
But we may bound $|2\Re(b_\theta' \overline{b_\theta})|\leq \frac{1}{2}r|b'_\theta|^2+8p^{-1}|b_\theta|^2$, and so, as soon as $k^2$ is large enough, it holds 
\begin{equation*}
\Re\langle -S b, b\rangle_{L^2(r \dd r)} \geq \eps \int_p^\infty \frac{1}{2}r|\partial_r b|^2 +r^{-1}|b|^2  \dd r,
\end{equation*}
from where the proof proceeds entirely as before.
\end{proof}

\begin{lemma}\label{lemma:injectivity2}
Let $\mathcal{I}=[0,\infty)$ or $\mathcal{I}=[p,\infty)$, for some $p >0$. Suppose that 
\begin{equation}
\label{eq:T growing}
\left|\Omega(r)-\frac{\Omega'(r_0)}{U'(r_0)}U(r)\right| \to \infty
\end{equation}
as $r \to \infty$, and moreover suppose that there exists $C>0, \Lambda>0$ such that
\[
\sup_{r\in (0,\infty)} \frac{|r\Omega'(r)|^2}{\left|\Omega(r)-\frac{\Omega'(r_0)}{U'(r_0)}U(r)+\Lambda \right|} \leq C.
\]
Assume further that $\Omega(r)-\frac{\Omega'(r_0)}{U'(r_0)}U(r)$ is bounded from below (or from above).
Then, for all $|m|, |k|$ sufficiently large, there are at most countably many $\lambda \in \mathbb{C}$ with $\Re(\lambda)>0 $, where $\mathcal{L}-\lambda$ is not injective on $X_{\mathcal{I}}$, and each such $\lambda$ is isolated. 
\end{lemma}

\begin{proof}
We provide a proof in the case where $\mathcal{I}=[0,\infty)$, $\Omega(r)-\frac{\Omega'(r_0)}{U'(r_0)}U(r)$ is bounded from below, and $m \to +\infty$. The general case follows similarly. We aim to show that $\mathcal{L}$ has a compact resolvent on $L^2(r \dd r)$. Once this is established, by classical spectral theory (see e.g. \cite{Kato}*{III-6, Theorem 6.29}) $\mathcal{L}$ has empty essential spectrum on $L^2(r\dd r)$, and then the result follows exactly as in the previous Lemma. We once again take $(S,\mathcal{D}(S))$ defined in \eqref{eq:dissipative-operator} as the ``base'' operator around which we perturb. Recall that by the proof of Lemma \ref{lemma:injectivity1} we know that $(S,\mathcal{D}(S))$ is closable, and its closure is dissipative. In fact, we claim that for all $\Re(\lambda)$ large enough, and for all $b \in \mathcal{D}(\overline{S})$ there holds that 
\begin{equation}
\label{eq:confining}
\Im\langle (i\lambda-\overline{S})b, b\rangle_{L^2(r\dd r)} \geq  \int_0^\infty rm(\Omega(r)-\frac{\Omega'(r_0)}{U'(r_0)}U(r)+\frac{\lambda}{m})|b|^2 \dd r.
\end{equation}
Indeed, let $\Re(\lambda)\geq -\inf_{r \in (0,\infty)}m(\Omega(r)-\frac{\Omega'(r_0)}{U'(r_0)}U(r))$, then the inequality \eqref{eq:confining} certainly holds for any $b \in \mathcal{D}(S)$. Proceeding as in Lemma \ref{lemma:injectivity1}, we take a sequence $b_n \to b$, $Sb_n \to Sb$, $b \in \mathcal{D}(\overline{S})$, and observe that there must hold
\begin{equation}
\label{eq:confining2}
\Im\langle (i\lambda-\overline{S})b, b\rangle_{L^2(r\dd r)} \geq \liminf_{n \to \infty}\int_0^\infty rm(\Omega(r)-\frac{\Omega'(r_0)}{U'(r_0)}U(r)+\frac{\lambda}{m})|b_n|^2 \dd r.
\end{equation}
Since the right-hand side of \eqref{eq:confining2} is non-negative, by Fatou's Lemma we in fact see that we obtain \eqref{eq:confining} as claimed. Furthermore, from Lemma \ref{lemma:injectivity1} we also know that 
\[ 
\Re\langle -\overline{S}b,b\rangle_{L^2(r\dd r)} \geq \eps \int_0^\infty r |\partial_r b|^2\dd r
\]
for all $b \in \mathcal{D}(\overline{S})$. Once again we abuse notation and identify $S$ with its closure. From \eqref{eq:T growing}, \eqref{eq:confining} we can conclude by an identical argument as in the proof of Lemma \ref{lemma:injectivity1} that $S$ has a compact resolvent on $L^2((0,\infty),r\dd r)$.

Next, we show that in fact $S+K$ is also a closed operator with compact resolvent. Indeed, let us begin by showing that $(S+K,\mathcal{D}(S))$ is closed. To do so, we let $\lambda=m\lambda_0$, where $\lambda_0 \geq -\inf_{r \in (0,\infty)}(\Omega(r)-\frac{\Omega'(r_0)}{U'(r_0)}U(r))+1$. Then, the inequality \eqref{eq:confining} yields
\begin{equation*}
\|B\|_{L^2( r \dd r)}\|(i\lambda-S)^{-1}b\|_{L^2(r \dd r)}\geq m\|(i\lambda-S)^{-1} b\|_{L^2(r \dd r)}^2,
\end{equation*}
and so
\begin{equation*}
\|B\|_{L^2(r \dd r)} \geq m \|(i\lambda-S)^{-1}b\|_{L^2(r \dd r)}.
\end{equation*}
In particular, we observe that 
\begin{equation*}
\|(i\lambda -S)^{-1}\|_{L^2(r \dd r) \to L^2(r \dd r)} \leq m^{-1}.
\end{equation*}
Furthermore, we see from \eqref{eq:confining} that 
\begin{align*}
&\|K(i\lambda -S)^{-1}b\|_{L^2(r \dd r)}^2 \leq \int_0^\infty r|r\Omega'(r)|^2|(i\lambda -S)^{-1}b|^2 \dd r \\
&\leq m^{-1}\sup_{r \geq 0}\frac{|r\Omega'(r)|^2}{|\Omega(r)-\frac{\Omega'(r_0)}{U'(r_0)}U(r)+\lambda|}\int_0^\infty rm|\Omega(r)-\frac{\Omega'(r_0)}{U'(r_0)}U(r)+\lambda||(i\lambda -S)^{-1}b|^2 \dd r\\
& \leq Cm^{-1}\|b\|_{L^2(r\dd r)}\|(i\lambda -S)^{-1}b\|_{L^2(r \dd r)}\leq Cm^{-2}\|b\|_{L^2(r \dd r)}^2.
\end{align*}
Therefore, we conclude that for $b \in \mathcal{D}(S)$,
\begin{equation*}
\|Kb\|\leq C m^{-1}\|(i\lambda -S)b\|_{L^2(r \dd r)} \leq Cm^{-1}\|Sb\|_{L^2(r \dd r)}+Cm^{-1}|\lambda|\|b\|_{L^2(r \dd r)}.
\end{equation*}
In the language of \cite{Kato}*{Chapter III-4}, $K$ is $S$-bounded with $S$-bound $Cm^{-1}$. Thus, picking $m$ large enough, by Theorem 1.1 from \cite{Kato}*{Chapter III-4}, it holds that $(S+K,D(S))$ is closed for all $m$ large enough. Next, we show that it still admits a compact resolvent. To do so, we once again use the observation that $\|K(i\lambda -S)\|_{L^2(r \dd r) \to L^2(r \dd r)}\leq Cm^{-1}$. We write 
\begin{equation}
(K+S-i\lambda)=(K(S-i\lambda)^{-1}+\Id)(S-i\lambda).
\end{equation}
Since $\|K(i\lambda -S)^{-1}\|_{L^2(r \dd r) \to L^2(r \dd r)}<1$ for $m$ large, we deduce that $(K(S-i\lambda)^{-1}+\Id)$ is invertible, and hence so is $(K+S-i\lambda)$, as a product of invertible operators. In fact, we have 
\begin{equation}
(K+S-i\lambda)^{-1}=(S-i\lambda)^{-1}(K(S-i\lambda)^{-1}+\Id)^{-1}
\end{equation}
is a compact operator, as a composition of a compact operator with a bounded operator. Hence, $S+K$ has a compact resolvent, and so the proof is complete. 
\end{proof}

\subsection{The divergence-free condition}\label{s:divfree}

The next crucial application of the a-priori bound from Lemma \ref{lemma:eigenfunctions are smooth} is to the question of whether a given growing mode is in fact divergence free. It is worth noting that, at least on a formal level, this property should certainly be satisfied. Taking divergence of the equation 
\begin{equation}
\eps \Delta B+B\cdot \nabla u-u \cdot \nabla B=\lambda B,
\end{equation}
we see that $\div(B)$ satisfies the stationary advection diffusion equation 
\begin{equation}\label{eq:div advection diffusion}
\eps \Delta \div(B)-u\cdot \nabla \div(B)=\lambda \div(B),
\end{equation}
see also e.g.\ \cite{ChildressGilbert}.
If it were possible to integrate by parts in \eqref{eq:div advection diffusion}, the divergence free property would immediately follow as soon as $\Re(\lambda)>0$. Therefore, we need at least that $\div(B) \in H^1$, which crucially follows from Lemma \ref{lemma:eigenfunctions are smooth}. Hence, we have the following result.

\begin{lemma}
\label{lemma:divergence free}
Let $b\in X$ define a growing mode for the kinematic dynamo equation with Ponomarenko velocity field via $B(r,\theta,z)=b(r)\e^{i(m\theta+kz)}$ on $\mathbb{R}^2 \times \mathbb{T}$. Then, $\div(B)=0$.
\end{lemma}

\begin{proof}
Recall that we have found $\lambda$ with $\Re(\lambda)>0$, and a vector field $B$ so that 
\[
\eps \Delta B+B\cdot \nabla u-u\cdot \nabla B=\lambda B,
\]
and so from \eqref{eq:div advection diffusion} we have that $\div(B)$ satisfies
\begin{equation}
\label{eq:divergence equation}
\eps \Delta\div(B)-u \cdot \nabla \div(B)=\lambda \div(B).
\end{equation}
We claim this is enough to deduce that $\div(B)=0$. Indeed, rewriting this in cylindrical coordinates, we see that 
\[
\eps \left(\partial_r^2+\frac{1}{r}\partial_r-\frac{m^2}{r^{2}}-k^2 \right)\div(B)-i(m\Omega(r)+kU(r))\div(B)=\lambda \div(B).
\]
But now, recall that in cylindrical coordinates, the divergence is given by 
\[
\div(B)=\frac{1}{r}(B_r+mB_\theta)+\partial_r B _r+kB_z.
\]
In particular, following Lemma \ref{lemma:eigenfunctions are smooth}, we find that if $B\in X=L^\infty((0,\infty),(1+(\eps^{-1/3}|r-r_0|)^N)\dd r)$, then there holds that
\begin{enumerate}
    \item $\div(B) \in L^\infty((0,\infty),(1+(\eps^{-1/3}|r-r_0|)^{N-1})\dd r)$,
    \item $\partial_r\div(B) \in L^\infty((0,\infty),(1+(\eps^{-1/3}|r-r_0|)^{N-2})\dd r)$.
\end{enumerate}
Furthermore, iterating the argument from Lemma \ref{lemma:eigenfunctions are smooth}, we in fact deduce that $\div(B) \in C^2((0,\infty))$. Hence, we now multiply \eqref{eq:divergence equation} by $\overline{\div(B)}$ and integrate by parts in the left-hand side to get, for any $R>0$,
\begin{align*}
&\int_0^R r \left[\eps \left(\frac{1}{r}\partial_r(r\partial_r)-\frac{m^2}{r^{2}}-k^2 \right)\div(B)-i(m\Omega(r)+kU(r))\right]\div(B) \overline{\div(B)} \dd r\\
& \quad = -\int_0^R \left[ \eps \left(r |\partial_r \div(B)|^2 + \frac{m^2}{r}|\div(B)|^2+k^2r|\div(B)|^2 \right)+ir(m\Omega(r)+kU(r))|\div(B)|^2 \right] \dd r\\
& \qquad + \eps |R\div B(R)\partial_r \div B(R)|.
\end{align*}
Taking $R \to \infty$, we observe from \eqref{eq:divergence equation}
\[
\begin{split}
    -\int_0^\infty & \eps \left( \left[ r|\partial_r \div(B)|^2+\left(\frac{m^2}{r}+rk^2 \right)|\div(B)|^2 \right] + ir(m\Omega(r)+kU(r))|\div(B)|^2 \right) \dd r \\
    & = \lambda \int_0^\infty r |\div(B)|^2 \dd r.
\end{split}
\]
If $\div(B)$ is non-zero, the real part of the left-hand side is negative, whilst the real part of the right-hand side is positive. Therefore, we deduce that $\div(B)=0$, and the claim of the lemma follows.
\end{proof}

When proving a similar result on domains with boundaries, we need to take additional care, since the perfectly conducting boundary conditions need to first be transformed into appropriate boundary conditions for the divergence. Nevertheless, we have the following result.
\begin{lemma}
Let $b$ define a growing mode for the kinematic dynamo equations with Ponomarenko velocity field via $B(r,\theta,z)=b(r)\e^{i(m\theta+kz)}$ on a domain of the form $[p,q] \times \mathbb{T}^2$, where $p$ may be either zero or a finite number, and $q$ may be either a finite number larger than $p$ or $\infty$. Suppose further that $b$ satisfies perfectly conducting boundary conditions on the boundary of the domain. Then, $\div(B)=0$.
\end{lemma}
\begin{proof}
Denote by $\hat n$ the outward normal to the domain, and take the dot product of the eigenvalue equation with $\hat{n}$, to deduce (using that $B \cdot \hat n=0$ on the boundary)
\begin{equation}
(u \cdot \nabla B-B \cdot \nabla u)\cdot \hat n=\eps \Delta B \cdot  \hat n.
\end{equation}
Now since $\hat n$ is simply a vector pointing in the radial direction, $ (B\cdot \nabla u )\cdot \hat n=0$, and similarly, using the modal form we see $(u \cdot \nabla B)\cdot \hat n=(im\Omega(r)+ikU(r))B_r =0$ on the boundary of the domain. Therefore, on $\partial \mathcal{M}$, it holds $\Delta B \cdot \hat n=0$. Furthermore, note that $\Delta=\nabla \times \nabla -\nabla (\div)$. An explicit computation shows that $(\nabla \times \nabla \times  B) \cdot\hat n=0$ on $\partial\mathcal{M}$, and so it holds $\nabla (\div B)\cdot \hat n=0$. But this is precisely what we need, since now it remains to multiply \eqref{eq:divergence equation} by $\overline{\div(B)}$ and observe that the boundary terms vanish due to the condition $\nabla (\div B)\cdot \hat n=0$ on $\partial\M$. Thus we deduce that $\div(B)=0$.
\end{proof}

\section{Proof of Theorems \ref{thm:dynamo} and \ref{thm:full dynamo} in the full space} \label{s:proof}

This section is devoted to the proof of the main theorems, that provide the existence of a growing mode, and gives an asymptotic profile of such mode as $\eps\to 0$. For the convenience of the reader, let us state a more concise version of Theorem \ref{thm:dynamo} specific for the case $\M=\R^2$, whose proof will be the central point of this section.

\begin{theorem}\label{theorem:mainResult}
Under Assumptions \ref{H0}--\ref{H2} for the velocity field, for all $\eps$ small enough, the operator $\L$ has an eigenvalue in $X$ of the form
\begin{equation}\label{eq:eigenvalue}
\lambda =\eps^{1/3}\left[-M^2\left(\frac{1}{r_0^2}+\frac{\Omega'_0}{U'_0}\right) + \sqrt{\frac{-2iM\Omega'_0}{r_0}}-c_2^{1/2}+o_{\eps \to 0}(1)\right].
\end{equation}
If moreover the velocity field satisfies \ref{H3}, then the associated eigenfunction $\psi\in X$ is given by 
\begin{equation}\label{eq:eigenvector}
\psi(r) = C\begin{pmatrix}
\e^{-\frac{1}{2}\eps^{-2/3}c_2^{1/2}(r-r_0)^2}\\
0
\end{pmatrix} + \psi_{\mathrm{err}}(r),
\end{equation}
with $C\in\Co\setminus\{0\}$ and such that $\psi_\err$ satisfies $\|\psi_{\mathrm{err}}\|_X = o_{\eps\to 0}(1)$.
\end{theorem}

To prove Theorem \ref{theorem:mainResult} we need to introduce an actual inverse of $\L-\lambda$, given in terms of $\G^\lambda$ as the following composition of operators
\[
(\L-\lambda)^{-1} = \G^\lambda \sum_{n=0}^\infty (-1)^n \left( (\L-\lambda)\G^\lambda - \Id \right)^n.
\]
We thus crucially need to make use of the results presented in Section \ref{s:injectivity}, as well as the following proposition that gives quantitative estimates on the operator norm of $\G^\lambda$, and measures the error produced by this Green's function, as and approximation the actual inverse near Gilbert's eigenvalue
\begin{equation}
\label{eq:gilbert eigenvalue}
\lambda^\star = \eps^{1/3}\mu^\star = \eps^{1/3}\left[ -M^2\left(\frac{1}{r_0^2}+\frac{\Omega'_0}{U'_0}\right) + \sqrt{\frac{-2iM\Omega_0'}{r_0}}-c_2^{1/2}\right].
\end{equation}
The proposition reads as follows.

\begin{proposition}\label{proposition:approx-Green's}
Let $\eta\in \Co\setminus\{0\}$ be such that $|\Re(-c_2^{1/2}\eta)|\leq 2$, pick $\gamma \in (\frac{2}{9},\frac{1}{3})$ and $\omega\in(\frac{2}{9},\gamma)$. Assume \ref{H0}--\ref{H2}, and let $\mathcal{R}$ be the compact interval in $r_0$ guaranteed by Lemma \ref{lemma:consequences of assumptions}.
Fix any compact set $\mathcal{N} \subset \mathbb{R}\setminus \{0\}$, and let
\[
\lambda = \eps^{1/3}\left( \eta - M^2\left[ \frac{1}{r_0^2} + \left(\frac{\Omega_0'}{U_0'}\right)^2 \right] + \sqrt{\frac{-2iM\Omega_0'}{r_0}} - c_2^{1/2} \right).
\]
Then, there exists a Green's function $\G^\lambda$ such that the following estimates are satisfied uniformly for $f\in Y$ with $\supp(f)\subset [0,\infty)$, $\eps>0$ sufficiently small, $r_0 \in \mathcal{R}$, and $M \in \mathcal{N}$,
\begin{equation}\label{eq:G}
\| \G^\lambda f\|_X \lesssim \varepsilon^{-1/3} \| f\|_Y,
\end{equation}
\begin{equation}\label{eq:G-error}
\| ((\L-\lambda)\G^\lambda-\Id) f\|_Y \lesssim \left( \eps^{\gamma} + \eps^{3\gamma-2/3} + \eps^{2/3-2\gamma} + \eps^\beta + \eps^{2\omega-4/9} \right) \| f\|_Y,
\end{equation}
where the coefficient $\beta>0$ is defined in \ref{P3} and all bounds are uniform in compact subsets of $\eta$.
In particular, $\G^\lambda$ may be written as 
\begin{equation*}
\G^\lambda f= \G_2^\lambda (f\mathds{1}_{[r_0-\eps^\gamma, r_0+\eps^\gamma]}) +\G_{\mathrm{extra}}(f\mathds{1}_{\mathbb{R}_+ \setminus [r_0-\eps^\gamma, r_0+\eps^\gamma]}).
\end{equation*}
Here $\G_{\mathrm{extra}}$ does not depend on $\lambda$, and $\mathcal{G}_{2}^\lambda$ is the exact inverse of $\L_2-\lambda$ on $(-\infty,\infty)$, where $\L_2$ denotes a differential operator for which $\lambda_\star\in\Co$ as defined in \eqref{eq:gilbert eigenvalue}
is an isolated eigenvalue, with associated eigenfunction
\begin{equation}\label{eq:fvec0}
f_\star = \begin{pmatrix}
    \e^{-\frac{1}{2}\eps^{-2/3}c_2^{1/2}(r-r_0)^2} \\
    0
\end{pmatrix}
\end{equation}
on $L^\infty((-\infty,\infty)\max\{1,r^2\}w_\eps(r-r_0))$. 
\end{proposition}

The proof of this statement shall be postponed to the subsequent Section \ref{s:approx-Green's}, since the construction of the Green's function $\G^\lambda$ entails some of the most technical and convoluted arguments presented in this paper.

As a last extra step before proceeding with the proof of Theorem \ref{theorem:mainResult}, we need a small interlude to discuss some preliminary results about the properties of the operator $\L$.

\begin{lemma}\label{lemma:L-closed}
The unbounded operator $(\L,\D(\L))$ with domain $\mathcal{D}(\L)$ given by
\[
\mathcal{D}(\L)=\left\lbrace b \in H^2_{\mathrm{loc}}(0,\infty)\cap X:\mathcal{L}b \in X\right\rbrace
\]
is closed.
\end{lemma}

\begin{proof}
We begin by noting that, since any $b \in \mathcal{D}(\mathcal{L})$ is in $H^2_{\mathrm{loc}}$, the action of $\mathcal{L}$ on $\mathcal{D}(\mathcal{L})$ is well defined, and hence the definition of $\mathcal{D}(\mathcal{L})$ makes sense. Suppose now that we have a sequence $(B_n)_n$ and two elements $b,H\in X$ such that 
\[
b_n \to b, \quad \L b_n \to h
\]
as $n\to\infty$ in $X$. Fix a smooth compactly supported vector valued $\phi \in C_c^\infty(0,\infty)$. Then, it holds 
\[
\int_{0}^\infty \mathcal{L}B_n(x) \cdot \phi(x) \dd x = \int_{0}^\infty b_n(x) \cdot \mathcal{L}^*\phi(x) \dd x \to \int_0^\infty b(x) \cdot \mathcal{L}^*\phi(x) \dd x
\]
as $n \to \infty$, where $\mathcal{L}^*$ denotes the formal adjoint of $\mathcal{L}$, and convergence follows since $\mathcal{L}^*\phi$ is a compactly supported, bounded function. However, by assumption it also holds that 
\[
\int_{0}^\infty \mathcal{L}b_n(x) \cdot \phi(x) \dd x \to \int_0^\infty h(x) \cdot \phi(x) \dd x
\]
and so we deduce that 
\begin{equation}
\label{eq:distributionalFormulation}
\int_0^\infty b(x) \cdot \mathcal{L}^*\phi(x) \dd x=\int_0^\infty h(x) \cdot \phi(x) \dd x
\end{equation}
for any $\phi \in C_c^\infty(0,\infty)$. We thus conclude that $\mathcal{L}b=h$ on $(0,\infty)$ in a distributional sense. We note that since $H \in X$, it holds that $H \in L^2_{\mathrm{loc}}(0,\infty)$. Therefore, by elliptic regularity, see for instance \cite{Folland96}*{Theorem 6.33}, there holds that $b \in H^2_{\mathrm{loc}}(0,\infty)$, so that in fact $b \in \mathcal{D}(\mathcal{L})$. Finally, since $b \in H^2_{\mathrm{loc}}(0,\infty)$ and $\phi$ has compact support we may integrate by parts in \eqref{eq:distributionalFormulation} to obtain 
\[
\int_0^\infty \mathcal{L}b\cdot \phi(x) dx=\int_0^\infty h(x) \cdot \phi(x) dx
\]
which by the Fundamental Lemma of Calculus of Variations implies that $\mathcal{L}b=h$ almost everywhere on $(0,\infty)$, and this completes the proof.
\end{proof}

With this tool in our bag, we can directly move on to the proof of Theorem \ref{theorem:mainResult}.

\begin{proof}[Proof of Theorem \ref{theorem:mainResult}]
The proof of this theorem will be split in three different pieces. First of all we will use the approximated Green's functions to construct a resolvent operator, secondly we will define a Riesz projector to prove that $\L$ has a spectral point nearby the candidate eigenvalue, and finally we will prove the existence of the corresponding eigenfunction.

\emph{Step 1}. We begin by using the approximate Green's function $\G^\lambda$ to construct a resolvent operator. As further elaborated in Section \ref{s:approx-Green's}, $\G^\lambda$ is given in terms of a linear combination of different Green's functions associated to approximations of the problem in different regions,
\[
\G^\lambda f = \G_2^\lambda(f \dsOne_{[r_0-\eps^\gamma,r_0+\eps^\gamma)}) + \G_{\mathrm{extra}}(f \dsOne_{\R_+\setminus [r_0-\eps^\gamma,r_0+\eps^\gamma)}),
\]
where $G_{\mathrm{extra}}$ is defined accordingly depending on the number of zeroes of the transport function $T(r)$, e.g.\ if $\F_0$ from Assumption \ref{H1} is empty, then
\[
\G_{\mathrm{extra}}(f \dsOne_{\R_+\setminus [r_0-\eps^\gamma,r_0+\eps^\gamma)}) = \G_0(f \dsOne_{[0,\eps^\gamma)}) + \G_1(f \dsOne_{[\eps^\gamma,r_0-\eps^\gamma)}) + \G_3(f \dsOne_{[r_0+\eps^\gamma,\infty)}).
\]
Crucially, the superscript $\lambda$ appears exclusively in $\G_2^\lambda$ since it is the only piece of the approximate Green's function that depends on the choice of $\lambda = \mu\eps^{1/3}$, regardless of the number of zeroes of $T(r)$. Hence, we define $\mu$ to be given by an expression of the form
\[
\mu = -M^2\left(\frac{1}{r_0^2}+\frac{\Omega'_0}{U'_0}\right) + \sqrt{\frac{-2iM\Omega'_0}{r_0}}-c_2^{1/2}-\eta,
\]
where $\eta\in\Co$ is chosen as in Proposition \ref{proposition:approx-Green's}, namely $\eta\neq 0$ and sufficiently small in modulus. For any fixed $\eta$, we find that $\lambda = \mu\eps^{1/3}$ traces out a circumference of radius $\eps^{1/3}|\eta|>0$ around Gilbert's eigenvalue
\[
\eps^{1/3}\mu_\star = \eps^{1/3} \left[-M^2\left(\frac{1}{r_0^2}+\frac{\Omega'_0}{U'_0}\right) + \sqrt{\frac{-2iM\Omega'_0}{r_0}}-c_2^{1/2}\right]
\]
in the complex plane. We shall denote this curve by $\Gamma\subset\Co$. Next in order, notice from Proposition \ref{proposition:approx-Green's} that there exists a positive number $a>0$ such that for any $f\in Y$ and any $\eps>0$ sufficiently small,
\begin{equation}\label{eq:error-estimate-total}
\|(\L-\lambda)\G^\lambda f -  f\|_Y \lesssim \eps^a \|f\|_Y.
\end{equation}
In addition, we claim that the inverse of the exact operator $\L-\lambda$ is given by the map
\begin{equation}\label{eq:inverseL}
\G_\L^\lambda = \G^\lambda \sum_{n=0}^\infty (-1)^n\left((\L-\lambda)\G^\lambda - \Id\right)^n.
\end{equation}
First of all we will show that \eqref{eq:inverseL} is a right inverse. Indeed, notice from \eqref{eq:error-estimate-total} that
\[
\left\lVert \left((\L-\lambda)\G^\lambda - \Id\right)^n \right\rVert \lesssim \eps^{na},
\]
which is summable in $n$ for $\eps$ small. Here we use $\|\cdot\|$ to denote the operator norm $Y\to Y$. Moreover, since $\G^\lambda$ maps $Y$ to $X$, we deduce that the operator $(\L-\lambda)^{-1}$ is well-defined. In addition, for any fixed $N\in \N$ we can estimate
\[
\begin{split}
    & \left\lVert (\L-\lambda) \G^\lambda \sum_{n=0}^N (-1)^n\left((\L-\lambda)\G^\lambda - \Id\right)^n f - f \right\rVert_Y \\
& \quad = \left\lVert \sum_{n=0}^{N} (-1)^{n}\left((\L-\lambda)\G^\lambda - \Id\right)^{n+1} f + \sum_{n=0}^{N} (-1)^{n}\left((\L-\lambda)\G^\lambda - \Id\right)^n f - f \right\rVert_Y  \\
& \quad = \left\lVert \sum_{n=1}^{N+1} (-1)^{n-1}\left((\L-\lambda)\G^\lambda - \Id\right)^n f - \sum_{n=0}^{N} (-1)^{n-1}\left((\L-\lambda)\G^\lambda - \Id\right)^n f - f  \right\rVert_Y \\
& \quad = \left\lVert (-1)^{N}\left((\L-\lambda)\G^\lambda - \Id\right)^{N+1} f \right\rVert_Y \lesssim \eps^{(N+1)a} \left\lVert f \right\rVert_Y,
\end{split}
\]
which converges to zero as $N\to\infty$. Therefore, mimicking the proof of closedness of the operator from Lemma \ref{lemma:L-closed}, there holds that $\G_\L^\lambda$ maps $X$ into $\mathcal{D}(\L)$, and $(\L-\lambda)\G_\L^\lambda= \Id$. We now mention a subtle point that needs to be dealt with. Throughout this section, we will frequently argue by saying that $\G_\L^\lambda=(\L-\lambda)^{-1}$, where $(\L-\lambda)^{-1}$ denotes the functional analytic resolvent, and thus derive properties about said resolvent using our explicit form for $\G_\L^\lambda$. However, thus far we have only showed that $\G_\L^\lambda$ is a \emph{right inverse} to $(\L-\lambda)$, and so $(\L-\lambda)^{-1}$ may not even exist. From the existence of a right inverse, we know that $(\L-\lambda)$ is surjective, but injectivity is still an open question. However, in Lemmas \ref{lemma:injectivity1} and \ref{lemma:injectivity2} from Section \ref{s:injectivity}, we show that under \ref{H3}, $\L :\mathcal{D}(\L) \to X$ admits at most countably many, isolated eigenvalues with positive real part, which do not have an accumulation point. Thus, given our Jordan curve $\Gamma \subset \{\Re(z) >0\}$ for which $\G_\L^\lambda$ defines a right inverse of $\L-\lambda$ for all $\lambda $ in a neighbourhood of $\Gamma$, we may simply deform $\Gamma$ slightly and conclude that there exists a curve $\tilde \Gamma$ on which both $\L-\lambda$ is injective, and $\G_\L^\lambda$ is a right inverse for $\L-\lambda$. Hence, for any $\lambda \in \tilde{\Gamma}$, by surjectivity we know that there exists a (possibly unbounded) inverse map $(\L-\lambda)^{-1}$, and for any $f \in X$ it holds write
\[
(\L-\lambda)((\L-\lambda)^{-1}f-\G_\L^\lambda f) = 0,
\]
so that by the injectivity of $\L-\lambda$, it must hold $(\L-\lambda)^{-1}f=\G_\L^\lambda f$ for all $f \in X$. Hence $(\L-\lambda)^{-1}=\G_\L^\lambda$, and since $\G_\L^\lambda$ is bounded, we conclude that $\lambda \in \rho(\L)$, the resolvent set of $\L$. 
Hence, \eqref{eq:inverseL} gives a well-defined right and left inverse of the operator $\L-\lambda$ for all $\lambda\in\tilde\Gamma$ and $\eps>0$ sufficiently small. From now on, we shall abuse notation and identify $\Gamma$ with $\tilde \Gamma$. We claim that from this, we can deduce the existence of a growing mode for $\L$, and thus complete the proof of Theorem \ref{theorem:mainResult}.

\emph{Step 2}. At this stage we will show that the operator $\L$ has an spectral point inside the curve $\Gamma$. In order to do so we define the Riesz projector
\[
Pf = \frac{1}{2\pi i} \int_{\Gamma} (\lambda-\L)^{-1} f \dd \lambda.
\]
A detailed account of the Riesz projector may be found in \cite{Kato}*{III-5,6,7}. For a summary of these results, see also \cite{CZSV25}*{Appendix A}. We introduce this object due to the fact that standard properties of Riesz projectors ensure that if there exists $\phi\in X$ such that $P\phi\neq 0$, then $\L$ has a spectral point in the interior of $\Gamma$, see e.g.\ \cite{Kato}. The goal now is to find such $\phi$. We use Proposition \ref{proposition:approx-Green's} to find the estimate
\[
\begin{split}
    \left\lVert \G^\lambda \sum_{n=1}^\infty (-1)^n \left( (\L-\lambda)\G^\lambda - \Id \right)^n f \right\rVert_X & \lesssim \eps^{-1/3} \left\lVert \sum_{n=1}^\infty (-1)^n \left( (\L-\lambda)\G^\lambda - \Id \right)^n f \right\rVert_Y \\
    & \lesssim \eps^{-1/3} \sum_{n=1}^\infty \eps^{na} \|f\|_Y \leq \eps^{a-1/3}\|f\|_Y.
\end{split}
\]
In particular, note that
\[
\begin{split}
   Pf - \frac{1}{2\pi i}\int_\Gamma (-\G^\lambda)f\dd \lambda & = \frac{1}{2\pi i}   \int_\Gamma \left[(\lambda-\L)^{-1} - (\lambda-\L)^{-1}(\lambda-\L)(-\G^\lambda)\right] f\dd \lambda   \\
    & = \frac{1}{2\pi i}   \int_\Gamma (\L-\lambda)^{-1} \left((\L-\lambda)\G^\lambda - \Id \right) f\dd \lambda  \\
    & = \frac{1}{2\pi i}   \int_\Gamma \G^\lambda \sum_{n=0}^\infty (-1)^n \left((\L-\lambda)\G^\lambda - \Id \right)^{n+1} f\dd \lambda,
\end{split}
\]
hence, combining both expressions we obtain
\[
\begin{split}
    \left\lVert Pf - \frac{1}{2\pi i}\int_\Gamma (-\G^\lambda)f\dd \lambda  \right\rVert_X & \lesssim \eps^{1/3} |\eta|\sup_{\lambda\in\Gamma} \left\lVert \G^\lambda \sum_{n=1}^\infty (-1)^{n-1} \left( (\L-\lambda)\G^\lambda - \Id \right)^n f \right\rVert_X \\
    & \lesssim \eps^a|\eta| \|f\|_Y,
\end{split}
\]
where recall that $\eps^{1/3}|\eta|$ is the radius of the  circle $\Gamma\subset\Co$. Furthermore, note that when integrating $\G^\lambda$ in $\Gamma$, all components of the Green's function $\G^\lambda$ that do not depend on $\lambda$---namely $G_{\mathrm{extra}}$---will vanish due to Cauchy's Theorem. Hence, using the decomposition of Proposition \ref{proposition:approx-Green's}, we can write
\[
\frac{1}{2\pi i}\int_\Gamma (-\G^\lambda)f\dd \lambda = \frac{1}{2\pi i}\int_\Gamma (-\G_2^\lambda)\left(f\dsOne_{[r_0-\eps^\gamma,r_0+\eps^\gamma)}\right)\dd \lambda.
\]
Therefore, as a candidate we want to choose a function that has the property of being an eigenfunction for the approximated operator $\L_2$, namely the operator corresponding to Green's function $\G_2^\lambda$. This operator is defined by
\[
\L_2 V = \eps\partial_r^2 V - \left[c_2\eps^{-1/3}(r-r_0)^2 - \eps^{1/3}M^2\left(\frac{1}{r_0^2} + \frac{\Omega_0'}{U_0'}\right)\right] V + \sqrt{\frac{-2iM\Omega_0'}{r_0}} \begin{pmatrix}
    \displaystyle V_1 \\
    \displaystyle - V_2
\end{pmatrix}
\]
see Section \ref{s:around-r0} for further explanation. We pick
\begin{equation*}
f_\star = \begin{pmatrix}
    \e^{-\frac{1}{2}\eps^{-2/3}c_2^{1/2}(r-r_0)^2} \\
    0
\end{pmatrix},
\end{equation*}
because it is an eigenfunction of the operator $\L_2$ in $L^\infty(\mathbb{R},w_\eps(r-r_0)\max\{1,r^2\})$ with an isolated eigenvalue
\begin{equation}
\label{eq:lambda star}
\lambda_\star = \mu_\star\eps^{1/3} = \eps^{1/3} \left[-M^2\left(\frac{1}{r_0^2} + \frac{\Omega_0'}{U_0'}\right) + \sqrt{\frac{-2iM\Omega_0'}{r_0}} - c_2^{1/2} \right].
\end{equation}
The obtention of this eigensystem is a straightforward computation that we include in the Appendix \ref{s:around-r0-app} for the convenience of the reader. Since $\G_2^\lambda$ is the Green's function associated to the operator $\L_2-\lambda$, there holds that
\begin{equation}\label{eq:riesz-prop1}
\frac{1}{2\pi i}\int_\Gamma (-\G_2^\lambda)f_\star \dd\lambda = f_\star,
\end{equation}
due to standard properties of the Riesz projector, since $\lambda_\star \in \mathrm{int}(\Gamma)$. Moreover, for any $\eps>0$ sufficiently small we have that
\[
\|f_\star(r) - f_\star\dsOne_{[r_0-\eps^\gamma,r_0+\eps^\gamma)}(r)\|_X =\sup_{|r-r_0|\geq \eps^{\gamma}}|w_\eps(r-r_0)\e^{-\eps^{-2/3}(r-r_0)^2\Re(c_2^{1/2})|}| \leq \eps\|f_\star\|_X,
\]
where we crucially need $\gamma<1/3$. Hence, using \eqref{eq:riesz-prop1} and that the Riesz projector is a bounded operator with norm $1$, we observe that
\[
\begin{split}
    \left\lVert \frac{1}{2\pi i}\int_\Gamma (-\G^\lambda)f_\star\dd \lambda \right\rVert_X & \geq \|f_\star\|_X - \left\lVert \frac{1}{2\pi i}\int_\Gamma (-\G^\lambda)\left(f_\star - f_*\dsOne_{[r_0-\eps^\gamma,r_0+\eps^\gamma)} \right)\dd \lambda \right\rVert_X \\
    & \geq (1-\eps)\|f_\star\|_X.
\end{split}
\]
Finally, we combine all the estimates to conclude with
\[
\begin{split}
    \|P f_*\|_X & \geq \left\lVert \frac{1}{2\pi i}\int_\Gamma (-\G^\lambda)f_\star\dd \lambda \right\rVert_X - \left\lVert Pf_*- \frac{1}{2\pi i}\int_\Gamma (-\G^\lambda)f_\star \dd \lambda \right\rVert_X\\
    &\gtrsim (1-\eps-\eps^a)\|f_\star\|_X.
\end{split}
\]
Therefore we see that if $\eps>0$ is small enough, then $\| P\phi\|_X>0$, and thus the range of the Riesz projector is non-empty.

\emph{Step 3}. Last, we move on to showing that we can actually infer the existence of an eigenfunction, since so far we have only shown that there exists a spectral value of $\L$ in the interior of $\Gamma$. To show that an eigenfunction exists it suffices to prove that the range of $P$ is finite dimensional, and indeed, we will show that it is one-dimensional. We claim that for any $\psi\in X$ in the range of $P$ with $\|\psi\|_X=1$, there exist a constant $C\in\mathbb{C}$, $C\neq 0$, and a function $\psi_\err\in X$ so that 
\[
\psi = Cf_\star + \psi_\err.
\]
Moreover, $\psi_\err$ satisfies $\|\psi_\err\|_X\leq c_0 \eps^a$ for some $a>0$, $c_0>0$ independent of $\psi$. Once this is established, an application of Riesz' Lemma immediately implies that the range of $P$ is one dimensional. To see this, on the one hand we notice that the range of the operator
\[
\frac{1}{2\pi i}\int_\Gamma (-\G^\lambda_2)\dd \lambda,
\]
is precisely $\vspan(f_\star)$.  On the other hand, using these properties we can write 
\[
\begin{split}
    \| \psi - Cf_\star \|_X & \leq \left\lVert \psi - \frac{1}{2\pi i}\int_\Gamma (-\G_2^\lambda)(\mathds{1}_{[r_0-\eps^\gamma,r_0+\eps^\gamma)}\psi)\dd \lambda  \right\rVert_X \\
    & \quad + \left\lVert \frac{1}{2\pi i}\int_\Gamma (-\G_2^\lambda)(\mathds{1}_{[r_0-\eps^\gamma,r_0+\eps^\gamma)}\psi) \dd \lambda - Cf_\star \right\rVert_X.
\end{split}
\]
The first addend in the right hand side is bounded by $\eps^a|\eta|$, following the same techniques that we addressed in Step 2. The second addend directly vanishes, upon an appropriate choice of the constant $C$. All in all, the claim of the theorem follows from here, assuming \ref{H3}.

Finally, we remark that, if \ref{H3} is not satisfied and the injectivity of the operator $\L-\lambda$ cannot be established, we can still guarantee the existence of a growing mode with an eigenvalue $\lambda\in\Gamma$ of the form stated in Theorem \ref{theorem:mainResult}. Indeed, if injectivity fails for some $\lambda$ close to $\eps^{1/3}\mu_\star$, then by definition there exists some $\phi \in \ker(\L-\lambda)$, and so we have a growing mode. Hence, the proof of the theorem is complete.
\end{proof}

With this result under our belt, we now prove the existence of growing modes for the kinematic dynamo equations under the additional assumption that the compact set $\mathcal{R}_0$ from \ref{H2} contains an interval.

\begin{proof}[Proof of Theorem \ref{thm:full dynamo}.]
Under \ref{H2}, the bounds in Proposition \ref{proposition:approx-Green's} are uniform for $r_0 \in \mathcal{R}$, and $M$ in any compact subset $\mathcal{N} \subset \mathbb{R}\setminus \{0\}$.  Therefore, for any fixed compact $\mathcal{N} \subset \mathbb{R}\setminus \{0\}$, we may pick $\eps>0$ small enough so that the conclusion of Theorem \ref{theorem:mainResult} holds uniformly in $r_0 \in \mathcal{R}$, $M \in \mathcal{N}$. Now, in order for a function of the form $b(r)$ to define a growing mode of the kinematic dynamo equations via the modal form $B(r,\theta,z)=b(r)e^{i\eps^{-1/3}(M\theta+Kz)}$, it must be that $\eps^{-1/3}M, \eps^{-1/3}K \in \mathbb{Z}$. However, note that $K$ is related to $M$ via 
\[
K=-\frac{\Omega'(r_0)}{U'(r_0)}M.
\]
Thus, fix now some $M \in \mathcal{N}$, and  note that for all $\eps>0$ small enough, there exists $\tilde{M}(\eps) \in [-1,1]$ so that $(M+\eps^{1/3}\tilde{M}(\eps))\eps^{-1/3} \in \mathbb{Z}$. Next, note that the quotient $
\Omega'(r_0)/U'(r_0)$ is not constant for $r_0 \in \mathcal{R}$. Indeed, suppose that it was constant for $r_0 \in \mathcal{R}$. Then, we compute 
\begin{equation*}
T'(r)=U'(r)\left ( \frac{\Omega'(r)}{U'(r)}-\frac{\Omega'(r_0)}{U'(r_0)}\right )=0
\end{equation*}
for $r \in \mathcal{R}$. But this implies that $T(r)=T(r_0)=0$ for all $r \in \mathcal{R}$, which contradicts \ref{H1}. Furthermore, $\Omega'(r_0)/U'(r_0)$ defines a continuous function in $r_0$, since by assumption $\Omega, U$ are $C^3$.
In particular, there exists $r_0 \in \mathcal{R}$ so that 
\[
r_0 \left |\frac{\dd }{\dd r}\log\left |\frac{\Omega'(r_0)}{U'(r_0)}\right | \right |<4,
\]
and this must further in fact hold true in some subinterval of $\mathcal{R}$, and so we replace $\mathcal{R}$ by this subinterval.
By the intermediate value theorem, the image of $\mathcal{R}$ under the map 
\[
r \mapsto \frac{\Omega'(r)}{U'(r)}
\]
contains a non-empty interval $[p_1,q_1] \subset \mathbb{R}$. Thus, having fixed $M$ so that $M\eps^{-1/3} \in \mathbb{Z}$, as soon as $|q_1-p_1|>M^{-1}\eps^{1/3}$, there exists $r_0 \in \mathcal{R}$ so that 
\[
\frac{\Omega'(r_0)}{U'(r_0)}M\eps^{-1/3} \in \mathbb{Z}.
\]
Hence, for all $\eps>0$ sufficiently small, we may pick $M \in \mathcal{N}$, $r_0 \in \mathcal{R}$ according to the above recipe, deduce the existence of a growing mode for $\L$ with these choices of $M, r_0$, and thus define a growing mode for the kinematic dynamo equations of the form $B(r,\theta,z,t)=b(r)e^{i\eps^{-1/3}(M\theta+Kz)+\lambda t}$,
with $\Re(\lambda)=\eps^{1/3}(\mu+o_{\eps \to 0}(1))$, as desired.
\end{proof}

\section{Construction of the Green's functions: proof of Proposition \ref{proposition:approx-Green's}}\label{s:approx-Green's}

In this section, we conduct a detailed analysis of the operator $\G^\lambda$, which serves as an approximation to the inverse of $\L-\lambda$ in the space $Y$. As outlined in Section \ref{s:proof}, this operator plays a crucial role in defining the true inverse \eqref{eq:inverseL}, hence this section is mostly devoted to proving the quantitative estimates from Proposition \ref{proposition:approx-Green's}. 
Following the framework introduced in Section \ref{s:proof}, we assume throughout this section that the problem is posed on the full space $\M=\R^2$, thereby avoiding complications related to boundary conditions. The arguments here presented will be extended to the case of domains with boundaries in Section \ref{s:boundaries}.

The complexity of the Green's function $\G^\lambda$ is directly influenced by the number of zeros of the transport function $T(r)$, which is determined by the (given) velocity field $u$. Under Assumption \ref{H1}, $T(r)$ will always vanish up to second order only at $r=r_0$, and it might vanish linearly in a finite collection of points $s_j\neq r_0$. Since the construction of $\G^\lambda$ is intricate, we will first consider the following extra assumption:
\begin{itemize}
    \item[\namedlabel{SH}{SH}] The transport function $T(r)$ only vanishes at $r=r_0$, i.e.\ the set $\F_0$ from \ref{H1} is empty.
\end{itemize}

This simplification allows us to present the core ideas in a more streamlined manner.
We maintain Assumption \ref{SH} throughout Sections \ref{s:around-r0}, \ref{sec:tow-inf}, \ref{s:tow-zero} and \ref{sec:near-0}, removing it only for Section \ref{s:linear-zeroes}, where Proposition \ref{proposition:approx-Green's} is proved in full generality.

Assumptions \ref{H0}--\ref{H3} together with \ref{SH} already gives a very wide class of admissible vector fields, including polynomials, or examples with finite kinetic energy, e.g.\ $\Omega(r) = \e^{-ar^2}$, $U(r) = \e^{-br^2}$ with $a\neq b$. In this scenario, the approximate Green's function $\G^\lambda$ is defined as the sum of the four contributions,
\begin{equation}\label{eq:Glambda-total}
\G^\lambda f = \G_0 \left(f\dsOne_{[0,\eps^\gamma)}\right) +  \G_1 \left(f\dsOne_{[\eps^\gamma,r_0-\eps^\gamma)}\right) +  \G_2^\lambda \left(f\dsOne_{[r_0-\eps^\gamma,r_0+\eps^\gamma)}\right) +  \G_3 \left(f\dsOne_{[r_0+\eps^\gamma,\infty)}\right),
\end{equation}
associated respectively to the contributions near the origin, towards the origin, around the critical radius, and towards infinity. The operators $\G_0$, $\G_1$, $\G_2^\lambda$ and $\G_3$ are themselves Green's functions associated to differential equations that approximate the exact problem \eqref{eq:1}--\eqref{eq:2} in a tailor-made manner for each of their domains of definition.

The claim of Proposition \ref{proposition:approx-Green's} under the additional assumption \ref{SH} is a straightforward consequence of Propositions \ref{proposition:r0}, \ref{proposition:rinf}, \ref{proposition:rtow0} and \ref{prop:G0}, that will be discussed in the subsequent sections. As it can be observed from the definition of $\G^\lambda$ in \eqref{eq:Glambda-total}, we divide the support of $f\in Y$ in four distinctive parts: $[0,\eps^\gamma)$, $[\eps^\gamma,r_0-\eps^\gamma)$, $[r_0-\eps^\gamma,r_0+\eps^\gamma)$, and $[r_0+\eps^\gamma,\infty)$. Each piece of this puzzle presents its own intrinsic difficulties. 
In order to prove Proposition \ref{proposition:approx-Green's} in full generality we need to combine the arguments from Sections \ref{s:around-r0}--\ref{sec:near-0} with Lemma \ref{lemma:linear-vanish}. In particular, the complete proof is discussed by the end of Section \ref{s:linear-zeroes}.

\subsection{Green's function around the critical radius}\label{s:around-r0}

We start by studying the problem in the region where we expect the growth we are seeking to occur, that is, near the critical radius $r_0>0$. We look for an operator for which we can compute exact inverses, and that approximates the exact operator$\L-\lambda$ for values of $r$ very close to $r_0$. We will do so by looking at $\L-\lambda$ in \eqref{eq:1}--\eqref{eq:2} and removing terms that are sufficiently small in this regime. Thus, we drop $D_\eps^+$ and $D_\eps^-$ since they will be very small as $r\to r_0$. Moreover, we eliminate all terms in $\L$ of order \emph{strictly} larger than $1/3$ in $\varepsilon$, which effectively produces an approximation of equations \eqref{eq:1} and \eqref{eq:2} that is fully decoupled. Notice that crucially we do not drop the term $\lambda = \mu\eps^{1/3}$.

We also need to find suitable estimates for the transport term $T(r)$, which corresponds to a term of order $-1/3$ in $\eps$. To do so, we fall back on the intuition provided by Gilbert in \cite{gilbert1988} (see also Section \ref{s:gilbert-scaling}). Observe that a Taylor expansion for $\Omega$ and $U$ around $r_0$ yields
\[
\Omega-\Omega_0 = \Omega_0'(r-r_0) + \frac{1}{2}\Omega_0''(r-r_0)^2 + \Omega_{\text{hot}},
\]
\[
U-U_0 = U_0'(r-r_0) + \frac{1}{2}U_0''(r-r_0)^2 + U_{\text{hot}},
\]
where both $\Omega_{\text{hot}}$ and $U_{\text{hot}}$ are $O(|r-r_0|^3)$. Therefore, we can write
\[
i\eps^{-1/3}T(r) = \frac{i\eps^{-1/3}}{2}\left( \Omega_0''-\frac{\Omega_0'}{U_0'}U_0'' \right) (r-r_0)^2 + T_{\text{hot}},
\]
since the linear terms in $r$ vanish thanks to Gilbert's condition \eqref{eq:Gilbert2}, i.e.\ $M\Omega_0'+KU_0'=0$. It is important to notice that the remainder $T_{\textrm{hot}}$ is of order
\begin{equation}\label{eq:THOT}
T_{\text{hot}} = O(\varepsilon^{-1/3}|r-r_0|^3),
\end{equation}
and by Lemma \ref{lemma:consequences of assumptions}, \ref{P2}, the $O(\varepsilon^{-1/3}|r-r_0|^3)$ is in fact \emph{uniform} in $r_0 \in \mathcal{R}$.
Finally, we will also omit the first order derivatives $\varepsilon r^{-1}\partial_r$ for the approximated operator. All in all, we obtain the following equations,
\begin{equation}\label{eq:r0-1}
\varepsilon\partial_r^2V_1 - c_2\varepsilon^{-1/3}(r-r_0)^2V_1 - \varepsilon^{1/3}\left(  M^2\left[ \frac{1}{r_0^2} + \left(\frac{\Omega_0'}{U_0'}\right)^2 \right] +\mu -\sqrt{\frac{-2iM\Omega_0'}{r_0}} \right)V_1 = f_1,
\end{equation}
\begin{equation}\label{eq:r0-2}
\varepsilon\partial_r^2V_2 - c_2\varepsilon^{-1/3}(r-r_0)^2V_2 - \varepsilon^{1/3}\left(  M^2\left[ \frac{1}{r_0^2} + \left(\frac{\Omega_0'}{U_0'}\right)^2 \right] +\mu +\sqrt{\frac{-2iM\Omega_0'}{r_0}} \right)V_2 = f_2,
\end{equation}
where recall that following Gilbert's notation \eqref{eq:P2}, we introduce the constant
\begin{equation*}
c_2 = \frac{iM}{2}\left( \Omega_0''-\frac{\Omega_0'}{U_0'}U_0'' \right).
\end{equation*}
We denote by $\L_2$ the operator associated to this system of ODEs, excluding the factor $\lambda=\mu\eps^{1/3}$ in its definition, namely the system \eqref{eq:r0-1}--\eqref{eq:r0-2} can be written as
\[
(\L_2-\lambda)V = f.
\]
Throughout this paper, unless otherwise specified, $V_1$ and $V_2$ will denote the components of the vector $V=(V_1,V_2)$, and similarly we write $f=(f_1,f_2)$. 

As is standard, one way of inverting the operator $\L_2-\lambda$ is by computing a Green's operator $\G_2^\lambda$, which is related to its corresponding Green's kernel $G_2^\lambda$ by
\[
(\L_2-\lambda)V(r) = f(r), \quad V(r) = \G_2^\lambda f(r) = \int_0^\infty G_2^\lambda(r,s) f(s)\dd s.
\]
We now endeavour to find estimates on this Green's operator $\G_2^\lambda$. Indeed, for any $f$ compactly supported in a small neighbourhood of $r_0$ we find the following result.

\begin{proposition}\label{proposition:r0}
Let $\eta\in\Co$, $\eta\neq 0$, be such that $|\Re(-c_2^{1/2}\eta)|\leq 2$,  assume \ref{H0}--\ref{H2}, and choose a coefficient $\gamma>\frac{2}{9}$. Let $\mathcal{R}$ be the compact set from Lemma \ref{lemma:consequences of assumptions}, and let $\mathcal{N} \subset \mathbb{R}\setminus \{0\}$ be any compact set.
Let $\G_{2}^\lambda$ be the Green's function associated to the system of equations \eqref{eq:r0-1}--\eqref{eq:r0-2}, with
\[
\mu = \eta - M^2\left[ \frac{1}{r_0^2} + \left(\frac{\Omega_0'}{U_0'}\right)^2 \right] + \sqrt{\frac{-2iM\Omega_0'}{r_0}} - c_2^{1/2}.
\]
Then, uniformly for $r_0 \in \mathcal{R}$, $M \in \mathcal{N}$, and for $f\in Y$ such that $\supp(f)\subset [r_0-\varepsilon^\gamma,r_0+\varepsilon^\gamma)$, the following estimates hold true for all $\eps>0$ sufficiently small.
\begin{equation}\label{eq:Gr0}
\|\G_2^\lambda f\|_X \lesssim \varepsilon^{-1/3} \|f\|_Y,
\end{equation}
\begin{equation}\label{eq:drGr0}
\| \partial_r\G_{2}^\lambda f\|_X \lesssim \eps^{-1+\gamma} \|f\|_Y,
\end{equation}
\begin{equation}\label{eq:Gr0-error}
\|((\L-\mu\varepsilon^{1/3})\G_{2}^\lambda-\Id) f \|_Y \lesssim \left( \eps^\gamma + \eps^{3\gamma-2/3} \right) \|f\|_Y.
\end{equation}
Furthermore, all estimates are uniform in compact subsets of $\eta$.
\end{proposition}

\begin{remark}\label{rmk:exp-decay}
    As it will be clear throughout the proof of the proposition, estimate \eqref{eq:G} corresponds to the worst case scenario, that occurs when $r\in \supp(f)$. If $r\not\in\supp(f)$ we crucially find an exponentially decaying estimate of the form
    \[
    |w_\eps(r-r_0)\G^\lambda_2 f(r)|\lesssim \eps^{-1/3} \e^{-\frac{1}{2}\chi\eps^{-2/3}\left(|r-r_0|^2-\eps^{2\gamma}\right)}\|f\|_Y,
    \]
    for some constant $\chi>0$ defined in Lemma \ref{lemma:bounds-v1v2}. This property will be essential in order to ensure that the \emph{global} error estimate \eqref{eq:Gr0-error} holds.
\end{remark}

Before proceeding with the proof of the proposition, let us make a comment on how to derive the Green's function associated to the approximated problem \eqref{eq:r0-1}--\eqref{eq:r0-2}. Notice that now \eqref{eq:r0-1}--\eqref{eq:r0-2} consists of two decoupled equations that can be studied individually, hence we can derive a separate Green's function for each component, 
\[
\G_2^\lambda f = \begin{pmatrix}
    \G_{2,1}^\lambda f_1\\
    \G_{2,2}^\lambda f_2
\end{pmatrix}.
\]
We start by analysing $\G_{2,1}^\lambda$, then it will become clear that the analysis of $\G_{2,2}^\lambda$ is completely analogous. In order to derive a Green's function, we first look for two independent solutions of \eqref{eq:r0-1}. This equation falls into the umbrella term of \emph{Weber differential equations}, which are ordinary differential equations of the form
\begin{equation}\label{eq:weber}
y''(z) + \left(\nu + \frac{1}{2} - \frac{1}{4}z^2\right)y(z) = 0,
\end{equation}
with $\nu\in\Co$.
Two independent solutions to this Weber equations can be found by means of the \emph{parabolic cylinder functions}: $y_1(z) = D_\nu(z)$, $y_2(z) = D_\nu(-z)$, and their Wronskian is given by
\[
W(y_1(z),y_2(z)) = \frac{\sqrt{2\pi}}{\Gamma(-\nu)},
\]
see \cite{Lebedev72} and Appendix \ref{s:appendix-parabolic} for extended information about this type of special functions. We can rewrite equation \eqref{eq:r0-1} in the form
\[
v''(z) - \zeta^2\eps^{2\beta}\left(q(\mu)\varepsilon^{-2/3} + c_2\zeta^2\eps^{2\beta-4/3}z^2\right)v(z) = 0,
\]
where we made the change of variables $\zeta \eps^\beta z=r-r_0$, and set
\[
q(\mu) = M^2\left( \frac{1}{r_0^2} + \left(\frac{\Omega_0'}{U_0'}\right)^2 \right) +\mu -\sqrt{\frac{-2iM\Omega_0'}{r_0}}.
\]
In order to obtain an equation as \eqref{eq:weber}, we make the choice $\beta = 1/3$ and $c_2\zeta^4 = 1/4$, and without loss of generality we pick 
\[
\zeta = \frac{c_2^{-1/4}}{\sqrt{2}},
\]
so that
\[
v''(z) + \left(-\frac{1}{2}c_2^{-1/2}q(\mu)-\frac{1}{4}z^2\right)v(z) = 0.
\]
Therefore, we find solutions to this Weber differential equation in terms of parabolic cylinder functions with index
\[
\nu = -\frac{1}{2}\left(c_2^{-1/2}q(\mu) +1\right) = -\frac{1}{2}c_2^{-1/2}\eta,
\]
for some number $\eta\in\Co$ such that
\[
\eta = q(\mu) + c_2^{1/2} = \mu + M^2\left( \frac{1}{r_0^2} + \left(\frac{\Omega_0'}{U_0'}\right)^2 \right)  -\sqrt{\frac{-2iM\Omega_0'}{r_0}} + c_2^{1/2}.
\]
Notice that to study the second component of the differential equation \eqref{eq:r0-2}, we follow the same analysis with the sole exception that in the definition of $q(\mu)$ we write
\[
q(\mu) = \mu + M^2\left( \frac{1}{r_0^2} + \left(\frac{\Omega_0'}{U_0'}\right)^2 \right) +\sqrt{\frac{-2iM\Omega_0'}{r_0}}.
\]
We crucially choose the branch $\Re(\sqrt{-i})>0$, in this way we observe that the (real part of the) index $\nu$ corresponding the second component \eqref{eq:r0-2} will be strictly smaller that the one associated to the first component and it sufficies to study the first component, see Remark \ref{rmk:eta}. Putting everything together, we find two independent solutions to the (first component of the) Weber differential equation \eqref{eq:r0-1},
\begin{equation}\label{eq:v1-r0}
v_1(r) = D_{-\frac{1}{2}c_2^{-1/2}\eta}\left(\sqrt{2}c_2^{1/4}\eps^{-1/3}(r-r_0)\right),
\end{equation}
\begin{equation}\label{eq:v2-r0}
v_2(r) = D_{-\frac{1}{2}c_2^{-1/2}\eta}\left(-\sqrt{2}c_2^{1/4}\eps^{-1/3}(r-r_0)\right),
\end{equation}
with Wronskian given by
\begin{equation}\label{eq:wronskian}
W(v_1(r),v_2(r)) = W(v_1(z),v_2(z))\frac{\dd z}{\dd r} = \frac{2\sqrt{\pi}c_2^{1/4}}{\Gamma\left(\frac{1}{2}c_2^{-1/2}\eta\right)}\eps^{-1/3} =: \frac{1}{\mathfrak{w}(\eta)}\eps^{-1/3}.
\end{equation}
The parabolic cylinder functions admit a representation in terms of the \emph{Hermite functions}, denoted by $H_\nu(z)$. The relation between these special functions is of the form
\[
D_\nu(z) = 2^{-\frac{1}{2}\nu} \e^{-\frac{1}{4}z^2} H_\nu\left(\frac{z}{\sqrt{2}}\right).
\]
Since for some specific calculations it will be more convenient to use this different representation of the two independent solutions. We see thus that \eqref{eq:v1-r0}--\eqref{eq:v2-r0} can be written as
\[
v_1(r) = 2^{\frac{1}{4}c_2^{-1/2}\eta}\e^{-\frac{1}{2}c_2^{1/2}\eps^{-2/3}(r-r_0)^2}H_{-\frac{1}{2}c_2^{-1/2}\eta}\left(c_2^{1/4}\eps^{-1/3}(r-r_0)\right),
\]
\[
v_2(r) = 2^{\frac{1}{4}c_2^{-1/2}\eta}\e^{-\frac{1}{2}c_2^{1/2}\eps^{-2/3}(r-r_0)^2}H_{-\frac{1}{2}c_2^{-1/2}\eta}\left(-c_2^{1/4}\eps^{-1/3}(r-r_0)\right),
\]
where $H_\nu(z)$ denotes the Hermite function of order $\nu\in\Co$. For further information about these special functions and some relevant related results we refer to Appendix \ref{s:appendix-parabolic}.

The next step is to define a Green's function associated to equation \eqref{eq:r0-1}, as well as to obtain suitable bounds on this Green's function. In order to get there, we begin by noting the following convenient bounds for the solutions $v_1(r)$ and $v_2(r)$,

\begin{lemma}\label{lemma:bounds-v1v2}
Let us define
\[
\chi  = \frac{1}{2}|c_2|^{-1/2}>0, \quad \text{and} \quad \nu=-\frac{1}{2}c_2^{-1/2}\eta.
\]
Then, $v_1(r)$ and $v_2(r)$ defined in \eqref{eq:v1-r0}--\eqref{eq:v2-r0} satisfy the following bounds uniformly for $c_2, \nu$ in compact sets of $\mathbb{C}$:
\begin{enumerate}
    \item If $r<r_0$,
    \begin{align*}
        |v_1(r)| & \lesssim \e^{\chi\eps^{-2/3}(r-r_0)^2}, \\
        |v_2(r)| & \lesssim \e^{-\chi\eps^{-2/3}(r-r_0)^2}\left(1+\eps^{-1/3}|r-r_0|\right)^{\Re(\nu)}.
    \end{align*}
    \item If $r\geq r_0$,
    \begin{align*}
        |v_1(r)| & \lesssim \e^{-\chi\eps^{-2/3}(r-r_0)^2}\left(1+\eps^{-1/3}|r-r_0|\right)^{\Re(\nu)}, \\
        |v_2(r)| & \lesssim \e^{\chi\eps^{-2/3}(r-r_0)^2}.
    \end{align*}
\end{enumerate}
\end{lemma}

The proof of this lemma follows from standard properties of the parabolic cylinder functions, exploiting the fact that they can be written in terms of Hermite functions. In particular, it is a direct consequence of Lemma \ref{lemma:parabolic cylinder bounds} in Appendix \ref{s:appendix-parabolic}.

\begin{remark}
As is the case throughout the document, many of our estimates are written using the $\lesssim$ symbol for a more convenient presentation of results. In particular, we typicaly absorb all constants that do not depend on $\eps$ into the symbol $\lesssim$. However, it is worth noting that these constants may depend on $\nu\in\Co$, and hence on $\eta\in\Co$, which is a parameter that must be treated carefully later on. However, it is clear from our estimates that they are upper bounded by the case when $\Re(\nu)=1$. Thus, proving the result in this case automatically yields the entire range $|\Re(\nu)|\leq 1$.
\end{remark}

The last preparation required before proceeding with the proof of Proposition \ref{proposition:r0} is the construction of the Green's function, that we do in the usual fashion. Taking the two independent solutions $v_1(r)$ and $v_2(r)$ and its Wronskian, we write a Green's kernel
\[
G_{2,1}^\lambda(r,s) = \mathfrak{w}(\eta)\eps^{-2/3}\left\lbrace
\begin{array}{ll}
    v_1(r)v_2(s) & \text{if } s<r, \\
    v_1(s)v_2(r) & \text{if } r\leq s.
\end{array}
\right.
\]
We find the solution to the first component of our problem \eqref{eq:r0-1} via 
\begin{equation}\label{eq:Green's-r0}
\begin{split}
v(r) = (\mathcal{G}_{2,1}^{\lambda}f_1)(r) & = \int_0^\infty G_{2,1}^\lambda(r,s)f_1(s)\dd s \\
& = \mathfrak{w}(\eta)\eps^{-2/3}v_1(r)\int_0^rv_2(s)f_1(s)\dd s + \mathfrak{w}(\eta)\eps^{-2/3}v_2(r)\int_r^\infty v_1(s)f_1(s)\dd s.
\end{split}
\end{equation}
Using the bounds from Lemma \ref{lemma:bounds-v1v2}, we can write a suitable pointwise estimate for the Green's function. This estimate is tailor made for the case near $r_0$, therefore we crucially assume that 
\[
\supp(f_1)\subset [r_0-\eps^\gamma,r_0+\eps^\gamma),
\]
for some $\gamma>0$ to be specified later. Recall that we use the notation from Lemma \ref{lemma:bounds-v1v2},
\[
\chi  = \frac{1}{2}|c_2|^{-1/2}>0, \quad \nu = -\frac{1}{2}c_2^{-1/2}\eta.
\]
Finding upper bounds for the Green's function will turn out to be rather cumbersome, since it requires a careful splitting of the integrals in \eqref{eq:Green's-r0} depending on the value of $r$. Indeed, using Lemma \ref{lemma:bounds-v1v2} we find the following estimate for the Green's function,
\begin{equation}\label{eq:G-bounds-I-IV}
|\G_{2,1}^{\lambda}f(r)| \lesssim \eps^{-2/3}\left(\I_\eps + \I'_\eps + \II_\eps + \III_\eps + \IV_\eps + \IV'_\eps\right),
\end{equation}
where
\[
\I_\eps = \dsOne_{\{r_0-\eps^\gamma\leq r<r_0\}}\e^{\chi\eps^{-2/3}(r-r_0)^2}\int_{r_0-\eps^\gamma}^{r} \left(1+\eps^{-1/3}|s-r_0|\right)^{\Re(\nu)} \e^{-\chi\eps^{-2/3}(s-r_0)^2}|f_1(s)|\dd s,
\]
\[
\begin{split}
\I'_\eps & = \dsOne_{\{r_0\leq r\}} \left(1+\eps^{-1/3}|r-r_0|\right)^{\Re(\nu)} \e^{-\chi\eps^{-2/3}(r-r_0)^2} \\
& \quad \times \int_{r_0-\eps^\gamma}^{r_0} \left(1+\eps^{-1/3}|s-r_0|\right)^{\Re(\nu)}\e^{-\chi\eps^{-2/3}(s-r_0)^2}|f_1(s)|\dd s,
\end{split}
\]
\[
\II_\eps = \left(1+\eps^{-1/3}|r-r_0|\right)^{\Re(\nu)} \e^{-\chi\eps^{-2/3}(r-r_0)^2} \int_{\min\{r,r_0\}}^r \e^{\chi\eps^{-2/3}(s-r_0)^2}|f_1(s)| \dd s,
\]
\[
\III_\eps = \left(1+\eps^{-1/3}|r-r_0|\right)^{\Re(\nu)} \e^{-\chi\eps^{-2/3}(r-r_0)^2} \int_r^{\max\{r,r_0\}} \e^{\chi\eps^{-2/3}(s-r_0)^2}|f_1(s)| \dd s,
\]
\[
\IV_\eps = \dsOne_{\{r_0\leq r < r_0+\eps^\gamma\}}\e^{\chi\eps^{-2/3}(r-r_0)^2}\int_{r}^{r_0+\eps^\gamma} \left(1+\eps^{-1/3}|s-r_0|\right)^{\Re(\nu)} \e^{-\chi\eps^{-2/3}(s-r_0)^2}|f_1(s)|\dd s,
\]
\[
\begin{split}
\IV'_\eps & = \dsOne_{\{r<r_0\}} \left(1+\eps^{-1/3}|r-r_0|\right)^{\Re(\nu)}\e^{-\chi\eps^{-2/3}(r-r_0)^2} \\
& \quad \times\int_{r_0}^{r_0+\eps^\gamma} \left(1+\eps^{-1/3}|s-r_0|\right)^{\Re(\nu)}\e^{-\chi\eps^{-2/3}(s-r_0)^2}|f_1(s)| \dd s.
\end{split}
\]

We leave it to the reader to verify that these bounds are precisely what one obtains by applying Lemma \ref{lemma:bounds-v1v2} carefully to the expression \eqref{eq:Green's-r0}. Note that by our assumptions, $|\Re(\nu)|\leq 1$, and we shall thus simply prove all our upper bounds for $\Re(\nu)=1$, showing that the estimates are indeed uniform for $|\Re(\nu)|\leq 1$.
\begin{remark}\label{rmk:eta}
    In the proof of Proposition \ref{proposition:r0} we will always argue for \eqref{eq:r0-1}, namely only the first component of equations. We claim that our arguments readily carry over to the second component \eqref{eq:r0-2} because of the choice of the branch 
    \[
    \Re\left(\sqrt{\frac{-2iM\Omega_0'}{r_0}}\right) > 0,
    \]
    in equations \eqref{eq:r0-1}--\eqref{eq:r0-2}. This option ensures that the real part of the index $\nu$ corresponding to the second component \eqref{eq:r0-2} will be strictly smaller, and therefore 
     the condition
    \[
    \Re(\nu) = \Re\left(-\frac{1}{2}c_2^{1/2}\eta \right) \leq 1,
    \]
    required for \eqref{eq:r0-1}, will suffice as well for \eqref{eq:r0-2}. Everything else follows analogously.
\end{remark}

We now proceed with the proof of Proposition \ref{proposition:r0}. Since many of the arguments throughout Section \ref{s:approx-Green's} will be quite similar in flavour to those appearing in this proof, we endeavour to provide a substantial amount of detail in the following.

\begin{proof}[Proof of Proposition \ref{proposition:r0}]
We will argue for equation \eqref{eq:r0-1}, since, as mentioned in Remark \ref{rmk:eta}, the same arguments carry over to the analysis of equation \eqref{eq:r0-2}. Let $f\in Y$ be such that $\supp(f)\subset [r_0-\eps^\gamma,r_0+\eps^\gamma)$. First of all we address \eqref{eq:Gr0}, and in order to do so we will find suitable bounds for each of the six addends in \eqref{eq:G-bounds-I-IV}. To do so, we must carefully split our analysis into sections, depending on the value of $r$. Furthermore, we shall take care to ensure that our estimates hold true \emph{uniformly} for $r_0 \in \mathcal{R}$, $M \in \mathcal{N}$. In particular, we this dependency primarily shows up in the guise of the constant $\chi$.

\emph{Step 1}. Assume $0\leq r<r_0-\eps^\gamma$, then all the addends in \eqref{eq:G-bounds-I-IV} vanish except for $\III_\eps$ and $\IV'_\eps$. For the third addend we can write
\[
\begin{split}
\eps^{-2/3} & \III_\eps  \leq \eps^{-2/3}\left(1+\eps^{-1/3}|r-r_0|\right) \e^{-\chi\eps^{-2/3}(r-r_0)^2} \int_{r_0-\eps^\gamma}^{r_0} \e^{\chi\eps^{-2/3}(s-r_0)^2}|f(s)| \dd s \\
& \leq \|f\|_{Y}\frac{\eps^{-2/3}\left(1+\eps^{-1/3}|r-r_0|\right)}{(r_0-\eps^\gamma)^2} \e^{-\chi\eps^{-2/3}(r-r_0)^2} \int_{r_0-\eps^\gamma}^{r_0} \frac{\e^{\chi\eps^{-2/3}(s-r_0)^2}}{1+(\eps^{-1/3}|s-r_0|)^N} \dd s.
\end{split}
\]
The main goal is to bound each of these addends in $X$, hence we need to make sure that when multiplied by $\max\{1,r^2\}w_\eps(r-r_0)$, everything stays uniformly bounded in $r$ and behaves well for $\eps$ small. First of all notice that if $r<r_0-\eps^\gamma$ we have the trivial bound 
\[
\max\{1,r^2\} \leq 1+r_0^2 \lesssim 1.
\]
After performing a change of variables $z=\eps^{-1/3}(r_0-r)$, respectively $x=\eps^{-1/3}(r_0-s)$ for the integral variable, we can write
\[
\eps^{-2/3}\max\{1,r^2\}w_\eps(r-r_0)\III_\eps \lesssim \eps^{-1/3} \|f\|_{Y} \left(1+z\right) \left(1+z^N\right) \e^{-\chi z^2} \int_0^{\eps^{-1/3+\gamma}} \frac{\e^{\chi x^2}}{1+x^N} \dd x.
\]
Notice that since $r\in [0,r_0-\eps^\gamma)$, this corresponds to $z\in [\eps^{-1/3+\gamma},r_0\eps^{-1/3})$. Therefore, for all $z$ in this region, we can write the estimate
\[
\eps^{-2/3}\max\{1,r^2\}w_\eps(r-r_0)\III_\eps \lesssim \eps^{-1/3} \|f\|_{Y} P(z),
\]
where
\begin{equation}\label{eq:P}
P(z) = \left(1+z\right) \left(1+z^N\right) \e^{-\chi z^2} \int_0^z \frac{\e^{\chi x^2}}{1+x^N} \dd x.
\end{equation}
We claim that $P:[0,\infty) \to \mathbb{R}_+$ is a bounded function for any $\chi>0$. Since it is a smooth, non-negative function with $P(0)=0$, it suffices to show that the limit as $z\to\infty$ remains bounded. Indeed, let us use L'H\^opital's rule,
\[
\frac{\dd}{\dd z} \int_0^z \frac{\e^{\chi x^2}}{1+x^N} \dd x = \frac{\e^{\chi z^2}}{1+z^N},
\]
\[
\begin{split}
    \frac{\dd}{\dd z} \left[\left(1+z\right)^{-1} \left(1+z^N\right)^{-1} \e^{\chi z^2}\right] & = -\left(1+z\right)^{-2}\left(1+z^N\right)^{-1} \e^{\chi z^2} \\
    & \quad - Nz^{N-1}\left(1+z\right)^{-1}\left(1+z^N\right)^{-2} \e^{\chi z^2} \\
    & \quad + 2\chi z \left(1+z\right)^{-1} \left(1+z^N\right)^{-1} \e^{\chi z^2},
\end{split}
\]
and hence,
\[
\lim_{z\to\infty} \frac{\frac{\dd}{\dd z} \int_0^z \left(1+x^N\right)^{-1} \e^{\chi x^2} \dd x}{\frac{\dd}{\dd z} \left[\left(1+z\right)^{-1} \left(1+z^N\right)^{-1} \e^{\chi z^2}\right]} = \lim_{z\to\infty} \frac{\left(1+z\right)}{2\chi z}=\frac{1}{2\chi}.
\]
This argument shows that, under the assumption that for any fixed $\chi >0$, we have $P(z)\lesssim 1$. Next, to show uniformity in $\chi$, simply differentiate $P(z)$ in $\chi$, yielding 
\begin{equation*}
\frac{\dd}{\dd \chi}P(z)=-z^2P(z)+(1+z)(1+z^N)\e^{-\chi z^2}\int_0^z \frac{\e^{\chi x^2}x^2}{1+x^n} \dd x.
\end{equation*}
But we can simply bound 
$$
(1+z)(1+z^N)\int_0^z \frac{\e^{\chi x^2}x^2}{1+x^n} \dd x \leq z^2(1+z)(1+z^N) \int_0^z \frac{\e^{\chi x^2}x^2}{1+x^n} \dd x=z^2P(z),
$$
and so we deduce $\frac{\dd}{\dd \chi}P(z) \leq 0$, for all $z>0$. Thus, for $\chi$ in any compact interval of $(0,\infty)$, we have that $P(z) \lesssim 1$ \emph{uniformly} in $\chi$. Therefore we can write
\[
\eps^{-2/3}\max\{1,r^2\}w_\eps(r-r_0)\III_\eps \lesssim \eps^{-1/3} \|f\|_{Y}.
\]
In fact, note that 
\begin{align*}
&\eps^{-2/3}w_\eps(r-r_0)\III_\eps(r)\\
&\leq \|f\|_{Y}\eps^{-2/3}w_\eps(\eps^\gamma)(1+\eps^{-1/3+\gamma})\e^{-\chi \eps^{-2/3+2\gamma}}\int_{r_0-\eps^\gamma}^{r_0} \frac{\e^{\chi\eps^{-2/3}(s-r_0)^2}}{1+(\eps^{-1/3}|s-r_0|)^N} \dd s  \times L(r,\eps),
\end{align*}
where 
\begin{equation*}
L(r,\eps)=\frac{w_\eps(r-r_0)(1+\eps^{-1/3}|r-r_0|)}{w_\eps(\eps^\gamma)(1+\eps^{-1/3+\gamma})}\e^{-\chi \eps^{-2/3}((r-r_0)^2-\eps^{2\gamma})}.
\end{equation*}
But by Lemma \ref{lemma:properties-weight}, we know that there exists a constant $C>0$ independent of $\eps, r$, depending continuously on $\chi$, so that
\begin{equation*}
\frac{w_\eps(r-r_0)(1+\eps^{-1/3}|r-r_0|)}{w_\eps(\eps^\gamma)(1+\eps^{-1/3+\gamma})}\leq C \e^{\frac{1}{2}\chi\eps^{-2/3}((r-r_0)^2-\eps^{2\gamma})}.
\end{equation*}
Hence, we conclude in fact that 
\begin{equation*}
\eps^{-2/3}w_\eps(r-r_0)\III_\eps(r) \lesssim \|f\|_{Y} \eps^{-2/3}w_\eps(\eps^\gamma)\III_\eps(r_0-\eps^\gamma)L(r,\eps) \lesssim \|f\|_{Y}\eps^{-1/3}\e^{-\frac{1}{2}\chi\eps^{-2/3}((r-r_0)^2-\eps^{2\gamma})},
\end{equation*}
uniformly for $r \leq r_0-\eps^\gamma$, $r_0 \in \mathcal{R}$, $M \in \mathcal{N}$.

For the term $\IV'_\eps$ we use that $s^{-2}w_\eps(s-r_0)^{-1}$ is a decreasing function for all $s>r_0$, together with the bound
\[
\left(1+\eps^{-1/3}|s-r_0|\right)\e^{-\chi\eps^{-2/3}(s-r_0)^2} \lesssim \e^{-\frac{2}{3}\chi\eps^{-2/3}(s-r_0)^2}
\]
to get the estimate
\[
\begin{split}
    \eps^{-2/3}\IV'_\eps & = \eps^{-2/3}\frac{1}{r_0^2w_\eps(0)} \e^{-\frac{2}{3}\chi\eps^{-2/3}(r-r_0)^2} \int_{r_0}^{r_0+\eps^\gamma} \e^{-\frac{2}{3}\chi\eps^{-2/3}(s-r_0)^2}|f(s)| \dd s \\
    & \lesssim \|f\|_Y\eps^{-1/3} \e^{-\frac{2}{3}\chi\eps^{-2/3}(r-r_0)^2}.
\end{split}
\]
Hence
\[
\eps^{-2/3}|\max\{1,r^2\} w_\eps(r-r_0)\IV'_\eps| \lesssim \eps^{-1/3}\|f\|_Y w_\eps(r-r_0) \e^{-\frac{2}{3}\chi\eps^{-2/3}(r-r_0)^2}.
\]
Finally, note that $w_\eps(r-r_0)\e^{-\frac{2}{3}\chi\eps^{-2/3}(r-r_0)^2} \lesssim \e^{-\frac{1}{2}\chi\eps^{-2/3}(r-r_0)^2}$, which decays very fast as $\eps\to 0$ due to the effect of the exponential.
Putting all pieces together, we find that uniformly for $0\leq r<r_0-\eps^\gamma$, $r_0 \in \mathcal{R}$, $M \in \mathcal{N}$, there holds
\begin{equation}\label{eq:G-est-1}
|\max\{1,r^2\}w_\eps(r-r_0)\G_{2,1}^{\lambda}f(r)| \lesssim \eps^{-1/3}e^{-\frac{1}{2}\chi \eps^{-2/3}\left((r-r_0)^2-(\eps^{\gamma})^2\right)}\|f\|_Y.
\end{equation}

\emph{Step 2}. Assume now that $r_0-\eps^\gamma\leq r<r_0$. Note that in this regime, we once again can neglect the $\max\{1,r^2\}$ weight up to a multiplicative constant independent of $\eps$. Furthermore, in this scenario we have $\I_\eps'=0$, $\II_\eps = 0$, and $\IV_\eps=0$, but the other three terms contribute to the error. For the first addend we write
\[
\begin{split}
\eps^{-2/3}\I_\eps & \leq \eps^{-2/3}\e^{\chi\eps^{-2/3}(r-r_0)^2}\int_{r_0-\eps^\gamma}^{r} \left(1+\eps^{-1/3}|s-r_0|\right) \e^{-\chi\eps^{-2/3}(s-r_0)^2}|f(s)|\dd s \\
& \leq \|f\|_Y \frac{\eps^{-2/3}}{(r_0-\eps^\gamma)^2} \e^{\chi\eps^{-2/3}(r-r_0)^2}\int_{r_0-\eps^\gamma}^{r} \frac{\left(1+\eps^{-1/3}|s-r_0|\right)}{1 + \left(\eps^{-1/3}|s-r_0|\right)^N} \e^{-\chi\eps^{-2/3}(s-r_0)^2}\dd s.
\end{split}
\]
We want to argue using L'H\^opital's rule as for Step 1. In this regard, notice that after performing the change of variables $z=\eps^{-1/3}(r_0-r)$, and $x=\eps^{-1/3}(r_0-s)$ respectively for the integral variables, we obtain the bound
\[
\eps^{-2/3}w_\eps(r-r_0)\I_\eps \lesssim \eps^{-1/3}\|f\|_YQ(z),
\]
with
\begin{equation}\label{eq:Q}
Q(z) = \left(1 + z^N\right)\e^{\chi z^2}\int_z^{\infty} \frac{\left(1+x\right)}{1 + x^N} \e^{-\chi x^2}\dd x.
\end{equation}
The upper bound of the integral after the change of variables corresponds to $\eps^{-1/3+\gamma}$, however we can brutally upper-bound it by infinity. As before, we want to show that $Q(z)\lesssim 1$, indeed, this is a non-negative smooth function with the property
\[
0 \leq Q(0) = \int_0^{\infty} \frac{\left(1+x\right)}{1 + x^N} \e^{-\chi x^2}\dd x \lesssim \int_0^{\infty} \e^{-\chi x^2}\dd x \lesssim 1.
\]
Therefore, as before, it suffices to show that $Q(z)$ remains bounded in the limit $z\to\infty$ in order to obtain the desired result. To see this we make use again of L'H\^opital's rule,
\[
\frac{\dd}{\dd z} \int_z^{\infty} \frac{\left(1+x\right)}{1 + x^N} \e^{-\chi x^2}\dd x = -\frac{\left(1+z\right)}{1 + z^N} \e^{-\chi z^2},
\]
\[
\frac{\dd}{\dd z}\left[\left(1 + z^N\right)^{-1}\e^{-\chi z^2}\right] = -Nz^{N-1}(1+z^N)^{-2}\e^{-\chi z^2} - 2\chi z (1+z^N)^{-1} \e^{-\chi z^2},
\]
and therefore
\[
\lim_{z\to \infty} \frac{\frac{\dd}{\dd z} \int_z^{\infty} \left(1+x\right)\left(1 + x^N\right)^{-1} \e^{-\chi x^2}\dd x}{\frac{\dd}{\dd z}\left[\left(1 + z^N\right)^{-1}\e^{-\chi z^2}\right]} = \lim_{z\to\infty} \frac{\left(1+z\right)}{2\chi z},
\]
which, as in the previous case, is bounded by a constant, provided that $\Re(\nu)\leq 1$. All in all, we derived the estimate
\[
\eps^{-2/3}|\max\{1,r^2\}w_\eps(r-r_0)\I_\eps| \lesssim \eps^{-1/3}\|f\|_Y,
\]
for all $r\in [r_0-\eps^\gamma,r_0)$. Finally, we deduce that $\frac{\dd}{\dd \chi}Q(z)\leq 0$, and so we can extend the bound to be uniform in compact intervals in $\chi$. Next, for the third addend we write
\[
\begin{split}
\eps^{-2/3}\III_\eps & = \eps^{-2/3}\left(1+\eps^{-1/3}|r-r_0|\right) \e^{-\chi\eps^{-2/3}(r-r_0)^2} \int_r^{r_0} \e^{\chi\eps^{-2/3}(s-r_0)^2}|f(s)| \dd s \\
& \leq \|f\|_Y \frac{\eps^{-2/3}\left(1+\eps^{-1/3}|r-r_0|\right)}{(r_0-\eps^\gamma)^2} \e^{-\chi\eps^{-2/3}(r-r_0)^2} \int_r^{r_0} \frac{\e^{\chi\eps^{-2/3}(s-r_0)^2}}{1+\left(\eps^{-1/3}|s-r_0|\right)^N} \dd s 
\end{split}
\]
which, after multiplication by the weights from the $X$ norm, and the usual change of variables $z=\eps^{-1/3}(r_0-r)$, $x=\eps^{-1/3}(r_0-s)$, can be bounded by
\[
\eps^{-2/3}|\max\{1,r^2\}w_\eps(r-r_0)\III_\eps| \lesssim \eps^{-1/3}\|f\|_Y P(z),
\]
where $P(z)$ is defined in \eqref{eq:P}. In particular, since we have already proved in Step 1 that $P(z)\lesssim 1$, we directly obtain the estimate
\[
\eps^{-2/3}|\max\{1,r^2\}w_\eps(r-r_0)\III_\eps| \lesssim \eps^{-1/3}\|f\|_Y.
\]
Lastly, for the fourth term we write just as in Step 1,
\[
\begin{split}
\eps^{-2/3}\IV'_\eps & \leq \|f\|_Y \frac{\eps^{-2/3}\left(1+\eps^{-1/3}|r-r_0|\right)}{r_0^2} \e^{-\chi\eps^{-2/3}(r-r_0)^2} \\
& \quad \times \int_{r_0}^{r_0+\eps^\gamma} (1+\eps^{-1/3}|r-r_0|) \frac{\e^{-\chi\eps^{-2/3}(s-r_0)^2}}{1+\left(\eps^{-1/3}|s-r_0|\right)^N} \dd s.
\end{split}
\]
As before, we make the change of variables $z=\eps^{-1/3}(r_0-r)$ and $x=\eps^{-1/3}(r_0-s)$ so that 
\[
\eps^{-2/3}w_\eps(r-r_0)\IV'_\eps \lesssim \eps^{-1/3} \|f\|_{Y} \left(1+z\right) \left(1+z^N\right) \e^{-\chi z^2} \int_0^{\infty} (1+x)^{\Re(\nu)} \frac{\e^{-\chi x^2}}{1+x^N} \dd x.
\]
Due to the finiteness of the integral and the boundedness on the prefactor we can write
\[
\left(1+z\right) \left(1+z^N\right) \e^{-\chi z^2} \int_0^{\infty} (1+x) \frac{\e^{-\chi x^2}}{1+x^N} \dd x \lesssim 1,
\]
which readily yields 
\[
\eps^{-2/3}|\max\{1,r^2\}w_\eps(r-r_0)\IV_\eps'| \lesssim \eps^{-1/3}\|f\|_Y.
\]
All in all, we find that for any $r_0 \in \mathcal{R}$, $M \in \mathcal{N}$, $r\in [r_0-\eps^\gamma,r_0)$, the approximate Green's function satisfies the estimate
\begin{equation}\label{eq:G-est-2}
|\max\{1,r^2\}w_\eps(r-r_0)\G_{2,1}^{\lambda}f(r)| \lesssim \eps^{-1/3}\|f\|_Y.
\end{equation}

\emph{Step 3}. We assume now that $r_0\leq r<r_0+\eps^\gamma$, which directly implies $\I_\eps=0$, $\III_\eps=0$, and $\IV_\eps'=0$. The analysis in this regime follows very closely the ideas introduced in the previous two steps. First of all we have,
\[
\begin{split}
    \eps^{-2/3}\I'_\eps &\leq \|f\|_Y \frac{\eps^{-2/3}\left(1+\eps^{-1/3}|r-r_0|\right)}{(r_0-\eps^\gamma)^2} \e^{-\chi\eps^{-2/3}(r-r_0)^2} \\
    & \quad \times \int_{r_0-\eps^\gamma}^{r_0} (1+\eps^{-1/3}|r-r_0|)\frac{\e^{-\chi\eps^{-2/3}(s-r_0)^2}}{1+\left(\eps^{-1/3}|r-r_0|\right)^N} \dd s.
\end{split}
\]
Once again, the usual change of variables $z=\eps^{-1/3}(r_0-r)$, and $x=\eps^{-1/3}(r_0-s)$ gives the estimate
\[
\eps^{-2/3}w_\eps(r-r_0)\I'_\eps \lesssim \eps^{-1/3} \|f\|_{Y} \left(1+z\right) \left(1+z^N\right) \e^{-\chi z^2} \int_0^{\infty} \frac{\e^{-\chi x^2}}{1+x^N}(1+x) \dd x,
\]
hence using that this is a product of a finite integral and a bounded function, we readily obtain the bound
\[
\eps^{-2/3}|\max\{1,r^2\}w_\eps(r-r_0)\I_\eps'| \lesssim \eps^{-1/3}\|f\|_Y.
\]
Next in order, for the second term we observe that
\[
\begin{split}
\eps^{-2/3}\II_\eps & = \eps^{-2/3}\left(1+\eps^{-1/3}|r-r_0|\right) \e^{-\chi\eps^{-2/3}(r-r_0)^2} \int_{r_0}^r \e^{\chi\eps^{-2/3}(s-r_0)^2}|f(s)| \dd s \\
& \leq \|f\|_Y \frac{\eps^{-2/3}\left(1+\eps^{-1/3}|r-r_0|\right)}{r_0^2} \e^{-\chi\eps^{-2/3}(r-r_0)^2} \int_{r_0}^r \frac{\e^{\chi\eps^{-2/3}(s-r_0)^2}}{1+\left(\eps^{-1/3}|r-r_0|\right)^N} \dd s,
\end{split}
\]
so that, invoking once again the boundedness of $P(z)$ defined in \eqref{eq:P}, we directly obtain the estimate
\[
\eps^{-2/3}|\max\{1,r^2\}w_\eps(r-r_0)\II_\eps| \lesssim \eps^{-1/3}\|f\|_Y.
\]
Finally, the forth term in the definition of the approximate Green's function can be controlled instead by using the boundedness of $Q(z)$, defined in \eqref{eq:Q}, that has been proved in Step 2. We thus write,
\[
\begin{split}
\eps^{-2/3}\IV_\eps & = \eps^{-2/3}\e^{\chi\eps^{-2/3}(r-r_0)^2}\int_{r}^{r_0+\eps^\gamma} \left(1+\eps^{-1/3}|s-r_0|\right) \e^{-\chi\eps^{-2/3}(s-r_0)^2}|f(s)|\dd s \\
& \leq \|f\|_Y \frac{\eps^{-2/3}}{r^2}\e^{\chi\eps^{-2/3}(r-r_0)^2}\int_{r}^{r_0+\eps^\gamma} \frac{\left(1+\eps^{-1/3}|s-r_0|\right)}{1+\left(\eps^{-1/3}|r-r_0|\right)^N} \e^{-\chi\eps^{-2/3}(s-r_0)^2}\dd s 
\end{split}
\]
Therefore, since $r\in [r_0,r_0+\eps^\gamma)$ we get
\[
\eps^{-2/3}|\max\{1,r^2\}w_\eps(r-r_0)\IV_\eps| \lesssim \eps^{-1/3}\|f\|_Y Q(z) \lesssim \eps^{-1/3}\|f\|_Y.
\]
Putting everything together, we find that for the values $r_0\leq r<r_0+\eps^\gamma$, the approximate Green's function, weighted with $\max\{1,r^2\}w_\eps(r-r_0)$, is pointwise bounded by
\begin{equation}\label{eq:G-est-3}
|\max\{1,r^2\}w_\eps(r-r_0)\G_{2,1}^{\lambda}f(r)| \lesssim \eps^{-1/3}\|f\|_Y.
\end{equation}

\emph{Step 4}. The last piece missing concerns the regime $r\geq r_0+\eps^\gamma$, where we need to be slightly more careful due to the effect of the weight $\max\{1,r^2\}$. We have $\I_\eps=0$, $\III_\eps=0$, and $\IV_\eps=\IV_\eps'=0$, for the other two terms we argue as before. On the one hand, for $\I_\eps'$ we can write
\[
\begin{split}
    \eps^{-2/3}\I'_\eps  & \leq \|f\|_Y \frac{\eps^{-2/3}\left(1+\eps^{-1/3}|r-r_0|\right)}{(r_0-\eps^\gamma)^2} \e^{-\chi\eps^{-2/3}(r-r_0)^2} \\
    & \quad \times \int_{r_0-\eps^\gamma}^{r_0} (1+\eps^{-\frac{1}{3}}|s-r_0|)\frac{\e^{-\chi\eps^{-2/3}(s-r_0)^2}}{1+\left(\eps^{-1/3}|r-r_0|\right)^N} \dd s.
\end{split}
\]
Since both exponentials have the correct sign, as in Step 1 we find that 
\[
\max\{1,r^2\}\eps^{-2/3}w_\eps(r-r_0)\I'_\eps \lesssim \eps^{-1/3}\|f\|_Y,
\]
and in fact, we note that 
\begin{align*}
\max\{1,r^2\}\eps^{-2/3}\I_\eps'(r) & \lesssim \|f\|_{Y}   \eps^{-2/3}(1+\eps^{-1/3+\gamma})w_\eps(\eps^\gamma)\max\{1,r^2\}L(r,\eps) \\
& \quad \times \e^{-\chi \eps^{-2/3}(r-r_0)^2}\int_{r_0-\eps^\gamma}^{r_0}(1+\eps^{-\frac{1}{3}}|s-r_0|)\frac{\e^{-\chi\eps^{-2/3}(s-r_0)^2}}{1+\left(\eps^{-1/3}|r-r_0|\right)^N} \dd s,
\end{align*}
where as before we define 
\begin{equation*}
L(r,\eps)=\frac{w_\eps(r-r_0)(1+\eps^{-1/3}|r-r_0|)}{w_\eps(\eps^\gamma)(1+\eps^{-1/3+\gamma})}\e^{-\chi \eps^{-2/3}((r-r_0)^2-\eps^{2\gamma})}.
\end{equation*}
Once again, we observe by Lemma \ref{lemma:properties-weight} that 
\begin{equation*}
L(r,\eps) \lesssim \e^{-\frac{1}{2}\chi \eps^{-2/3}((r-r_0)^2-\eps^{2\gamma})},
\end{equation*}
uniformly in $r, \eps$, as well as $\chi$ in compact intervals of $(0,\infty)$, and so we deduce that 
\begin{equation*}
\eps^{-2/3}w_\eps(r-r_0)\I'_\eps \lesssim \eps^{-1/3}e^{-\frac{1}{2}\chi \eps^{-2/3}((r-r_0)^2-(\eps^{\gamma})^2)}\|f\|_Y,
\end{equation*}
if $r \geq r_0+\eps^\gamma$. Finally, to include the $\max\{1,r^2\}$ weight, we note that 
\[
\max\{1,r^2\} \e^{-\frac{1}{2}\chi \eps^{-2/3}((r-r_0)^2-\eps^{2\gamma})} \lesssim \e^{-\frac{1}{4}\chi \eps^{-2/3}((r-r_0)^2-\eps^{2\gamma})},
\]
and so we deduce 
\begin{equation*}
\max\{1,r^2\}\eps^{-2/3}w_\eps(r-r_0)\I'_\eps \lesssim \eps^{-1/3}e^{-\frac{1}{4}\chi \eps^{-2/3}((r-r_0)^2-(\eps^{\gamma})^2)}\|f\|_Y,
\end{equation*}
Next, we analyse $\II_\eps$. We write,
\[
\begin{split}
\eps^{-2/3}&w_\eps(r-r_0)\II_\eps \lesssim \eps^{-2/3}r^2 w_\eps(r-r_0)\left(1+\eps^{-1/3}|r-r_0|\right)\\
& \qquad \qquad \qquad \times \e^{-\chi\eps^{-2/3}(r-r_0)^2}\int_{r_0}^r \e^{\chi\eps^{-2/3}(s-r_0)^2}|f(s)| \dd s \\
& \leq \|f\|_Y \eps^{-2/3} w_\eps(r-r_0)\left(1+\eps^{-1/3}|r-r_0|\right)\e^{-\chi\eps^{-2/3}(r-r_0)^2}\int_{r_0}^{r_0+\eps^\gamma} \frac{\e^{\chi\eps^{-2/3}(s-r_0)^2}}{w_\eps(s-r_0)} \dd s.
\end{split}
\]
On the one hand we can see that via the usual change of variables,
\[
\eps^{-2/3} w_\eps(r-r_0)\II_\eps \lesssim \|f\|_Y \eps^{-1/3} P(z)
\]
where $P(z)$ is the uniformly bounded function defined in \eqref{eq:P}. On the other hand, we once again notice that by Lemma \ref{lemma:properties-weight}, the following estimate holds true,
\begin{equation*}
L(r,\eps) \e^{-\chi\eps^{-2/3}(r^2 - (r_0+\eps^\gamma)^2)} \lesssim \e^{-\frac{1}{2}\chi \eps^{-2/3} ((r-r_0)^2 - (\eps^\gamma)^2)}.
\end{equation*}
Therefore we can write
\[
\eps^{-2/3} w_\eps(r-r_0)\II_\eps \lesssim \|f\|_Y \eps^{-1/3}\e^{-\frac{1}{2}\chi \eps^{-2/3} ((r-r_0)^2 - (\eps^\gamma)^2)},
\]
for any $r\geq r_0+\eps^\gamma$, $r_0 \in \mathcal{R}$, $M \in \mathcal{N}$, and where the constants absorbed in $\lesssim$ are independent of all the relevant quantities. Finally, as in the previous step, we can absorb the $\max\{1,r^2\}$ weight up to reducing the coefficient in the decaying exponential, and so we obtain the estimate
\[
\eps^{-2/3} \max\{1,r^2\} w_\eps(r-r_0)\II_\eps \lesssim \|f\|_Y \eps^{-1/3}\e^{-\frac{1}{4}\chi \eps^{-2/3} ((r-r_0)^2 - (\eps^\gamma)^2)},
\]
uniformly for $r \geq r_0+\eps^\gamma$.
Hence, putting all these estimates together, we find that
\begin{equation}\label{eq:G-est-4}
|\max\{1,r^2\}w_\eps(r-r_0)\G_{2,1}^{\lambda}f(r)| \lesssim \eps^{-1/3}e^{-\frac{1}{4}\chi \eps^{-1/3}((r-r_0)^2-\eps^{2\gamma})}\|f\|_Y,
\end{equation}
for $r \geq r_0+\eps^\gamma$. 
Therefore, by combining \eqref{eq:G-est-1}, \eqref{eq:G-est-2}, \eqref{eq:G-est-3} and \eqref{eq:G-est-4}, we arrive at the first claim in the proposition \eqref{eq:Gr0}. As mentioned earlier, the argument for the second component of the approximate Green's function is completely analogous.

\emph{Estimate for $\partial_r\G_{2}^\lambda$}. In order to derive the sought estimate for the derivatives of $\G_2^\lambda$, it is convenient to regard the independent solutions $v_1(r)$ and $v_2(r)$, \eqref{eq:v1-r0}--\eqref{eq:v2-r0}, in terms of Hermite functions instead of parabolic cylinder functions. In particular, we know that we can represent these solutions as
\[
\begin{split}
v_1(r) & = D_{-\frac{1}{2}c_2^{-1/2}\eta}\left(\sqrt{2}c_2^{1/4}\eps^{-1/3}(r-r_0)\right) \\
& = 2^{\frac{1}{4}c_2^{-1/2}\eta}\e^{-\frac{1}{2}c_2^{1/2}\eps^{-2/3}(r-r_0)^2}H_{-\frac{1}{2}c_2^{-1/2}\eta}\left(c_2^{1/4}\eps^{-1/3}(r-r_0)\right),
\end{split}
\]
\[
\begin{split}
v_2(r) & = D_{-\frac{1}{2}c_2^{-1/2}\eta}\left(-\sqrt{2}c_2^{1/4}\eps^{-1/3}(r-r_0)\right) \\
& = 2^{\frac{1}{4}c_2^{-1/2}\eta}\e^{-\frac{1}{2}c_2^{1/2}\eps^{-2/3}(r-r_0)^2}H_{-\frac{1}{2}c_2^{-1/2}\eta}\left(-c_2^{1/4}\eps^{-1/3}(r-r_0)\right),
\end{split}
\]
where $H_\nu(z)$ denote the Hermite functions. Once we have the bounds for the Green's operator $\G_{2}^\lambda$, we can now easily bound its derivatives using a special relation
\begin{equation}\label{eq:derivative-H}
H_\nu'(z) = 2\nu H_{\nu-1}(z).
\end{equation}
We refer to \cite{Lebedev72} and Appendix \ref{s:appendix-parabolic} for more detailed information. Again, we argue for the first component of the approximate Green's function, and the second component follows readily. Differentiating the Green's function $\G_{2,1}^\lambda$ from \eqref{eq:Green's-r0} we find that
\[
\partial_r (\G_{2,1}^\lambda f)(r) = \mathfrak{w}(\eta)\eps^{-2/3}\left(v_1'(r)\int_0^r v_2(s)f(s)\dd s + v_2'(r)\int_r^\infty v_1(s)f(s)\dd s\right).
\]
In addition, using \eqref{eq:derivative-H} we get
\[
\begin{split}
v_1'(r) & = - c_2^{1/2}\eps^{-2/3}(r-r_0)v_1(r) \\
& \quad -  2^{\frac{1}{4}c_2^{-1/2}\eta-1} c_2^{-1/4}\eta \eps^{-1/3} \e^{-\frac{1}{2}c_2^{1/2}\eps^{-2/3}(r-r_0)^2}H_{-\frac{1}{2}c_2^{-1/2}\eta-1}\left(c_2^{1/4}\eps^{-1/3}(r-r_0)\right),
\end{split}
\]
\[
\begin{split}
v_2'(r) & = - c_2^{1/2}\eps^{-2/3}(r-r_0)v_2(r) \\
& \quad + 2^{\frac{1}{4}c_2^{-1/2}\eta-1} c_2^{-1/4}\eta \eps^{-1/3} \e^{-\frac{1}{2}c_2^{1/2}\eps^{-2/3}(r-r_0)^2}H_{-\frac{1}{2}c_2^{-1/2}\eta-1}\left(-c_2^{1/4}\eps^{-1/3}(r-r_0)\right).
\end{split}
\]
If we want to compare this with the result we just proved for the Green's operator before taking the derivative, we see that $\partial_r (\G_{2,1}^\lambda f)(r)$ will have two main contributions: $\partial_r (\G_{2,1}^\lambda f)(r) = \A'_\eps + \A''_\eps$. Here we define $\A'_\eps$ to be the collection of terms that arise when the derivative hits the exponential functions, whereas $\A''_\eps$ consists of the collection of terms arising when the derivative hits the Hermite functions. For the first of the addends we can write
\[
\begin{split}
\A'_\eps & = - \mathfrak{w}(\eta)\eps^{-2/3}c_2^{1/2}\eps^{-2/3}(r-r_0)\left(v_1(r)\int_0^r v_2(s)f(s)\dd s + v_2(r)\int_r^\infty v_1(s)f(s)\dd s\right) \\
& = -c_2^{1/2}\eps^{-2/3}(r-r_0) (\G_{2,1}^\lambda f)(r).
\end{split}
\]
Adding the weights corresponding to $X$, we arrive at the estimate
\[
|\max\{1,r^2\} w_\eps(r-r_0) \A'_\eps|\lesssim \eps^{-2/3}|r-r_0||\max\{1,r^2\}w_\eps(r-r_0)\G_{2,1}^\lambda f(r)|.
\]
In particular, we see that if $r$ is sufficiently close to $r_0$, e.g.\ $|r-r_0|\leq 2\eps^\gamma$, then we can directly use \eqref{eq:Gr0} and write the estimate
\[
|\max\{1,r^2\} w_\eps(r-r_0) \A'_\eps|\lesssim \eps^{-1+\gamma}\|f\|_Y.
\]
For the remaining uncovered regions for $r$ we need to be make use of the exponential decay. If $r<r_0-2\eps^\gamma$, then we are strictly inside the regime covered in Step 1 of this proof. Therefore, a direct reproduction of the proof in such case yields that $\I_\eps=\I'_\eps=0$, $\II_\eps=0$ and $\IV_\eps=0$. Hence, recall from Step 1 that we have the following estimate for $\III_\eps$, 
\begin{equation*}
\eps^{-2/3}|\max\{1,r^2\}w_\eps(r-r_0)\III_\eps|\lesssim \eps^{-1/3}e^{-\frac{1}{2}\chi \eps^{-2/3}\left((r-r_0)^2-(\eps^{\gamma})^2\right)}\|f\|_{Y},
\end{equation*}
and thus we write
\[
|\max\{1,r^2\} w_\eps(r-r_0) \A'_\eps|\lesssim \eps^{-1/3}\e^{-\frac{1}{2}\chi \eps^{-2/3}\left((r-r_0)^2-(\eps^{\gamma})^2\right)}\|f\|_{Y}.
\] 
This argument carries over identically for the term $\IV'_\eps$. Moreover, following a completely analogous argument for the regime $r>r_0+2\eps^\gamma$, we observe that, for any $r$ it holds
\[
|\max\{1,r^2\} w_\eps(r-r_0) \A_\eps|\lesssim \eps^{-1+\gamma}\|f\|_Y\left( 1 + \e^{-\frac{1}{4}\chi \eps^{-2/3}\left((r-r_0)^2-(\eps^{\gamma})^2\right)} \dsOne_{\{|r-r_0|\geq 2\eps^\gamma\}}(r)\right).
\]
The second contribution to the derivative of the Green's function concerns the case in which the derivatives fall into the Hermite functions. We find thus an expression of the form
\[
\begin{split}
\A''_\eps & = C\mathfrak{w}(\eta)\eps^{-1} 2^{\nu/2} \e^{-\frac{1}{2}\zeta(r)^2}H_{\nu-1}(\zeta(r)) \int_{\zeta(0)}^{\zeta(r)} \e^{-\frac{1}{2}\zeta(s)^2}H_{\nu}(-\zeta(s))f(s)\dd s  \\
& \quad - C\mathfrak{w}(\eta)\eps^{-1} 2^{\nu/2} \e^{-\frac{1}{2}\zeta(r)^2}H_{\nu-1}(-\zeta(r)) \int_{\zeta(0)}^{\zeta(r)} \e^{-\frac{1}{2}\zeta(s)^2}H_{\nu}(\zeta(s))f(s)\dd s,
\end{split}
\]
where in order to shorten the notation, we define $\zeta(r) = c_2^{1/4}\eps^{-1/3}(r-r_0)$, and $C$ is a new constant absorbing only $\eps$ and $r$--independent terms. Now, in order to find appropriate estimates for $\A''_\eps$, we simply notice that the bounds from Lemma \ref{lemma:bounds-v1v2} are still applicable to this new Green's function, with the exception that, when dealing with the derivatives, the exponent $\Re(\nu)$ turns into $\Re(\nu)-1$. This yields no additional constraints since the only condition that we needed to impose on $\Re(\nu)$ is $\Re(\nu)\leq 1$. All in all, we can yet again use the bounds we derived for $\G_{2}^\lambda$ in \eqref{eq:Gr0}, keeping in mind that now there is an extra $\eps^{-1/3}$ multiplying everything. Therefore, we find as before that the largest contribution comes from the values $|r-r_0|\leq \eps^\gamma$ and, it is bounded by
\[
|\max\{1,r^2\}w_\eps(r-r_0)\A''_\eps| \lesssim \eps^{-2/3}\|f\|_Y.
\]
Since we assume $\gamma<1/3$, the estimate from the $\A'_\eps$ is indeed worse, and we obtain
\[
|\max\{1,r^2\}w_\eps(r-r_0)\partial_r\G_{2,1}^\lambda f(r)| \lesssim \eps^{-1+\gamma}\|f\|_Y.
\]
Once again, all these estimates hold uniformly in $r_0 \in \mathcal{R}$, $M \in \mathcal{N}$.
As pointed before, the estimate for the second component of the approximate Green's function follows by a completely analogous argument.

\emph{The error estimate}. The last bit remaining to be proved from Proposition \ref{proposition:r0} is the error estimate \eqref{eq:Gr0-error}, that quantifies how good the approximated operator $\L_2-\lambda$ is. For this we compare equations \eqref{eq:1} associated to the operator $\L-\lambda$, and \eqref{eq:r0-1} associated to the Green's function $\G_{2}^\lambda$. On the one hand notice that by definition of the Green's function $\G_{2}^\lambda$, we have the identities
\[
\begin{split}
((\L-\lambda)\G_{2}^\lambda - \Id)f & =
(\L-\lambda)V - (\L_{2}-\lambda)V = (\L-\L_{2})\G_{2}^\lambda f.
\end{split}
\]
Notice that, using the representation of the operator developed at the beginning of the section, for the first component we may write
\[
((\L-\L_{2})V)_1 = \frac{\eps}{r}\partial_r V_1 - T_{\text{hot}}V_1 - \frac{iM\alpha\eps}{r^2}V_1 + D_\eps^-V_1 + D_\eps^+V_2+\frac{iM\alpha\eps}{r^2}V_2,
\]
and analogously for the second component. Taking this into account, we write for $|r-r_0|\leq 2\eps^\gamma$
\[
\begin{split}
|(\L-\L_{2})\G_{2}^\lambda f| & \lesssim \frac{\eps}{r} |\partial_r \G_{2}^\lambda f| + \frac{\eps}{r^2}|\G_{2}^\lambda f| + \eps^{-1/3}|r-r_0|^3|\G_{2}^\lambda f| + (|D_\eps^-| + |D_\eps^+|)|\G_{2}^\lambda f|,
\end{split}
\]
where in particular, following \eqref{eq:THOT} we use that $T_{\text{hot}}\lesssim \eps^{-1/3}|r-r_0|^3$, uniformly for $r_0 \in \mathcal{R}$. Now, by definition of $D_\eps^+$, $D_\eps^-$ we derive the bounds
\[
|D_\eps^+|,|D_\eps^-| \lesssim \eps^{1/3}\left( \left|\frac{1}{r^2}-\frac{1}{r_0^2}\right| + |r\Omega'-r_0\Omega_0'| \right) \lesssim \eps^{1/3}|r-r_0|,
\]
therefore, using \eqref{eq:Gr0}--\eqref{eq:drGr0} we obtain the estimate for $|r-r_0|\leq 2\eps^\gamma$,
\[
\begin{split}
|r^2(\L-\L_{2})\G_{2}^\lambda f| & \lesssim \frac{\eps}{r} |\partial_r \G_{2}^\lambda f| + \left( \frac{\eps}{r^2} + \eps^{-1/3}|r-r_0|^3 + \eps^{1/3}|r-r_0|\right)|\G_{2}^\lambda f| \\
& \lesssim \eps |\partial_r \G_{2}^\lambda f| + \left( \eps + \eps^{-1/3+3\gamma} + \eps^{1/3+\gamma}\right)|\G_{2}^\lambda f| \\
& \lesssim \left( \eps^{2/3} + \eps^{-2/3+3\gamma} + \eps^{\gamma}\right) \|f\|_Y.
\end{split}
\]
Observe that the right hand side converges to zero as $\eps\to 0$ provided that $3\gamma-2/3>0$.
Next, for $|r-r_0|\geq 2\eps^\gamma$, noting that by Assumption \ref{H2}, $T$ and $\Omega$ may be bounded by some polynomial of the form $C(1+r)^{N_1}$, $C(1+r)^{N_2}$ respectively for some $N_1,N_2\in\N$ and some constant $C>0$, we find the error estimate 
\[
\begin{split}
|r^2(\L-\L_{2})\G_{2}^\lambda f| & \lesssim \eps |r\partial_r \G_{2}^\lambda f| + \left( \eps + r^2\eps^{-1/3}(1+r)^{N_1} + \eps^{1/3}r^2(1+r)^{N_2}\right)|\G_{2}^\lambda f| \\
& \lesssim \eps^\gamma\|f\|_{Y}+\left( \eps^{2/3}+\eps^{-2/3}r^2(1+r)^{N_1}+r^2(1+r)^{N_2}\right)e^{-\frac{1}{4}\chi \eps^{-\frac{2}{3}}((r-r_0)^2-\eps^{2\gamma})} \|f\|_Y \\
&\lesssim \eps^\gamma\|f\|_{Y},
\end{split}
\]
uniformly for $\eps>0$ small enough, $r_0 \in \mathcal{R}$, and $M \in \mathcal{N}$, and thus, the proof of Proposition \ref{proposition:r0} is complete.
\end{proof}

\subsection{Green's function towards infinity}\label{sec:tow-inf}

In this section we will take $f$ to be supported in the interval $[r_0+\eps^\gamma,\infty)$. The key idea to investigate the problem in this region is similar to how argued in the previous section, but with suitable changes to account for the unbounded domain and possible growth at infinity. We need to consider an appropriate approximation of \eqref{eq:1}--\eqref{eq:2} in this regime, and therefore, the behaviour of the general transport functions $\Omega$ and $U$ will play a crucial role.

The strategy to find such approximated version of \eqref{eq:1}--\eqref{eq:2} will be to neglect all terms that converge to zero as $r\to \infty$ and $\eps\to 0$, up to order $\eps^{1/3}$ (included). We find the system
\[
\eps\partial_r^2 V_1 - iM\eps^{-1/3}\left((\Omega-\Omega_0) - \frac{\Omega_0'}{U_0'}(U-U_0)\right)V_1 + \frac{\alpha\eps^{1/3}}{2}r\Omega'V_2 = f_1,
\]
\[
\eps\partial_r^2 V_2 - iM\eps^{-1/3}\left((\Omega-\Omega_0) - \frac{\Omega_0'}{U_0'}(U-U_0)\right)V_2 - \frac{\alpha\eps^{1/3}}{2}r\Omega'V_1 = f_2.
\]
Additionally, we will neglect in our approximated equation the coupled terms involving $\Omega'$. The possible growth of $\Omega'$ at infinity must be addressed later in the analysis, but will turn that we can treat it as an error, as long as the transport term $T(r)$ grows at least as quickly as $r\Omega'(r)$. In this way we have decoupled the ODE, and moreover, we observe that both components $V_1$ and $V_2$ satisfy the same equation
\begin{equation}\label{eq:r-large}
\eps\partial_r^2 V - iM\eps^{-1/3}T(r)V = f,
\end{equation}
where, as usual,
\[
T(r) = \left((\Omega(r)-\Omega_0) - \frac{\Omega_0'}{U_0'}(U(r)-U_0)\right).
\]
The transport term $T(r)$ at this stage could be very complex, so our strategy must be general enough to allow for a variety of behaviours of $T$. The crucial insight is that, as long as $T$ is non-vanishing, the equation \eqref{eq:r-large} is ``very far away'' from its spectrum. Therefore, we most certainly expect the operator associated to \eqref{eq:r-large} to be invertible, and in fact we may hope that its inverse is very well behaved. As such, we can get away with a rather simple approach for constructing an approximate inverse, which consists of pointwise linear approximations to the function $T(r)$, which, using the exponential decay of the associated Green's functions, may be glued together to produce a global inverse. 

Indeed, let $\delta>0$ to be chosen later. Select a sequence of points $\{r_j\}_{j\geq 1}$ such that the domain $[r_0+\eps^\gamma,\infty)$ is partitioned in intervals of the form $[r_j,r_j+\eps^\delta)$ where we identify 
\[
r_{j+1} = r_j+\eps^\delta,
\]
namely $r_1=r_0+\eps^\gamma$ and
\[
[r_0+\eps^\gamma,\infty) = \bigcup_{j=1}^\infty [r_j,r_j+\eps^\delta).
\]
We will denote this partition by $\mathcal{P}_{3}(\eps^\delta)$, i.e.\ $r_j = r_1+(j-1)\eps^\delta\in\mathcal{P}_3(\eps^\delta)$ for all $j\geq 1$. Now we evaluate the transport term precisely at $r_j$ on each interval to get rid of the $r-$dependent coefficients. Indeed, we can approximate \eqref{eq:r-large} by
\begin{equation}\label{eq:r-large-linear}
\eps\partial_r^2 V(r) - i\eps^{-1/3}MT(r_j)V(r) = f, \quad r\in [r_j,r_j+\eps^\delta).
\end{equation}
Setting $f=0$, this ODE has two nontrivial independent solutions
\[
v_1(r) = \e^{\sqrt{iMT(r_j)}\eps^{-2/3}r}, \quad v_2(r) = \e^{-\sqrt{iMT(r_j)}\eps^{-2/3}r},
\]
with Wronskian given by $W(v_1(r),v_2(r)) = 2\sqrt{iMT(r_j)}\eps^{-2/3}$. Therefore, following the usual method, we can construct a Green's function $\G_{3,j}$ associated to the linearly approximated ODE \eqref{eq:r-large-linear} in the domain $[r_j,r_j+\eps^\delta)$ with the correct decay at infinity, that is
\[
\G_{3,j}f(r) = \frac{\eps^{-1/3}}{2\sqrt{iMT(r_j)}}\left(\int_0^r\e^{-\sqrt{iMT(r_j)}\eps^{-2/3}(r-s)}f(s)\dd s + \int_r^\infty\e^{-\sqrt{iMT(r_j)}\eps^{-2/3}(s-r)}f(s)\dd s \right),
\]
where we always choose the branch $\Re(\sqrt{iMT(r_j)})> 0$, regardless of the sign of $MT(r_j)$.
Then we define the Green's function associated to the complete interval $[r_0+\eps^\gamma,\infty)$ by a linear combination
\begin{equation}\label{eq:Ginf-def}
\G_3 f(r) = \sum_{j=1}^\infty \G_{3,j}(f\mathds{1}_{[r_j,r_j+\eps^\delta)})(r).
\end{equation}
Observe that to have a well-defined Green's function, we need to make sure, among other things, that $T(r_j)\neq 0$ for all $j\geq 1$, which is directly ensured by Assumptions \ref{H0}--\ref{H3} together with the extra assumption \ref{SH}. 

Without further ado, let us introduce the main results corresponding to the solution towards infinity. We start with a preliminary pointwise estimate for each term $\G_{3,j}$.

\begin{lemma}\label{lemma:pointGrj-towinf}
Let $\gamma,\delta>0$ and consider the partition $\mathcal{P}_3(\eps^\delta)$ of the interval $[r_0+\eps^\gamma,\infty)$ in subintervals of size $\eps^\delta$. Let $\G_{3,j}$ be the Green's function associated to equation \eqref{eq:r-large-linear}, and let $f\in Y$. Then, under Assumptions \ref{H0}-\ref{H2}, and \ref{SH}, if 
\[
\frac{2}{3}-\gamma-\delta>0, \quad \gamma<\delta, \quad \gamma<\frac{1}{3},
\]
there exists $c_0>0$ independent of $\eps, r>0, r_0 \in \mathcal{R}, M \in \mathcal{N}$, where $\mathcal{R}$ is from Lemma \ref{lemma:consequences of assumptions}, and $\mathcal{N} \subset \mathbb{R} \setminus \{0\}$ is a compact set, so that for any $\eps>0$ sufficiently small independently of $r, r_j \geq r_0+\eps^\gamma$, and any $r\geq 0$ the following estimates hold true,
\begin{equation}\label{eq:Grj-towinf}
\begin{split}
    |w_\eps(r-r_0)\G_{3,j} (f\dsOne_{[r_j,r_j+\eps^\delta)})(r)| & \lesssim \|f\|_Y\frac{\eps^{1/3}}{r_j^2|MT(r_j)|} \e^{-c_0\eps^{-2/3+\gamma}(r_j-r)}\dsOne_{\{r<r_j-\eps^\delta\}}(r) \\
    &  + \|f\|_Y\frac{\eps^{1/3}}{r_j^2|MT(r_j)|}\dsOne_{\{r_j-\eps^\delta\leq r\leq r_j+2\eps^\delta\}}(r) \\
    &  + \|f\|_Y\frac{\eps^{1/3}}{r_j^2|MT(r_j)|} \e^{-c_0\eps^{-2/3+\gamma}(r-r_j-\eps^\delta)}\dsOne_{\{r>r_j+2\eps^\delta\}}(r),
\end{split}
\end{equation}
\begin{equation}\label{eq:drGrj-towinf}
\begin{split}
    |w_\eps(r-r_0)\partial_r\G_{3,j} (f\dsOne_{[r_j,r_j+\eps^\delta)})(r)| & \lesssim \|f\|_Y\frac{\eps^{-1/3}}{r_j^2\sqrt{|MT(r_j)|}} \e^{-c_0\eps^{-2/3+\gamma}(r_j-r)}\dsOne_{\{r<r_j-\eps^\delta\}}(r) \\
    &  + \|f\|_Y\frac{\eps^{-1/3}}{r_j^2\sqrt{|MT(r_j)|}}\dsOne_{\{r_j-\eps^\delta\leq r\leq r_j+2\eps^\delta\}}(r) \\
    &  + \|f\|_Y\frac{\eps^{-1/3}}{r_j^2\sqrt{|MT(r_j)|}} \e^{-c_0\eps^{-2/3+\gamma}(r-r_j-\eps^\delta)}\dsOne_{\{r>r_j+2\eps^\delta\}}(r),
\end{split}
\end{equation}
where the implicit constants from $\lesssim$ are \emph{independent} of $r, r_j \geq r_0+\eps^\gamma, \eps$, $r_0 \in \mathcal{R}$, $M \in \mathcal{N}$.
\end{lemma}

\begin{proof}
Let $\G_{3,j}$ be the Green's function associated to the linearised equation \eqref{eq:r-large-linear}. Fix $j\in\N$, and take $f\in L^\infty$ with $\supp(f)\subset [r_j,r_j+\eps^\delta)$. Using the branch $\Re(\sqrt{iMT(r_j)})>0$, and since $\Re(\sqrt{i})=1/\sqrt{2}$ regardless of the sign of $T(r_j)$, we find the pointwise estimate
\[
\begin{split}
    |\G_{3,j}f(r)| & \leq  \frac{\eps^{-1/3}}{2\sqrt{|MT(r_j)|}} \int_0^r\e^{-\frac{1}{\sqrt{2}}\sqrt{|MT(r_j)|}\eps^{-2/3}(r-s)}|f(s)|\dd s \\
    & \quad + \frac{\eps^{-1/3}}{2\sqrt{|MT(r_j)|}}\int_r^\infty\e^{-\frac{1}{\sqrt{2}}\sqrt{|MT(r_j)|}\eps^{-2/3}(s-r)}|f(s)|\dd s.
\end{split}
\]
Since $\supp(f)\subset [r_j,r_j+\eps^\delta)$, if $r<r_j-\eps^\delta$ we can write
\[
\begin{split}
|w_\eps(r-r_0)\G_{3,j}f(r)| & \leq  \frac{\eps^{-1/3}w_\eps(r-r_0)}{2\sqrt{|MT(r_j)|}}\int_{r_j}^{r_j+\eps^\delta}\e^{-\frac{1}{\sqrt{2}}\sqrt{|MT(r_j)|}\eps^{-2/3}(s-r)}|f(s)|\dd s \\
& \leq \|f\|_Y  \frac{\eps^{-1/3}w_\eps(r-r_0)}{2r_j^2\sqrt{|MT(r_j)|}}\int_{r_j}^\infty w_\eps(s-r_0)^{-1}\e^{-\frac{1}{\sqrt{2}}\sqrt{|MT(r_j)|}\eps^{-2/3}(s-r)}\dd s
\end{split}
\]
and therefore, using that $w_\eps(s-r_0)^{-1}$ is a decreasing function for all $s>r_0$, we obtain
\begin{equation}
\label{eq:weight bound linear interpolation}
|w_\eps(r-r_0)\G_{3,j}f(r)|  \lesssim \|f\|_Y\frac{\eps^{1/3}w_\eps(r-r_0)}{w_\eps(r_j-r_0)r_j^2|MT(r_j)|}\e^{-\frac{1}{\sqrt{2}}\sqrt{|MT(r_j)|}\eps^{-2/3}(r_j-r)}.
\end{equation}
Now, by Lemma \ref{lemma:properties-weight}, there exists a constant $C>0$ so that uniformly for $r_j \geq r_0+\eps^\gamma$, $r \leq r_0$ it holds
\begin{equation*}
\frac{w_\eps(r-r_0)}{w_\eps(r_j-r_0)}\leq Ce^{\eps^{-\gamma}(r_j-r_0)}.
\end{equation*}
In particular, notice that $|T(r_j)|\gtrsim \eps^{\gamma}$ uniformly for $r_j \geq r_0+\eps^\gamma$, $r_0 \in \mathcal{R}$ by Lemma \ref{lemma:consequences of assumptions}. Therefore, $\frac{1}{\sqrt{2}}\sqrt{|MT(r_j)|}\eps^{-2/3+\gamma}\gtrsim \eps^{-2/3+2\gamma} \to \infty$ as $\eps \to 0$, provided that $\gamma<\frac{1}{3}$. In particular, for all $\eps>0$ sufficiently small, it holds $\eps^{-\gamma}\leq \frac{1}{2\sqrt{2}}\sqrt{|MT(r_j)|}\eps^{-2/3}$, uniformly for $r_j \geq r_0+\eps^\gamma$, $r_0 \in \mathcal{R}$, $M \in \mathcal{N}$.
Hence, we may further upper bound the expression from \eqref{eq:weight bound linear interpolation} by
\begin{equation}
\label{eq:linear interpolation weight bound 2}
\|f\|_{Y} \frac{\eps^{1/3}}{r_j^2 |MT(r_j)|}e^{-(\frac{1}{\sqrt{2}}\sqrt{|MT(r_j)|}\eps^{-\frac{2}{3}}-\eps^{-\gamma})(r_j-r)} \leq C\|f\|_{Y} \frac{\eps^{1/3}}{r_j^2 |MT(r_j)|}e^{-\frac{1}{2\sqrt{2}}\sqrt{|MT(r_j)|}\eps^{-\frac{2}{3}}(r_j-r)}.
\end{equation}
Certainly this argument still works uniformly for $r_j-\eps^\delta\leq r<r_j$, but in that particular case the contribution of the exponential may be negligible, since in the limit $r\to r_j$ we do not gain any decay. Thus, for $r_j-\eps^\delta\leq r<r_j$ we write
\[
|w_\eps(r-r_0)\G_{3,j}f(r)| \lesssim \|f\|_Y\frac{\eps^{1/3}}{r_j^2|MT(r_j)|}.
\]
If $r\in (r_j,r_j+2\eps^\delta)$ we write,
\[
\begin{split}
    |w_\eps(r-r_0)\G_{3,j}f(r)| & \leq  \frac{\eps^{-1/3}w_\eps(r-r_0)}{2\sqrt{|MT(r_j)|}} \int_{r_j}^{r}\e^{-\frac{1}{\sqrt{2}}\sqrt{|MT(r_j)|}\eps^{-2/3}(r-s)}|f(s)|\dd s \\
    & \quad + \frac{\eps^{-1/3}w_\eps(r-r_0)}{2\sqrt{|MT(r_j)|}} \int_{r} ^{r_j+\eps^\delta}e^{-\frac{1}{\sqrt{2}}\sqrt{|MT(r_j)|}\eps^{-2/3}(s-r)}|f(s)| \dd s.
\end{split}
\]
We deal with the terms separately. Note first that 
\begin{align*}
&\frac{\eps^{-1/3}w_\eps(r-r_0)}{2\sqrt{|MT(r_j)|}}\int_{r} ^{r_j+\eps^\delta}e^{-\frac{1}{\sqrt{2}}\sqrt{|MT(r_j)|}\eps^{-2/3}(s-r)}f(s) \dd s\\
& \qquad \leq \|f(r)\|_Y\frac{\eps^{-1/3}w_\eps(r-r_0)}{2\sqrt{|MT(r_j)|}}\int_{r} ^{r_j+\eps^\delta}e^{-\frac{1}{\sqrt{2}}\sqrt{|MT(r_j)|}\eps^{-2/3}(s-r)}w_\eps(s-r_0)^{-1} \dd s\\
& \qquad \leq \|f(r)\|_Y\frac{\eps^{1/3}}{r_j^2|MT(r_j)|}.
\end{align*}
Here we used again that $w_\eps(s-r_0)^{-1}$ is decreasing for $s > r_0$. Next, we estimate
\[
\begin{split}
|w_\eps(r-r_0)\G_{3,j}f(r)| & \leq  \frac{\eps^{-1/3}w_\eps(r-r_0)}{2\sqrt{|MT(r_j)|}} \int_{r_j}^{r}\e^{-\frac{1}{\sqrt{2}}\sqrt{|MT(r_j)|}\eps^{-2/3}(r-s)}f(s)\dd s\\
& \lesssim  \|f\|_Y\frac{\eps^{-1/3}w_\eps(r-r_0)}{2r_j^2\sqrt{|MT(r_j)|}}\int_{r_j}^{r}w_\eps(s-r_0)^{-1}\e^{-\frac{1}{\sqrt{2}}\sqrt{|MT(r_j)|}\eps^{-2/3}(r-s)}\dd s \\
& \lesssim  \|f\|_Y \frac{\eps^{1/3}w_\eps(r-r_0)}{w_\eps(r_j-r_0)r_j^2|MT(r_j)|}.
\end{split}
\]
By  Lemma \ref{lemma:properties-weight}, we find that the quotient
\[
\frac{w_\eps(r-r_0)}{w_\eps(r_j-r_0)}
\]
is bounded by a constant, uniformly for $r_j \geq r_0+\eps^\gamma$, $|r-r_j|<2\eps^\delta$ (provided $\gamma<\delta)$, and thus we can write
\begin{equation*}
    |w_\eps(r-r_0)\G_{3,j}f(r)| \lesssim \|f\|_Y \frac{\eps^{1/3}}{r_j^2|MT(r_j)|}. 
\end{equation*}
Finally, if $r \geq r_j+2\eps^\delta$, then we estimate
\begin{align*}
|w_\eps(r-r_0)\G_{3,j}f(r)|&\lesssim \|f\|_Y \frac{\eps^{1/3}w_\eps(r-r_0)}{w_\eps(r_j-r_0)r_j^2|MT(r_j)|} \e^{-\frac{1}{\sqrt{2}}\sqrt{|MT(r_j)|}\eps^{-2/3}(r-r_j)}\e^{\frac{1}{\sqrt{2}}\sqrt{|MT(r_j)|}\eps^{-2/3+\delta}}\\
& = \|f\|_Y\frac{\eps^{1/3}w_\eps(r-r_0)}{w_\eps(r_j-r_0)r_j^2|MT(r_j)|} \e^{-\frac{1}{\sqrt{2}}\sqrt{|MT(r_j)|}\eps^{-2/3}(r-r_j-\eps^\delta)}.
\end{align*}
It remains to note the bound
\begin{equation*}
    \frac{w_\eps(r-r_0)}{w_\eps(r_j-r_0)}\e^{-\frac{1}{\sqrt{2}}\sqrt{|MT(r_j)|}\eps^{-2/3}(r-r_j-\eps^\delta)} \leq C\e^{-\frac{1}{2\sqrt{2}}\sqrt{|MT_\eps|}\eps^{-2/3}(r-r_j-\eps^\delta)},
\end{equation*}
for some $C>0$ \emph{independent} of $r>0$, $r_j \geq r_0+\eps^\gamma$, and $\eps>0$ small enough, which can be deduced from Lemma \ref{lemma:properties-weight} in precisely the same way as \eqref{eq:linear interpolation weight bound 2}. Noting finally that by Lemma \ref{lemma:consequences of assumptions}, $|MT(r_j)|\gtrsim \eps^{2\gamma}$ uniformly for $r_j \geq r_0+\eps^\gamma$, $r_0 \in \mathcal{R}$, $M \in \mathcal{N}$, we arrive at the first claim of the lemma, estimate \eqref{eq:Grj-towinf}.

In order to deal with the $r$ derivatives of the Green's functions, we directly compute them using the definition of $\G_{3,j}$,
\[
\partial_r \G_{3,j}f(r) = -\frac{\eps^{-1}}{2}\left( \int_0^r \e^{-\sqrt{iMT(r_j)}\eps^{-2/3}(r-s)}f(s)\dd s + \int_r^\infty\e^{-\sqrt{iMT(r_j)}\eps^{-2/3}(s-r)}f(s)\dd s \right),
\]
Therefore, 
\[
|\partial_r\G_{3,j}f(r)| \leq  \frac{\eps^{-1}}{2}\left(\int_0^r\e^{-\frac{1}{\sqrt{2}}\sqrt{|MT(r_j)|}\eps^{-2/3}(r-s)}|f(s)|\dd s + \int_r^\infty\e^{-\frac{1}{\sqrt{2}}\sqrt{|MT(r_j)|}\eps^{-2/3}(s-r)}|f(s)|\dd s \right),
\]
and similarly to how we argued for the previous case but noticing the different exponents on $\varepsilon$ and $|T(r_j)|$ now, we obtain the analogous pointwise estimate \eqref{eq:drGrj-towinf}.
\end{proof}

Lemma \ref{lemma:pointGrj-towinf} yields a pointwise estimate in $L^\infty((0,\infty),w_\eps(r-r_0)\dd r)$ for each addend $\G_{3,j}$ that constitute the Green's function $\G_3$. In order to obtain an analogous estimate in the full norm $X$ we must include the extra weight $\max\{1,r^2\}$. In this aspect, it is convenient to obtain estimates---akin to \eqref{eq:Grj-towinf} and \eqref{eq:drGrj-towinf}---with an $r^2$ weight. From Lemma \ref{lemma:pointGrj-towinf}, it is rather direct to see that we have the following pointwise bounds
\begin{equation}\label{eq:Ginf-pointwise}
\begin{split}
    |r^2w_\eps(r-r_0)\G_{3,j} & (f\dsOne_{[r_j,r_j+\eps^\delta)})(r)| \lesssim \|f\|_Y\frac{\eps^{1/3}}{|MT(r_j)|} \e^{-c_0\eps^{-2/3+\gamma}(r_j-r)}\dsOne_{\{r_j-r>\eps^\delta\}}(r) \\
    & + \|f\|_Y\frac{\eps^{1/3}}{|MT(r_j)|}\dsOne_{\{r_j-\eps^\delta\leq r\leq r_j+2\eps^\delta\}}(r) \\
    & + \|f\|_Y\frac{\eps^{1/3}}{|MT(r_j)|}\left(1+\frac{|r-r_j|^2}{r^2_j}\right) \e^{-c_0\eps^{-2/3+\gamma}(r-r_j-\eps^\delta)}\dsOne_{\{r-r_j>2\eps^\delta\}}(r),
\end{split}
\end{equation}
as well as,
\begin{equation}\label{eq:drGinf-pointwise}
\begin{split}
    |r^2w_\eps(r-r_0)\partial_r \G_{3,j} & (f\dsOne_{[r_j,r_j+\eps^\delta)})(r)| \lesssim \|f\|_Y\frac{\eps^{-1/3}}{\sqrt{|MT(r_j)|}} \e^{-c_0\eps^{-2/3+\gamma}(r_j-r)}\dsOne_{\{r_j-r>\eps^\delta\}}(r) \\
    & + \|f\|_Y\frac{\eps^{-1/3}}{\sqrt{|MT(r_j)|}}\dsOne_{\{r_j-\eps^\delta\leq r\leq r_j+2\eps^\delta\}}(r)\\
     & + \|f\|_Y\frac{\eps^{-1/3}}{\sqrt{|MT(r_j)|}}\left(1+\frac{|r-r_j|^2}{r^2_j}\right) \e^{-c_0\eps^{-2/3+\gamma}(r-r_j-\eps^\delta)}\dsOne_{\{r-r_j>2\eps^\delta\}}(r). \\
\end{split}
\end{equation}
Notice that, after introducing the $r^2$ factor, when it is not comparable to $r_j^2$ we use the relation
\[
\frac{r^2}{r_j^2} \leq 2\left(1+\frac{|r-r_j|^2}{r_j^2}\right ).
\]
These estimates, together with Lemma \ref{lemma:pointGrj-towinf}, provide the key to prove the main proposition of the section, that we present next.

\begin{proposition}\label{proposition:rinf}
Let $\gamma,\delta>0$ be such that
\[
\gamma<\frac{1}{3}, \quad \gamma + \delta <\frac{2}{3}, \quad \gamma<\delta,
\]
consider the partition $\mathcal{P}_3(\eps^\delta)$ of the interval $[r_0+\eps^\gamma,\infty)$ in subintervals of length $\eps^\delta$, and let $\G_3$ be the Green's function introduced in \eqref{eq:Ginf-def}. Let $f\in Y$ with $\supp(f)\subset [r_0+\eps^\gamma,\infty)$. Under Assumptions \ref{H0}--\ref{H2} and \ref{SH}, the following estimates hold true uniformly for any $\eps>0$ sufficiently small, $r_0 \in \mathcal{R}$, and $M$ in any compact set $\mathcal{N} \subset \mathbb{R}\setminus \{0\}$.
\begin{equation}\label{eq:Grinf}
\|\G_3 f\|_X \lesssim \eps^{1/3-2\gamma}\|f\|_Y,
\end{equation}
\begin{equation}\label{eq:drGrinf}
\|\partial_r\G_3 f\|_X \lesssim \eps^{-1/3-\gamma}\|f\|_Y.
\end{equation}
Moreover, the error caused by the approximated operator can be estimated by
\begin{equation}\label{eq:Grinf-error}
\| ((\L-\lambda)\G_{3}-\Id) f\|_Y \lesssim \left(\eps^{2/3-2\gamma} + \eps^{2/3-2\gamma} + \eps^{\beta}\right) \|f\|_Y,
\end{equation}
where the coefficient $\beta>0$ is from \ref{P3}.
\end{proposition}

\begin{proof}
\emph{Part 1: Estimates on $\G_3$ and $\partial_r\G_3$.} First of all we will derive the $X$ and $Y$ bounds on the Green's function and its first order derivatives. These will be a direct consequence of the estimates derived in Lemma \ref{lemma:pointGrj-towinf}. In particular, for $\G_3$ we see that we can combine \eqref{eq:Grj-towinf} and \eqref{eq:Ginf-pointwise} as follows
\[
\begin{split}
    |\max\{1,r^2\}w_\eps(r-r_0)\G_{3,j} & (f\dsOne_{[r_j,r_j+\eps^\delta)})(r)| \lesssim \|f\|_Y\frac{\eps^{1/3}}{|MT(r_j)|} \e^{-c_0\eps^{-2/3+\gamma}(r_j-r)}\dsOne_{\{r_j-r>\eps^\delta\}}(r) \\
    & + \|f\|_Y\frac{\eps^{1/3}}{|MT(r_j)|}\dsOne_{\{r_j-\eps^\delta\leq r\leq r_j+2\eps^\delta\}}(r) \\
    & + \|f\|_Y\frac{\eps^{1/3}}{|MT(r_j)|}\left(1+|r-r_j|^2\right) \e^{-c_0\eps^{-2/3+\gamma}(r-r_j-\eps^\delta)}\dsOne_{\{r-r_j>2\eps^\delta\}}(r).
\end{split}
\]
Using linearity from the definition \eqref{eq:Ginf-def} we can write 
\[
|\max\{1,r^2\}w_\eps(r-r_0)\G_3 f(r)| \leq \sum_{j=1}^\infty |\max\{1,r^2\}w_\eps(r-r_0)\G_{3,j}(f\dsOne_{[r_j,r_j+\eps^\delta)})(r)|.
\]
Now, say that $r\in [r_k,r_k+\eps^\delta)$, where $r_k = r_1+(k-1)\eps^\delta$ is an element of the partition $\mathcal{P}_3(\eps^\delta)$. Then the pointwise estimate after summing over all $j$ can be bounded by
\[
\begin{split}
    |\max\{1,r^2\}w_\eps(r-r_0)\G_3 f(r)| & \lesssim \|f\|_Y \sum_{j=1}^{k-2} \eps^{1/3-2\gamma}\left(1+|r-r_j|^2\right)\e^{-c_0\eps^{-2/3+\gamma}(r-r_j-\eps^\delta)} \\
    & \quad + \|f\|_Y \sum_{j=k+2}^{\infty} \eps^{1/3-2\gamma} \e^{-c_0\eps^{-2/3+\gamma}(r_j-r)} \\
    & \quad + \|f\|_Y\eps^{1/3}\left(\frac{1}{|MT(r_{k-1})|}+\frac{1}{|MT(r_{k})|}+\frac{1}{|MT(r_{k+1})|}\right).
\end{split}
\]
Naturally, if $r<r_0+\eps^\gamma$ we cannot write $r\in [r_k,r_k+\eps^\delta)$, but in such case one can simply obtain the corresponding contributions using the same bounds, namely if $r<r_1-2\eps^\delta$
\[
|\max\{1,r^2\}w_\eps(r-r_0)\G_3 f(r)| \lesssim \|f\|_Y \sum_{j=1}^{\infty} \eps^{1/3-2\gamma} \e^{-c_0\eps^{-2/3+\gamma}(r_j-r)},
\]
and the claim of the proposition follows readily.
Let us thus analyse the case $r\in [r_k,r_{k+1})$, by starting with the contributions of $j=k-1$, $j=k$ and $j=k+1$. By Lemma \ref{lemma:consequences of assumptions}, these may be estimated by
\[
\|f\|_Y\eps^{1/3}\left(\frac{1}{|MT(r_{k-1})|}+\frac{1}{|MT(r_{k})|}+\frac{1}{|MT(r_{k+1})|}\right) \lesssim \|f\|_Y\eps^{1/3-2\gamma}.
\]
For the other two terms we observe that we have a converging geometric series that goes to zero quickly as $\eps\to 0$, namely using that $r_j = r_1+(j-1)\eps^\delta$,
\[
\begin{split}
    \sum_{j=k+2}^{\infty} \eps^{1/3-2\gamma} \e^{-c_0\eps^{-2/3+\gamma}(r_j-r)} & \leq \sum_{j=k+2}^{\infty} \eps^{1/3-2\gamma} \e^{-c_0\eps^{-2/3+\gamma}(r_1+(j-1)\eps^\delta-r)} \\\
    & = \eps^{1/3-2\gamma} \e^{-c_0\eps^{-2/3+\gamma}(r_1-\eps^\delta-r)} \sum_{j=k+2}^{\infty} \e^{-c_0\eps^{-2/3+\gamma+\delta}j} \\
    & \lesssim \eps^{1/3-2\gamma} \e^{-c_0\eps^{-2/3+\gamma}(r_1+(k+1)\eps^\delta-r)},
\end{split}
\]
where the last inequality holds uniformly in $\eps>0$ sufficiently small, provided that $\gamma+\delta<2/3$.
Taking into account that $r_1+(k+1)\eps^\delta = r_{k+2}$, and $r<r_{k+1}$ we write
\[
\sum_{j=k+2}^{\infty} \eps^{1/3-2\gamma} \e^{-c_0\eps^{-2/3+\gamma}(r_j-r)} \lesssim \eps^{1/3-2\gamma} \e^{-c_0\eps^{-2/3+\gamma}(r_{k+2}-r)} \leq \eps^{1/3-2\gamma} \e^{-c_0\eps^{-2/3+\gamma+\delta}}.
\]
Hence, crucially using the condition $\gamma+\delta < 2/3$, we obtain the quantitative converging estimate as $\eps\to 0$. In particular, this bound holds \emph{uniformly} for $r_0 \in \mathcal{R}$, $M \in \mathcal{N}$. The estimate for the remaining term follows from the estimate,
\[
\sum_{j=1}^{k-2} \eps^{1/3-2\gamma}\left(1+|r-r_j|^2\right)\e^{-c_0\eps^{-2/3+\gamma}(r-r_j-\eps^\delta)} \lesssim \sum_{j=1}^{k-2} \eps^{1/3-2\gamma}\e^{-\frac{1}{2}c_0\eps^{-2/3+\gamma}(r-r_j-\eps^\delta)},
\]
thus, since $r\geq r_k$ we have that $r-r_{k-2}-\eps^\delta\geq \eps^\delta$, and
\[
\sum_{j=1}^{k-2} \eps^{1/3-2\gamma}\e^{-\frac{1}{2}c_0\eps^{-2/3+\gamma}(r-r_j-\eps^\delta)} \lesssim \eps^{1/3-2\gamma}\e^{-\frac{1}{2}c_0\eps^{-2/3+\gamma}(r-r_{k-2}-\eps^\delta)} \lesssim \eps^{1/3-2\gamma} \e^{-\frac{1}{2}c_0\eps^{-2/3+\gamma+\delta}}.
\]
\[
\begin{split}
    \eps^{1/3-2\gamma}\sum_{j=1}^{k-2} \e^{-\frac{1}{2}c_0\eps^{-2/3+\gamma}(r-r_1-(j-1)\eps^\delta-\eps^\delta)} & \lesssim \eps^{1/3-2\gamma} \e^{-\frac{1}{2}c_0\eps^{-2/3+\gamma}(r-r_1)} \e^{\frac{1}{2}c_0\eps^{-2/3+\gamma+\delta}(k-2)} \\
    & \leq \eps^{1/3-2\gamma} \e^{-c_0\eps^{-2/3+\gamma+\delta}}.
\end{split}
\]
All in all, putting everything together we find for $r\in [r_k,r_{k+1})$,
\[
|\max\{1,r^2\}w_\eps(r-r_0)\G_3 f(r)| \lesssim \|f\|_Y\eps^{1/3-2\gamma}\left( \dsOne_{[r_k,r_k+\eps^\delta)}(r) + \e^{-\frac{1}{2}c_0\eps^{-2/3+\gamma+\delta}}\dsOne_{[0,r_k)\cap[r_{k+1},\infty)}(r) \right),
\]
with the implicit constants in the $\lesssim$ symbol independent of $\eps>0$, $r_k \geq r_0+\eps^\gamma$, which in particular implies \eqref{eq:Grinf}.

The estimate on $\partial_r\G_2$ follows similarly, we can put together the pointwise estimates for $\partial_r\G_{r_j}$ from \eqref{eq:drGinf-pointwise}, i.e.\
\[
\begin{split}
    |\max\{1,r^2\}w_\eps(r-r_0)\partial_r &\G_{r_j} (f\dsOne_{[r_j,r_j+\eps^\delta)})(r)| \lesssim \|f\|_{Y}\frac{\eps^{-1/3}}{\sqrt{|MT(r_j)|}} \e^{-c_0\eps^{-2/3+\gamma}(r_j-r)}\dsOne_{\{r_j-r>\eps^\delta\}}(r) \\
    & + \|f\|_Y\frac{\eps^{-1/3}}{\sqrt{|MT(r_j)|}}\dsOne_{\{r_j-\eps^\delta\leq r\leq r_j+2\eps^\delta\}}(r)\\
     & + \|f\|_Y\frac{\eps^{-1/3}}{\sqrt{|MT(r_j)|}}\left(1+\frac{|r-r_j|^2}{r_j^2}\right) \e^{-c_0\eps^{-2/3+\gamma}(r-r_j-\eps^\delta)}\dsOne_{\{r-r_j>2\eps^\delta\}}(r). \\
\end{split}
\]
Therefore, a completely analogous argument to that used for $\G_3$ yields
\[
|\max\{1,r^2\}w_\eps(r-r_0)\partial_r\G_2 f(r)| \lesssim \|f\|_Y\eps^{-1/3-\gamma}\left( \dsOne_{[r_k,r_k+\eps^\delta)}(r) + \e^{-\frac{1}{2}c_0\eps^{-2/3+\gamma+\delta}}\dsOne_{[0,r_k)\cap[r_{k+1},\infty)}(r) \right),
\]
and thus, as a by-product, we obtain \eqref{eq:drGrinf}.

\emph{Part 2: The error estimate.} We now move on to computing the error committed by the linear approximation in the regime $r\geq r_0+\eps^\gamma$.
We define,
\[
\L_{3,j}V(r) = \eps\partial_r^2 V(r) - i \eps^{-1/3}MT(r_j)V(r),
\]
so that
\[
\L_{3,j}\G_{3,j}(f\dsOne_{[r_j,r_j+\eps^\delta)})(r) = (f\dsOne_{[r_j,r_j+\eps^\delta)})(r),
\]
for any $j\in\N$. The exact operator $\L-\lambda$ has two components, and we want to approximate both with the same operator $\L_3$. We can write the exact operator as
\[
(\L-\lambda)\begin{pmatrix}
    V_1\\
    V_2
\end{pmatrix} = \eps\left(\partial_r^2+\frac{1}{r}\partial_r\right)\begin{pmatrix}
    V_1\\
    V_2
\end{pmatrix} - \begin{pmatrix}
    A_\eps + \lambda + P_\eps & - Q_\eps \\
    Q_\eps & A_\eps + \lambda - P_\eps
\end{pmatrix}\begin{pmatrix}
    V_1\\
    V_2
\end{pmatrix},
\]
where we define
\begin{equation}
\label{eq:PQ definition}
\begin{split}
P_\eps(r) = &\frac{iM\alpha\eps}{r^2}-\sqrt{\frac{-2iM\Omega_0'}{r_0}}\eps^{1/3} - D_\eps^-(r), \quad Q_\eps(r) = \frac{iM\alpha\eps}{r^2}+D_\eps^+(r),\\
&A_\eps(r) = iM\eps^{-1/3}T(r) + \frac{\eps}{r^2} + \eps^{1/3}M^2\left(\frac{1}{r^2} + \left(\frac{\Omega_0'}{U_0'}\right)^2\right),
\end{split}
\end{equation}
with $D_\eps$ defined as in Section \ref{s:equations}.
With this information we can now evaluate the size of the error. For the first component, let $j\in\N$, so that we can write
\[
\begin{split}
(\L-\lambda-\L_{3,j})V_1(r) & = \frac{\eps}{r}\partial_rV_1(r) - iM\eps^{-1/3}(T(r)-T(r_j))V_1(r) \\
& \quad -\left[ \frac{\eps}{r^2} + \eps^{1/3}M^2\left(\frac{1}{r^2} + \left(\frac{\Omega_0'}{U_0'}\right)^2\right)+\lambda + P_\eps(r) \right]V_1 + Q_\eps(r)V_2(r).
\end{split}
\]
Now, we get the following pointwise estimates
\[
\begin{split}
|D_\eps^+(r)|,|D_\eps^-(r)| & \lesssim \eps^{1/3}\left(\frac{r^2-r_0^2}{r_0^2r^2}+ |r\Omega'(r) - r_0\Omega'_0|\right) \\
& \lesssim \eps^{1/3}\left(1 + \frac{1}{r^2} + |r\Omega'(r)|\right),
\end{split}
\]
\[
|P_\eps(r)| \lesssim \frac{\eps}{r^2} + \eps^{1/3} + |D_\eps^-(r)| \lesssim \frac{\eps}{r^2} + \eps^{1/3}\left(1 + \frac{1}{r^2} + |r\Omega'(r)| \right),
\]
\[
|Q_\eps(r)| \lesssim \frac{\eps}{r^2} + |D_\eps^+(r)| \lesssim \frac{\eps}{r^2} + \eps^{1/3}\left(1 + \frac{1}{r^2} +  |r\Omega'(r)| \right).
\]
Therefore, noting that 
\begin{equation*}
(\L-\lambda)\G_3 f-f=\sum_{j \geq 1}(\L-\lambda)\G_{3,j}(f\mathds{1}_{[r_j,r_j+\eps^\delta)})-\sum_{j \geq 1}\L_{3,j}\G_{3,j}(f\mathds{1}_{[r_j,r_j+\eps^\delta)}),
\end{equation*}
and using $\lambda = \mu\eps^{1/3}$ we obtain
\begin{equation}\label{eq:error-Ginf-mid}
\begin{split}
    \left|r^2w_\eps(r-r_0)((\L-\lambda)\G_3-\Id)f\right| & \lesssim \eps\left|rw_\eps(r-r_0)\partial_r\G_3f\right| + \eps^{1/3}\left|\max\{1,r^2\}w_\eps(r-r_0)\G_3f\right| \\
    &  + \eps^{-1/3}\sum_{j=1}^\infty\left|r^2w_\eps(r-r_0)M(T(r)-T(r_j))\G_{3,j}(f\dsOne_{[r_j,r_j+\eps\delta)})\right| \\
 &  + \eps^{1/3}|r\Omega'(r)|\left|r^2w_\eps(r-r_0)\G_3f\right|.
 \end{split} 
\end{equation}
For the term involving the derivative of the Green's function we find that
\[
\eps \left|rw_\eps(r-r_0)\partial_r\G_3f\right| \leq \eps \|\partial_r\G_3f\|_X \lesssim \eps^{2/3-2\gamma}\|f\|_Y,
\]
by a direct application of \eqref{eq:drGrinf}. The second addend in the right hand side in \eqref{eq:error-Ginf-mid} can be bounded using \eqref{eq:Grinf} by
\[
\eps^{1/3}\left|\max\{1,r^2\}w_\eps(r-r_0)\G_3f \right| \leq \eps^{2/3-2\gamma}\|f\|_Y.
\]
We further see that all of the above bounds are uniform for $r_0 \in \mathcal{R}$, $M \in \mathcal{N}$.
In order to estimate the terms involving $T(r)$ it will not suffice to use \eqref{eq:Grinf}, and so we shall employ the more precise bound \eqref{eq:Ginf-pointwise}. We write for the first component
\[
\begin{split}
    \eps^{-1/3}\sum_{j=1}^\infty & \left|r^2w_\eps(r-r_0)M(T(r)-T(r_j))\G_{3,j}(f_1\dsOne_{[r_j,r_j+\eps^\delta)})(r)\right| \\
    & \leq \|f_1\|_Y\sum_{j=1}^\infty\frac{|T(r)-T(r_j)|}{|T(r_j)|} \left(1+\frac{|r-r_j|^2}{r_j^2}\right) \e^{-c_0\eps^{-2/3+\gamma}(r-r_j-\eps^\delta)}\dsOne_{\{r>r_{j+1}\}} \\
    & \quad + \|f_1\|_Y\sum_{j=1}^\infty\frac{|T(r)-T(r_j)|}{|T(r_j)|}\e^{-c_0\eps^{-2/3+\gamma}(r_j-r)}\dsOne_{\{r<r_{j-1}\}} \\
    & \quad + \|f_1\|_Y\sum_{j=1}^\infty\frac{|T(r)-T(r_j)|}{|T(r_j)|}\dsOne_{\{r_{j-1}\leq r\leq r_{j+1}\}}\\
    & = \I_{\eps}(r) + \II_{\eps}(r) + \III_{\eps}(r). 
\end{split}
\]
To deal with the first two terms in the right hand side, we will use the fast decay of the exponentials to compensate for the (possible) growth of $T(r)$ as $r\to\infty$. Throughout this analysis we shall assume that $r \geq r_0+\eps^\gamma$, since the general case is analogous (and in fact easier). Let us start analysing the first addend. On the one hand, notice that there must exist some $k \geq 0$ so that $r\in[r_k,r_k+\eps^\delta)$ with $k>2$, and hence by Property \ref{P3} we have
\[
\begin{split}
    \I_{\eps}(r) & = \|f_1\|_Y\sum_{j=1}^{k-2}\frac{|T(r)-T(r_j)|}{|T(r_j)|} \left(1+\frac{|r-r_j|^2}{r_j^2}\right)\e^{-c_0\eps^{-2/3+\gamma}(r-r_j-\eps^\delta)} \\
    & \leq \|f_1\|_Y\sum_{n=1}^{N_1}  B_n(\eps) \sum_{j=1}^{k-2} |r-r_j|^n \left(1+|r-r_j|^2\right)\e^{-c_0\eps^{-2/3+\gamma}(r-r_j-\eps^\delta)}.
\end{split}
\]
In order to compute the partial sums of a polynomial multiplying a decaying exponential we argue as follows. Let $n\in\N$ and $r_j = r_1 + (j-1)\eps^\delta$, then we write
\[
\begin{split}
    \sum_{j=1}^{k-2} |r-r_j|^n \e^{-c_0\eps^{-2/3+\gamma}(r-r_j-\eps^\delta)} & \leq n!\sum_{j=1}^{k-2} \e^{n(r-r_j)-c_0\eps^{-2/3+\gamma}(r-r_j-\eps^\delta)} \\
    & = n!\e^{-\left(c_0\eps^{-2/3+\gamma}-n\right)(r-r_1)} \sum_{j=1}^{k-2} \e^{\left(c_0\eps^{-2/3+\gamma}-n\right)j\eps^\delta} \\
    & \lesssim \e^{-\left(c_0\eps^{-2/3+\gamma}-n\right)(r-r_1)} \e^{\left(c_0\eps^{-2/3+\gamma}-n\right)(k-2)\eps^\delta} \\
    & = \e^{-\left(c_0\eps^{-2/3+\gamma}-n\right)(r-r_{k-1})} \lesssim \e^{-\left(c_0\eps^{-2/3+\gamma}-n\right)\eps^\delta}
\end{split}
\]
uniformly in $k \geq 0$, provided that we choose $\eps>0$ small enough so that $\eps^{-1/3+\gamma}>c_0^{-1}n$. Therefore we arrive at the estimate
\[
\I_{\eps}(r) \lesssim \|f_1\|_Y\e^{-\left(c_0\eps^{-2/3+\gamma}-2-N_1\right)\eps^\delta}\sum_{n=1}^{N_1} B_n(\eps) \lesssim \|f_1\|_Y\eps^{-N_1\delta+\beta}\e^{-\left(c_0\eps^{-2/3+\gamma}-2-N_1\right)\eps^\delta},
\]
uniformly in $k$, and where $\beta>0$ comes from Property \ref{P3}. Notice that it might possible that $-N_1\delta+\beta<0$, however this estimate is still rapidly converging to zero as $\eps\to 0$. We apply a similar argument for the second term,
\[
\begin{split}
    \II_{\eps}(r) & = \|f_1\|_Y\sum_{j=k+2}^\infty\frac{|T(r)-T(r_j)|}{|T(r_j)|}\e^{-c_0\eps^{-2/3+\gamma}(r_j-r)} \\
    & \leq \|f_1\|_Y\sum_{n=1}^{N_1} B_n(\eps)\sum_{j=k+2}^{\infty}|r_j-r|^n\e^{-c_0\eps^{-2/3+\gamma}(r_j-r)} \\
    & \lesssim \|f_1\|_Y\e^{-\left(c_0\eps^{-2/3+\gamma}-N\right)\eps^\delta}\sum_{n=1}^{N_1}  B_n(\eps) \\
    & \lesssim \|f_1\|_Y\eps^{-N_1\delta+\beta}\e^{-\left( c_0\eps^{-2/3+\gamma}-N_1\right)\eps^\delta}
\end{split}
\]
Last, for the third addend in the error estimate we use Property \ref{P3} once again to achieve
\[
\begin{split}
    \III_{\eps}(r) & = \|f_1\|_Y\left(\frac{|T(r)-T(r_{k-1})|}{|T(r_{k-1})|} + \frac{|T(r)-T(r_k)|}{|T(r_k)|} + \frac{|T(r)-T(r_{k+1})|}{|T(r_{k+1})|}\right) \\
    & \leq \|f_1\|_Y\sum_{n=1}^{N_1} B_n(\eps)\left(|r-r_{k-1}|^n + |r-r_{k}|^n + |r-r_{k+1}|^n\right).
\end{split}
\]
Moreover, since $r\in [r_k,r_k+\eps^\delta)$, we can use that $|r-r_{k-1}|\leq 2\eps^\delta$, $|r-r_{k}|\leq \eps^\delta$, and $|r-r_{k+1}|\leq 2\eps^\delta$, to write
\[
\III_{\eps}(r) \lesssim \|f_1\|_Y\sum_{n=1}^{N_1} B_n(\eps)\eps^{n\delta} \lesssim \eps^\beta \|f_1\|_Y.
\]
Since $r \geq r_0+\eps^\gamma$ was arbitrary, all in all we may write
\[
\begin{split}
    \eps^{-1/3}\sum_{j=1}^\infty & \left|w_\eps(r-r_0)r^2(T(r)-T(r_j))\G_{r_j}(f_1\dsOne_{[r_j,r_j+\eps^\delta)})(r)\right| \\
    & \lesssim \eps^\beta \|f_1\|_Y\dsOne_{[r_{k-1},r_{k+1})}(r) \\
    & \quad + \|f_1\|_Y\eps^{-N\delta+\beta}\e^{-\left(c_0\eps^{-2/3+\gamma}-2-N\right)\eps^\delta}\dsOne_{[r_0+\eps^\gamma,r_{k-1})\cap [r_{k+1},\infty)}(r) \\
    & \lesssim \eps^\beta \|f_1\|_Y,
\end{split}
\]
for all $r \geq r_0+\eps^\gamma$, and it is straightforward to see that it holds in fact for all $r \geq 0$.
There is only one more term to be addressed in \eqref{eq:error-Ginf-mid}, namely the term concerning $\Omega'(r)$. For this term, we once again need to make use of the more precise estimate \eqref{eq:Ginf-pointwise}. As before, we argue for the first component. Let $r\in [r_k,r_k+\eps^\delta)$, then we write
\[
\begin{split}
    \eps^{1/3}|r\Omega'(r)||r^2w_\eps(r-r_0)\G_{\infty}f_1(r)| & \leq \eps^{2/3}\|f_1\|_Y \sum_{j=1}^{k-2} \left(1+\frac{|r-r_j|^2}{r_j^2}\right) \frac{|r\Omega'(r)|}{|MT(r_j)|} \e^{-c_0\eps^{-2/3+\gamma}(r-r_j-\eps^\delta)} \\
    & \quad + \eps^{2/3}\|f_1\|_Y \sum_{j=k+2}^{\infty} \frac{|r\Omega'(r)|}{|MT(r_j)|} \e^{-c_0\eps^{-2/3+\gamma}(r_j-r)} \\
    & \quad + \eps^{2/3}\|f_1\|_Y \left( \frac{|r\Omega'(r)|}{|T(r_{k-1})|} + \frac{|r\Omega'(r)|}{|MT(r_k)|} + \frac{|r\Omega'(r)|}{|T(r_{k+1})|}\right) \\
    & = \I_{k,\eps}(r) + \II_{k,\eps}(r) + \III_{k,\eps}(r).
\end{split}
\]
In order to deal with these terms we need to impose a new set of conditions on the function $\Omega$ and its relation with the function $T$, and this is where Property \ref{P4} comes into play. Using this hypothesis we can directly estimate the third term by noticing that
\[
\frac{|r\Omega'(r)|}{|T(r_{k-1})|} \leq \eps^{-2\gamma}\left( D_0 + \sum_{n=1}^{N_2} D_n|r-r'|^n\right) \lesssim \eps^{-2 \gamma}\left( 1 + \sum_{n=1}^{N_2} \eps^{n\delta}\right) \lesssim \eps^{-2\gamma},
\]
and analogously with $r_k$ and $r_{k+1}$. Thus, we arrive at the estimate,
\[
\III_{k,\eps}(r) \lesssim \eps^{2/3}\|f_1\|_Y \left( \frac{|r\Omega'(r)|}{|T(r_{k-1})|} + \frac{|r\Omega'(r)|}{|T(r_k)|} + \frac{|r\Omega'(r)|}{|T(r_{k+1})|}\right) \lesssim \eps^{2/3-2\gamma}\|f_1\|_Y.
\]
For the second addend we write,
\[
\begin{split}
    \II_{\eps,k}(r) & \leq \eps^{2/3-2\gamma}\|f_1\|_Y \sum_{j=k+2}^{\infty} \left( D_0 + \sum_{n=1}^{N_2} D_n|r-r'|^n\right) \e^{-c_0\eps^{-2/3+\gamma}(r_j-r)}.
\end{split}
\]
The interactions between the growing polynomial and decreasing exponential inside the infinite sum can be controlled for any $n\in\N$ by
\[
\begin{split}
    \sum_{j=k+2}^{\infty} \left(r_j-r\right)^n\e^{-c_0\eps^{-2/3+\gamma}(r_j-r)} & \leq n!\sum_{j=k+2}^{\infty} \e^{n(r_j-r)-c_0\eps^{-2/3+\gamma}(r_j-r)} \\
    & \lesssim \e^{-\left(c_0\eps^{-2/3+\gamma}-n\right)(r_{k+2}-r)} \lesssim \e^{-\left(c_0\eps^{-2/3+\gamma}-n\right)\eps^\delta}
\end{split}
\]
provided that $\eps>0$ is sufficiently small in comparison to ${N_2}$. Notice that in order to obtain this upper bound, we are only using that $x^n \leq n!e^x$ for any $x\geq 0$, which in not very sharp, but sufficient for our purposes. Putting things together we find that
\[
\begin{split}
\II_{\eps,k}(r) & \lesssim \eps^{2/3-2\gamma}\|f_1\|_Y \e^{-\left(c_0\eps^{-2/3+\gamma}-{N_2}\right)\eps^\delta} \sum_{n=1}^{N_2} D_n \\
& \lesssim \eps^{2/3-2\gamma-{N_2}\delta} \e^{-\left(c_0 \eps^{-2/3+\gamma}-{N_2}\right)\eps^\delta} \|f_1\|_Y.
\end{split}
\]
Finally, for the remaining term we argue similarly, 
\[
\begin{split}
    \III_{\eps,k}(r) & \leq \eps^{2/3-2\gamma}\|f_1\|_Y \sum_{j=1}^{k-2} \left(1+\frac{|r-r_j|^2}{r_j^2}\right) \left(D_0 + \sum_{n=1}^{N_2} D_n|r-r'|^n\right) \e^{-c_0\eps^{-2/3+\gamma}(r-r_j-\eps^\delta)} \\
    & \lesssim \eps^{2/3-2\gamma}\|f_1\|_Y \e^{-\left(c_0\eps^{-2/3+\gamma}-({N_2}+2)\right)\eps^\delta} \sum_{n=1}^{N_2} D_n \\
    & \lesssim \eps^{2/3-2\gamma-{N_2}\delta}\|f_1\|_Y \e^{-\left(c_0\eps^{-2/3+\gamma}-({N_2}+2)\right)\eps^\delta}.
\end{split}
\]
Putting all these estimates together, we find the bound
\[
\begin{split}
    \eps^{1/3}|r\Omega'(r)||r^2w_\eps(r-r_0)& \G_{3}f_1(r)| \lesssim \|f_1\|_Y \eps^{2/3-2\gamma}\dsOne_{[r_{k-1},r_{k+1})}(r) \\
    & + \|f_1\|_Y\eps^{2/3-2\gamma-{N_2}\delta} \e^{-\left(c_0\eps^{-2/3+\gamma}-({N_2}+2)\right)\eps^\delta}\dsOne_{[r_0+\eps^\gamma,r_{k-1})\cap [r_{k+2},\infty)}(r),
\end{split}
\]
namely,
\[
\eps^{1/3}|r\Omega'(r)||r^2w_\eps(r-r_0)\G_{3}f_1(r)| \lesssim \eps^{2/3-2\gamma}\|f_1\|_Y.
\]
Coming back to \eqref{eq:error-Ginf-mid} and introducing all the estimates we were able to derive, we can conclude that for any $r\geq r_0+\eps^\eps$ the following estimate holds true uniformly for $r_0 \in \mathcal{R}$, $M \in \mathcal{N}$,
\[
\left|r^2w_\eps(r-r_0)(\L-\lambda-\L_2)\G_3f\right| \lesssim \left( \eps^{2/3-2\gamma}+ \eps^{\beta} + \eps^{2/3-2\gamma} \right)\|f_1\|_Y.
\]
We remark here that if $r<r_0+\eps^\gamma$ we cannot say that $r\in [r_k,r_{k+1})$ for some $k\in\N$. However, the estimates \eqref{eq:Grj-towinf}, \eqref{eq:drGrj-towinf}, \eqref{eq:Ginf-pointwise} and \eqref{eq:drGrj-towinf} still hold true: with only an exponentially decaying term surviving in $r<r_0+\eps^\gamma-\eps^\delta$, or at most one of the constant contributions if the distance between $r$ and $r_1 = r_0+\eps^\gamma$ is sufficiently small. Therefore, all these arguments carry over for the scenario $r<r_0+\eps^\gamma$, and we arrive at the claim of the proposition.
\end{proof}

\subsection{Green's function towards zero}\label{s:tow-zero}

Next in order we focus on studying the solution towards zero, that is, the solution when the support of $f$ is contained in an interval of the form $[\eps^\gamma,r_0-\eps^\gamma)$, for some $\gamma>0$ to be chosen accordingly. The method we will use is analogous to the method introduced for the solution towards infinity. Again, our approximate operator will be ``strongly invertible'', and so it suffices to approximate the exact operator by means of linear interpolation on a partition in the same vein as for the previous section.

Fix $r_0>0$, $\gamma,\delta>0$, and let $\eps>0$ be arbitrarily small, then we define $\ell=\ell(\eps)$ to be 
\begin{equation}\label{eq:ell-natural}
    \ell = \left\lceil \frac{r_0-2\eps^\gamma}{\eps^\delta} \right\rceil.
\end{equation}
Here $\lceil a\rceil$ denotes the smallest $n\in\N$ for which $n\geq a$. Then, we define the partition
\[
[\eps^\gamma,r_0-\eps^\gamma) = \bigcup_{j=1}^{\ell-1}[r_j,r_j+\eps^\delta) \cup [r_\ell,r_0-\eps^\gamma),
\]
where we identify $r_1 = \eps^\gamma$, and $r_{j+1}=r_j+\eps^\delta$ . We denote this partition by $\mathcal{P}_1(\eps^\delta)$, which is finite for any fixed value of $\eps>0$.

Following the ideas and the notation of Section \ref{sec:tow-inf}, we approximate the exact equations \eqref{eq:v1-r0}--\eqref{eq:v2-r0} for the case of $f$ supported in $[\eps^\gamma,r_0-\eps^\gamma)$, with the simpler ordinary differential equation,
\[
\eps\partial_r^2 V - i\eps^{-1/3}MT(r)V = f,
\]
where the transport function is defined as usual by
\[
T(r) = \left(\Omega(r)-\Omega_0-\frac{\Omega_0'}{U_0'}(U(r)-U_0)\right).
\]
As before, we further to approximate solutions to this simpler ODE by a piecewise linear evaluation of the transport function $T(r)$ in the points $r_j$ of the partition $\mathcal{P}_1(\eps^\delta)$, that is
\begin{equation}\label{eq:r-tow0-linear}
\eps\partial_r^2 V - i\eps^{-1/3}MT(r_j)V = f, \quad r\in [r_j,r_j+\eps^\delta),
\end{equation}
which, as shown in the previous section, has two nontrivial independent solutions of the form
\[
v_1(r) = \e^{\sqrt{iMT(r_j)}\eps^{-2/3}r}, \quad v_1(r) = \e^{-\sqrt{iMT(r_j)}\eps^{-2/3}r}.
\]
In the same manner as for the section towards infinity, and taking into account the their Wronskian is given by $W(v_1(r),v_2(r)) = 2\sqrt{iMT(r_j)}\eps^{-2/3}$, we construct a Green's function associated to \eqref{eq:r-tow0-linear} with the correct decay at infinity as
\[
(\G_{1,j}f)(r) = \frac{\eps^{-1/3}}{2\sqrt{iMT(r_j)}}\left(\int_0^r\e^{-\sqrt{iMT(r_j)}\eps^{-2/3}(r-s)}f(s)\dd s + \int_r^\infty\e^{-\sqrt{iMT(r_j)}\eps^{-2/3}(s-r)}f(s)\dd s \right),
\]
where we make the choice of the branch $\Re(\sqrt{iMT(r_j)})>0$ regardless of the sign of $MT(r_j)$, which again we know it is never zero because of Assumptions \ref{H0}--\ref{H2} and \ref{SH}.

Since the approximated operator is exactly the same as the one defined for the solution towards infinity, we quickly review the essential definitions before analysing the error produced by the approximation in this regime. We define the Green's function associated to the whole interval $[\eps^\gamma,r_0+\eps^\gamma)$ by the linear combination
\begin{equation}\label{eq:Gtow0-def}
\G_1 f(r) = \sum_{j=1}^{\ell-1} \G_{1,j}(f\mathds{1}_{[r_j,r_j+\eps^\delta)})(r)+\G_{1,\ell}(f\mathds{1}_{[r_\ell,r_0-\eps^\gamma)})(r).
\end{equation}
The first result we present here is a pointwise estimate for the terms $\G_{1,j}$ constituting the Green's function $\G_1$. Since the definition of the operator is exactly the same, and the only change comes from the domain in which this operator is applied, these estimates presented now follow trivially from the proof of Lemma \ref{lemma:pointGrj-towinf}, simply being extra careful with the fact that both extrema of the interval now are $\eps$ dependent.

\begin{lemma}\label{lemma:pointGrj-tow0}
Let $\gamma,\delta>0$ and consider the partition $\mathcal{P}_1$ of the interval $[\eps^\gamma,r_0-\eps^\gamma)$ in subintervals of length $\eps^\delta$. Let $\G_{1,j}$ be the Green's function associated to equation \eqref{eq:r-tow0-linear}, and let $f\in Y$. Then, under Assumptions \ref{H0}--\ref{H2} and \ref{SH}, if 
\[
\frac{2}{3}-\gamma-\delta>0, \quad \gamma<\delta,
\]
there exists $c_0>0$ independent of $r_0 \in \mathcal{R}$, $M \in \mathcal{N}$, so that for any $\eps>0$ sufficiently small, independent of $j$, and any $r\geq 0$ the following estimates hold true,
\begin{equation}\label{eq:Grj-tow0}
\begin{split}
    |w_\eps(r-r_0)\G_{1,j} (f\dsOne_{[r_j,r_j+\eps^\delta)})(r)| & \lesssim \|f\|_Y\frac{\eps^{1/3}}{r_j^2|MT(r_j)|} \e^{-c_0\eps^{-2/3+\gamma}(r_j-r)}\dsOne_{\{r_j-r>\eps^\delta\}}(r) \\
    & \quad + \|f\|_Y\frac{\eps^{1/3}}{r_j^2|MT(r_j)|}\dsOne_{\{r_j-\eps^\delta\leq r\leq r_j+2\eps^\delta\}}(r) \\
    & \quad + \|f\|_Y\frac{\eps^{1/3}}{r_j^2|MT(r_j)|} \e^{-c_0\eps^{-2/3+\gamma}(r-r_j-\eps^\delta)}\dsOne_{\{r-r_j>2\eps^\delta\}}(r),
\end{split}
\end{equation}
\begin{equation}\label{eq:drGrj-tow0}
\begin{split}
    |w_\eps(r-r_0)\partial_r\G_{1,j} (f\dsOne_{[r_j,r_j+\eps^\delta)})(r)| & \lesssim \|f\|_Y\frac{\eps^{-1/3}}{r_j^2\sqrt{|MT(r_j)|}} \e^{-c_0\eps^{-2/3+\gamma}(r_j-r)}\dsOne_{\{r_j-r>\eps^\delta\}}(r) \\
    & \quad + \|f\|_Y\frac{\eps^{-1/3}}{r_j^2\sqrt{|MT(r_j)|}}\dsOne_{\{r_j-\eps^\delta\leq r\leq r_j+2\eps^\delta\}}(r) \\
    & \quad + \|f\|_Y\frac{\eps^{-1/3}}{r_j^2\sqrt{|MT(r_j)|}} \e^{-c_0\eps^{-2/3+\gamma}(r-r_j-\eps^\delta)}\dsOne_{\{r-r_j>2\eps^\delta\}}(r).
\end{split}
\end{equation}
\end{lemma}

The proof of this result follows straightforwardly from the proof of Lemma \ref{lemma:pointGrj-towinf}, so we omit details here for brevity. Nonetheless, we make the following remark regarding the contribution of the smaller extra element in the partition of this domain.

\begin{remark}
    As before, one of the conditions for this lemma to hold true is the relation $\gamma<\delta$, which comes from the fact that the length of our partition should not exceed $\eps^\gamma$. Now, our partition has an extra interval $[r_\ell,r_0-\eps^\gamma)$, but this interval by definition is \emph{at most} of size $\eps^\delta<\eps^\gamma$, and hence this extra piece of the partition does not cause any issues.
\end{remark}

As in the previous section, we also want to estimate the Green's function with an $r^2$ weight to fully recover the $X$ norm. As a direct consequence of \eqref{eq:Grj-tow0} and \eqref{eq:drGrj-tow0} we find that
\begin{equation}\label{eq:Gell-pointwise}
\begin{split}
    |r^2w_\eps(r-r_0)\G_{1,j} (f&\dsOne_{[r_j,r_j+\eps^\delta)})(r)| \lesssim \|f\|_Y\frac{\eps^{1/3}}{|T(r_j)|} \e^{-c_0\eps^{-2/3+\gamma}(r_j-r)}\dsOne_{\{r_j-r>\eps^\delta\}}(r) \\
    &  + \|f\|_Y\frac{\eps^{1/3}}{|T(r_j)|}\dsOne_{\{r_j-\eps^\delta\leq r\leq r_j+2\eps^\delta\}}(r) \\
    &  + \|f\|_Y\frac{\eps^{1/3}}{|T(r_j)|} \left(1+\frac{|r-r_j|^2}{r_j^2}\right) \e^{-c_0\eps^{-2/3+\gamma}(r-r_j-\eps^\delta)}\dsOne_{\{r-r_j>2\eps^\delta\}}(r),
\end{split}
\end{equation}
whereas for the derivative,
\begin{equation}\label{eq:drGell-pointwose}
\begin{split}
    |r^2w_\eps(r-r_0)\partial_r\G_{1,j} (f&\dsOne_{[r_j,r_j+\eps^\delta)})(r)| \lesssim \|f\|_Y\frac{\eps^{-1/3}}{\sqrt{|MT(r_j)|}} \e^{-c_0\eps^{-2/3+\gamma}(r_j-r)}\dsOne_{\{r_j-r>\eps^\delta\}}(r) \\
    &  + \|f\|_Y\frac{\eps^{-1/3}}{\sqrt{|MT(r_j)|}}\dsOne_{\{r_j-\eps^\delta\leq r\leq r_j+2\eps^\delta\}}(r) \\
    &  + \|f\|_Y\frac{\eps^{-1/3}}{\sqrt{|MT(r_j)|}} \left(1+\frac{|r-r_j|^2}{r_j^2}\right)\e^{-c_0\eps^{-2/3+\gamma}(r-r_j-\eps^\delta)}\dsOne_{\{r-r_j>2\eps^\delta\}}(r).
\end{split}
\end{equation}

With these estimates in hand we proceed with the main result of the section, that provides quantitative estimates for the $X$ norm of $\G_1$ and $\partial_r\G_1$, and gives an estimate for the error produced by the approximated operator.

\begin{proposition}\label{proposition:rtow0}
Let $\gamma,\delta>0$ be such that
\[
\gamma<\frac{1}{3}, \quad \gamma<\delta, \quad \gamma+\delta<\frac{2}{3},
\]
consider the partition $\mathcal{P}_1(\eps^\delta)$ of the interval $[\eps^\gamma,r_0-\eps^\gamma)$ in subintervals of length (at most) $\eps^\delta$, and let $\G_1$ be the Green's function introduced in \eqref{eq:Gtow0-def}. Let $f\in Y$ with $\supp(f)\subset [\eps^\gamma,r_0-\eps^\gamma)$. Under Assumptions \ref{H0}--\ref{H2} and \ref{SH}, the following estimates hold true uniformly for any $\eps>0$ sufficiently small, $r_0 \in \mathcal{R}$, $M \in \mathcal{N}$,
\begin{equation}\label{eq:Grell}
\|\G_1 f\|_X \lesssim \eps^{1/3-2\gamma} \|f\|_Y,
\end{equation}
\begin{equation}\label{eq:drGrell}
\|\partial_r\G_1 f\|_X \lesssim \eps^{-1/3-2\gamma}\|f\|_Y,
\end{equation}
Moreover, the error carried out by the approximated operator can be estimated by
\begin{equation}\label{eq:Grell-error}
\begin{split}
    &\left\lVert ((\L-\lambda)\G_{1}-\Id) f\right\rVert_Y \lesssim \left( \eps^{2/3-2\gamma}+ \eps^{2/3-2\gamma} + \eps^\beta \right) \|f\|_Y,
\end{split}
\end{equation}
where the coefficient $beta>0$ is from \ref{P3}.
\end{proposition}

\begin{proof}
\emph{Part 1: Estimates on $\G_1$.} We will obtain the bounds on the corresponding weighted $L^\infty$ spaces using linearity
\[
(\G_1 f)(r) = \sum_{j=1}^{\ell_\eps} \G_{1,j}(f\mathds{1}_{[r_j,r_j+\eps^\delta)})(r),
\]
and Lemma \ref{lemma:pointGrj-tow0}. We begin by assuming that $r\in [\eps^\gamma,r_0-\eps^\gamma)$, and hence $r\in [r_k,r_k+\eps^\delta)$ for some $k$. By means of \eqref{eq:Grj-tow0}, we find the estimate
\[
\begin{split}
    |w_\eps(r-r_0)\G_1 f(r)| & \lesssim \|f\|_Y \sum_{j=1}^{k-2} \frac{\eps^{1/3}}{r_j^2|MT(r_j)|}\e^{-c_0\eps^{-2/3+\gamma}(r-r_j-\eps^\delta)} \\
    & \quad + \|f\|_Y \sum_{j=k+2}^{\ell} \frac{\eps^{1/3}}{r_j^2|MT(r_j)|} \e^{-c_0\eps^{-2/3+\gamma}(r_j-r)} \\
    & \quad + \|f\|_Y\eps^{1/3}\left(\frac{1}{r_{k-1}^2|MT(r_{k-1})|}+\frac{1}{r_k^2|MT(r_{k})|}+\frac{1}{r_{k+1}^2|MT(r_{k+1})|}\right).
\end{split}
\]
Let us start analysing the third term in the estimate. By Assumptions \ref{H1} and \ref{SH} there holds that $|MT(r_k)|\gtrsim \eps^{2\gamma}$, with $|T(r_k)|\sim \eps^{2\gamma}$ only near $r_0$. Since $r_k\geq \eps^\gamma$, but $r_k\sim \eps^\gamma$ only near the origin, i.e.\ far away from $r_0$, we can write
\[
\frac{\eps^{1/3}}{r_k^2|MT(r_{k})|} \lesssim \eps^{1/3-2\gamma},
\]
uniformly in $k$.
The other two addends follow readily as in the proof of Proposition \ref{proposition:rinf}. That is, on the one hand we have
\[
\begin{split}
    \|f\|_Y \sum_{j=k+2}^{\ell} \frac{\eps^{1/3}}{r_j^2|MT(r_j)|} \e^{-c_0\eps^{-2/3+\gamma}(r_j-r)} & \lesssim \|f\|_Y \eps^{1/3-2\gamma} \sum_{j=k+2}^{\infty} \e^{-c_0\eps^{-2/3+\gamma}(r_j-r)} \\
    & \lesssim \|f\|_Y \eps^{1/3-2\gamma} \e^{-c_0\eps^{-2/3+\gamma}(r_{k+2}-r)} \\
    & \leq \|f\|_Y \eps^{1/3-2\gamma} \e^{-c_0\eps^{-2/3+\gamma+\delta}}. 
\end{split}
\]
On the other hand we also have the bound,
\[
\begin{split}
    \|f\|_Y\sum_{j=1}^{k-2} \frac{\eps^{1/3}}{r_j^2|MT(r_j)|} \e^{-c_0\eps^{-2/3+\gamma}(r-r_j-\eps^\delta)} & \lesssim \|f\|_Y \eps^{1/3-2\gamma}\sum_{j=1}^{k-2} \e^{-c_0\eps^{-2/3+\gamma}(r-r_j-\eps^\delta)} \\
    & \lesssim  \|f\|_Y \eps^{1/3-2\gamma}\e^{-c_0\eps^{-2/3+\gamma}(r-r_{k-1})}\\
    & \leq \|f\|_Y \eps^{1/3-2\gamma}\e^{-c_0\eps^{-2/3+\gamma+\delta}}.
\end{split}
\]
Note that for the above bounds we once again require the condition $\gamma+\delta<2/3$.
Since all of these bounds are uniform in $k$, we arrive at the estimate
\[
\begin{split}
    |w_\eps(r-r_0)\G_1 f(r)| \lesssim \eps^{1/3-2\gamma}\|f\|_Y ,
\end{split}
\]
which in tandem with the observation that $\max\{1,r^2\}\lesssim 1$ for $r \leq r_0-\eps^\gamma$ implies that
\begin{equation}
|\max\{1,r^2\}w_\eps(r-r_0)\G_1 f(r)|  \lesssim \eps^{1/3-2\gamma}\|f\|_Y
\end{equation}
for $\eps^\gamma \leq r \leq r_0-\eps^\gamma$. 

If now $r\in[0,\eps^\gamma)$, using \eqref{eq:Grj-tow0} we find the corresponding estimate
\[
\begin{split}
    |\max\{1,r^2\}w_\eps(r-r_0)\G_1 f(r)| & \lesssim \|f\|_Y \sum_{j=3}^{\ell} \frac{\eps^{1/3}}{r_j^2|MT(r_j)|} \e^{-c_0\eps^{-2/3+\gamma}(r_j-r)} \\
    & \quad + \|f\|_Y\eps^{1/3}\left(\frac{1}{r_1^2|MT(r_{1})|}+\frac{1}{r_{2}^2|MT(r_{2})|}\right).
\end{split}
\]
Proceeding exactly as before, this is bounded in magnitude by $\eps^{-1/3+2\gamma}\|f\|_{Y}$. On the contrary, if $r>r_0+\eps^\gamma$ we find an analogous expression via a combination of \eqref{eq:Grj-tow0} and \eqref{eq:Gell-pointwise}. We write
\[
\begin{split}
    | \max\{1,r^2\} w_\eps(r-r_0)\G_1 f(r)| & \lesssim \|f\|_Y \sum_{j=1}^{\ell-1} \frac{\eps^{1/3}}{|MT(r_j)|}\left(1+\frac{|r-r_j|^2}{r_j^2}\right)\e^{-c_0\eps^{-2/3+\gamma}(r-r_j-\eps^\delta)} \\
    &  \quad + \|f\|_Y\frac{\eps^{1/3}}{|MT(r_{\ell})|}\dsOne_{\{|r-r_\ell|\leq 2\eps^\delta\}}.
\end{split}
\]
Observe that the last addend vanishes if $r$ and the last point of the partition $r_\ell$ are not sufficiently close. Arguing as in the previous section, we thus see that
\[
| \max\{1,r^2\} w_\eps(r-r_0)\G_1 f(r)| \lesssim \|f\|_Y \eps^{1/3-2\gamma},
\]
uniformly for $r \geq r_0-\eps^\gamma$.
Combining the estimates for the intervals $[0,\eps^\gamma]$, $[\eps^\gamma, r_0-\eps^\gamma]$, $[r_0-\eps^\gamma, \infty)$, we find the estimate \eqref{eq:Grell} in the claim of the proposition.

\emph{Part 2: Estimates on $\partial_r\G_1$.} We want to derive bounds for $\partial_r\G_1$ in $X$. We start from the pointwise bound \eqref{eq:Grj-tow0} from Lemma \ref{lemma:pointGrj-tow0}. Arguing as in the previous part of the proof, for any $r\in [r_k,r_{k+1})$ we obtain an estimate of the form
\[
\begin{split}
    |\max\{1,r^2\}w_\eps(r-r_0)& \partial_r \G_1 f(r)| \lesssim \|f\|_Y \sum_{j=1}^{k-2} \frac{\eps^{-1/3}}{r_j^2\sqrt{|MT(r_j)|}} \e^{-c_0\eps^{-2/3+\gamma}(r-r_j-\eps^\delta)} \\
    &  + \|f\|_Y \sum_{j=k+2}^{\ell} \frac{\eps^{-1/3}}{r_j^2\sqrt{|MT(r_j)|}} \e^{-c_0\eps^{-2/3+\gamma}(r_j-r)} \\
    &  + \|f\|_Y\eps^{-1/3}\left(\frac{1}{r_{k-1}^2\sqrt{|MT(r_{k-1})|}}+\frac{1}{r_k^2\sqrt{|MT(r_k)|}}+\frac{1}{r_{k+1}^2\sqrt{|MT(r_{k+1})|}}\right).
\end{split}
\]
Since $r_j\in [\eps^\gamma,r_0-\eps^\gamma)$ we can write
\[
\frac{\eps^{-1/3}}{r_j^2\sqrt{|MT(r_j)|}} \lesssim \max\left\lbrace \eps^{-1/3-2\gamma}, \eps^{-1/3-\gamma}\right\rbrace = \eps^{1/3-2\gamma},
\]
where again, we are using Assumptions \ref{H1} and \ref{SH} to argue that $r_j$ cannot be close to $\eps^\gamma$ at the same time as $|T(r_j)|$ is close to $\eps^{2\gamma}$. Therefore, following the same arguments that we used for the bound of $\G_1$ in $X$, we find
\[
\begin{split}
    |\max\{1,r^2\}w_\eps(r-r_0)\partial_r\G_1 f(r)| & \lesssim \eps^{-1/3-2\gamma} \|f\|_Y,
\end{split}
\]
uniformly for $r \in [\eps^\gamma, r_0-\eps^\gamma]$.
Similarly, for the case $r\not\in [\eps^\gamma,r_0+\eps^\gamma)$ we obtain again
\[
| \max\{1,r^2\} w_\eps(r-r_0)\partial_r\G_1 f(r)| \lesssim \|f\|_Y \eps^{-1/3-2\gamma},
\]
and thus we arrive at the claim of the proposition \eqref{eq:drGrell}.

\emph{Part 3: The error estimate}. This part of the claim entails the biggest differences with respect to the the case of the solution towards infinity. Even though the setting is initially the same, now we need to address the terms involving $T(r)$ and $\Omega'(r)$, whilst bearing in mind the degeneracies that may occur as $r\to 0$. As before, we define the operator $\L_{1,j}$ by \eqref{eq:r-tow0-linear}. We present here the proof for the first component of the operator, since the argument for the second component is identical. The difference between the exact and the approximate operators is given by
\[
\begin{split}
(\L-\lambda-\L_{1,j})V_1(r) & = \frac{\eps}{r}\partial_rV_1(r) - iM\eps^{-1/3}(T(r)-T(r_j))V_1(r) \\
& \quad -\left[ \frac{\eps}{r^2} + \eps^{1/3}M^2\left(\frac{1}{r^2} + \left(\frac{\Omega_0'}{U_0'}\right)^2\right)+\lambda + P_\eps(r) \right]V_1 + Q_\eps(r)V_2(r),
\end{split}
\]
where $P_\eps(r)$ and $Q_\eps(r)$ are defined by \eqref{eq:PQ definition}. Let $f\in Y$ with $\supp(f)\subset [\eps^\gamma, r_0-\eps^\delta)$. Arguing as in the proof of Proposition \ref{proposition:rinf}, we get that
\[
\begin{split}
    |r^2w_\eps(r-r_0)((\L-\mu\eps^{1/3})\G_1-\Id)f| & \lesssim \eps\left|w_\eps(r-r_0)r\partial_r\G_{1}f_1\right| + \eps^{1/3}\left|\max\{1,r^2\} w_\eps(r-r_0)\G_{1}f\right| \\
    & \quad + \eps^{-1/3}\sum_{j=1}^{\ell}\left|r^2w_\eps(r-r_0)M(T(r)-T(r_j))\G_{1,j}(f\dsOne_{[r_j,r_j+\eps\delta)})\right| \\
 & \quad + \eps^{1/3} |r\Omega'(r)|\left|w_\eps(r-r_0)r^2\G_{1}f\right|.
 \end{split} 
\]
To deal with the first two addends we can directly apply the bounds previously derived in this proposition. On the one hand,
\[
\eps\left|w_\eps(r-r_0)r\partial_r\G_{1}f\right| \lesssim \eps^{2/3-2\gamma} \|f\|_Y.
\]
On the other hand, by \eqref{eq:Grell} we obtain the estimate
\[
\eps^{1/3}\left|\max\{1,r^2\}w_\eps(r-r_0)\G_{1}f\right| \lesssim  \eps^{2/3-2\gamma}  \|f\|_Y.
\]
Only the terms involving $T(r)$ and $\Omega'(r)$ are left to be analysed. Let us start with the addend concerning the transport function $T(r)$. Assume for now that $r\in [r_k,r_k+\eps^\delta)$ for some $k\in\N$. Hence, by means of Lemma \ref{lemma:pointGrj-tow0} we write
\[
\eps^{-1/3}\sum_{j=1}^{\ell}r^2 M|T(r)-T(r_j)||w_\eps(r-r_0)\G_{1,j}(f_1\dsOne_{[r_j,r_j+\eps^\delta)})| \leq \I_{\eps,k}(r) + \II_{\eps,k}(r) + \III_{\eps,k}(r),
\]
where
\[
\I_{\eps,k}(r) = \|f\|_Y \sum_{j=k+2}^{\ell}\frac{r^2|T(r)-T(r_j)|}{r_j^2|T(r_j)|} \e^{-c_0\eps^{-2/3+\gamma}(r_j-r)},
\]
\[
\II_{\eps,k}(r)  = \|f\|_Y \left(\frac{r^2|T(r)-T(r_{k-1})|}{r_{k-1}^2|T(r_{k-1})|} + \frac{r^2|T(r)-T(r_{k})|}{r_k^2|T(r_{k})|} + \frac{r^2|T(r)-T(r_{k+1})|}{r_{k+1}^2|T(r_{k+1})|}\right),
\]
\[
\III_{\eps,k}(r) = \|f\|_Y \sum_{j=1}^{k-2}\frac{r^2|T(r)-T(r_j)|}{r_j^2|T(r_j)|} \e^{-c_0\eps^{-2/3+\gamma}(r-r_j-\eps^\delta)}.
\]
We start by analysing the second addend. Observe that since $r\in [r_k,r_{k+1})$, then  
\[
\frac{r^2}{r_k^2}, \frac{r^2}{r_{k-1}^2}, \frac{r^2}{r_{k+1}^2} \lesssim 1.
\]
This together with Property \ref{P3} yields
\[
\begin{split}
   \II_{\eps,k}(r) & \lesssim \|f\|_Y\frac{{r^2}|T(r)-T(r_{k})|}{r_k^2|T(r_{k})|} \lesssim \|f\|_Y\sum_{n=1}^{N_1} B_n(\eps)|r-r_k|^n \leq  \|f\|_Y\sum_{n=1}^{N_1} B_n(\eps)\eps^{n\delta} \lesssim \|f\|_Y\eps^{\beta}.
\end{split}
\]
Next in order we analyse the first term, corresponding to the indices $k+2\leq j\leq \ell$. We write
\[
\begin{split}
    \I_{\eps,k}(r) & \leq \|f\|_Y\sum_{j=k+2}^{\ell}\sum_{n=1}^{N_1} B_n(\eps)\frac{r^2|r-r_j|^n}{r_j^2} \e^{-c_0\eps^{-2/3+\gamma}(r_j-r)} \\
    & \leq  \|f\|_Y\sum_{n=1}^{N_1} B_n(\eps)\sum_{j=k+2}^{\infty}|r-r_j|^n(1+\eps^{-2\gamma}|r-r_j|^2)\e^{-c_0\eps^{-2/3+\gamma}(r_j-r)},
\end{split}
\]
and therefore, up to a constant depending exclusively on $N_1$ that we absorb with the symbol $\lesssim$,
\[
\begin{split}
    \I_{\eps,k}(r) & \leq  \|f\|_Y \eps^{-2\gamma}\sum_{n=1}^{N_1} B_n(\eps)\sum_{j=k+2}^{\infty}\e^{-\frac{1}{2}c_0\eps^{-2/3+\gamma}(r_j-r)} \lesssim  \|f\|_Y\eps^{-2\gamma}\e^{-\frac{1}{2}c_0\eps^{-2/3+\gamma}(r_{k+2}-r)}\sum_{n=1}^{N_1} B_n(\eps). 
\end{split}
\]
Additionally, notice that Property \ref{P3} implies that
\begin{equation}\label{eq:sum-Bn}
\sum_{n=1}^{N_1} B_n(\eps) \leq \eps^{-N_1\delta}\sum_{n=1}^{N_1} B_n(\eps)\eps^{n\delta} \lesssim \eps^{\beta-N_1\delta},
\end{equation}
and thus
\[
\I_{\eps,k}(r) \lesssim \|f\|_Y\eps^{\beta-2\gamma-N_1\delta}\e^{-\frac{1}{2}c_0\eps^{-2/3+\gamma+\delta}},
\]
which quickly converges to zero as $\eps\to 0$ provided that $\delta+\gamma<2/3$ due to the effect of the exponential.
Last, we address the third addend using a completely analogous argument,
\[
\begin{split}
    \III_{\eps.k}(r) & \lesssim  \|f\|_Y\sum_{j=1}^{k-2}\sum_{n=1}^{N_1} B_n(\eps) (1+\eps^{-2\gamma}|r-r_j|^2)|r-r_j|^n \e^{-c_0\eps^{-2/3+\gamma}(r-r_{j+1})} \\
    & \lesssim   \|f\|_Y\eps^{\beta-2\gamma-N_1\delta}\e^{-\frac{1}{2}c_0\eps^{-2/3+\gamma+\delta}}.
\end{split}
\]
Observe that if $r\not\in [\eps^\gamma,r_0-\eps^\gamma)$ the estimates obtained are analogous, since \eqref{eq:Grj-tow0}, \eqref{eq:drGrj-tow0}, \eqref{eq:Gell-pointwise} and \eqref{eq:drGell-pointwose} hold true still in this range and the bounds the produce are analogous. The principle one should bear in mind is that if there exists $r_k\in \mathcal{P}_1(\eps^\delta)$ such that $|r-r_k|\leq 2\eps^\delta$, then the error is bounded by $\eps^\beta$ as shown before. Otherwise the error will be bounded by a more rapidly decaying exponential function, and can be neglected. In particular, we find that for the term with the transport function, the error is bounded by
\[
\eps^{-1/3}\sum_{j=1}^{\ell}M|T(r)-T(r_j)||w_\eps(r-r_0)\G_{1,j}(f\dsOne_{[r_j,r_j+\eps^\delta)})| \lesssim \eps^\beta,
\]
where $\beta$ is defined in Property \ref{P3}. The last contribution to be addressed is the one corresponding to the terms involving $\Omega'$. Again, for any $r\in [r_k,r_k+\eps^\delta)$, we can write
\[
\eps^{1/3}  |r\Omega'(r)|\left|r^2w_\eps(r-r_0)\G_{1}f\right| \lesssim \I_{\eps,k}(r) + \II_{\eps,k}(r) + \III_{\eps,k}(r),
\]
where
\[
\I_{\eps,k}(r) = \eps^{2/3} \|f\|_Y \sum_{j=k+2}^{\ell} \frac{|r\Omega'(r)|}{r_j^2|MT(r_j)|} \e^{-c_0\eps^{-2/3+\gamma}(r_j-r)},
\]
\[
\II_{\eps,k}(r) = \eps^{2/3} \|f_1\|_Y\sum_{j\in\{k-1,k,k+1\}}\frac{|r\Omega'(r)|}{|MT(r_j)|},
\]
\[
\III_{\eps,k}(r) = \eps^{2/3} \|f\|_Y \sum_{j=1}^{k-2} \frac{r^2|r\Omega'(r)|}{r_j^2|MT(r_j)|}  \e^{-c_0\eps^{-2/3+\gamma}(r-r_j-\eps^\delta)}.
\]
As before, we start by the second addend. On the one hand, observe that by means of Property \ref{P4} we can write
\[
\begin{split}
    \II_{\eps,k}(r) & \lesssim \eps^{2/3-2\gamma}\|f\|_Y \left(D_0 + \sum_{n=1}^{N_2} D_n|r-r_k|^n\right) \\
    & \lesssim \eps^{2/3-2\gamma} \|f\|_Y \left(D_0 + \sum_{n=1}^{N_2} D_n\eps^{n\delta}\right) \lesssim \eps^{2/3-2\gamma} \|f_1\|_Y.
\end{split}
\]
Next in order, we look at the term $\I_{\eps,k}(r)$, for which Lemma \ref{lemma:consequences of assumptions} yields the estimate
\[
\begin{split}
    \I_{\eps,k}(r) & \lesssim \eps^{2/3-4\gamma}\|f\|_Y \sum_{j=k+2}^{\infty} \left(D_0 + \sum_{n=1}^{N_2} D_n|r_j-r|^n\right) \e^{-c_0\eps^{-2/3+\gamma+\delta}} \\
    & \lesssim \|f\|_Y\eps^{2/3-4\gamma}  \e^{-c_0\eps^{-2/3+\gamma+\delta}},
\end{split}
\]
which quickly converges to zero as $\eps\to 0$ due to the effect of the exponential. The last term to be addressed can be analogously bounded by
\[
\begin{split}
    \III_{\eps,k}(r) & \lesssim \eps^{2/3-2\gamma}\|f\|_Y \sum_{j=1}^{k-2} (1+\eps^{-2\gamma}|r-r_j|^2)\left(D_0 + \sum_{n=1}^{N_2} D_n|r-r_j|^n\right)\e^{-c_0\eps^{-2/3+\gamma}(r-r_j-\eps^\delta)} \\
    & \lesssim \|f\|_Y \eps^{2/3-4\gamma} \e^{-\frac{1}{2}c_0\eps^{-2/3+\gamma+\delta}}.
\end{split}
\]
As argued before, we find the ``worst" estimates for values of $r$ such that $|r-r_k|\leq 2\eps^\delta$ for some $r_k\in\mathcal{P}_1(\eps^\delta)$, whereas if $|r-r_k|> 2\eps^\delta$. This argument yields analogous bounds for the case $r\not\in [\eps^\gamma, r_0-\eps^\gamma)$, and thus we can write in general
\[
\begin{split}
    \eps^{1/3}  |r\Omega'(r)|\left|w_\eps(r-r_0)r^2\G_{1}f_1\right| & \lesssim \eps^{2/3-2\gamma}\|f\|_Y.
\end{split}
\]
Combining all the estimates together we find the claim of the proposition,
\[
\begin{split}
    \left|r^2w_\eps(r-r_0)((\L-\lambda)\G_1 f-f)\right| \lesssim  \max\left\lbrace \eps^{2/3-2\gamma}, \eps^{\beta}, \eps^{2/3-2\gamma} \right\rbrace \|f\|_Y,
\end{split}
\]
which, under the assumption $\gamma<\delta$, converges to zero as $\eps\to 0$.
\end{proof}

\subsection{Green's function near zero}\label{sec:near-0}

In this section we consider the last bit of the (positive) real line that is left to study, namely $[0,\eps^\gamma)$. This last bit of the Green's function will turn out to be of vital importance, since it can be shown to have a \emph{regularizing} effect on the solution. Indeed, when computing the error for the Green's functions we have defined so far in Section \ref{s:approx-Green's}, the error bounds we have derived are all in the space $Y$, which allows for quadratic blow up as $r \to 0$. Since we would like to invert $\L$ on the space $X$, it is however essential that our inverse does not blow up near $r=0$. In this regard, we show that the approximate Green's function near zero not only has the properties we expect of ``generic'' approximate Green's functions (such as rapid decay outside of the support of the input), but also defines a bounded map from $Y \to X$. This allows us to close the construction of the operator $\mathcal{G}^\lambda$ in such a way that it defines a bounded linear map $\mathcal{G}^\lambda:Y \to X$, and hence terms of the form $\mathcal{G}^\lambda (((\mathcal{L}-\lambda)\G^\lambda -\Id)^n)$ are well defined maps.

That being said, in order to derive such a regularising map, substantial care has to be taken in defining the approximate Green's function, since the terms of order $r^{-2}$ in the definition of $\L$ may no longer be treated as errors. In fact, as $r\to 0$, these will be the dominant terms in $\mathcal{L}$. We thus begin with some algebraic manipulations of the equation defining $\mathcal{L}$ in order to obtain an expression more amenable to analysis.

Recall from Section \ref{s:equations}, that the kinematic dynamo equations can be written in the following form,
\begin{align*}
    \eps\left( \partial_r^2V_1 + \frac{1}{r}\partial_rV_1\right) & = (A_\eps+\lambda) V_1 - \frac{1}{2}\left(\frac{1}{\alpha}\eps^{-1/3}B_\eps + \alpha\eps^{1/3}C_\eps\right)V_1 - \frac{1}{2}\left(\frac{1}{\alpha}\eps^{-1/3}B_\eps - \alpha\eps^{1/3}C_\eps\right)V_2, \\
    \eps\left( \partial_r^2V_2 + \frac{1}{r}\partial_rV_2\right) & = (A_\eps+\lambda) V_2 + \frac{1}{2}\left(\frac{1}{\alpha}\eps^{-1/3}B_\eps + \alpha\eps^{1/3}C_\eps\right)V_2 + \frac{1}{2}\left(\frac{1}{\alpha}\eps^{-1/3}B_\eps - \alpha\eps^{1/3}C_\eps\right)V_1,
\end{align*}
where
\[
A_\eps = \frac{\eps}{r^2} + \eps^{1/3}M^2\left( \frac{1}{r^2}+\left( \frac{\Omega_0'}{U_0'} \right)^2 \right)  + i\eps^{-1/3}MT(r),
\]
\[
B_\eps = \frac{2iM\eps^{2/3}}{r^2}, \quad C_\eps = -r\Omega' - \frac{2iM\eps^{2/3}}{r^2}, 
\]
and
\[
\alpha^2 = -\frac{2iM}{r_0^3\Omega_0'}.
\]
We now conjugate this equation by the invertible matrix 
\begin{equation}\label{eq:Pmatrix}
P=\begin{pmatrix}
1 & 1\\
1& -1
\end{pmatrix},
\end{equation}
and define the operator $\tilde \L$ via 
\[
(\tilde{\L}-\lambda)\begin{pmatrix}
V_+\\
V_-
\end{pmatrix}=P(\L-\lambda)P^{-1}\begin{pmatrix}
V_+\\
V_-
\end{pmatrix},
\]
with
\begin{align}\label{eq:near0-addsubtract}
    (\tilde\L-\lambda)\begin{pmatrix}
    V_+\\
    V_-
    \end{pmatrix}=\left(\eps \partial_r^2 + \eps \frac{1}{r}\partial_r  -(A_\eps+\lambda) \right)\begin{pmatrix}
    V_+\\
    V_-
    \end{pmatrix}+ \begin{pmatrix}
    \alpha\eps^{1/3}C_\eps V_-\\
    \frac{1}{\alpha}\eps^{-1/3}B_\eps V_+
    \end{pmatrix},
\end{align}
where we denote $V_+ = V_1+V_2$ and $V_-=V_1-V_2$. Recall now that---in the coordinate system $V=(V_1,V_2)$ and $f=(f_1,f_2)$---we want to solve the equation $(\L-\lambda)V=f$. Suppose for a second that we manage to find an operator $\tilde \G_0$, so that for all $g$ supported in $[0,\eps^\gamma)$, it holds 
\begin{equation}\label{eq:estimate-for-g}
\|((\tilde \L-\lambda)\tilde\G_0 -\Id)g\|_{Y} \lesssim \eps^{a}\|g\|_{Y},
\end{equation}
for some $a>0$. Then, setting $f=P^{-1} g$, $\G_0=P^{-1}\tilde \G_0 P$, we have 
\begin{align*}
((\L-\lambda)\G_0 -\Id)f& = (\L-\lambda)P^{-1}\tilde \G_0 g -P^{-1}g\\
&=P^{-1}(P(\L-\lambda)P^{-1}\tilde \G_0 g-g)=P^{-1}((\tilde \L-\lambda)\tilde \G_0 g-g).
\end{align*}
Therefore, since as a linear mapping the matrix $P$ satisfies $\|P\|,\|P^{-1}\|\lesssim 1$, there holds 
\begin{equation}
\label{eq:change of variables Green's function}
\|((\L-\lambda)\G_0-\Id)f\|_{Y} \leq \|P^{-1}\| \|((\tilde \L-\lambda)\tilde \G_0-\Id)g\|_Y \lesssim \eps^a \|P f\|_{Y}\lesssim \eps^a \|f\|_{Y}.
\end{equation}
Thus, it suffices to find an operator $\tilde{\G_0}$ so that for any $g$ supported in $[0,\eps^\gamma)$, estimate \eqref{eq:estimate-for-g} holds true for some $a>0$.
As we have been doing so far, we want to find $\tilde\G_0$ as an exact inverse of an approximate ODE for the system $\tilde\L-\lambda$. In order to write the appropriate approximated operator we use Assumptions \ref{H0} and \ref{H1}, which in particular ensure that $T\in C^3([0,\infty))$ is well defined up to $r=0$, and moreover
\[
\lim_{r\to 0}T(r) = \tau \neq 0.
\]
Taking this into account, we may now define the operator
\begin{equation}\label{eq:L0}
\tilde\L_0\begin{pmatrix}
    V_+\\
    V_-
    \end{pmatrix}=\left(\eps \partial_r^2 + \eps \frac{1}{r}\partial_r  -\frac{\eps^{1/3}M^2}{r^2} + iM\eps^{-1/3}\tau \right)\begin{pmatrix}
    V_+\\
    V_-
    \end{pmatrix}-\frac{\eps^{1/3}}{r^2}\sqrt{-2iMr_0^3\Omega_0'}\begin{pmatrix}
    0\\
    V_+
\end{pmatrix}
\end{equation}
which approximates \eqref{eq:near0-addsubtract}.
To derive it, we follow the same criteria as for the previous sections, neglecting the contributions of the terms in \eqref{eq:near0-addsubtract} that are either multiplied by a sufficiently large power of $\eps$, or in this case, by a non-negative power of $r$, since this region corresponds to the limit $r\to 0$.  We denote this approximated operator by $\tilde\L_0$, and its corresponding Green's function by $\tilde\G_0$, namely 
\[
\tilde\G_0\begin{pmatrix}
    f_+\\
    f_-
\end{pmatrix} = \begin{pmatrix}
    V_+\\
    V_-
\end{pmatrix}
\]
solves
\[
\tilde\L_0\begin{pmatrix}
    V_+\\
    V_-
\end{pmatrix} = \begin{pmatrix}
    f_+\\
    f_-
\end{pmatrix},
\]
with the further boundary condition that $(V_+,V_-) \to 0$ as $r \to \infty$.

We now begin the analysis of the operator $\tilde\L_0$. Observe that the first component of $\tilde\L_0$ in \eqref{eq:L0} can be studied independently of the second component. We introduce a new Green's function associated just to this first component. We consider the operator
\begin{equation}\label{eq:ode-0}
\eps\left( \partial_r^2V + \frac{1}{r}\partial_rV\right) -\left(\frac{\eps^{1/3}M^2}{r^2} + i\eps^{-1/3}M\tau\right)V,
\end{equation}
and we write $\G_0^1$ to denote its Green's function, i.e.\ 
\[
(\tilde\L_0)_1\G_0^1f = f.
\]
We start by studying this first Green's function. Just like in the previous sections, we look for two independent solutions to the approximated ODE \eqref{eq:ode-0}. In general, ODEs of the form
\[
v''(r) + \frac{1}{r}v'(r) = \left(\frac{a}{r^2}+b\right)v(r)
\]
have two independent solutions given by Bessel functions of the first and second kind, see \cite{Olver74} for detailed information. Since our arguments are necessarily complex, we make the choice of defining the two independent solutions as \emph{modified Bessel functions},
\[
v_1(r) = I_{a^{1/2}}\left(b^{1/2}r\right), \quad v_2(r) = K_{a^{1/2}}\left(b^{1/2}r\right),
\]
where $I_\nu(x)$ represents a modified Bessel function of the first kind, whereas $K_\nu(x)$ represents a modified Bessel function of the second kind. For more details about these special functions, we refer the reader to Appendix \ref{s:appendix-bessel}. For our specific case \eqref{eq:ode-0}, notice that
\[
a_\eps = M^2\eps^{-2/3}, \quad b_\eps = i\tau M\eps^{-4/3},
\]
where we include the subindex to stress their $\eps$ dependency now.
The Wronskian between the modified Bessel functions $I_\nu$ and $K_\nu$ is given by
\[
W(I_\nu(z),K_\nu(z)) = I_\nu(z)K_\nu'(z) - I_\nu'(z)K_\nu(z) = -\frac{1}{z},
\]
so therefore, applied to the solutions found for \eqref{eq:ode-0} we get,
\[
W(v_1(r),v_2(r)) = -\frac{1}{r}.
\]
With this in mind, we can construct the following Green's function associated to the problem \eqref{eq:ode-0}. In this section we take $\supp(f)\subset [0,\eps^\gamma)$, for some $\gamma>0$ to be chosen accordingly. Then we write,
\begin{equation}\label{eq:Green-near-0}
\begin{split}
    \G_0^1f(r) & = -\eps^{-1}v_2(r) \int_0^r s v_1(s) f(s)\dd s - \eps^{-1} v_1(r) \int_r^\infty s v_2(s) f(s)\dd s,
\end{split}
\end{equation}
namely
\[
\begin{split}
    \G_0^1f(r) & = -\eps^{-1}K_{a_\eps^{1/2}}\left(b_\eps^{1/2}r\right) \int_0^r sI_{a_\eps^{1/2}}\left(b_\eps^{1/2}s\right) f(s)\dd s \\
    & \quad - \eps^{-1}I_{a_\eps^{1/2}}\left(b_\eps^{1/2}r\right) \int_r^\infty sK_{a_\eps^{1/2}}\left(b_\eps^{1/2}s\right) f(s)\dd s.
\end{split}
\]
Since all we want to compute now are weighted $L^\infty$ bounds for this Green's function, first of all what we need are suitable asymptotic bounds for the modified Bessel functions. From Lemma \ref{lemma:Bessel bounds} we have the bounds which hold uniformly as $\nu \to \infty$ for all $z\in\Co$ with $|\arg(z)|=\frac{\pi}{4}$.
\begin{equation}\label{modBesselI}
|I_\nu(\nu z)|\leq C\frac{\e^{\nu \xi}}{(2\pi \nu)^{1/2}(1+z^2)^{1/4}},
\end{equation}
\begin{equation}\label{modBesselK}
|K_{\nu}(\nu z)|  \leq C \left(\frac{\pi}{2 \nu}\right)^{1/2}\frac{\e^{-\nu \xi}}{(1+z^2)^{1/4}},
\end{equation}
where
\[
\xi = (1+z^2)^{1/2} + \log\left( \frac{z}{1+(1+z^2)^{1/2}} \right).
\]
Since $v_1$ and $v_2$ are represented in terms of modified Bessel function with indices of the form $\nu = a_\eps^{1/2} = M\eps^{-1/3}$, and $|\arg(b_\eps^{1/2}r)|=\frac{\pi}{4}$, these are indeed the bounds we are interested in.
For simplicity's sake, and for the sake of a clearer presentation of the results, we will select the value $M=1$ in the following. The analysis remains largely unchanged for other values of $M$, and it will be clear from the analysis that all our bounds depend continuously on $M$. With this in mind we obtain from \eqref{modBesselI}, \eqref{modBesselK} the estimates
\[
|v_1(r)| \lesssim \eps^{1/6}\frac{|\tau|^{\frac{1}{2}\eps^{-1/3}} \eps^{-\frac{1}{3}\eps^{-1/3}}r^{\eps^{-1/3}}}{\left|1+i\tau\eps^{-2/3}r^2\right|^{1/4}\left|1+(1+i\tau\eps^{-2/3}r^2)^{1/2}\right|^{\eps^{-1/3}}} \e^{\Re\left[\eps^{-1/3}(1+i\tau\eps^{-2/3}r^2)^{1/2}\right]},
\]
\[
|v_2(r)| \lesssim \eps^{1/6}\frac{\left|1+(1+i\tau\eps^{-2/3}r^2)^{1/2}\right|^{\eps^{-1/3}}}{\left|1+i\tau\eps^{-2/3}r^2\right|^{1/4} |\tau|^{\frac{1}{2}\eps^{-1/3}} \eps^{-\frac{1}{3}\eps^{-1/3}}r^{\eps^{-1/3}}} \e^{-\Re\left[\eps^{-1/3}(1+i\tau\eps^{-2/3}r^2)^{1/2}\right]}.
\]
In order to obtain precise bounds from here, we shall break down the required computations into the following pieces
\[
\Re\left[\eps^{-1/3}(1+i\tau\eps^{-2/3}r^2)^{1/2}\right] = \frac{\eps^{-1/3}}{\sqrt{2}} \left( \left(1 + \eps^{-4/3}r^4\tau^2\right)^{1/2} + 1\right)^{1/2};
\]
\[
\left|1+(1+i\tau\eps^{-2/3}r^2)^{1/2}\right|^{\eps^{-1/3}} = \left( 1+\left[2+2\left(1+\eps^{-4/3}r^4\tau^2\right)^{1/2}\right]^{1/2} + \left(1+\eps^{-4/3}r^4\tau^2\right)^{1/2} \right)^{\frac{1}{2}\eps^{-1/3}};
\]
\[
\left|1+i\tau\eps^{-2/3}r^2\right|^{1/4} = \left( 1 + \eps^{-4/3} r^4 \tau^2\right)^{1/8}.
\]
Putting everything together, we find the estimates
\[
\begin{split}
    |v_1(r)| & \lesssim \frac{|\tau|^{\frac{1}{2}\eps^{-1/3}} \eps^{\frac{1}{6}-\frac{1}{3}\eps^{-1/3}} r^{\eps^{-1/3}} \e^{\frac{1}{\sqrt{2}}\eps^{-1/3} \left( \left(1 + \eps^{-4/3}r^4\tau^2\right)^{1/2} + 1\right)^{1/2}} }{\left( 1 + \eps^{-4/3} r^4 \tau^2\right)^{1/8} \left( 1+\left[2+2\left(1+\eps^{-4/3}r^4\tau^2\right)^{1/2}\right]^{1/2} + \left(1+\eps^{-4/3}r^4\tau^2\right)^{1/2} \right)^{\frac{1}{2}\eps^{-1/3}}};
\end{split}
\]
\[
\begin{split}
    |v_2(r)| & \lesssim \frac{|\tau|^{-\frac{1}{2}\eps^{-1/3}} \eps^{\frac{1}{6}+\frac{1}{3}\eps^{-1/3}} r^{-\eps^{-1/3}} \e^{-\frac{1}{\sqrt{2}}\eps^{-1/3} \left( \left(1 + \eps^{-4/3}r^4\tau^2\right)^{1/2} + 1\right)^{1/2}} }{\left( 1 + \eps^{-4/3} r^4 \tau^2\right)^{1/8} \left( 1+\left[2+2\left(1+\eps^{-4/3}r^4\tau^2\right)^{1/2}\right]^{1/2} + \left(1+\eps^{-4/3}r^4\tau^2\right)^{1/2} \right)^{-\frac{1}{2}\eps^{-1/3}}}.
\end{split}
\]
In order to plug in these bounds into a estimate for the Green's function \eqref{eq:Green-near-0}, we first realise that given any positive $n$, the function
\[
x \mapsto \frac{1}{x^{1/8}\left( 1+\sqrt{2}(1+x)^{1/2} + x \right)^n} \e^{\sqrt{2}n(1+x)^{1/2}}
\]
is monotone increasing for all $x>1$. In particular, setting $x=(1+\eps^{-4/3}r^4\tau^2)^{1/2}$, $n=\frac{1}{2}\eps^{-\frac{1}{3}}$, for the first addend in Green's function \eqref{eq:Green-near-0} we can write
\[
\begin{split}
    \left| \eps^{-1} v_2(r) \int_0^r s v_1(s) f(s)\dd s \right| & \lesssim \frac{|\tau|^{-\frac{1}{2}\eps^{-1/3}} \eps^{-\frac{5}{6}+\frac{1}{3}\eps^{-1/3}} r^{-\eps^{-1/3}} }{\left( 1 + \eps^{-4/3} r^4 \tau^2\right)^{1/4}}   \int_0^r |\tau|^{\frac{1}{2}\eps^{-1/3}} \eps^{\frac{1}{6}-\frac{1}{3}\eps^{-1/3}} s^{1+\eps^{-1/3}}  |f(s)|\dd s \\
    & \leq \|f\|_Y \frac{\eps^{-2/3} r^{-\eps^{-1/3}} }{\left( 1 + \eps^{-4/3} r^4 \tau^2\right)^{1/4}}   \int_0^r w_\eps(s-r_0)^{-1}s^{-1+\eps^{-1/3}} \dsOne_{\{s<\eps^\gamma\}}(s) \dd s.
\end{split}
\]
An analogous game allows us to obtain a similar bound for the second addend in \eqref{eq:Green-near-0}. Here we simply use that the mapping
\[
x \mapsto \frac{1}{\left( 1+\sqrt{2}(1+x)^{1/2} + x \right)^{-n}} \e^{-\sqrt{2}n(1+x)^{1/2}}
\]
is monotone decreasing for all $x\geq 1$ and $n$ positive, and that $x(r)=\left( 1 + \eps^{-4/3} r^4 \tau^2\right)^{1/2}$ is monotone increasing for all $r>0$. Therefore, the second addend may be bounded above by
\[
\begin{split}
    \left| \eps^{-1} v_1(r) \int_r^\infty s v_2(s) f(s)\dd s \right| & \lesssim \frac{|\tau|^{\frac{1}{2}\eps^{-1/3}} \eps^{-\frac{5}{6}-\frac{1}{3}\eps^{-1/3}} r^{\eps^{-1/3}} }{\left( 1 + \eps^{-4/3} r^4 \tau^2\right)^{1/8}}   \int_r^\infty \frac{|\tau|^{-\frac{1}{2}\eps^{-1/3}} \eps^{\frac{1}{6}+\frac{1}{3}\eps^{-1/3}} s^{1-\eps^{-1/3}} }{\left( 1 + \eps^{-4/3} s^4 \tau^2\right)^{1/8}} |f(s)|\dd s  \\
    & \leq \|f\|_Y \frac{\eps^{-2/3} r^{\eps^{-1/3}} }{\left( 1 + \eps^{-4/3} r^4 \tau^2\right)^{1/4}}   \int_r^\infty w_\eps(s-r_0)^{-1}  s^{-1-\eps^{-1/3}} \dsOne_{\{s<\eps^\gamma\}}(s)\dd s.
\end{split}
\]
Hence, going back to the expression for the Green's function \eqref{eq:Green-near-0} we arrive at the estimate
\begin{equation}\label{eq:G01-near0}
\begin{split}
    |\G_0^1f(r)| & \lesssim \frac{\eps^{-2/3}\|f\|_Y}{\left( 1 + \eps^{-4/3} r^4 \tau^2\right)^{1/4}} r^{-\eps^{-1/3}}\int_0^r w_\eps(s-r_0)^{-1}s^{-1+\eps^{-1/3}} \dd s \\
    & \quad + \frac{\eps^{-2/3}\|f\|_Y}{\left( 1 + \eps^{-4/3} r^4 \tau^2\right)^{1/4}}r^{\eps^{-1/3}}\int_r^\infty w_\eps(s-r_0)^{-1}s^{-1-\eps^{-1/3}} \dd s,
\end{split}
\end{equation}
for any function $f$.
Next in order, we want to recover an expression for the Green's function $\tilde\G_0$ associated to the full approximated variables-changed operator $\tilde\L_0$ defined in \eqref{eq:L0},
\[
\tilde\L_0 \begin{pmatrix}
    V_+\\
    V_-
\end{pmatrix} = \begin{pmatrix}
    f_+\\
    f_-
\end{pmatrix}.
\]
The first component in \eqref{eq:L0} is already sorted out, i.e.\ $V_+ = \G_0^1f_+$. For the second component we can use the structure of the operator we want to invert,
\[
f_- = (\tilde\L_0)_2V_- = (\tilde\L_0)_1V_- - \frac{\eps^{1/3}}{r^2}\sqrt{-2iMr_0^3\Omega_0'}V_+ = (\tilde\L_0)_1V_- - \frac{\eps^{1/3}}{r^2}\sqrt{-2iMr_0^3\Omega_0'}\G_0^1 f_+,
\]
namely,
\[
V_- = \G_0^1 f_- + \G_0^1 \left(\frac{\eps^{1/3}}{r^2}\sqrt{-2iMr_0^3\Omega_0'}\G_0^1 f_+\right).
\]
Therefore, collecting everything in a main expression we obtain that the total Green's function is given by
\begin{equation}
\label{eq:Green's function near 0}
\begin{pmatrix}
    V_+\\
    V_-
\end{pmatrix} = \tilde\G_0\begin{pmatrix}
    f_+\\
    f_-
\end{pmatrix} = \begin{pmatrix}
    \G_0^1f_+ \\
    \G_0^1 f_- + \eps^{1/3}\sqrt{-2iMr_0^3\Omega_0'}\G_0^1 \left(r^{-2}\G_0^1 f_+\right)
\end{pmatrix}.
\end{equation}
An important observation here is that the Green's function $\G_0^1$ now acts additionally on the function 
\[
h(r) = \frac{1}{r^2}\G_0^1(f_-)(r),
\]
which is no longer supported in the interval $[0,\eps^\gamma)$. This should be carefully taken into account when addressing the corresponding estimates. 

Without further ado, let us introduce the main result about the weighted $L^\infty$ estimates on the Green's function $\tilde\G_0$, and the error produced by the approximated operator $\tilde\L_0$. Notice, that in comparison with the previous sections, we keep the first order derivative in the approximated operator. This implies that we will not need to estimate $\partial_r\tilde \G_0$ as before in order to compute the error of the approximation.

\begin{proposition}\label{prop:G0}
Let $\gamma>2/9$. Define $\tilde\G_0$ as in \eqref{eq:Green's function near 0} and set $\G_0=P^{-1}\tilde \G_0 P$, where $P$ is the invertible matrix defined in \eqref{eq:Pmatrix}. Let $f\in Y$ with $\supp(f)\subset [0,\eps^\gamma)$. Under Assumptions \ref{H0}--\ref{H2}, the following estimates hold true uniformly for all $r_0 \in \mathcal{R}$, $M \in \mathcal{N}$, and for any $\eps>0$ sufficiently small,
\begin{align}
    \left|\max\{1,r^2\}w_\eps(r-r_0)\G_0f(r)\right| & \lesssim \eps^{-1/3} \dsOne_{\{0\leq r<2\eps^\gamma\}}(r) \|f\|_Y \label{eq:G0-Linf} \\
    & \quad + \eps^{-1/3}\left(\frac{\eps^\gamma}{r}\right)^{\frac{1}{4}\eps^{-1/3}} \dsOne_{\{r\geq 2\eps^\gamma\}}(r) \|f\|_Y, \notag \\
    \left|r^2w_\eps(r-r_0)\G_0f(r)\right| & \lesssim \eps^{-1/3+2\gamma} \dsOne_{\{0\leq r<2\eps^\gamma\}}(r) \|f\|_Y \label{eq:G0-Y} \\
    & \quad + \eps^{-1/3+2\gamma}\left(\frac{\eps^\gamma}{r}\right)^{\frac{1}{4}\eps^{-1/3}} \dsOne_{\{r\geq 2\eps^\gamma\}}(r) \|f\|_Y, \notag
\end{align}
and the same bounds also hold true with $\G_0$ replaced by $\tilde \G_0$.
Moreover, the error produced by the approximated operator is bounded by
\begin{equation}\label{eq:G0-error}
    \left\lVert ((\L-\lambda)\G_0-\Id) f\right\rVert_Y \lesssim \left( \eps^{2/3} + \eps^{2\gamma} + \eps^{3\gamma-2/3} \right) \|f\|_Y.
\end{equation}
\end{proposition}

\begin{remark}
    Notice that from \eqref{eq:G0-Linf} we straightforwardly obtain the following estimate for the approximate Green's function in $X$,
    \begin{equation}\label{eq:G0-X-Y}
    \|\G_0f\|_X \lesssim \eps^{-1/3}\|f\|_Y,
    \end{equation}
    provided that $\supp(f)\subset [0,\eps^\gamma)$.
\end{remark}

\begin{proof}
\emph{Part 1: Estimate on $\tilde \G_0$.} We want to derive estimates for a Green's function that can be written as in equation \eqref{eq:G0-Linf}. This implies that we need to find appropriate bounds for expressions of the form $|\G_0^1f(r)|$, and $|\G_0^1(r^{-2}\G_0^1f)(r)|$, where in both cases $f$ is a bounded function supported in the interval $[0,\eps^\gamma)$. Let us begin with the first case, since we have already presented suitable estimates at the beginning of Section \ref{sec:near-0}. In particular, since $\supp(f)\subset [0,\eps^\gamma)$, we can use the estimate \eqref{eq:G01-near0} to write
\begin{equation}\label{eq:near zero expression}
\begin{split}
    |w_\eps(r-r_0)\G_0^1f(r)|  & \lesssim \frac{\eps^{-2/3}\|f\|_Yw_\eps(r-r_0)}{\left( 1 + \eps^{-4/3} r^4 \tau^2\right)^{1/4}}  r^{-\eps^{-1/3}}\int_0^r w_\eps(s-r_0)^{-1}s^{-1+\eps^{-1/3}} \dd s \\
    & \quad + \frac{\eps^{-2/3}\|f\|_Y w_\eps(r-r_0)}{\left( 1 + \eps^{-4/3} r^4 \tau^2\right)^{1/4}} r^{\eps^{-1/3}}\int_r^{\eps^\gamma} w_\eps(s-r_0)^{-1} s^{-1-\eps^{-1/3}} \dd s.
\end{split}
\end{equation}
The second addend is nonzero only when $r<\eps^\gamma$, therefore let us start by taking $r\in [0,\eps^\gamma)$ and estimate both terms separately. Firstly, we note that $w_\eps(s-r_0)^{-1}$ is increasing for $s \in [0,r_0]$, and so the first addend from \eqref{eq:near zero expression} may be bounded by 
\[
w_\eps(r-r_0)\eps^{-2/3}\|f\|_{Y} r^{-\eps^{-1/3}}\int_0^r s^{-1+\eps^{-1/3}}w_\eps(s-r_0)^{-1}\dd s \leq \eps^{-1/3}\|f\|_{Y},
\]
for $r \leq \eps^\gamma$.
The second term in \eqref{eq:near zero expression} will require slightly more care. Indeed, we want to bound
\[
w_\eps(r-r_0)r^{\eps^{-1/3}}\int_r^{\eps^\gamma}w_\eps(s-r_0)^{-1}s^{-1-\eps^{-1/3}}\dd s.
\]
Since $w_\eps(s-r_0)^{-1}$ is increasing, we estimate it by its value at the right endpoint, ending up with an upper bound of the form 
\[
\frac{w_\eps(r-r_0)}{w(\eps^\gamma-r_0)}r^{\eps^{-1/3}}\int_r^{\eps^\gamma}s^{-1-\eps^{-1/3}}\dd s \leq \frac{w_\eps(r-r_0)}{w(\eps^\gamma-r_0)} \eps^{1/3}.
\]
For $r \in [0,\eps^\gamma)$, we further show in Lemma \ref{lemma:properties-weight} that the quotient
\[
\frac{w_\eps(r-r_0)}{w(\eps^\gamma-r_0)}
\]
is bounded for any $\eps$ sufficiently small by a constant that depends only on $N$. Therefore, we deduce that there exists a constant $C>0$ independent of $\eps \to 0$, so that for all $r \in [0,\eps^\gamma)$,
\[
|w_\eps(r-r_0)\G_0^1f(r)|\leq C\eps^{-1/3}\|f\|_Y,
\]
for any $f$ with support contained in $[0,\eps^\gamma)$.
Furthermore, if $r >\eps^\gamma$, the second addend in \eqref{eq:near zero expression} vanishes and we can estimate 
\begin{align}
\label{eq:near zero bounds}
\begin{split}
    |w_\eps(r-r_0)\G_0^1f(r)| & \leq w_\eps(r-r_0)\eps^{-2/3}\|f\|_Y r^{-\eps^{-1/3}}\int_0^{\eps^\gamma}w_\eps(s-r_0)^{-1} s^{-1+\eps^{-1/3}} \dd s  \\
    &\lesssim \frac{w_\eps(r-r_0)}{w_\eps(\eps^\gamma-r_0)} \eps^{-1/3}\|f\|_Y \left(\frac{\eps^\gamma}{r}\right)^{\eps^{-1/3}}.
\end{split}
\end{align}
Now, thanks to Lemma \ref{lemma:properties-weight} about the properties of the weight function, we argue that there exists a (uniform) constant $C$ such that for all $\eps$ small and all $r\geq \eps^\gamma$ there holds
\[
\frac{w_\eps(r-r_0)}{w_\eps(\eps^\gamma-r_0)}\leq \left (\frac{\eps^\gamma}{r} \right )^{-\frac{1}{2}\eps^{-1/3}},
\]
and thus we can write
\begin{equation*}
\frac{w_\eps(r-r_0)}{w_\eps(\eps^\gamma-r_0)}\left(\frac{\eps^\gamma}{r}\right)^{\eps^{-1/3}} \leq C \left (\frac{\eps^\gamma}{r} \right )^{\frac{1}{2}\eps^{-1/3}}.
\end{equation*}
All in all, we obtain the estimate
\[
|w_\eps(r-r_0)\G_0^1f(r)| \lesssim \eps^{-1/3}\|f\|_Y\left(\dsOne_{\{0\leq r<\eps^\gamma\}}(r) + \left(\frac{\eps^\gamma}{r}\right)^{\frac{1}{2}\eps^{-1/3}}\dsOne_{\{r\geq\eps^\gamma\}}(r)  \right).
\]
From this, we also easily obtain bounds with an extra $r^2$ weight in the $L^\infty$ norm. Indeed, if $r\in [0,\eps^\gamma)$, then trivially we obtain
\[
|r^2w_\eps(r-r_0)\G_0^1f(r)| \lesssim r^2\eps^{-1/3}\|f\|_Y \leq \eps^{2\gamma-1/3}\|f\|_Y.
\]
On the other hand, for $r\geq \eps^\gamma$ we write $r^2 =\eps^{2\gamma}(r\eps^{-\gamma})^2$ so that
\begin{equation}
\label{eq:near0 decay bound}
|r^2w_\eps(r-r_0)\G_0^1f(r)| \lesssim \eps^{-1/3}\|f\|_Y r^2\left(\frac{\eps^\gamma}{r}\right)^{\frac{1}{2}\eps^{-1/3}} \lesssim \eps^{-1/3+2\gamma}\|f\|_Y\left(\frac{\eps^\gamma}{r}\right)^{\frac{1}{4}\eps^{-1/3}}.
\end{equation}
With this, we obtained the desired estimate for the $X$ norm of $\G_0^1f$. To find the sought estimate for $\tilde\G_0f$ we also need to control terms of the form
\[
\eps^{1/3}\G_0^1\left( \frac{1}{r^2}\G_0^1f \right),
\]
where $\supp(f)\subset [0,\eps^\gamma)$, but $\supp(r^{-2}\G_0^1f)$ is not contained in the interval $[0,\eps^\gamma)$. In this case we split the function inside the Green's function in two contributions,
\[
\frac{1}{r^2}\G_0^1f = h_1(r) + h_2(r),
\]
with
\[
h_1(r) = \left(\frac{1}{r^2}\G_0^1f\right)\dsOne_{\{0\leq r<2\eps^\gamma\}}(r), \quad h_2(r) = \left(\frac{1}{r^2}\G_0^1f\right)\dsOne_{\{r\geq 2\eps^\gamma\}}(r).
\]
The term corresponding to $h_1$ is handled identically to before, and we obtain the pointwise upper bound 
\[
|w_\eps(r-r_0)\G_0^1h_1(r)|\lesssim \eps^{-1/3}\|h_1(r)\|_Y \left(\mathds{1}_{\{r \leq 2\eps^\gamma\}}+\left(\frac{2\eps^\gamma}{r}\right)^{\frac{1}{2}\eps^{-1/3}}\mathds{1}_{\{r >2\eps^\gamma\}}\right).
\]
Recall further that we already showed that $\|h_1(r)\|_Y \lesssim \eps^{-1/3}\|f(r)\|_Y$. Therefore, we observe 
\[
|w_\eps(r-r_0)\G_0^1 h_1(r)|\lesssim \eps^{-2/3}\|f(r)\|_Y\left(\mathds{1}_{\{r \leq 2\eps^\gamma\}}+\left(\frac{2\eps^\gamma}{r}\right)^{\frac{1}{2}\eps^{-1/3}}\mathds{1}_{\{r >2\eps^\gamma\}}\right).
\]
The remaining bit $h_2(r)$ is no longer supported in a small interval around zero. On the contrary, it is supported strictly away from zero, and it is precisely this property that we seek to exploit in order to derive suitable bounds. From \eqref{eq:G01-near0}, the Green's function $\G_0^1$ applied to $h_2(r)$ is bounded by an expression of the form
\[
\begin{split}
    |w_\eps(r-r_0)& \G_0^1h_2(r)| \\
    & \lesssim \frac{\eps^{-2/3}w_\eps(r-r_0)}{\left( 1 + \eps^{-4/3} r^4 \tau^2\right)^{1/4}} r^{-\eps^{-1/3}}\int_{2\eps^\gamma}^{\max\{r,2\eps^\gamma\}} w_\eps(s-r_0)^{-1}s^{-1+\eps^{-1/3}} |s^2 w_\eps(s-r_0) h_2(s)|\dd s \\
    & \quad + \frac{\eps^{-2/3}w_\eps(r-r_0)}{\left( 1 + \eps^{-4/3} r^4 \tau^2\right)^{1/4}}r^{\eps^{-1/3}}\int_{\max\{r,2\eps^\gamma\}}^\infty w_\eps(s-r_0)^{-1}s^{-1-\eps^{-1/3}} |s^2 w_\eps(s-r_0) h_2(s)|\dd s.
\end{split}
\]
Recalling the bounds \eqref{eq:near0 decay bound} for $|\G_0^1f(s)|$ when $s\geq \eps^\gamma$, we have that for $r\in [0,2\eps^\gamma)$,
\[
\begin{split}
    |w_\eps(r-r_0)\G_0^1h_2(r)| & \lesssim \eps^{-2/3} w_\eps(r-r_0)r^{\eps^{-1/3}}\int_{2\eps^\gamma}^\infty w_\eps(s-r_0)^{-1}s^{-1-\eps^{-1/3}} |w_\eps(s-r_0)\G_0^1f(s)| \dd s. \\
    & \lesssim \eps^{-1}\|f\|_Y w_\eps(r_0)\int_{2\eps^\gamma}^\infty s^{-1-\eps^{-1/3}} \left(\frac{\eps^\gamma}{s}\right)^{\frac{1}{2}\eps^{-1/3}} \dd s\\
    & \lesssim \eps^{-\frac{2}{3}}\|f\|_Y w_\eps(r_0) 2^{-\frac{1}{2}\eps^{-1/3}}.
\end{split}
\]
In the second inequality we used both the facts that $w_\eps(r-r_0) \leq w_\eps(r_0)$ for $r \leq r_0$, and that $w_\eps(s-r_0)^{-1}\leq 1$ for all $s \geq 0$.
From the definition of the weight function \eqref{eq:weight} it automatically follows that $w_\eps(r_0) \lesssim \eps^{-N/3}$. We conclude that, since the prefactor $2^{-\frac{1}{2}\eps^{-1/3}}$ decays geometrically fast to zero as $\eps\to 0$, we can bound this term uniformly as $\eps\to 0$ for any $r \leq 2\eps^\gamma$ by
\[
|w_\eps(r-r_0)\G_0^1h_2(r)|\leq \eps^{-2/3}2^{-\frac{1}{4}\eps^{-1/3}}\|f\|_Y
\]
If $r\geq 2\eps^\gamma$ instead, we can compute the estimate 
\begin{align*}
    |w_\eps(r-r_0)\G_0^1h_2(r)| &\lesssim \eps^{-2/3} w_\eps(r-r_0) r^{-\eps^{-1/3}}\int_{2\eps^\gamma}^r s^{-1+\eps^{-1/3}}  w_\eps(s-r_0)^{-1}| w_\eps(s-r_0)\G_0^1f(s)|\dd s \\
    & \quad + \eps^{-2/3} w_\eps(r-r_0) r^{\eps^{-1/3}}\int_r^\infty s^{-1-\eps^{-1/3}}  w_\eps(s-r_0)^{-1}| w_\eps(s-r_0)\G_0^1f(s)| \dd s,
\end{align*}
thus, using the previous estimates for $\G_0^1f$ we obtain
\[
\begin{split}
    |w_\eps(r-r_0)\G_0^1h_2(r)| &\lesssim \eps^{-1}\|f\|_Yw_\eps(r-r_0) r^{-\eps^{-1/3}}\int_{2\eps^\gamma}^r s^{-1+\eps^{-1/3}}\left(\frac{\eps^\gamma}{s}\right)^{\frac{1}{2}\eps^{-1/3}} \dd s \\
    & \quad + \eps^{-1}\|f\|_Y w_\eps(r-r_0) r^{\eps^{-1/3}}\int_r^\infty s^{-1-\eps^{-1/3}}\left(\frac{\eps^\gamma}{s}\right)^{\frac{1}{2}\eps^{-1/3}}  \dd s \\
    & \lesssim \eps^{-2/3}\|f\|_Y w_\eps(r-r_0) \left(\frac{\eps^\gamma}{r}\right)^{\frac{1}{2}\eps^{-1/3}}.
\end{split}
\]
But observe from Lemma \ref{lemma:properties-weight} about the properties of the wight function, since $w_\eps(\eps^\gamma-r_0)\geq 1$, we can readily deduce that for any $r \geq 2\eps^\gamma$ there holds 
\[
w_\eps(r-r_0) \leq \left (\frac{\eps^\gamma}{r}\right )^{-\frac{1}{4}\eps^{-1/3}}.
\]
Therefore, we deduce that
\[
|w_\eps(r-r_0)\G_0^1h_2(r)| \lesssim \eps^{-2/3}\left(\frac{\eps^\gamma}{r}\right)^{\frac{1}{4}\eps^{-1/3}}\|f\|_Y
\]
for any $r \geq 2\eps^\gamma$. All in all, we obtain a final estimate for the Green's function $\G_0^1$ acting on $r^{-2}\G_0^1f$ of the form
\[
\begin{split}
    \left|w_\eps(r-r_0)\G_0^1\left(\frac{1}{r^2}\G_0^1f\right)\right| & \lesssim \eps^{-2/3}\|f\|_Y \left( 1+2^{-\frac{1}{4}\eps^{-1/3}}\right) \dsOne_{\{0\leq r<2\eps^\gamma\}}(r) \\
    & \quad + \eps^{-2/3}\|f\|_Y \left(\frac{\eps^\gamma}{r}\right)^{\frac{1}{4}\eps^{-1/3}} \dsOne_{\{r\geq 2\eps^\gamma\}}(r),
\end{split}
\]
which implies in particular
\[
\begin{split}
    \left|w_\eps(r-r_0)\G_0^1\left(\frac{1}{r^2}\G_0^1f\right)\right| & \lesssim \eps^{-2/3}\|f\|_Y \dsOne_{\{0\leq r<2\eps^\gamma\}}(r) \\
    & \quad + \eps^{-2/3}\|f\|_Y  \left(\frac{\eps^\gamma}{r}\right)^{\frac{1}{4}\eps^{-1/3}} \dsOne_{\{r\geq 2\eps^\gamma\}}(r).
\end{split}
\]
Finally, in order to obtain the claim for the $X$ norm, we need to obtain an analogous estimate with an additional $r^2$ weight. This is thus straightforwardly deduced from the previous estimate,
\[
\begin{split}
    \left|r^2w_\eps(r-r_0)\G_0^1\left(\frac{1}{r^2}\G_0^1f\right)\right| & \lesssim \eps^{2\gamma-2/3}\|f\|_Y \dsOne_{\{0\leq r<2\eps^\gamma\}}(r) \\
    & \quad + \eps^{2\gamma-2/3}\|f\|_Y  \left(\frac{\eps^\gamma}{r}\right)^{\frac{1}{4}\eps^{-1/3}} \dsOne_{\{r\geq 2\eps^\gamma\}}(r).
\end{split}
\]
Now we can go back to the explicit expression for the Green's function $\tilde\G_0$ in \eqref{eq:Green's function near 0} and bring together all the estimates we derived thus far. Something to bear in mind is that the first estimate derived for $\tilde\G_0^1f$ is split into the intervals $[0,\eps^\gamma)$ and $[\eps^\gamma,\infty)$. However a quick inspection reveals that this estimate is also true if split into the intervals $[0,2\eps^\gamma)$ and $[2\eps^\gamma,\infty)$. We analyse the second component of \eqref{eq:Green's function near 0}, since it is ever so slightly more complicated and the estimate for the first component follows analogously.
\[
\begin{split}
    |w_\eps(r-r_0)\tilde \G_0f_{\pm}| & \lesssim |w_\eps(r-r_0)\G_0^1f_-| + \eps^{1/3}\left|w_\eps(r-r_0)(\G_0^1\left(\frac{1}{r^2}\G_0^1(f_+)\right)\right| \\
    & \lesssim \left(\eps^{-1/3} \dsOne_{\{0\leq r<2\eps^\gamma\}}(r) + \eps^{-1/3}\left(\frac{\eps^\gamma}{r}\right)^{\frac{1}{4}\eps^{-1/3}} \dsOne_{\{r\geq 2\eps^\gamma\}}(r) \right) \|f_{\pm}\|_Y.
\end{split}
\]
Here we have used the notation $f_\pm=(f_+,f_-)$.
Moreover, if we put together the pointwise estimates with the $r^2$ weight, we can write
\[
\begin{split}
    |r^2w_\eps(r-r_0)\tilde \G_0f_{\pm}| & \lesssim |w_\eps(r-r_0)\G_0^1f_-| + \eps^{1/3}\left|w_\eps(r-r_0)\G_0^1\left(\frac{1}{r^2}\G_0^1(f_+)\right)\right| \\
    & \lesssim \left(\eps^{2\gamma-1/3} \dsOne_{\{0\leq r<2\eps^\gamma\}}(r) + \eps^{2\gamma-1/3}\left(\frac{\eps^\gamma}{r}\right)^{\frac{1}{4}\eps^{-1/3}} \dsOne_{\{r\geq 2\eps^\gamma\}}(r) \right) \|f_{\pm}\|_Y.
\end{split}
\]
We obtained in fact \eqref{eq:G0-Y} as claimed in the lemma, and combining the last two estimates we arrive as well at \eqref{eq:G0-Linf}.

\emph{Part 2: The error estimate}. The next claim to be addressed corresponds to the error estimate \eqref{eq:G0-error}. We need to compare the exact operator with the approximated operator that we constructed for functions $f$ supported in a neighbourhood of $r=0$. The goal is to find an estimate for 
\[
\left ((\tilde\L-\lambda)-\tilde \L_0 \right )\tilde\G_0 f_\pm.
\]
Indeed, if we find a suitable estimate on this difference, by \eqref{eq:change of variables Green's function} we conclude \eqref{eq:G0-error}.
Combining \eqref{eq:near0-addsubtract} with \eqref{eq:L0} and writing $\lambda = \mu\eps^{1/3}$, we find that
\[
\begin{split}
(\tilde\L-\lambda)\begin{pmatrix}
    V_+\\
    V_-
\end{pmatrix} - \tilde\L_0\begin{pmatrix}
    V_+\\
    V_-
\end{pmatrix} & = - \left( \frac{\eps}{r^2} + \eps^{1/3}M^2\left(\frac{\Omega_0'}{U_0'}\right)^2 + \mu\eps^{1/3} \right) \begin{pmatrix}
    V_+\\
    V_-
\end{pmatrix} \\
& \quad - i\eps^{-1/3}(T(r)-\tau)\begin{pmatrix}
    V_+\\
    V_-
\end{pmatrix} - \alpha\eps^{1/3}\left(r\Omega' + \frac{2iM\eps^{2/3}}{r^2}\right)\begin{pmatrix}
    V_-\\
    0
\end{pmatrix}.
\end{split}
\]
Introducing weights corresponding to the $Y$ norm, we observe that it suffices to estimate 
\begin{equation}
\label{eq:error expression near zero}
\begin{split}
&\left| r^2w_\eps(r-r_0) \left ((\tilde\L-\lambda)-\tilde \L_0 \right )\tilde\G_0 f_\pm \right| \lesssim \left |r^2w_\eps(r-r_0)\eps^{-1/3}(T(r)-\tau)\tilde{\G_0}f_\pm\right |   \\
& \quad \lesssim \left |r^2 w_\eps(r-r_0)\left( \frac{\eps}{r^2} + \eps^{1/3}M^2\left(\frac{\Omega_0'}{U_0'}\right)^2 + \mu\eps^{1/3}\right) \tilde{\G_0}f_\pm\right |\\
& \qquad +\left |w_\eps(r-r_0)\left (\alpha \eps^{1/3}(r^3\Omega'(r)+2iM\eps^{2/3}\right )\tilde{\G_0}f_\pm\right |.
\end{split}
\end{equation}
Using \eqref{eq:G0-Linf}, \eqref{eq:G0-Y} the first second from \eqref{eq:error expression near zero} may be controlled by
\begin{align*}
\eps |w_\eps(r-r_0)\tilde{\G_0}f_\pm|+\eps^{1/3}\left | M^2\left(\frac{\Omega_0'}{U_0'}\right)^2 + \mu\right | \left |r^2 w_\eps(r-r_0)\tilde \G_0 f_\pm  \right|\lesssim (\eps^{2/3}+\eps^{2\gamma})\|f_\pm\|_{Y}.
\end{align*}
To find the corresponding bound for the first addend in \eqref{eq:error expression near zero}, we first note that by a Taylor expansion near zero, there holds $|T(r)-\tau|\lesssim r\mathds{1}_{\{0 \leq r \leq 2\eps^\gamma\}} + \BigO_{\eps\to 0}(\eps^{2\gamma})$. On the other hand, far away from the origin we use Hypothesis \ref{H2}, which in turn produces the estimate
\[
|T(r)-\tau| \lesssim r\dsOne_{\{0\leq r<2\eps^\gamma\}}(r) + \left(1+\sum_{n=1}^{N_1} B_n(\eps)r^n \right)\dsOne_{\{r \geq 2\eps^\gamma\}}(r)
\]
Putting this together with \eqref{eq:G0-Y} we find
\[
\begin{split}
    \eps^{-1/3}|T(r)-\tau| \left|r^2w_\eps(r-r_0)\tilde\G_0 f_\pm\right| & \lesssim  \eps^{2\gamma-2/3} r \dsOne_{\{0\leq r<2\eps^\gamma\}}(r) \|f_\pm\|_Y \\
& \quad +  \eps^{2\gamma-2/3} \left(1+\sum_{n=1}^{N_1} B_n(\eps)r^n \right)\left(\frac{\eps^\gamma}{r}\right)^{\frac{1}{4}\eps^{-1/3}} \dsOne_{\{r\geq 2\eps^\gamma\}}(r) \|f_\pm\|_Y  \\
& \lesssim  \eps^{3\gamma-2/3} \|f_\pm\|_Y, \\
\end{split}
\]
where we used the estimate from Property \ref{P3}
\[
\sum_{n=1}^{N_1} B_n(\eps) \lesssim \eps^{\beta-N_1\delta}.
\]
Hence, we obtain
\[
 \eps^{-1/3}|T(r)-\tau| \left|r^2w_\eps(r-r_0)\tilde\G_0 f_\pm\right| \lesssim \eps^{3\gamma-2/3}\|f_\pm\|_Y.
\]
The last term to be addressed can be split in two. On the one hand, the term without the $\Omega'$ can be directly bounded by
\[
\eps^{1/3}\left|2iM\eps^{2/3}\right|\left|w_\eps(r-r_0)\tilde\G_0 f_\pm\right| \lesssim \eps\left|w_\eps(r-r_0)\tilde\G_0 f_\pm\right|,
\]
which together with \eqref{eq:G0-Linf} yields
\[
\eps\left|w_\eps(r-r_0)\tilde \G_0 f_\pm\right| \lesssim \eps^{2/3}\|f_\pm\|_Y.
\]
On the other hand, the term with the azimuthal component of the vector field can be estimated by using Property \ref{P4} far away from the origin, and the mean value theorem near the origin. Namely, there exist $N_2\in \N$ such that
\[
|r\Omega'(r)| \lesssim  r\dsOne_{\{0\leq r<2\eps^\gamma\}}(r) + C\left (1 + \sum_{n=1}^{N_2} D_nr^n\right)\dsOne_{\{2\eps^\gamma\leq r\}}(r).
\]
Hence, putting this together with \eqref{eq:G0-Y} we can write the estimate
\[
\begin{split}
    \eps^{1/3}\left|r\Omega'\right|\left|r^2w_\eps(r-r_0)\tilde\G_0f_\pm\right| & \lesssim r\eps^{2\gamma}\dsOne_{\{0\leq r<2\eps^\gamma\}}(r)\|f_\pm\|_Y \\
& \quad + \eps^{-1/3}\left(1 + \sum_{n=1}^{N_2} D_nr^n\right)\left(\frac{\eps^\gamma}{r}\right)^{\frac{1}{4}\eps^{-1/3}}\dsOne_{\{2\eps^\gamma\leq r\}}(r)\|f_\pm\|_Y \\
& \lesssim \left( \eps^{3\gamma} + \eps^{-1/3-N_2\delta}\left(\frac{1}{2}\right)^{\frac{1}{8}\eps^{-1/3}} \right) \|f_\pm\|_Y.
\end{split}
\]

As in the previous scenario, the term that is exponentially decaying becomes negligible in comparison with the other term. All in all we find that 
\[
r^2w_\eps(r-r_0)\left|(\tilde\L-\lambda)\tilde\G_0 f_\pm - f_\pm\right| \lesssim \left( \eps^{2/3} + \eps^{2\gamma} + \eps^{3\gamma-2/3} \right)\|f_\pm\|_Y,
\]
and so we conclude the proof of the Proposition by \eqref{eq:change of variables Green's function}.
\end{proof}

\subsection{Transport functions with linear zeroes.}\label{s:linear-zeroes}

Thus far we have assumed that $T(r) \neq 0$ for all $r \neq r_0$, however Assumption \ref{SH} is not strictly necessary. In fact, we shall show that if there exists a (finite) collection of points $s_j>0$ with $T(s_j)=0$, and $T'(s_j) \neq 0$, we can still construct a local inverse to our operator in a neighbourhood of $s_j$, which can be glued together with the remaining inverses as before. This yields the only bit remaining in order to obtain the claim of Proposition \ref{proposition:approx-Green's}

As mentioned at the beginning of Section \ref{s:approx-Green's}, the form of the Green's operator $\G^\lambda$ depends on the specific form of the transport function $T(r)$. We want to adapt the proof of the previous sections to the general case allowing for linear zeros, and we will do so by dividing the support of $f$ in different regions as earlier, but taking into account the new zeros. In particular, around the origin $\supp(f)\subset[0,\eps^\gamma)$, and around the critical radius $\supp(f)\subset[r_0-\eps^\gamma,r_0+\eps^\gamma)$, we define $\G_0$ and $\G_2^\lambda$ just like in Sections \ref{sec:near-0} and \ref{s:around-r0} respectively, whereas around every linear zero $s_j$ we introduce new operators $\G^\LV_j$ defined for $\supp(f)\subset [s_j-\eps^\omega,s_j+\eps^\omega)$. The regions between the origin and the first zero of $T$, as well as between any two zeros, will be covered by $\G_1$. The region from the largest zero $\max\{r_0,s_j\}$ to infinity will be covered by $\G_3$ as in the previous section. All in all, we can write the Green's function as in Section \ref{s:proof}
\[
\G^\lambda f = \G_2^\lambda (f\dsOne_{[r_0-\eps^\gamma,r_0+\eps^\gamma)}) + \G_{\mathrm{extra}}(f\dsOne_{\R_+\setminus [r_0-\eps^\gamma,r_0+\eps^\gamma)}),
\]
where crucially all Green's operators included in $\G_{\mathrm{extra}}$ are independent of the choice of $\lambda$.

Notice that we can glue together all these Green's function due to the exponential decay exhibited for all $r\not\in\supp(f)$ on each corresponding case---see Remark \ref{rmk:exp-decay} for further details. Therefore, to prove Proposition \ref{proposition:approx-Green's} it only remains to focus on the local problem around one of the linear zeros $s_j$, since every other region follows from the analysis in the previous sections. 

Suppose that $s_j>0$ is a linear zero of $T(r)$, and set $MT'(s_j)=\tau_j$, which without loss of generality we take $\tau_j >0$. In this section we will assume that $\supp(f)\subset [s_j-\eps^\omega,s_j+\eps^\omega)$, for some $\omega>0$ to be specified later, and we define the following operator 
\begin{equation}
\label{eq:airy approximation}
\mathcal{L}^{\mathrm{LV}}_j=\eps \partial_r^2 -i\eps^{-1/3}\tau_j(r-s_j).
\end{equation}
that we claim it approximates the true $\L$ sufficiently well if $f$ is supported around $s_j$.
We now set $v(r)$ as a solution of the differential equation
\begin{equation*}
\partial_r^2 v-r v=0,
\end{equation*}
and let $w(r)=v(\gamma (r-s_j))$. Then, 
\begin{equation*}
\L^{\mathrm{LV}}_j w(r)=\eps \gamma^2 v''(\gamma(r-s_j))-i \eps^{-1/3}\gamma^{-1}\tau_j (\gamma(r-s_j)v(\gamma(r-s_j))).
\end{equation*}
In particular, as soon as $i\eps^{-4/3}\gamma^{-3}\tau_j=1$, then $\L^{\LV}_j w=0$. Therefore, we pick $\gamma=i^{1/3}\tau_j^{1/3}\eps^{-4/9}$, where we also make the choice of the branch in the cube root so that $\arg(i^{1/3})=\pi/6$.
Now, the solutions of 
\begin{equation*}
\partial_r^2 v-rv=0
\end{equation*}
are the so called \emph{Airy functions}. In particular, as further elaborated in Appendix \ref{s:Appendix airy}, we have two linearly independent solutions given by 
\begin{equation*}
v_1(r)=\Ai(r), \quad v_2(r)=\Ai(\e^{2 \pi i/3}r). 
\end{equation*}
Furthermore, subtracting the equations 
\begin{align}
v_2 v_1''-rv_1v_2=0\\
v_1 v_2''-rv_2v_1=0
\end{align}
we see that $\frac{\dd}{\dd r}W(v_1,v_2)(r)=0$, where $W(v_1,v_2)=v_1'v_2-v_2'v_1$ is the Wronskian. Hence it is clear that the Wronskian is given by $W(\Ai(r),\Ai(\e^{2 \pi i/3}r))=\text{constant}$.
Furthermore, from Lemma \ref{lemma:Airy bounds} in Appendix \ref{s:Appendix airy}, we have the following asymptotic bounds valid uniformly in $z$ with $|\arg(z)|=\frac{\pi}{6}$,
\begin{align*}
&|\Ai(z)|, |\Ai(-\e^{2\pi i/3}z)| \leq C\e^{-\frac{\sqrt{2}}{3}|z|^{3/2}},\\
&|\Ai(-z)|, |\Ai(\e^{2\pi i/3}z)| \leq C\e^{\frac{\sqrt{2}}{3}|z|^{3/2}},\\
&|\Ai'(z)|, |\Ai'(-\e^{2\pi i/3}z)|\leq C(1+|z|^\frac{1}{4}) \e^{-\frac{\sqrt{2}}{3} |z|^{3/2}},\\
&|\Ai'(-z)|, |\Ai'(\e^{2\pi i/3}z)|\leq C(1+|z|^\frac{1}{4}) \e^{\frac{\sqrt{2}}{3} |z|^{3/2}},
\end{align*}
where we write $\Ai'(z)$ to denote the (complex) derivative of $\Ai(z)$ with respect to $z$.
Therefore, under the boundary condition that our solutions should decay as $r \to \pm \infty$, we obtain the Green's function 
\begin{equation*}
\G^{\LV}_j f(r) = C\eps^{-5/9} \left[a_1(r)\int_{s_j-\eps^\omega}^ra_2(s) f(s) \dd s+a_2(r)\int_r^{s_j+\eps^\omega}a_1(s)f(s)\dd s \right],
\end{equation*}
where 
\[
a_1(r)=\Ai \left(i^{1/3}\tau_j^{1/3}\eps^{-4/9}(r-s_j)\right), \quad a_2(r)=\Ai \left(\e^{2\pi i/3}i^{1/3}\tau_j^{1/3}\eps^{-4/9}(r-s_j)\right).
\]
Recall that here we assume that $\supp(f)\subset [s_j-\eps^\omega,s_j+\eps^\omega)$, for some $\omega>0$ to be specified later. 

We now shall state the main result of the section.

\begin{lemma}\label{lemma:linear-vanish}
Suppose there exists a point $s_j>0$ so that $T(s_j) =0$, but $T'(r_j) \neq 0$, and let $\omega>\frac{2}{9}$.
Then, there exists an operator $\G^{\LV}_j$ so that for all $f$ with $\supp (f) \subset [r_j-\eps^\omega, r_j+\eps^\omega)$ there holds 
\begin{equation}\label{eq:LV-X}
\|\G^{\LV}_j f\|_X \lesssim \eps^{-1/9}\|f\|_Y, \quad \|\partial_r \G^{\LV}_j f\|_{X} \lesssim \eps^{-5/9}\|f\|_Y.
\end{equation}
Furthermore, we have the following error estimate 
\begin{equation}\label{eq:error-LV}
\|((\mathcal{L}-\lambda)\G^{\LV}_j - \Id)f\|_Y \lesssim \left(\eps^{2/9} + \eps^{2\omega-4/9} \right)\|f\|_{Y}.
\end{equation}
\end{lemma}

\begin{remark}\label{rmk:why Gilbert condition}
An interesting consequence of our analysis of the Airy functions is its connection to Gilbert's condition $M\Omega'(r_0)+KU'(r_0)=0$ from Section \ref{s:gilbert-scaling}, see \eqref{eq:Gilbert2}. Indeed, if we assume $M\Omega'(r_0)+KU'(r_0) \neq 0$, then by an order of magnitude analysis, the dominant contribution of the equation near $r_0$ would be given by $\eps \partial_r^2 -i(M\Omega'(r_0)+KU'(r_0))\eps^{-1/3}(r-r_0)$, just as in \eqref{eq:airy approximation}. Hence, to first order, any growing mode must look like a 
\[
v(i^{1/3}\tau_j^{1/3}\eps^{-4/9}(r-r_0))
\]
near $r_0$, where $v$ is a solution to Airy's equation. But from Lemma \ref{lemma:Airy bounds}, we see that any such solution will grow exponentially either for $r>r_0$ or $r<r_0$. Therefore, we do not expect such an approximate to be a good candidate for a ``boundary layer" solution, i.e.\ a solution for which the growth occurs in a neighbouring area of a critical radius $r_0$.
\end{remark}
\begin{proof}
We proceed following the ideas of Section \ref{s:around-r0}. Set 
\[
\kappa =\frac{\sqrt{2}}{3}\tau_j^{1/3},
\]
then we find the estimate for the Green's function
\[
|\G^{\LV}_jf(r)| \lesssim C\|f\|_Y \eps^{-5/9} \left(\I_\eps + \I'_\eps + \II_\eps + \III_\eps + \IV_\eps + \IV'_\eps\right),
\]
where
\[
\I_\eps = \dsOne_{\{s_j-\eps^\omega\leq r<s_j\}}\e^{\kappa \eps^{-2/3}(s_j-r)^{3/2}}\int_{s_j-\eps^\omega}^{r} \frac{\e^{-\kappa \eps^{-2/3}|s_j-s|^{3/2}}}{s^2w_\eps(s-r_0)}\dd s,
\]
\[
\I'_\eps = \dsOne_{\{s_j\leq r\}}  \e^{-\kappa \eps^{-2/3}(r-s_j)^{3/2}}\int_{s_j-\eps^\omega}^{s_j} \frac{\e^{-\kappa \eps^{-2/3}(s_j-s)^{3/2}}}{s^2w_\eps(s-r_0)}\dd s,
\]
\[
\II_\eps = \e^{-\kappa \eps^{-2/3}(r-s_j)^{3/2}} \int_{\min\{r,s_j\}}^{\min\{r,s_j+\eps^\omega\}} \frac{\e^{\kappa \eps^{-2/3}(s-s_j)^{3/2}}}{s^2w_\eps(s-r_0)} \dd s,
\]
\[
\III_\eps = \e^{-\kappa \eps^{-2/3}(s_j-r)^{3/2}} \int_{\max\{r,s_j-\eps^\omega\}}^{\max\{r,s_j\}} \frac{\e^{\kappa \eps^{-2/3}(s_j-s)^{3/2}}}{s^2w_\eps(s-r_0)} \dd s,
\]
\[
\IV_\eps = \dsOne_{\{s_j\leq r < s_j+\eps^\omega\}}\e^{\kappa \eps^{-2/3}(r-s_j)^{3/2}}\int_{r}^{s_j+\eps^\omega} \frac{\e^{-\kappa \eps^{-2/3}|s_j-s|^{3/2}}}{s^2w_\eps(s-r_0)}\dd s,
\]
\[
\IV'_\eps = \dsOne_{\{r<s_j\}}\e^{-\kappa \eps^{-2/3}(s_j-r)^{3/2}} \int_{s_j}^{s_j+\eps^\omega} \frac{\e^{-\kappa \eps^{-2/3}(s-s_j)^{3/2}}}{s^2w_\eps(s-r_0)} \dd s.
\]
It thus only remains to estimate these terms. Since the arguments are very similar to those we have seen before in Section \ref{s:around-r0}, we shall restrict ourselves to showing the required bounds for $r \in \supp(f)$. For $r \notin \supp(f)$, the Green's function will decay exponentially as in Section \ref{s:around-r0}, and thus the errors are negligible. Without loss of generality we also assume that $s_j>r_0$, since the case $s_j<r_0$ follows in a similar manner. 

We start by considering $r\in [s_j-\eps^\omega,s_j)$. In this case, $\I'_\eps=0$, $\II_\eps=0$ and $\IV_\eps=0$. For the first term we find, after introducing the weight corresponding to the $X$ norm,
\begin{equation*}
w_\eps(r-r_0)\I_\eps = w_\eps(r-r_0)\e^{\kappa \eps^{-2/3}(s_j-r)^{3/2}}\int_{s_j-\eps^\omega}^{r} \frac{\e^{-\kappa \eps^{-2/3}|s_j-s|^{3/2}}}{s^2 w_\eps(s-r_0)} \dd s.
\end{equation*}
However $s^{-2}$ is a decreasing function bounded by its value in the left extreme of the integral, and moreover, as we show in Lemma \ref{lemma:properties-weight}, there exists a constant $C$ only depending on $N$ such that
\begin{equation*}
\frac{w_\eps(r-r_0)}{w_\eps(s-r_0)} \leq C
\end{equation*}
for any $s_j-\eps^\omega\leq s\leq r<s_j$. Notice that $C$ is independent of $\eps$ provided that $\eps$ is small enough. Up to the change of variables $v=\eps^{-4/9}(s_j-s)$, $u=\eps^{-4/9}(s_j-r)$, we find the estimate
\begin{align*}
w_\eps(r-r_0)\I_\eps \lesssim \eps^{4/9} \e^{\kappa u^{3/2}} \int_{u}^{\eps^{\omega-4/9}} \e^{-\kappa v^{3/2}} \dd v \lesssim \eps^{4/9} \e^{\kappa u^{3/2}} \int_{u}^{\infty} \e^{-\kappa v^{3/2}} \dd v.
\end{align*}
Furthermore, arguing as in Section \ref{s:around-r0}, an application of L'H\^opital's rule shows that the function 
\begin{equation}
\label{eq:lhopital Airy}
u \mapsto \e^{\kappa u^{3/2}} \int_u^\infty \e^{-\kappa z^{3/2}} \dd z
\end{equation}
is uniformly bounded in $u$, and decays like $u^{-1/2}$ as $u \to \infty$. Therefore, we observe that the first term is bounded by 
\begin{equation*}
\|f\|_Y \eps^{-5/9} w_\eps(r-r_0)\I_\eps \lesssim \eps^{-1/9}\|f\|_{Y},
\end{equation*}
and differentiating the expression in $\kappa$, we observe uniformity of the estimate in compact intervals of $\kappa$ as in Section \ref{s:around-r0}.
For the third term we argue similarly,
\[
\begin{split}
    w_\eps(r-r_0)\III_\eps & \leq w_\eps(r-r_0)\e^{-\kappa \eps^{-2/3}(s_j-r)^{3/2}} \int_r^{s_j} \frac{\e^{\kappa \eps^{-2/3}(s_j-s)^{3/2}}}{s^2w_\eps(s-r_0)} \dd s \\
    & = \eps^{4/9}\e^{-\kappa u^{3/2}} \int_0^{u} \e^{\kappa v^{3/2}} \dd v,
\end{split}
\]
under the change of variables $v=\eps^{-4/9}(s_j-s)$, $u=\eps^{-4/9}(s_j-r)$. This time, via L'H\^opital we find as well that
\begin{equation}
\label{eq:lhopital Airy2}
u \mapsto \e^{-\kappa s^{3/2}} \int_0^s \e^{\kappa z^{3/2}} \dd z
\end{equation}
is also uniformly bounded in $u$, and moreover, it decays like $u^{-1/2}$ as $u\to \infty$. Hence
\[
\|f\|_Y \eps^{-5/9} w_\eps(r-r_0) \III_\eps \lesssim \eps^{-1/9}\|f\|_{Y}.
\]
Again, uniformity in compact intervals of $\kappa$ follows as in Section \ref{s:around-r0}.
Lastly, the fourth term is easily boundable since both exponentials are decreasing towards infinity, 
\[
\begin{split}
    w_\eps(r-r_0)\IV'_\eps & \leq w_\eps(r-r_0)\e^{-\kappa \eps^{-2/3}(s_j-r)^{3/2}} \int_{s_j}^{s_j+\eps^\omega} \frac{\e^{-\kappa \eps^{-2/3}(s-s_j)^{3/2}}}{s^2w_\eps(s-r_0)} \dd s \\
    & \lesssim \eps^{4/9}\e^{-\kappa u^{3/2}} \int_{0}^{\eps^{\omega-4/9}} \e^{-\kappa v^{3/2}} \dd v \lesssim \eps^{4/9}.
\end{split}
\]
All in all, we obtain the estimate
\[
\|f\|_Y\eps^{-5/9} w_\eps(r-r_0)\IV'_\eps \lesssim \eps^{-1/9}\|f\|_Y \e^{-\eps^{-2/3}(s_j-r)^{3/2}},
\]
which quickly converges to zero as $\eps\to 0$ if $|r- s_j|\gtrsim O_{\eps\to 0}(1)$ due to the effect of the exponential, but is only bounded by $\eps^{-1/9}$ for the limit case $r=s_j$.
The argument exposed for this region carries over to the domain $r\in [s_j,s_j+\eps^\omega)$ by symmetry. Moreover, as mentioned before, if $r\not\in \supp(f)$, following the ideas from Section \ref{s:around-r0} we find that the Green's function decays exponentially, hence we obtain the first estimate in \eqref{eq:LV-X}.

Next we investigate the derivative of the Green's function. We recall the asymptotic bounds for the derivative of the Airy function from Appendix \ref{s:Appendix airy}, i.e.\ 
\begin{equation*}
\Ai'(z) \sim z^{1/4}\e^{-cz^{3/2}}.
\end{equation*}
Observing that the decay of the functions in \eqref{eq:lhopital Airy} and \eqref{eq:lhopital Airy2} is like $u^{-1/2}$ as $u \to \infty$, it follows that upon multiplying the function by $u^{1/4}$, it will still remain bounded uniformly in $u$, and therefore exactly the same arguments as before carry over and yield pointwise bounds for the derivative of the Green's function as stated in \eqref{eq:LV-X}.

It remains to comment on the error estimates \eqref{eq:error-LV}. A direct computation reveals that the error may be bounded by
\begin{equation*}
\begin{split}
    \|((\L-\lambda)\G^{\LV}_j-\Id)f\|_Y & \lesssim \|\eps \partial_r \G^{\LV}_j f\|_{Y}+\|\eps^{1/3}\G^{\LV}_j f\|_Y \\
    & \quad +\|r^{-2}\eps^{1/3}\G^{\LV}_j f\|_Y+\||r-s_j|^2\eps^{-{1/3}}\G^{\LV}_j f\|_Y,
\end{split}
\end{equation*}
where the $|r-s_j|^2$ contribution is \emph{uniform} in $r_0 \in \mathcal{R}$ due to Lemma \ref{lemma:consequences of assumptions}, \ref{P1}.
Due to the bounds on $\G^{\LV}_j f$, the first three error terms are certainly controlled by
\[
\|\eps \partial_r \G^{\LV}_j f\|_{Y} \lesssim \eps^{4/9}\|f\|_Y,
\]
\[
\|\eps^{1/3}\G^{\LV}_j f\|_Y \leq \eps^{1/3}\|\G^{\LV}_j f\|_X \lesssim \eps^{2/9}\|f\|_Y,
\]
\[
\|r^{-2}\eps^{1/3}\G^{\LV}_j f\|_Y \leq \eps^{1/3}\|\G^{\LV}_j f\|_X \lesssim \eps^{2/9}\|f\|_Y.
\]
Finally, by the $L^\infty$ bounds on $\G^{\LV}_j f$, we observe for the final term that we can write
\[
\left||r-s_j|^2 \eps^{-{1/3}}r^2 w_\eps(r-r_0)\G^{\LV}_j f\right| \lesssim \eps^{2\omega}\eps^{-4/9}\|f\|_{Y}
\]
if $|r-s_j|<\eps^\omega$, and due to the exponential decay of the Green's function away from $|r-s_j|<\eps^\omega$, this bound can be extended to all of $[0,\infty)$, which yields the claimed result \eqref{eq:error-LV}. 
\end{proof}

\begin{proof}[Proof of Proposition \ref{proposition:approx-Green's}]
With this result in hand, we can now readily conclude the proof of Proposition \ref{proposition:approx-Green's}. Indeed, by Assumptions \ref{H0}--\ref{H2}, we can cover any interval without zeroes of $T$ that is at least a distance $\eps^\gamma$ away from $\mathcal{F}_0 \cup\{r_0\}$ by linear interpolation steps as in Lemma \ref{lemma:pointGrj-towinf}. Thus, it remains to note that if $|r-s_j|<\eps^\gamma$, $s_j \in \mathcal{F}_0$, then as long as $\eps^\gamma<\eps^\omega$, this regime is covered by Lemma \ref{lemma:linear-vanish}. But recall that $\gamma>2/9$,  and in Lemma \ref{lemma:linear-vanish}, $\omega$ can be chosen arbitrarily in $(2/9,\infty)$. Therefore, simply picking $2/9<\omega<\gamma$, the claim of Proposition \ref{proposition:approx-Green's} follows.
\end{proof}

\section{Extension to domains with boundaries}\label{s:boundaries}

In this section, we extent our analysis to cover the case where the dynamo equations are posed on cylindrical domains with boundary. Given some interval $\mathcal{I}\subset [0,\infty)$ of the form $[0,q]$, $[p,q]$, or $[p,\infty)$, with $0<p<q<\infty$, we consider domains of the form 
\[
\mathcal{M}\times\T = \mathcal{I}\times \T\times\T \subseteq \R^3,
\]
i.e.\ (periodic) cylinders, annular cylinders, or the exterior of cylinders respectively. We hence attempt to solve $(\dyn-\lambda)B=0$ with suitable and physically relevant boundary conditions. 

Note that by construction the velocity field $u$ will always be tangential to the boundary, namely $u\cdot \hat n = 0$ on $\partial (\M\times \T)$, where we recall that $\hat n$ denotes the outward normal to the boundary unit vector. Thus, for the magnetic field we can naturally consider the boundary condition corresponding to \emph{perfectly conducting walls}, namely
\[
B\cdot \hat n = 0, \quad (\nabla\times B)\times \hat n = 0,
\]
we refer to \cites{GerbeauLeBris2007,roberts1967} for further information.
Componentwise, for cylindrical boundaries of the type here considered where the $\theta$ and $z$ variables are periodic, and solutions in modal form \[
B(t,r,\theta,z) = b(r) \e^{\lambda t + i(m\theta+kz)},
\]
this reads,
\[
b_r|_{\partial \cI}=0, \quad \partial_r(rb_\theta) |_{\partial \cI}=0, \quad \partial_r b_z |_{\partial \cI} = 0.
\]
From our analysis it will become clear that we could cover some other mathematically relevant, albeit not so physically interesting scenarios, such as homogeneous Dirichlet boundary conditions $B = 0$ on $\partial\M$. Since they are not standard for the MHD equations, we shall not provide any details however, and focus exclusively on the perfectly conducting walls scenario.

The boundary condition over $b_z$ will be omitted for the moment, and we will only address it later on when we discuss the $z$ component of the equation. We thus consider the the $r$ and $\theta$ components of the operator $\dyn$, i.e.\ corresponding to \eqref{eq:modal-eq-r} and \eqref{eq:modal-eq-theta}, that abusing the notation we will still denote by $\dyn$ for readability. The operator is thus defined on $\mathcal{D}(\dyn) \subset L^\infty(\cI)$, with 
\begin{equation*}
\mathcal{D}(\dyn)=\{B \in L^\infty(\cI) \cap H^2_{\mathrm{loc}}:\dyn B \in L^\infty(\cI), \  B_r|_{\partial \cI}=(rB_\theta)'|_{\partial \cI}=0\}.
\end{equation*}
As before, it is standard to show that $\dyn$ is closed on $\mathcal{D}(\dyn)$. 
The main result of this section will be the following.

\begin{theorem}[Growing modes with boundary conditions]
\label{thm:main result boundaries}
Let $\cI\subset [0,\infty)$ be a domain either of the form $[0,q]$, $[p,q]$, or $[p,\infty)$ with $0<p<q$, and assume \ref{H0}--\ref{H2}. Let $\mathcal{R}$ be the compact set provided by Lemma \ref{lemma:consequences of assumptions}, and let $\mathcal{N} \subset \mathbb{R}\setminus \{0\}$ be compact. Then, for all $\eps>0$ small, uniformly for $r_0 \in \mathcal{R}$, $M \in \mathcal{N}$, the operator $\dyn$ defined in \eqref{eq:modal-eq-r}--\eqref{eq:modal-eq-theta} has an eigenvalue of the form $\lambda=\eps^{1/3}(\mu_\star+o_{\eps \to 0}(1))$ as in Theorem \ref{thm:dynamo}, with an associated finite energy eigenfunction $b\in X_\mathcal{I}$, satisfying the boundary conditions 
\[
b_r|_{\partial \cI}=0, \quad (rb_\theta)'|_{\partial \cI}=0.
\]
Furthermore, assuming \ref{H3} only if $\cI$ is unbounded, the asymptotic expansion of Theorem \ref{thm:dynamo} for the eigenfunction $b$ is valid if $\Re(\mu_\star)>0$.
\end{theorem}

The strategy we follow will be similar to the case of the full space: we construct a resolvent for $\dyn-\lambda$ in a contour enclosing Gilbert's approximate eigenvalue $\lambda_\star = \mu_\star\eps^{1/3}$, and show that the associated Riesz projector has non-empty range. In order to construct the inverse operator, we shall simply employ the inverse we defined for the global equation, and add homogeneous solutions to correct the boundary conditions. Since the previous sections have emphasized uniformity in $r_0\in \mathcal{R}$, $M \in \mathcal{N}$ in great deal already, we shall omit any mention of it here, since our analysis builds entirely on that in Section \ref{s:approx-Green's}.

We give now full details of the proof of Theorem \ref{thm:main result boundaries} for the case of the annular cylinder $\cI=[p,q]$, and we will comment on the other two scenarios---corresponding to a cylinder $\cI=[0,q]$ and the exterior of a cylinder $\cI=[p,\infty)$---later on. Furthermore, throughout this section we shall make the standing assumption that Hypotheses \ref{H0}--\ref{H2} are satisfied, and in particular that the closed set $\mathcal{J}$ from the Hypotheses is given by the interval $[\frac{p}{2},2q]$.

Since there are a total of four different boundary conditions to satisfy in the case where $\mathcal{I}=[p,q]$, we shall construct four homogeneous solutions to $(\dyn-\lambda)B=0$ with $\lambda\in\Co$ being some contour enclosing Gilbert's approximate eigenvalue $\lambda_\star$. From there, it is a matter of inverting a $4\times 4$ matrix in order to fix the boundary conditions for our previously constructed inverse. It thus remains to construct precisely these homogeneous solutions. In order to make sure these are linearly independent, we shall attempt to construct solutions for which either the first or the second component is ``large'' at $p$, and ``small'' at $q$, or vice versa. Indeed, the main technical result of this section reads as follows.

\begin{lemma}
\label{lemma:homogeneous solutions}
Let $[p,q] \subset (0,\infty)$, $r_0 \in (p,q)$, and let $\lambda=\lambda_\star+\eps^{1/3}\eta\in \Co$ as in Proposition \ref{proposition:r0}. Then, for $\eps$ sufficiently small, there exist four homogeneous solutions $v_i=(v_{i,1},v_{i,2})$ with $i=1,2,3,4$, to $(\dyn-\lambda)v_i=0$ on $(p,q)$ such that $\|v_i\|_{L^\infty} \lesssim \eps^{1/3}$ for all $i$, and so that further 
\begin{enumerate}
    \item $|v_{1,1}(p)| \gtrsim \eps^{1/3} $, $|v_{1,1}(q)| \lesssim \eps^{2/3}$, $|\partial_rv_{1,2}(p)|\lesssim 1$, $|\partial_rv_{1,2}(q)|\lesssim \eps^{2/3}$.
    \item $|v_{2,1}(q)| \gtrsim \eps^{1/3}$, $|v_{2,1}(p)| \lesssim \eps^{2/3}$, $|\partial_rv_{2,2}(q)|\lesssim 1$, $|\partial_rv_{2,2}(p)|\lesssim \eps^{2/3}$.
    \item $|v_3(p)|\lesssim \eps^{2/3}$, $|v_3(q)|\lesssim \eps^{2/3}$, $|\partial_rv_{3,2}(p)|\gtrsim 1$, $|\partial_rv_{3,2}(q)|\lesssim \eps^{2/3}$,
     \item $|v_4(q)|\lesssim \eps^{2/3}$, $|v_4(p)|\lesssim \eps^{2/3}$, $|\partial_rv_{4,2}(q)|\gtrsim 1$, $|\partial_rv_{4,2}(p)|\lesssim \eps^{2/3}$.
\end{enumerate}
\end{lemma}

We remark that the first two estimates in \emph{(1)} and \emph{(2)} refer to the first component of $v_1$ and $v_2$ respectively, whereas the first two estimates in \emph{(3)} and \emph{(4)} refer to both components of the vectors $v_3$ and $v_4$.
Since the arguments for both endpoints will be essentially identical, we shall only prove the result for the left endpoint $p>0$.
Let us outline the main ideas of this proof, taking $v_1$ as an example. For this case, we shall set 
\[
v_1=(\dyn-\lambda)^{-1}\begin{pmatrix}
\mathds{1}_{[p-\eps^\delta,p]}\\
0
\end{pmatrix},
\]
where $(\dyn-\lambda)^{-1}$ is the inverse operator constructed on $L^\infty(\mathcal{J},\max\{1,r^2\}w_\eps(r-r_0))$.---In other words, it is the inverse constructed throughout section \ref{s:approx-Green's}, but restricted to functions supported on $\mathcal{J}$. In particular, this inverse still exists for functions $\Omega, U$ that are highly degenerate at infinity, since we only care about the local behaviour in the compact interval $\mathcal{J}$.
By the construction of $v_1$, it must hold that $(\dyn-\lambda) v_1=0$ on $(p,q)$. However, simply using the brute force iterative bounds from the construction of $(\dyn-\lambda)^{-1}$ in the previous sections will not be enough the obtain the more subtle quantitative bounds required for Lemma \ref{lemma:homogeneous solutions}---indeed, the most we could say would be that globally, $\|v_1\|_{X} \lesssim \eps^{-1/3}$, which is nowhere near sufficient for our purposes. Instead, we shall undertake a more subtle analysis, which carefully studies the dynamics of the iterative procedure defining $(\dyn-\lambda)^{-1}$. The key insight here is that each ``linear interpolation'' step as in Sections \ref{sec:tow-inf} and \ref{s:tow-zero} propagates the mass of the input forwards by at most $\eps^\delta$, subject to an exponentially small correction. Therefore, after $N$ iterations, the majority of the mass of 
\[
\sum_{n = 0}^N\G^\lambda ((\dyn-\lambda)\G^\lambda -\Id)^n \begin{pmatrix}
\mathds{1}_{[p-\eps^\delta,p]}\\0\end{pmatrix}
\]
will be contained in $[p-(N+2)\eps^\delta, p+(N+1)\eps^{\delta}]$---the additional $\eps^\delta$ term comes from the fact that the action of $\G^\lambda$ can itself be shown to propagate mass by $\eps^\delta$. Since simultaneously each iterative term yields progressive geometric decay, we can conclude that an initial datum concentrated near $p$ will have an inverse with very little mass at $q$ (and vice versa), provided $\eps$ is small enough. This intuition is the essential driving force behind our proof of Lemma \ref{lemma:homogeneous solutions}.

Before proceeding further with the proof, we must however acknowledge that we have only constructed an inverse for the operator $\L-\lambda$ so far, which arises from $\dyn-\lambda$ via a change of variables. More precisely, from Section \ref{s:equations} we see that 
\begin{equation*}
\L V = Q^{-1} \dyn Q V,
\end{equation*}
where $Q$ denotes the invertible matrix
\[
Q=\begin{pmatrix}
-\alpha \eps^{1/3} & \alpha \eps^{1/3}\\
1 & 1
\end{pmatrix}.
\]
Therefore, we see that 
\begin{equation*}
(\dyn -\lambda)^{-1}=Q (\L-\lambda)^{-1} Q^{-1}.
\end{equation*}

Furthermore, since our homogeneous solutions need only be local---and we focus only in the cylindrical shell defined by the radial interval $\cI=[p,q]$---we may define $(\L-\lambda)^{-1}$ using the operator 
\begin{equation*}
\G^\lambda f=\G_1 \left(f\dsOne_{[\frac{1}{2}p,r_0-\eps^\gamma)}\right) +  \G_2^\lambda \left(f\dsOne_{[r_0-\eps^\gamma,r_0+\eps^\gamma)}\right) +  \G_3 \left(f\dsOne_{[r_0+\eps^\gamma,2q]}\right),
\end{equation*}
i.e.\ we discard the components near zero and at infinity. Notice that now the operator $\G^\lambda$ depends on the choice of $p$ and $q$, nonetheless we retain the same notation as before to avoid introducing additional cumbersome symbols. Furthermore, we can write 
\begin{equation*}
(\L-\lambda)^{-1}=\G^\lambda \sum_{n \geq 0}(-1)^n Q^{-1}(Q ((\L-\lambda)\G^\lambda -\Id)Q^{-1})^n Q,
\end{equation*}
and so 
\begin{equation*}
Q(\L-\lambda)^{-1}Q^{-1}=\G^\lambda_{\mathrm{dyn}}\sum_{n \geq 0} (-1)^n ((\dyn -\lambda)\G^\lambda_{\mathrm{dyn}}-\Id)^n,
\end{equation*}
where 
\begin{equation*}
\G^\lambda_{\mathrm{dyn}}f =Q \G^\lambda Q^{-1}f=\G_1\left (f\mathds{1}_{[\frac{1}{2}p,r_0-\eps^\gamma)}\right )+Q \G_2^\lambda (Q^{-1}f\mathds{1}_{[r_0-\eps^\gamma,r_0+\eps^\gamma)})+\G_3\left (f\mathds{1}_{[r_0+\eps^\gamma,2q]}\right ).
\end{equation*}
For this last expression we have used that $\G_1, \G_3$ are diagonal operators with identical entries on the diagonal. The error bounds transfer immediately from Section \ref{s:approx-Green's} into the weighted space 
\[
L^\infty\left(\left[\frac{p}{2},2q\right], Q^{-1}\dd r\right).
\]
Namely, using Proposition \ref{proposition:approx-Green's}, we can write for some suitable $a>0$,
\begin{align}\label{eq:weighted-bounds}
\begin{split}
\|(\dyn -\lambda)\G^\lambda_{\mathrm{dyn}}-\Id)f\|_{L^\infty([\frac{1}{2}p,2q],Q^{-1}\dd r)} & =\|Q^{-1}((\dyn -\lambda)\G^\lambda_{\mathrm{dyn}}-\Id)f)\|_{L^\infty([\frac{1}{2}p,2q])}\\
& = \|((\L-\lambda)\G^\lambda -\Id)Q^{-1}f\|_{L^\infty[(\frac{1}{2}p,2q])} \\
& \lesssim \eps^a \|f\|_{L^\infty([\frac{1}{2}p,2q],Q^{-1}\dd r)}.
\end{split}
\end{align}
However, the unweighted $L^\infty$ bounds for the error are slightly different. Indeed, we collect them first for the Green's operator $\G_1$---that we recall it is defined in Section \ref{s:tow-zero}---in the following lemma.

\begin{lemma}[Unweighted bounds]
\label{lemma:unweighted bounds}
Fix $k_1,k_2>0$ such that $r_0-k_1\geq k_2$. Let $c_0>0$ be as defined in Section \ref{s:tow-zero}. If $f$ is such that $\supp(f)\subset [\frac{1}{2}p,k_1]$, then the following bounds hold true.
\begin{enumerate}
\item Let $f=(f_1,0)$, then the pointwise error bound by component is given by
\begin{align*}
|((\dyn-\lambda)\G_1 f-f)_1|(r) & \lesssim \eps^\beta \|f\|_{L^\infty}\left ( \mathds{1}_{[\frac{1}{2}p, k_1+2\eps^\delta]} +\e^{-c_0\eps^{-2/3}|r-(k_1+\eps^\delta)|}\right),\\
|((\dyn-\lambda)\G_1 f-f)_2|(r) & \lesssim \eps^{1/3} \|f\|_{L^\infty}\left ( \mathds{1}_{[\frac{1}{2}p, k_1+2\eps^\delta]} +\e^{-c_0\eps^{-2/3}|r-(k_1+\eps^\delta)|}\right),
\end{align*}
where $\beta$ is defined in Property \ref{P3}.
The second component of the Green's function satisfies $|(\G_1 f)_2|(r)=0$, whereas for the first component we obtain, 
\begin{equation*}
\begin{split}
|(\G_1 f)_1|(r)& \lesssim \eps^{1/3}\|f\|_{L^\infty}\left ( \mathds{1}_{[\frac{1}{2}p, k_1+2\eps^\delta]} +\e^{-c_0\eps^{-2/3}|r-(k_1+\eps^\delta)|}\right), \\
|\partial_r(\G_1 f)_1|(r) & \lesssim \eps^{-1/3}\|f\|_{L^\infty}\left ( \mathds{1}_{[\frac{1}{2}p, k_1+2\eps^\delta]} +\e^{-c_0\eps^{-2/3}|r-(k_1+\eps^\delta)|}\right).
\end{split}
\end{equation*}
\item If instead $f=(0,f_2)$, the pointwise error bound reads
\begin{align*}
|((\dyn-\lambda)\G_1 f-f)_1|(r) & \lesssim \eps \|f\|_{L^\infty}\left ( \mathds{1}_{[\frac{1}{2}p, k_1+2\eps^\delta]} +\e^{-c_0\eps^{-2/3}|r-(k_1+\eps^\delta)|}\right),\\
|((\dyn-\lambda)\G_1 f-f)_2|(r) & \lesssim \eps^{\beta} \|f\|_{L^\infty}\left ( \mathds{1}_{[\frac{1}{2}p, k_1+2\eps^\delta]} +\e^{-c_0\eps^{-2/3}|r-(k_1+\eps^\delta)|}\right),
\end{align*}
where again $\beta$ comes from Property \ref{P3}. The first component of the Green's function satisfies $|(\G_1 f)_1|(r) = 0$, whereas for the second component we obtain 
\begin{equation*}
\begin{split}
|(\G_1 f)_2|(r) & \lesssim \eps^{1/3}\|f\|_{L^\infty}\left ( \mathds{1}_{[\frac{1}{2}p, k_1+2\eps^\delta]} +\e^{-c_0\eps^{-2/3}|r-(k_1+\eps^\delta)|}\right), \\
|\partial_r(\G_1 f)_2|(r) & \lesssim \eps^{-1/3}\|f\|_{L^\infty}\left ( \mathds{1}_{[\frac{1}{2}p, k_1+2\eps^\delta]} +\e^{-c_0\eps^{-2/3}|r-(k_1+\eps^\delta)|}\right).
\end{split}
\end{equation*}
\end{enumerate}
\end{lemma}

These bounds are straightforward consequences of our analysis in Section \ref{s:approx-Green's}, taking into account that $\supp(f)\subset [\frac{1}{2}p,k_1]$, $k_1<r_0$, where the distance between $k_1$ and $r_0$ is independent of $\eps$.

At this point we want to introduce a partition of the interval $[\frac{1}{2}p,r_0-\eps^\gamma]$, but notice that we can always do it so that $p$ is a boundary point of the partition. Indeed, we set 
\[
r_j=\frac{p}{2}+(j-1)\eps^\delta, \quad j\geq 1,
\]
until $j_0\in\N$ such that $p-2\eps^\delta \leq \frac{1}{2}p + j_0\eps^\delta < p-\eps^\delta$. Then simply set $r_{j_0+1}=p-\eps^\delta$, and inductively $r_{j+1}=r_j+\eps^\delta$ from there. This way we create a partition $\bigcup_j [r_j,r_{j+1})$, which always contains the interval $[p-\eps^\delta,p)$. With this partition defined, we now define 
\[
\mathcal{P}_{\leq J}:=\bigcup_{j=1}^J [r_j,r_{j+1}),
\]
and $\mathcal{P}_{>J}:=[\frac{1}{2}p,2q]\setminus \mathcal{P}_{\leq J}$. With these preliminaries out of the way, we now state the lemma which will be the main workhorse of this section. 

\begin{lemma}
\label{lemma:bounded domains error 1}
Let $f \in L^\infty([\frac{1}{2}p,2q])$ be of the form $f=(f_1, f_2)$, and fix $k_2>0$ and $J$ such that $r_0-r_{J+1}\geq k_2$. Then, there hold the estimates 
\begin{equation*}
\begin{split}
& \|(((\dyn-\lambda)\G^\lambda_{\mathrm{dyn}}-\Id )f)_1\|_{L^\infty(\mathcal{P}_{\leq J+2})} \leq \eps^\beta \|f_1\|_{L^\infty(\mathcal{P}_{\leq J})}+\eps^{1/3}\|f_2\|_{L^\infty(\mathcal{P}_{\leq J})}+\eps^{a}\|Q^{-1}f\|_{L^\infty(\mathcal{P}_{>J})},\\
& \|(((\dyn-\lambda)\G^\lambda_{\mathrm{dyn}}-\Id )f)_2\|_{L^\infty(\mathcal{P}_{\leq J+2})} \leq \eps^\beta \|f_2\|_{L^\infty(\mathcal{P}_{\leq J})}+\eps\|f_1\|_{L^\infty(\mathcal{P}_{\leq J})}+\eps^{a}\|Q^{-1}f\|_{L^\infty(\mathcal{P}_{>J})},\\
& \|Q^{-1}((\dyn-\lambda)\G^\lambda_{\mathrm{dyn}}-\Id )f\|_{L^\infty(\mathcal{P}_{> J+2})} \leq \e^{-\frac{1}{2}c_0\eps^{-2/3+\delta}}\|f\|_{L^\infty(\mathcal{P}_{\leq J})}+\eps^{a}\|Q^{-1}f\|_{L^\infty(\mathcal{P}_{>J})}.
\end{split}
\end{equation*}
\end{lemma}

\begin{proof}
Writing $f=f\mathds{1}_{\mathcal{P}_{\leq J}}+f\mathds{1}_{\mathcal{P}_{>J}}$, we have 
\begin{equation*}
|((\dyn-\lambda)\G^\lambda_{\mathrm{dyn}}-\Id)f(r)|\leq |((\dyn-\lambda)\G_1-\Id)\mathds{1}_{\mathcal{P}_{\leq J}}f(r)|+|((\dyn-\lambda)\G^\lambda_{\mathrm{dyn}}-\Id)\mathds{1}_{\mathcal{P}_{>J}}f(r)|.
\end{equation*}
For the second addend in the right hand side we can directly deduce from \eqref{eq:weighted-bounds},
\begin{align*}
|((\dyn-\lambda)\G^\lambda_{\mathrm{dyn}}-\Id)\mathds{1}_{\mathcal{P}_{>J}}f(r)| \leq \|Q\| |Q^{-1}((\dyn-\lambda)\G^\lambda_{\mathrm{dyn}}-\Id)\mathds{1}_{\mathcal{P}_{>J}}f(r)| \lesssim \eps^a \|Q^{-1} f\|_{L^\infty(\mathcal{P}_{>J})}.
\end{align*}
For the first addend we use Lemma \ref{lemma:unweighted bounds}, so that all in all we can write for any $r \leq r_J+3\eps^\delta$,
\begin{equation*}
\begin{split}
&|(((\dyn-\lambda)\G^\lambda_{\mathrm{dyn}}-\Id)f)_1(r)| \lesssim \eps^\beta \|f_1\|_{L^\infty(\mathcal{P}_{\leq J})}+\eps^{1/3}\|f_2\|_{L^\infty(\mathcal{P}_{\leq J})}+\eps^{a}\|Q^{-1}f\|_{L^\infty(\mathcal{P}_{>J})},\\
&|(((\dyn-\lambda)\G^\lambda_{\mathrm{dyn}}-\Id)f)_2(r)| \lesssim \eps^\beta \|f_2\|_{L^\infty(\mathcal{P}_{\leq J})}+\eps\|f_1\|_{L^\infty(\mathcal{P}_{\leq J})}+\eps^{a}\|Q^{-1}f\|_{L^\infty(\mathcal{P}_{>J})}.
\end{split}
\end{equation*}
Similarly, for $r>r_j+3\eps^\delta$, again using Lemma \ref{lemma:unweighted bounds} we can write
\begin{align*}
|Q^{-1}((\dyn-\lambda)\G^\lambda_{\mathrm{dyn}}-\Id)f(r)| & \lesssim \|Q^{-1}\|\left(\eps^\beta + \eps^{1/3} \right) \e^{-c_0\eps^{-2/3+\delta}} \|f\|_{L^\infty(\mathcal{P}_{\leq J})}+\eps^{a}\|Q^{-1}f\|_{L^\infty(\mathcal{P}_{>J})}\\
&\lesssim \left( \eps^{\beta-1/3} + 1 \right) \e^{-c_0\eps^{-2/3+\delta}} \|f\|_{L^\infty(\mathcal{P}_{\leq J})}+\eps^{a}\|Q^{-1}f\|_{L^\infty(\mathcal{P}_{>J})},
\end{align*}
where we used that $\|Q^{-1}\| \lesssim \eps^{-1/3}$.
From here the proof follows provided that $\delta<2/3$, which is ensured by Lemma \ref{lemma:consequences of assumptions}.
\end{proof}

Before proceeding with the proof of Lemma \ref{lemma:homogeneous solutions} we first need to prove some technical results that will be required in the coming pages.

\begin{lemma}
\label{lemma:sequence lemma}
Let $\sigma>0$, $a\in (0,1/3)$, and consider the iteratively defined sequence 
\begin{equation*}
\begin{split}
&a_{n+1} \leq \eps^a(a_n+c_n)+\eps b_n,\\
&b_{n+1} \leq \eps^a ( b_n+c_n)+\eps^{1/3}a_n,\\
&c_{n+1} \leq \eps^a (c_n+\sigma(a_n+b_n)),
\end{split}
\end{equation*}
with initial configuration $a_0,b_0>0$, $c_0=0$.
Then there holds for all $n \geq 0$
\begin{equation*}
\begin{split}
&a_n \leq 2(\eps^{an/2} a_0+\eps^{1/3+a(n-1)/2}b_0 +\sqrt{\sigma}\eps^{an/2}(a_0+b_0)), \\
&b_n \leq 2(\eps^{an/2} b_0+\eps^{1/3+a(n-1)/2}a_0 +\sqrt{\sigma}\eps^{an/2}(a_0+b_0)), \\
&c_n \leq 2\sqrt{\sigma}\eps^{an/2}(a_0+b_0).
\end{split}
\end{equation*}
\end{lemma}

\begin{proof}
Since all constants are non-negative, the sequence is pointwise bounded above by the same sequence, but with the inequalities replaced by equalities, and further replacing the $\eps b_n$ contribution by $\eps^{1/3}b_n$ i.e. 
\begin{equation*}
\begin{split}
&a_{n+1} = \eps^a(a_n+c_n)+\eps^{1/3}b_n\\
&b_{n+1} = \eps^a ( b_n+c_n)+\eps^{1/3}a_n\\
&c_{n+1} = \eps^a (c_n+\sigma (a_n+b_n)).
\end{split}
\end{equation*}
Whilst this system may be solved via a matrix diagonalisation, the computation becomes extremely messy. Thus, we instead employ the following approach, which yields slightly less sharp bounds, but is sufficient for our purposes.
We first set $d_n=a_n+b_n$, and so obtain a reduced system 
\begin{equation*}
\begin{split}
&d_{n+1} \leq 2 \eps^a (d_n+c_n)\\
&c_{n+1} \leq 2 \eps^a (c_n+\sigma d_n),
\end{split}
\end{equation*}
which we may once again upper bound by replacing the inequalities by equalities.
Via a matrix diagonalization, we see that 
\begin{equation*}
\begin{pmatrix}
d_n\\
c_n
\end{pmatrix}
=\frac{1}{2}(2\eps^{a})^n\begin{pmatrix}
((1-\sqrt{\sigma})^n+(1+\sqrt{\sigma})^n)d_0+\frac{1}{\sqrt{\sigma}}(-(1-\sqrt{\sigma})^n+(1+\sqrt{\sigma})^n)c_0\\
\sqrt{\sigma}(-(1-\sqrt{\sigma})^n+(1+\sqrt{\sigma})^n)d_0+((1-\sqrt{\sigma})^n+(1+\sqrt{\sigma})^n)c_0
\end{pmatrix},
\end{equation*}
and so, since by assumption $c_0=0$,
\[
d_n \leq (2\eps^a (1+\sqrt{\sigma}))^n d_0, \quad c_n \leq \sqrt{\sigma} (2\eps^a (1+\sqrt{\sigma}))^n d_0.
\]
In particular, we can now use these bounds to study the equations for $a_n, b_n,$ which now, via the trivial estimate $(2\eps^a(1+\sqrt{\sigma}))^n \leq \eps^{an/2}$ for $\eps$ sufficiently small, may be bounded by
\begin{equation*}
a_{n+1}=\eps^{a}a_n+\eps^{1/3}b_n+\sqrt{\sigma}\eps^{an/2+a}d_0,\quad b_{n+1}=\eps^{a}b_n+\eps^{1/3}a_n+\sqrt{\sigma}\eps^{an/2+a}d_0.
\end{equation*}
Solving this difference equation is a rather arduous task, but still considerably easier than solving the full $3 \times 3$ system directly. Defining $v_n$ such that
\[
\begin{pmatrix}
    a_n\\
    b_n
\end{pmatrix} = Pv_n = \begin{pmatrix}
-1 & 1\\
1 & 1
\end{pmatrix}\begin{pmatrix}
    v_n^1\\
    v_n^2
\end{pmatrix}
\]
we find that $v_n$ satisfies the difference equation 
\begin{equation*}
v_{n+1}=\eps^{a} \begin{pmatrix}
1-\eps^{1/3-a} & 0\\
0 & 1+\eps^{1/3-a}
\end{pmatrix}v_n+ \eps^{an/2+a}\sqrt{\sigma}d_0 \begin{pmatrix}
0\\
1
\end{pmatrix}.
\end{equation*}
We thus directly observe that $v_n^1=\eps^{na}(1-\eps^{1/3-a})^n v_0^1$.
The computation for $v_n^2$ is more challenging, and can be done using the $z$-transform, a discrete time analogue of the Laplace transform, see \cite{Fleisch2022}*{Chapter 5}. Indeed, multiplying the $n^{\mathrm{th}}$ term in the difference equation by $z^{-n}$ and summing over $n$, we obtain
\begin{equation*}
zY(z)-zv_0^2=\eps^a (1+\eps^{1/3-a})Y(z)+\sqrt{\sigma} \eps^{a} \sum_{n \geq 0}\eps^{an/2}z^{-n}=\eps^a (1+\eps^{1/3-a})Y(z)+\sqrt{\sigma}d_0 \eps^{a} \frac{z}{z-\eps^{a/2}},
\end{equation*}
where we have set
\[
Y(z)=\sum_{n \geq 0}v^2_n z^{-n}.
\]
We thus solve this algebraic equation, and via a partial fractions computation, we thus have 
\begin{equation*}
Y(z)=\frac{z}{z-\eps^a (1+\eps^{1/3-a})}v_0^2+\frac{\sqrt{\sigma} d_0\eps^a}{\eps^{a/2}(1-\eps^{a/2}(1+\eps^{1/3-a}))}\left ( \frac{z}{z-\eps^{a/2}}-\frac{z}{z-\eps^{a}(1+\eps^{1/3-a})}\right ).
\end{equation*}
We recognise the right hand side as the $z$-transform of 
\begin{equation*}
v_0^2 (\eps^a (1+\eps^{1/3-a}))^n +\frac{\sqrt{\sigma}d_0 \eps^a}{\eps^{a/2}(1-\eps^{a/2}(1+\eps^{1/3-a}))}\left ( \eps^{an/2}-(\eps^a (1+\eps^{1/3-a}))^n\right ),
\end{equation*}
and so undoing the change of variables yields
\begin{equation}
\label{eq:a_n b_n expression}
\begin{split}
\begin{pmatrix}
a_n\\
b_n
\end{pmatrix}
& = \frac{1}{2}\eps^{na}\begin{pmatrix}
(a_0+b_0)(1+\eps^{1/3-a})^n-(b_0-a_0)(1-\eps^{1/3-a})^n \\
(a_0+b_0)(1+\eps^{1/3-a})^n+(b_0-a_0)(1-\eps^{1/3-a})^n
\end{pmatrix}\\
& \quad + \frac{\sqrt{\sigma} d_0\eps^a}{\eps^{a/2}(1-\eps^{a/2}(1+\eps^{1/3-a}))}\left ( \eps^{an/2}-(\eps^a (1+\eps^{1/3-a}))^n\right )\begin{pmatrix}
1\\
1
\end{pmatrix}.
\end{split}
\end{equation}
We see that we readily obtain from here the claim of the lemma using on the one hand the straightforward estimates
\[
\eps^{na}\left ((1+\eps^{1/3-a})^n+(1-\eps^{1/3-a})^n \right ) \leq \eps^{an/2}, 
\]
\[
\frac{\eps^a}{\eps^{a/2}(1-\eps^{a/2}(1+\eps^{1/3-a}))}\left ( \eps^{an/2}-(\eps^a (1+\eps^{1/3-a}))^n\right ) \leq \eps^{an/2},
\]
and on the other hand we claim
\begin{equation*}
\eps^{na}\left ((1+\eps^{1/3-a})^n-(1-\eps^{1/3-a})^n \right ) \leq 4\eps^{1/3+a(n-1)/2},
\end{equation*}
for all $n$ and $\eps$ small enough. Indeed, from the binomial theorem, it follows that 
\begin{equation*}
\left ((1+\eps^{1/3-a})^n-(1-\eps^{1/3-a})^n \right ) \leq 2((1+\eps^{1/3-a})^n-1).
\end{equation*}
In turn, we note the expression valid for any $x \neq 1$, 
$$
\frac{x^n-1}{x-1}=\sum_{k=0}^{n-1}x^k,
$$
and so 
$$
((1+\eps^{1/3-a})^n-1) \leq \eps^{1/3-a} \sum_{k=0}^{n-1}(1+\eps^{1/3-a})^k.
$$
Therefore, 
\begin{equation*}
\begin{split}
\eps^{na}((1+\eps^{1/3-a})^n-1) & \leq \eps^{1/3-a+na}+\eps^{1/3-a} \eps^{na/2}\sum_{k=1}^\infty  (\eps^{a/2}+\eps^{1/3-a/2})^k\\
& \leq \eps^{1/3-a+na}+\eps^{1/3-a+na/2}\frac{\eps^{a/2}+\eps^{1/3-a/2}}{1-(\eps^{a/2}+\eps^{1/3-a/2})} \leq 2 \eps^{1/3+a(n-1)/2},
\end{split}
\end{equation*}
as soon as $1/3-a/2>a/2$, i.e.\ as soon as $a<1/3$, and $\eps>0$ is small enough.
Thus, plugging this into \eqref{eq:a_n b_n expression}, the proof is complete. 
\end{proof}

Combining the previous lemmas, we arrive at the following result. 

\begin{lemma}
\label{lemma:propagation of mass}
Fix $k_2,N>0$ and suppose that $f=(f_1,0)$ is supported in $\mathcal{P}_{\leq J}$, where 
\[
r_J+(2N+1)\eps^\delta \leq r_0-k_2.
\]
Then, for all $n \leq N$ there holds
\begin{equation*}
\begin{split}
\|(((\dyn-\lambda)\G^\lambda_{\mathrm{dyn}}-\Id)^nf)_1\|_{L^\infty(\mathcal{P}_{\leq J+2n})} & \lesssim \eps^{an/2}\|f\|_{L^\infty(\mathcal{P}_{\leq J})}, \\
\|(((\dyn-\lambda)\G^\lambda_{\mathrm{dyn}}-\Id)^nf)_2\|_{L^\infty(\mathcal{P}_{ \leq J+2n})} & \lesssim \eps^{1/3+a(n-1)/2}\|f\|_{L^\infty(\mathcal{P}_{\leq J})}, \\
\|((\dyn-\lambda)\G^\lambda_{\mathrm{dyn}}-\Id)^nf\|_{L^\infty(\mathcal{P}_{>J+2n},Q^{-1} \dd r)} & \lesssim \eps^{\frac{1}{2}(\delta-1/3+na)}\e^{-\frac{1}{2}c_0\eps^{-2/3+\delta}}\|f\|_{L^\infty(\mathcal{P}_{\leq J})}.
\end{split}
\end{equation*}
If instead $f=(0,f_2)$, then 
\begin{equation*}
\begin{split}
\|(((\dyn-\lambda)\G^\lambda_{\mathrm{dyn}}-\Id)^nf)_1\|_{L^\infty(\mathcal{P}_{\leq J+2n})} & \lesssim \eps^{1/3+a(n-1)/2}\|f\|_{L^\infty(\mathcal{P}_{\leq J})}, \\
\|(((\dyn-\lambda)\G^\lambda_{\mathrm{dyn}}-\Id)^nf)_2\|_{L^\infty(\mathcal{P}_{ \leq J+2n})} & \lesssim \eps^{an/2}\|f\|_{L^\infty(\mathcal{P}_{\leq J})}, \\
\|((\dyn-\lambda)\G^\lambda_{\mathrm{dyn}}-\Id)^nf\|_{L^\infty(\mathcal{P}_{>J+2n},Q^{-1} \dd r)} & \lesssim \eps^{\frac{1}{2}(\delta-1/3+na)}\e^{-\frac{1}{2}c_0\eps^{-2/3+\delta}}\|f\|_{L^\infty(\mathcal{P}_{\leq J})}.
\end{split}
\end{equation*}
\end{lemma}
\begin{proof}
We set
\begin{align*}
    a_n & = \|(((\dyn-\lambda)\G^\lambda_{\mathrm{dyn}}-\Id)^nf)_1\|_{L^\infty(\mathcal{P}_{\leq J+2n})}, \\
    b_n & = \|(((\dyn-\lambda)\G^\lambda_{\mathrm{dyn}}-\Id)^nf)_2\|_{L^\infty(\mathcal{P}_{ \leq J+2n})}, \\
    c_n & = \|((\dyn-\lambda)\G^\lambda_{\mathrm{dyn}}-\Id)^nf\|_{L^\infty(\mathcal{P}_{>J+2n},Q^{-1} \dd r)}.
\end{align*}
Then, using the bounds from Lemma \ref{lemma:bounded domains error 1} and setting $\sigma=\eps^{\delta-1/3}\e^{-c_0\eps^{-2/3+\delta}}$, we may directly apply Lemma \ref{lemma:sequence lemma} and deduce the result.
\end{proof}

With this result in hand, we may now commence the proof of Lemma \ref{lemma:homogeneous solutions}.
\begin{proof}[Proof of Lemma \ref{lemma:homogeneous solutions}] The proof will be divided into multiple steps. 

\medskip

\emph{Step 1.} We begin by constructing $v_1$. Take $f=(f_1,f_2)$ with $f_1(r)=\mathds{1}_{[p-\eps^\delta,p]}(r)$, $f_2=0$, and set $v_1(r)=(\dyn-\lambda)^{-1}f$. In particular, it holds that $(\dyn-\lambda) v_1=0$ in $(p,q)$, and so we have a homogeneous solution. It remains to prove the claimed bounds at the endpoints. We begin by constructing the first iterate by hand. Indeed, we have via a direct computation that 
\begin{equation*}
f^0(r) :=\G^\lambda_{\mathrm{dyn}}f=\begin{pmatrix}
\frac{\eps^{1/3}}{2i T(p-\eps^\delta)}\e^{-\sqrt{iT(p-\eps^\delta)} \eps^{-2/3} r} \left(\e^{-\sqrt{iT(p-\eps^\delta)} \eps^{-2/3} p}-\e^{-\sqrt{iT(p-\eps^\delta)} \eps^{-2/3} (p-\eps^\delta)}\right )\\
0
\end{pmatrix},
\end{equation*}
for $r \geq p$. In particular, at $r=p$ we have 
\[
|f^1_0(r)|=\left|\frac{\eps^{1/3}}{2|T(p-\eps^\delta)|}(1-\e^{-\sqrt{iT(p-\eps^\delta)}\eps^{-2/3+\delta}})\right| \gtrsim \eps^{1/3}.
\]
Fix now $N$ large enough so that $Na>4/3$, and $a+3N\eps^\delta<r_0-|r_0-p|/2$. Then, we write 
\begin{equation*}
v_1(r)=f^0(r)+\sum_{n=1}^N(-1)^n \G^\lambda_{\mathrm{dyn}} ((\dyn-\lambda)\G^\lambda_{\mathrm{dyn}}-\Id)^n f+ \sum_{n \geq N+1}^\infty (-1)^n \G^\lambda_{\mathrm{dyn}} ((\dyn-\lambda)\G^\lambda_{\mathrm{dyn}}-\Id)^n f.
\end{equation*}
By our error bounds, there holds 
\begin{align*}
&\left\|\sum_{n \geq N+1}^\infty (-1)^n \G^\lambda_{\mathrm{dyn}} ((\dyn-\lambda)\G^\lambda_{\mathrm{dyn}}-\Id)^n f\right\|_{L^\infty([\frac{p}{2},2q])}\\
&\leq \|Q \|\left\|\sum_{n \geq N+1}^\infty (-1)^n \G^\lambda_{\mathrm{dyn}} ((\dyn-\lambda)\G^\lambda_{\mathrm{dyn}}-\Id)^n f\right\|_{L^\infty([\frac{p}{2},2q],Q^{-1} \dd r)}\\
&\lesssim \eps^{-1/3} \sum_{n \geq N+1}\eps^{n a}\|f\|_{L^\infty([\frac{p}{2},2q],Q^{-1} \dd r)} \lesssim \eps^{-1/3+(N+1)a}\|Q^{-1}f\|_{L^\infty([\frac{p}{2},2q])}\lesssim \eps^{-2/3+(N+1)a}.
\end{align*}
Next, we note that we can split the error produced at the $\mathrm{n}^{th}$ iteration into two parts as follows:
\begin{align*}
((\dyn-\lambda)\G^\lambda_{\mathrm{dyn}}-\Id)^n f & = \left (((\dyn-\lambda)\G^\lambda_{\mathrm{dyn}}-\Id)^n f \right )\mathds{1}_{r \leq p+2n \eps^{\delta}}\\
& \quad +\left (((\dyn-\lambda)\G^\lambda_{\mathrm{dyn}}-\Id)^n f \right )\mathds{1}_{r > p+2n \eps^{\delta}} =:f_\leq^n+f_>^n.
\end{align*}
By Lemma \ref{lemma:propagation of mass}, we observe that 
\begin{equation*}
\|f_\leq^n\|_{L^\infty([\frac{p}{2},2q])} \lesssim \eps^{\frac{na}{2}}, \quad \|f_>^n\|_{L^\infty([\frac{p}{2},2q], Q^{-1} \dd r)} \lesssim \eps^{\frac{an}{2}}\e^{-\frac{c_0}{4}\eps^{-2/3+\delta}},
\end{equation*}
so long as $\eps$ is small enough.
Furthermore, $f_\leq^n$ is supported in $r<r_0-|r_0-p|/2$, meaning $\G^\lambda_{\mathrm{dyn}}f_\leq^n=\G_1 f_\leq^n$, and the bounds from Lemma \ref{lemma:unweighted bounds} are applicable. We begin by estimating
\begin{align*}
&|\G^\lambda_{\mathrm{dyn}} ((\dyn-\lambda)\G^\lambda_{\mathrm{dyn}}-\Id)^n f|(r) \leq |\G_1 f_\leq^n|(r)+\|Q\|\|\G_{\mathrm{dyn}}^\lambda f_>^n\|_{L^\infty([\frac{p}{2},2q],Q^{-1} \dd r)}.
\end{align*}
In particular, for $r \leq p+3n \eps^\delta$ we deduce from Lemma \ref{lemma:propagation of mass} the bound
\begin{equation*}
|\G^\lambda_{\mathrm{dyn}}f_\leq^n|(r) \leq \eps^{1/3+\frac{na}{2}}+\eps^{-1/3+\frac{na}{2}}\e^{-\frac{c_0} {4}\eps^{-2/3+\delta}}\lesssim \eps^{1/3+\frac{na}{2}},
\end{equation*}
whereas for $r >p+3n \eps^\delta$ it holds 
\begin{equation*}
|\G^\lambda_{\mathrm{dyn}} f_2^n|\leq \eps^{1/3+\frac{na}{2}} \e^{-c_0\eps^{-2/3+\delta}}+\eps^{-1/3+\frac{na}{2}}\e^{-\frac{c_0} {4}\eps^{-2/3+\delta}}\lesssim \eps^{\frac{na}{2}}\e^{-\frac{c_0}{8}\eps^{-2/3+\delta}}.
\end{equation*}
Therefore, we deduce
\begin{align*}
|v_{1,1}(p)| \gtrsim \eps^{1/3}-\sum_{n=1}^\infty \eps^{1/3+\frac{na}{2}}-\eps^{-2/3+(N+1)a} \gtrsim \eps^{1/3},
\end{align*}
and similarly
\begin{equation*}
|v_{1,1}(q)|\lesssim \e^{-\frac{c_0}{4}\eps^{-2/3+\delta}}\sum_{n=0}^\infty \eps^{\frac{an}{2}} +\eps^{-2/3+(N+1)a} \lesssim \eps^{2/3}.
\end{equation*}
It thus remains to verify the derivative bounds. Recalling that we split 
$$
f_\leq^n=\left (((\dyn-\lambda)\G^\lambda_{\mathrm{dyn}}-\Id)^n f \right )\mathds{1}_{r \leq p+2n \eps^{\delta}}, \quad f_>^n=\left (((\dyn-\lambda)\G^\lambda_{\mathrm{dyn}}-\Id)^n f \right )\mathds{1}_{r > p+2n \eps^{\delta}},
$$
we now denote the components of $f_\leq^n$ by $f_{\leq,1}^{n}, f_{\leq,2}^{n}$, and analogously for $f_<^n$. By Lemma \ref{lemma:propagation of mass}, recalling that $f_{\leq,2}^{0}=0$, we get the more precise bounds for $f_{\leq,1}^{n}, f_{\leq,2}^{n}$
\begin{equation*}
\|f_{\leq,1}^{n}\|_{L^\infty([\frac{p}{2},2q])} \lesssim \eps^{\frac{na}{2}}, \quad \|f_{\leq,2}^{n}\|_{L^\infty([\frac{p}{2},2q])} \lesssim \eps^{1/3+(n-1)\frac{a}{2}}.
\end{equation*}
Therefore, we may bound the derivative of the second component for $r<a+3n \eps^\delta$ by 
\begin{align*}
|\partial_r (\G^\lambda_{\mathrm{dyn}}((\dyn-\lambda)\G^\lambda_{\mathrm{dyn}}-\Id)^nf)^2|(r) & \leq |\partial_r \G_1 f_{\leq,2}^{n}|(r) +\|Q\|\|\partial_r \G_{\mathrm{dyn}}^\lambda f_{>}^{n}\|_{L^\infty([\frac{p}{2},2q],Q^{-1} \dd r)} \\
& \lesssim \eps^{(n-1)\frac{a}{2}}+\eps^{-2/3+\frac{n a}{2}}\e^{-\frac{c_0}{4}\eps^{-2/3+\delta}}.
\end{align*}
Thus,
\begin{equation*}
|(v_{1,2})'(p)|\lesssim \sum_{n=1}^\infty \eps^{\frac{(n-1)a}{2}}+\eps^{-4/3+(N+1)a}\lesssim 1.
\end{equation*}
Similarly, we can show that $|(v_{1,2})'(q)|\lesssim \eps^{2/3}$, and so the construction of $v_1$ is complete.
\medskip

\emph{Step 2}. We now construct $v_3$. To do so, we pick $g = (g_1,g_2)$ with $g_1 = 0$, $g_2=\eps^{1/3}\mathds{1}_{[p-\eps^\delta,p]}(r),$ and set $v_3(r)=(\dyn-\lambda)^{-1}g$. Following an entirely analogous strategy as for the construction of $v_1$, we compute the first iterate by hand, yielding for $r \geq p$
\begin{equation*}
g^0:=\G^\lambda_{\mathrm{dyn}}g=\begin{pmatrix}
0\\
\frac{\eps^{2/3}}{2i T(p-\eps^\delta)}\e^{-\sqrt{iT(p-\eps^\delta)} \eps^{-2/3} r} \left(\e^{-\sqrt{iT(p-\eps^\delta)} \eps^{-2/3} p}-\e^{-\sqrt{iT(p-\eps^\delta)} \eps^{-2/3} (p-\eps^\delta)}\right )
\end{pmatrix}.
\end{equation*}
In particular, its derivative at $r=p$ is of size $O(1)$ in $\eps$. Writing as before 
$$
g_\leq^n=\left (((\dyn-\lambda)\G^\lambda_{\mathrm{dyn}}-\Id)^n g \right )\mathds{1}_{r \leq p+2n \eps^{\delta}}, \quad g_>^n=\left (((\dyn-\lambda)\G^\lambda_{\mathrm{dyn}}-\Id)^n g \right )\mathds{1}_{r > p+2n \eps^{\delta}},
$$
with respective components $g_\leq^n = (g_{\leq,1}^{n}, g_{\leq,2}^{n})$, an analogously for $g_>^n$, we can bound 
\begin{align*}
|\partial_r (\G^\lambda_{\mathrm{dyn}}((\dyn-\lambda)\G^\lambda_{\mathrm{dyn}}-\Id)^ng)_2|(r) & \leq |\partial_r \G_1 g_{\leq,2}^{n}|(r) + \|Q\|\|\partial_r \G_{\mathrm{dyn}}^\lambda g_{>}^{n}\|_{L^\infty([\frac{p}{2},2q],Q^{-1} \dd r)} \\
& \lesssim \eps^{(n-1)\frac{a}{2}}+\eps^{-1/3+\frac{n a}{2}}\e^{-\frac{c_0}{4}\eps^{-2/3+\delta}}.
\end{align*}
Following the arguments of Step 1, we can show that $|(v_{3,2})'(p)|\gtrsim 1$.  For the bound of $|v_3(p)|$, we note that at the first iterate, its value is $\eps^{2/3}$. Furthermore, we have the bounds 
\begin{align*}
|\G^\lambda_{\mathrm{dyn}} ((\dyn-\lambda)\G^\lambda_{\mathrm{dyn}}-\Id)^n g|(r) & \leq |\G_1 g_\leq^{n}(r)| + \|Q\|\|\G^\lambda_{\mathrm{dyn}} g_>^n\|_{L^\infty([\frac{p}{2},2q],Q^{-1} \dd r)} \\
& \lesssim \eps^{2/3+\frac{na}{2}}+\eps^{-2/3+\frac{na}{2}}\e^{-\frac{c_0}{4}\eps^{-2/3+\delta}} \lesssim \eps^{2/3+(n-1)\frac{a}{2}}.
\end{align*}
Hence, summing over $n$ once again yields the desired upper bound $|v_3(p)|\lesssim \eps^{2/3}$, and following the arguments from Step 1 verbatim, we also see that $|v_3(q)|\lesssim \eps^{2/3}$, as desired. The bounds for $v_2$ and $v_4$ thus follow analogously.
\end{proof}

Finally, we can prove the analogue of Theorem \ref{thm:dynamo}. Indeed, we begin by setting 
\begin{equation*}
(\dyn-\lambda)^{-1}_{\bd}f=(\dyn-\lambda)^{-1}f-\mathcal{J}^{-1}\begin{pmatrix}
((\dyn-\lambda)^{-1}f)_1(p)\\
((\dyn-\lambda)^{-1}f)_1(q)\\
\partial_r (r((\dyn-\lambda)^{-1}f)_2(p)\\
\partial_r (r((\dyn-\lambda)^{-1}f)_2(q)
\end{pmatrix} \cdot
\begin{pmatrix}
v_1(r)\\
v_2(r)\\
v_3(r)\\
v_4(r)
\end{pmatrix},
\end{equation*}
where $\mathcal{J}$ is the matrix given by 
\begin{equation}\label{eq:matrix-J}
\mathcal{J} = \begin{pmatrix}
v_{1,1}(p) &v_{2,1}(p) & v_{3,1}(p) & v_{4,1}(p)\\
v_{1,1}(q) &v_{2,1}(q) & v_{3,1}(q) & v_{4,1}(q)\\
(rv_{1,2})'(p) &(rv_{2,2})'(p) & (rv_{3,2})'(p) & (rv_{4,2})'(p)\\
(rv_{1,2})'(q) &(rv_{2,2})'(q) & (rv_{3,2})'(q) & (rv_{4,2})'(q)
\end{pmatrix}.
\end{equation}
A computation reveals that with our definition, $(\dyn-\lambda)^{-1}_{\bd}$ indeed satisfies the desired boundary conditions on $[p,q]$. We note the following: $\mathcal{J}$ is invertible, and its inverse has operator norm of order $\eps^{-1/3}$. Indeed, write $\mathcal{J}=\mathcal{J}_\main+\mathcal{J}_{\error}$, with 
\begin{equation*}
\mathcal{J}_{\main}=\begin{pmatrix}
v_{1,1}(p) &0 & 0 & 0\\
0 &v_{2,1}(q) & 0 & 0\\
(rv_{1,2})'(p) &0 & (rv_{3,2})'(p) & 0\\
0 &(rv_{2,2})'(q) & 0 & (rv_{4,2})'(q)
\end{pmatrix}.
\end{equation*}
Using Lemma \ref{lemma:homogeneous solutions}, we see that $\|\mathcal{J}_{\error}\|\lesssim \eps^{2/3}$. Furthermore, a computation yields that 
\begin{equation*}
\mathcal{J}^{-1}_{\main}=\begin{pmatrix}
\frac{1}{v_{1,1}(p)} &0 & 0 & 0\\
0 &\frac{1}{v_{2,1}(q)} & 0 & 0\\
-\frac{(rv_{1,2})'(p)}{v_{1,1}(p)(rv_{3,2})'(p)} &0 & \frac{1}{(rv_{3,2})'(p)} & 0\\
0 &-\frac{(rv_{2,2})'(q)}{v_{2,1}(q)(rv_{4,2})'(q)} & 0 & \frac{1}{(rv_{4,2})'(q)}
\end{pmatrix},
\end{equation*}
which has components of order at most $\eps^{-1/3}$, and so its operator norm is bounded by $\eps^{-1/3}$. Thus, we write 
\begin{equation*}
\mathcal{J}^{-1}=\mathcal{J}_{\main}^{-1}(\Id+\mathcal{J}_{\error}\mathcal{J}_{\main}^{-1})^{-1},
\end{equation*}
where $(\Id+\mathcal{J}_{\error}\mathcal{J}_{\main}^{-1})^{-1}$ is well-defined and has operator norm of order $O(1)$ in $\eps$ because of $\|\mathcal{J}_{\error}\mathcal{J}_{\main}^{-1}\|\lesssim \eps^{1/3}$. Hence, we obtain $\|\mathcal{J}^{-1}\|\lesssim \eps^{-1/3}$ as claimed. We can now finally proceed with the proof of the existence of growing modes in the case where $\cI=[p,q]$, that we state in more detail in the following proposition.

\begin{proposition}[Growing mode with boundaries]
\label{prop:boundaries}
Consider the operator $\dyn$ from \eqref{eq:modal-eq-r}--\eqref{eq:modal-eq-theta} on the interval $\cI = [p,q]$ with boundary conditions $b_r|_{\partial \cI}=0$, $(rb_\theta)'|_{\partial \cI}=0$. Assume the conditions of \ref{thm:main result boundaries}. Then, there exists an eigenvalue $\lambda$ of $\dyn$ that satisfies $\lambda=\eps^{1/3}(\lambda_\star+o_{\eps \to 0}(1))$. Furthermore, there holds the asymptotic expansion for the growing mode
    \begin{align*}
    \begin{pmatrix}
    (b_{\mathrm{in}}^\eps)_{r}\\
    (b_{\mathrm{in}}^\eps)_{\theta}
    \end{pmatrix}
     = 
    \begin{pmatrix}
    \eps^{1/3}\sqrt{\frac{-2iM}{r_0^3\Omega'(r_0)}}\\
    1
    \end{pmatrix}\left (\e^{-\frac{1}{2}\eps^{-2/3}c_2^{1/2}(r-r_0)^2}+\psi_{\mathrm{err}}^\eps \right)
    \end{align*}
   for some explicitly computable constant $c_2 \in \mathbb{C}$ with $\Re(c_2^{1/2})>0$, and
   \[
   \sup_{r \in [0,\infty)}\left|\left(1+(|r-r_0|\eps^{-1/3})^N\right)\psi_{\mathrm{err}}^\eps\right| \lesssim \eps^a,
   \]
   for some $a>0$ and some $N>0$ large enough.
\end{proposition}

\begin{proof}
In this proof, we will make explicit the power of the weight $w_\eps$ in the definition of $X_{\mathcal{I}}$ by writing 
\[
X^{N}_{\mathcal{I}}=L^\infty \left([p,q],\max\{1,r^2\}(1+(\eps^{-1/3}|r-r_0|)^{N}\dd r\right).
\]
Analogously, we do the same with the Banach space $Y$ and denote
\[
Y^N_{\mathcal{I}} = L^\infty \left([p,q],r^2(1+(\eps^{-1/3}|r-r_0|)^{N}\dd r\right).
\]
Given $N \in \mathbb{N}$, consider $\eps$ small enough so that $(\L-\lambda)\G^\lambda -\Id$ is contractive on $Y_{2N}$. We then define the Riesz projector
\begin{equation*}
P_{\bd}=\frac{1}{2 \pi i}\int_{\Gamma} -(\dyn-\lambda)^{-1}_{\bd} \dd \lambda =QP Q^{-1}+\mathrm{boundary  \ terms},
\end{equation*}
where, following the conventions of this section we abuse notation slightly and set 
\[
P=\frac{1}{2 \pi i}\int_{\gamma} -(\L-\lambda)^{-1} \dd \lambda,
\]
with 
\[
(\L-\lambda)^{-1}= \G^\lambda\sum_{n \geq 0}(-1)^n((\L-\lambda) \G^\lambda-\Id)^n.
\]
We aim to show that there exists $\phi$ so that $P \phi \neq 0$. From the proof of Theorem \ref{thm:dynamo} in Section \ref{s:proof}, it is direct that there exists $a>0$ such that $\|Q^{-1} (QP Q^{-1}) Q f_\star\|_{X^N_{\mathcal{I}}} \gtrsim 1-\eps^{a} $, where $f_\star$ is defined in \eqref{eq:fvec0}. Therefore, it only remains to deal with the contributions from the boundary. Note that $\|v_i\|_{L^\infty} \lesssim \eps^{1/3}$, so that 
$$
\|v_i\|_{X^{N}_\mathcal{I}}\lesssim \eps^{1/3}\eps^{-N/3}\max\{|p-r_0|^{N}, |q-r_0|^{N}\}.
$$
Thus, we estimate
\begin{equation*}
\left \|\mathcal{J}^{-1}\begin{pmatrix}
((\dyn-\lambda)^{-1}f)^1(p)\\
((\dyn-\lambda)^{-1}f)^1(q)\\
\partial_r (r((\dyn-\lambda)^{-1}f)^2(p)\\
\partial_r (r((\dyn-\lambda)^{-1}f)^2(q)
\end{pmatrix}
\cdot \begin{pmatrix}
v_1(r)\\
v_2(r)\\
v_3(r)\\
v_4(r)
\end{pmatrix}\right \|_{X^{N}_\mathcal{I}} \lesssim \eps^{-N/3}\left | \begin{pmatrix}
((\dyn-\lambda)^{-1}f)^1(p)\\
((\dyn-\lambda)^{-1}f)^1(q)\\
\partial_r (r((\dyn-\lambda)^{-1}f)^2(p)\\
\partial_r (r((\dyn-\lambda)^{-1}f)^2(q)
\end{pmatrix}\right |.
\end{equation*}
However, note from our discussion in Section \ref{s:approx-Green's}, and the assumption at the beginning of the proof that 
$\|(\L-\lambda)^{-1}f \|_{X^{2 N}_\mathcal{I}} \lesssim \eps^{-1/3}\|f\|_{X_\mathcal{I}^{2  N}}$, $\|\partial_r(\L-\lambda)^{-1}f \|_{X^{2 N}_\mathcal{I}} \lesssim \eps^{-1}\|f\|_{X^{2  N}_\mathcal{I}}$. Recalling that we can recover $(\dyn-\lambda)^{-1}$ from $(\L-\lambda)^{-1}$ via conjugation by $Q$, we thus have 
\begin{equation*}
\|(\dyn-\lambda)^{-1} Qf \|_{X^{2  N}_\mathcal{I}} \lesssim \eps^{-1/3}\|f\|_{X^{2  N}_\mathcal{I}}, \quad \|\partial_r(\dyn-\lambda)^{-1} Qf \|_{X^{2  N}_\mathcal{I}} \lesssim \eps^{-1}\|f\|_{X^{2  N}_\mathcal{I}}.
\end{equation*}
But now, note that $\|f_\star\|_{X^{2  N}_\mathcal{I}} \lesssim 1$. Thus, it holds that 
\begin{equation*}
\eps^{-N/3}|(\dyn-\lambda)^{-1}Qf_\star|(p) \leq \frac{\eps^{-N/3}}{w_\eps(p-r_0)}\|(\dyn-\lambda)^{-1}Q f_\star\|_{X^{2 N}_\mathcal{I}}\lesssim \frac{\eps^{-(1+N)/3}}{1+\eps^{-2 N/3}|r_0-p|^{2N}}\lesssim \eps^{\frac{N-1}{3}},
\end{equation*}
and the same game can be played both for other endpoint and for the derivatives. Thus, for $N$ large enough, we see that the boundary terms are of order $\eps^{1/3}$ in $X_{N}$, and thus it holds 
\begin{equation*}
\|Q^{-1}PQf_\star\|_{X^{N}_\mathcal{I}} \gtrsim 1-\eps^{a}-\eps^{2/3} \gtrsim 1,
\end{equation*}
and the result is proven, with the asymptotic expansion being deduced precisely as in the proof of Theorem \ref{thm:dynamo} for the full space in Section \ref{s:proof}.
\end{proof}
\begin{remark}
To extend the result to other domains with boundaries, the construction used above is entirely sufficient. Indeed, for domains of the form $[0,q]$ or $[p,\infty)$, we only need to ``correct'' the boundary conditions at $r=q$ and $r=p$ respectively, which can be done using $v_2$ and $v_4$ for $r=q$, and $v_1$ and $v_3$ for $r=p$. From there, the proof proceeds exactly as for the case $\mathcal{I}=[p,q]$.
\end{remark}

\section{The vertical component}\label{s:z-component}

This section is devoted to the completion of the missing bit that is needed for the slow dynamo effect enunciated in Theorem \ref{thm:dynamo}. So far we have shown that the two-dimensional problem \eqref{eq:1}--\eqref{eq:2} has a growing mode with an eigenvalue of the form $\lambda = \mu\eps^{1/3}$. This yields the result for the radial and azimuthal components of the magnetic field $B_r$ and $B_\theta$. However, the dynamo problem is three dimensional, and the behaviour of the $z$ component must thus also be addressed. Recall from \eqref{eq:modal-eq-z} that the $z$ component solves a forced advection-diffusion equation, with forcing given by the $B_r$ component of our previously constructed growing mode.
\begin{equation}
\label{eq:z-component}
\partial_t B_z + (u\cdot\nabla)B_z = \eps\Delta B_z +  U'(r)B_r.
\end{equation}
Thus, to construct such $z$ component, we need not worry about finding an eigenvalue, but instead it suffices to simply find a finite energy solution to \eqref{eq:z-component}.
In other words, we want to invert the operator $\L_z-\lambda$, where, up to a change of the imaginary part of $\lambda \mapsto \lambda+i\eps^{-1/3}(M\Omega(r_0)+KU(r_0))$, we define
\[
\L_z=\eps \left(\partial_r^2+\frac{1}{r}\partial_r -\eps^{-2/3}M^2-\eps^{-2/3}K^2 \right)-i\eps^{-1/3}(M(\Omega(r)-\Omega(r_0))+K(U(r)-U(r_0))).
\]
Recall that we have showed that 
\begin{equation}\label{eq:lambda-z}
\lambda= \eps^{1/3} \left(\eta - M^2\left[ \frac{1}{r_0^2} + \left(\frac{\Omega_0'}{U_0'}\right)^2 \right] + \sqrt{\frac{-2iM\Omega_0'}{r_0}} - c_2^{1/2} \right),
\end{equation}
for some $\eta$ so that $|\Re(\eta c_2^{1/2})|\leq 2$. From here, note that we can proceed in essentially the same way as for the $(B_r,B_z)$ components. For the sake of brevity, we therefore shall only comment on the slight differences we encounter here. We once again subdivide the interval $[0,\infty)$ into a number of intervals depending on the number of zeros of $T(r)$---four intervals, if $T(r)$ only vanishes at $r_0$---, and construct approximate Green's functions for each of them. It will turn out that we may in fact simply use the same approximate inverses as for the $B_r, B_\theta$ components. The only real difference occurs in the range $r \sim r_0$, since there the cross terms from Section \ref{s:around-r0} disappear when studying the vertical component. Namely, in this regime the approximate operator is given by 
\[
\L_{z,2} - \lambda =\eps \partial_r^2 -i\eps^{-1/3}c_2 (r-r_0)^2 -\eps^{1/3} \left( M^2\left[ \frac{1}{r_0^2}+\left(\frac{\Omega_0'}{U_0'}\right)^2 \right] + \mu\right),
\]
or equivalently
\[
\L_{z,2} - \lambda =\eps \partial_r^2 -i\eps^{-1/3}c_2 (r-r_0)^2 -\eps^{1/3} \left( \eta + \sqrt{\frac{-2iM\Omega'(r_0)}{r_0}}-c_2^{1/2}\right).
\]
Proceeding as in Section  \ref{s:around-r0}, we see that the homogeneous solutions to this operator are given in terms of the parabolic cylinder functions by 
\[
v_1(r) = D_{-\frac{1}{2}c_2^{-1/2} \left(\eta+\sqrt{\frac{-2iM\Omega'(r_0)}{r_0}}\right)}\left(\sqrt{2}c_2^{1/4}\eps^{-1/3}(r-r_0)\right),
\]
\[
v_2(r) = D_{-\frac{1}{2}c_2^{-1/2} \left(\eta+\sqrt{\frac{-2iM\Omega'(r_0)}{r_0}}\right)}\left(-\sqrt{2}c_2^{1/4}\eps^{-1/3}(r-r_0)\right).
\]
From here, the estimates proceed exactly as in the case of the $B_r$ and $B_\theta$ components, see Section \ref{s:around-r0} for further details. All in all, we obtain the following result, where only for this statement, we make explicit the dependence of the Banach space $X$ on the parameter $N\in\N$ by writing
\begin{equation*}
X_{N}=L^\infty\left((0,\infty),\max\{1,r^2\}(1+(\eps^{-1/3}|r-r_0|)^N)\dd r\right).
\end{equation*}

\begin{lemma}
\label{lemma:z-component}
Assume \ref{H0}--\ref{H2}, and let $(B_r,B_\theta)$ be the solution we constructed in the previous sections to the two-dimensional problem \eqref{eq:r0-1}--\eqref{eq:r0-2}. Assume further that there exists $N'\in\N$ such that $|U'(r)| \lesssim r^{N'}$ as $ r \to \infty$. Then, letting $\mathcal{R}$ be the compact interval from Lemma \ref{lemma:consequences of assumptions}, and taking $\mathcal{N}\subset \mathbb{R}\setminus \{0\}$ to be any compact set, for any $N>N'$, and any $\eps>0$ sufficiently small (depending on $N$), there exists an operator $\G_z:X_\mathcal{I}^N \to \mathcal{D}(\dyn_\mathcal{I}) \subset X_\mathcal{I}^N$ so that $\L_z \G_z= \Id$. In particular, setting $B_z=\G_z(-U'(r)B_r)$, we find that 
\[
B=(B_r,B_\theta,B_z)\in X_\mathcal{I}^{N-N'}
\]
defines an eigenfunction to the equations \eqref{eq:dynamo} with eigenvalue $\lambda$ as defined in \eqref{eq:lambda-z}, and satisfies perfectly conducting boundary conditions in the case where $\partial \mathcal{I} \neq \emptyset$.
\end{lemma}

The proof of this lemma in the case where $\mathcal{I}=[0,\infty)$ follows from the arguments indicated prior to the statement, and from the observation that if $B_r\in X_N$ and $U'\in X_\mathcal{I}^{N'}$, then $f = U'(r)B_r\in X_\mathcal{I}^{N-N'}$. When $\partial \mathcal{I} \neq \emptyset$, the only adjustment that needs to be made is to add homogeneous solutions in order to correct the boundary condition for $B_z$, which reads $B'_z|_{\partial \mathcal{I}}=0$. But this can be done in precisely the same way as the construction of $v_3, v_4$ in Section \ref{s:boundaries}.

\appendix

\section{Special functions}\label{s:special-functions}

In this section we shall lay out the basic theory of the modified Bessel functions, the parabolic cylinder and Hermite functions, as well as the Airy functions, which play a crucial role in the analysis of the approximated operators defined along this paper. In particular we shall obtain results regarding their asymptotic expansions. Throughout this section, we shall work frequently with complex numbers and their fractional powers. To clarify notation, we thus fix the following convention: 
\begin{itemize}
    \item $\arg(z) \in (-\pi ,\pi]$ denotes the \emph{principal argument} of a complex number.\\
    \item When writing $z^{\alpha}$, for $\alpha \in \mathbb{R} \setminus \mathbb{N}$, we are referring to the \emph{principal value} of $z^{\alpha}$, i.e.\ for $z=|z|\e^{i \theta}$, $\theta \in (-\pi,\pi]$, $z^\alpha=|z|^\alpha \e^{i \alpha \theta}$.
\end{itemize}

\subsection{Modified Bessel functions}\label{s:appendix-bessel} 

The modified Bessel functions arise as solutions to the differential equation 
\begin{equation}
\label{eq:modified Bessel}
z^2\frac{\dd^2 v_\nu}{\dd z^2}+z \frac{\dd v_\nu}{\dd z}- z^2 v_\nu-\nu^2 v_\nu=0.
\end{equation}
In particular, they can be recovered from solutions to Bessel's equation 
\begin{equation*}
z^2\frac{\dd^2 w_\nu}{\dd z^2}+z \frac{\dd w_\nu}{\dd z}-z^2 w_\nu-\nu^2 w_\nu=0
\end{equation*}
via the change of variables $v_\nu(z)=w_\nu(iz)$. 

As a second order ODE, the solution space of \eqref{eq:modified Bessel} is a two-dimensional vector space. In particular, it is customary to pick the basis vectors $I_\nu(z)$ and $K_\nu(z)$, which are known as the \emph{modified Bessel functions of the first (resp.\ second) kind.} For real values of the parameter $\nu$ and all $z$ in a sector of the form $\{ z\in\Co : \arg(z) <\pi/2\}$, $I_\nu(z)$ is bounded at zero, and grows exponentially as $|z| \to \infty$, whereas $K_\nu(z)$ diverges near zero, and decays exponentially as $|z| \to \infty$. In, particular, from \cite{Olver74}*{Chapter 7.8} there holds 
\[
K_\nu(z) \sim \frac{1}{2}\Gamma(\nu)\left(\frac{z}{2}\right)^{-\nu}, \quad \text{for } \nu >0 \text{ and } z \to 0,
\]
and from the series expansion 
\begin{equation*}
I_\nu(z)=\left(\frac{z}{2}\right)^{\nu} \sum_{n \geq 0}\frac{(\frac{1}{4}z^2)^n}{n!\Gamma(\nu+n+1)}
\end{equation*}
it is clear that $I_\nu(z) \sim (\frac{1}{2}z)^\nu$ near zero.

In this work, our particular interests lie in the asymptotic behaviour of $I_\nu(z)$, $K_\nu(z)$ as $\nu \to \infty$. We therefore require bounds on these functions in the limit $\nu \to \infty$ which hold \emph{uniformly} for all values of $z$ in some complex sector as specified above. Fortunately, such bounds have been constructed in \cite{Olver74}*{Chapter 10.7}, and we aim to give a brief summary here of the results therein. In particular, since the bounds we require for the modified Bessel functions are the most subtle of all the special function bounds we employ, we shall sketch the proof of the following result.

\begin{lemma}
\label{lemma:Bessel bounds}
There exists a constant $C>0$ so that uniformly in $z\in\Co$ with $\arg(z)=\pi/4$, and for all $\nu \geq 1$ there holds 
\begin{align*}
&|I_\nu(\nu z)|\leq C\frac{\e^{\nu \xi}}{(2\pi \nu)^{1/2}(1+z^2)^{1/4}},\\
& |K_{\nu}(\nu z)|  \leq C \left(\frac{\pi}{2 \nu}\right)^{1/2}\frac{\e^{-\nu \xi}}{(1+z^2)^{1/4}},
\end{align*}
where we have set 
\[
\xi(z)=(1+z^2)^{1/2}+\log\left(\frac{z}{1+(1+z^2)^{1/2}}\right).
\]
\end{lemma}

\begin{remark}
The above result is actually valid in a much larger context, namely such a constant can be found uniformly for any sector of the form $|\arg(z)|\leq \pi/2-\delta$ for arbitrarily small $\delta>0$, see \cite{Olver74}*{Chapter 10.7}. However, since for our applications we only require $|\arg(z)|=\pi/4$, we shall only prove the result in this simplified case.
\end{remark}
We start by noting that in \cite{Olver74}*{Chapter 10.7.4} it is shown that for any $\nu >0$, $|\arg(z)|<\pi/2$, there holds for any $n\in\N$
\begin{align}
\label{eq:expansions modified bessel}
&I_\nu(\nu z)=\frac{1}{1+\eta_{n,1}(\nu,\infty)}\frac{\e^{\nu \xi}}{(2\pi \nu)^\frac{1}{2}(1+z^2)^{1/4}}\left (\sum_{s=0}^{n-1}\frac{U_s(p)}{\nu^s}+\eta_{n,1}(\nu,z)\right ),\\
\label{eq:expansions modified bessel2}
&K_{\nu}(\nu z) = \left(\frac{\pi}{2 \nu}\right)^{1/2}\frac{\e^{-\nu \xi}}{(1+z^2)^{1/4}}\left (\sum_{s=0}^{n-1}(-1)^s \frac{U_{s}(p)}{\nu^s}+\eta_{n,2}(\nu,z)\right),
\end{align}
where $p=(1+z^2)^{-1/2}$, and the $U_s(p)$ are polynomials defined recursively by 
\begin{equation*}
U_{s+1}(p)=\frac{1}{2}p^2(1-p^2)U_s'(p)+\frac{1}{8}\int_0^p(1-5q^2)U_s(q)\dd q
\end{equation*}
and $U_0(p)=1$. Finally, the error terms $\eta_{n,j}(\nu,z)$ satisfy
\begin{align*}
|\eta_{n,1}(\nu,z)|\leq 2 \exp{\left(\frac{2\mathcal{V}_{1,p}(U_1)}{\nu}\right)}\frac{\mathcal{V}_{1,p}(U_n)}{\nu^n}\\
|\eta_{n,2}(\nu,z)|\leq 2 \exp{\left(\frac{2\mathcal{V}_{0,p}(U_1)}{\nu}\right)}\frac{\mathcal{V}_{0,p}(U_n)}{\nu^n}
\end{align*}
where $\mathcal{V}_{1,p}$ (resp.\ $\mathcal{V}_{0,p}$) denotes the $1$-variation from $1$ (resp.\ from $0$) to $p$ along so called \emph{$\xi$--progressive paths}. It thus remains to derive uniform bounds on the variational terms. To do so, it is first imperative to define the notion of a $\xi$--progressive path. Indeed, we have the following definition from \cite{Olver74}*{Chapter 6.11.4}.
\begin{definition}
Let $s:[0,1] \to \mathbb{C}$ be a twice differentiable function so that 
\begin{enumerate}
    \item $s''(\tau)$ is continuous
    \item $s'(\tau) \neq 0$ for all $\tau \in [0,1]$.
\end{enumerate}
Then, $s$ is called an $R_2$--arc.
Let now $\mathcal{P}_j$, $j=1,2$ be paths in the complex plane connecting two points $a_j, b_j \in \mathbb{C}$, and let $\xi :\mathbb{C} \to \mathbb{C}$ be a conformal map. $\mathcal{P}_j$ is called a $\xi$--progressive path if 
\begin{enumerate}
    \item $\mathcal{P}_j$ consists of a finite number of $R_2$--arcs
    \item As $t$ passes along $\mathcal{P}_j$ from $b_j$ to $a_j$, $\Re(\xi(t))$ is non-decreasing for $j=1$, and non-increasing for $j=2$.
\end{enumerate}
\end{definition}

With this in mind, we now provide a proof of Lemma \ref{lemma:Bessel bounds}.

\begin{proof}[Proof of Lemma \ref{lemma:Bessel bounds}.]
It is clear that it suffices to bound $\mathcal{V}_{1,p}(U_1)$ uniformly for $|\arg(z)|=\pi/4$. For simplicities sake, we shall assume that $\arg(z)=\pi/4$, since the general case follows similarly. First, note that we may compute 
\[
U_1(p)=\frac{1}{8}\left(p-\frac{5}{3}p^3\right).
\]
Next, we consider the image of $\arg(z)=\pi/4$ under the map $p(z)=(1+z^2)^{-1/2}$. Letting $z=a\e^{i\pi/4}$, $a>0$, we have 

\[
p(a\e^{i\pi/4})=\left (\frac{1-ia^2}{1+a^4} \right )^{1/2}.
\]
Next, consider 
\begin{equation*}
\xi(a)=\left(\frac{2+a^4-ia^2}{1+a^4} \right)^{1/2}+\log\left(\frac{p(a\e^{i\pi/4})}{1+\left(\frac{2+a^4-ia^2}{1+a^4}\right)^{1/2}}\right).
\end{equation*}
A computation shows that $\Re(\xi(a))$ is in fact a non-increasing function in $a$ for positive values of $a$. Therefore, given any point $z=a\e^{i\pi/4}$, $\xi$--progressive paths from $1$ (resp.\ $0$) to $z$ can be constructed simply by taking $\gamma(t)=p(t\e^{i\pi/4})$. Given any smooth function $f:\mathbb{C} \to \mathbb{C}$, it therefore holds
\begin{equation*}
\mathcal{V}_{1,p}(f)\leq \int_0^\infty \left|\frac{it}{(1+it^2)^{3/2}}f'\left(\frac{1}{1+it^2}\right)\right| \dd t \leq \|f'\|_{L^\infty(|z|\leq 1)} \int_0^\infty \frac{t}{(1+t^4)^{3/4}}\dd t \leq C\|f'\|_{L^\infty(|z|\leq 1)}.
\end{equation*}
Similarly, it also holds $\mathcal{V}_{0,p}(f) \leq C\|f'\|_{L^\infty(|z|\leq 1)}$, uniformly for all $\arg(z)=\pi/4$. Therefore, we observe that the error bounds satisfy
\begin{align*}
|\eta_{1,1}(\nu,z)|\leq 2 \nu^{-1}C\exp(2\nu^{-1}C\|U'_1\|_{L^\infty(|z|\leq 1)})\|U'_1\|_{L^\infty(|z|\leq 1)},\\
|\eta_{1,2}(\nu,z)|\leq 2 \nu^{-1}C\exp(2\nu^{-1}C\|U'_1\|_{L^\infty(|z|\leq 1)})\|U'_1\|_{L^\infty(|z|\leq 1)},
\end{align*}
and so the result follows immediately from \eqref{eq:expansions modified bessel}--\eqref{eq:expansions modified bessel2}.
\end{proof}

\subsection{Parabolic cylinder and Hermite functions}\label{s:appendix-parabolic}

In this section, we summarise basic facts about parabolic cylinder and Hermite functions. Most of this material may be found in \cite{Lebedev72}*{Chapter 10.6}. We consider the solutions to the second order differential equation
\begin{equation*}
\frac{\dd^2 w_\nu}{\dd z^2}+\left(\nu+\frac{1}{2}-\frac{z^2}{4}\right)w_\nu=0.
\end{equation*}
The solution space of this equation is spanned by the functions $D_\nu(z)$,  $D_\nu(-z)$, where we call $D_\nu(z)$ a parabolic cylinder function. Making the change of variables 
\[
w_\nu = \e^{-\frac{1}{4}z^2}u_\nu\left(\frac{z}{\sqrt{2}}\right),
\]
yields that $u_\nu$ satisfies the ODE 
\begin{equation}
\label{eq:hermite equation}
\frac{\dd^2 u_\nu}{\dd z^2}-2z\frac{\dd u_\nu}{\dd z}+2\nu u_\nu=0,
\end{equation}
which is called the \emph{Hermite's equation}. Solutions to \eqref{eq:hermite equation} are called Hermite functions, and in particular, the solution space is spanned by $H_\nu(z)$, $H_{\nu}(-z)$ so long as $\nu \neq 0,1,2 \dots$. In the case where $\nu=n=0,1,2 \dots$, $H_n$ corresponds to the $n^{\text{th}}$ Hermite polynomial. The Hermite functions enjoy the following pointwise asymptotics, which may be found in \cite{Lebedev72}*{Chapter 10.6}.
\begin{equation}\label{eq:asymp-H1}
H_\nu(z) = (2z)^\nu (1+\BigO(z^{-2})), \quad \text{if } \arg(z)\in \left(-\frac{3\pi}{4},\frac{3\pi}{4}\right),
\end{equation}
\begin{equation}\label{eq:asymp-H2}
H_\nu(z) = ((2z)^\nu-C(\nu)\e^{z^2} )(1+\BigO(z^{-2})), \quad \text{if } \arg(z)\in \left(\frac{\pi}{4},\frac{5\pi}{4}\right).
\end{equation}
(in \eqref{eq:asymp-H2} we have made the slight abuse of notation by writing $\arg(z)\in (\frac{\pi}{4},\frac{5\pi}{4})$ to mean $\arg(z) \in (\frac{ \pi}{4},\pi] \cup(-\pi,-\frac{3 \pi}{4})$).
With these asymptotic representations under our belt, we may now prove the following pointwise bounds for the parabolic cylinder functions. 
\begin{lemma}
\label{lemma:parabolic cylinder bounds}
For any fixed $\nu \neq 0,1,2 \dots$, there exists a constant $C>0$ so that 
\begin{equation}\label{eq:D-decay}
|D_\nu(z)|  \leq C \left|\e^{-\frac{1}{4}z^2} \right|\left(1+|z|\right)^{\Re(\nu)}, \quad \text{if } |\arg(z)|<\frac{3\pi}{4},
\end{equation}
\begin{equation}\label{eq:D-growth}
|D_\nu(z)| \leq C \left|\e^{\frac{1}{4}z^2} \right|, \quad \text{if } \arg(z)\in \left(\frac{3\pi}{4},\frac{5\pi}{4}\right).
\end{equation}
\end{lemma}
\begin{proof}
We begin by noting that we have the following relation
\[
D_\nu(z) = 2^{-\nu/2}\e^{-\frac{1}{4}z^2} H_\nu\left(\frac{z}{\sqrt{2}}\right).
\]
On the one hand, we know that $H_\nu(z)$ is a continuous function in $z\in\Co$, and moreover $|H_\nu(0)|\lesssim 1$ for any $\nu\in\Co$. On the other hand, we may now use \eqref{eq:asymp-H1}, to deduce that for $|\arg(z)|\leq 3\pi/4$, it holds
\[
|H_\nu(z)|\lesssim (1+|z|)^\gamma\left(1+(1+|z|)^{-2}\right) \lesssim (1+|z|)^\gamma,
\]
and hence we obtain \eqref{eq:D-decay}. Finally, in order to prove \eqref{eq:D-growth} we use \eqref{eq:asymp-H2} in a similar way. This time we write
\[
\left|\e^{-\frac{1}{4}z^2}H_\nu\left(\frac{z}{\sqrt{2}}\right)\right|\lesssim \max\left\lbrace \left|\e^{-\frac{1}{4}z^2}\right|(1+|z|)^\gamma, \left|\e^{\frac{1}{4}z^2}\right| \right\rbrace \left(1+(1+|z|)^{-2}\right) \lesssim \left|\e^{\frac{1}{4}z^2}\right|.
\]
At this point we observe that if $|z|\sim 1$ then both quantities in the maximum operator are comparable. However, for values of $z$ with $|z|\to \infty$ we have
\[
\left|\e^{-\frac{1}{4}z^2}\right|(1+|z|)^\gamma\lesssim \left|\e^{\frac{1}{4}z^2}\right|
\]
provided that $\Re(z)^2\geq \Im(z)^2$, or equivalently
\[
\arg(z)\in \left[-\frac{\pi}{4},\frac{\pi}{4}\right]\cup\left[\frac{3\pi}{4},\frac{5\pi}{4}\right].
\]
Since the first interval is not compatible with equation \eqref{eq:asymp-H2}, we discard it, and hence we find the sought estimate \eqref{eq:D-growth}.
\end{proof}

\subsection{Airy functions}\label{s:Appendix airy} 
We finally give a brief introduction to the theory of Airy functions. These are solutions to the second order differential equation 
\begin{equation}
\label{eq:airy equation}
\frac{\dd^2 v}{\dd z^2}-zv=0.
\end{equation}
The typical convention (see e.g.\ \cite{Olver74}) is to make the choice of basis vectors $\Ai(z)$, $\Bi(z)$ for the solution space of \eqref{eq:airy equation}. With this choice, for real values of $x$, $\Ai(x)$ decays exponentially for $x>0$ and oscillates for $x<0$, whereas $\Bi(x)$ grows exponentially for $x>0$ and osciallates for $x<0$. However, for the applications we have in mind, we are primarily interested in \emph{complex} values of $z$. For such values, it turns out (see \cite{Olver74}*{Chapter 11.8}) that $\Bi(z)$ is exponentially large, as long as $\arg(z) \neq \pi$. In order to derive a suitable Green's function for Section \ref{s:linear-zeroes}, we require that at all times, at least one of the solutions to \eqref{eq:airy equation} remains small. Therefore, we shall pick the function $\Ai(\e^{2\pi i/3}z)$ as a second solution of \eqref{eq:airy equation} instead of $\Bi(z)$. To verify that indeed $\Ai(z)$, $\Ai(\e^{2\pi i/3}z)$ are linearly independent, we examine the asymptotic bounds for the Airy functions $\Ai(z)$.
Indeed, from \cite{Olver74}*{Chapter 11.8} as well as \cite{DLMF} we set 
\[
\zeta = \frac{2}{3}z^{3/2},
\]
where $z^{3/2}$ denotes the \emph{principal branch} of $z^{3/2}$. Then, we have the asymptotic expansion
\begin{align}
\label{eq:Airy asymptotics}
\begin{split}
&\Ai(z) \sim \frac{e^{-\zeta}}{2\sqrt{\pi}z^\frac{1}{4}}\sum_{n \geq 0}(-1)^n \frac{u_n}{\zeta^n}, \quad |\arg(z)|\leq \pi -\delta,\\
&\Ai'(z) \sim -\frac{z^\frac{1}{4}e^{-\zeta}}{2\sqrt{\pi}}\sum_{n \geq 0}(-1)^n \frac{v_n}{\zeta^n}, \quad |\arg(z)|\leq \pi -\delta,
\end{split}
\end{align}
where $u_0=v_0=1$, and recursively 
\begin{equation*}
u_k=\frac{(6k-5)(6k-3)(6k-1)}{(2k-1)216k}u_{k-1}, \quad v_k=\frac{6k+1}{1-6k}u_k.
\end{equation*}
From this, we deduce the following lemma.

\begin{lemma}\label{lemma:Airy bounds}
Let $z=r\e^{i\frac{\pi}{6}}$. Then, there exists a constant $C>0$ uniform in $ r \in [0,\infty)$ such that the following bounds hold true,
\begin{align}\label{eq:useful airy bounds}
\begin{split}
&|\Ai(z)|, |\Ai(-\e^{2\pi i/3}z)| \leq C\e^{-\frac{\sqrt{2}}{3}r^{3/2}},\\
&|\Ai(-z)|, |\Ai(\e^{2\pi i/3}z)| \leq C\e^{\frac{\sqrt{2}}{3}r^{3/2}},\\
&|\Ai'(z)|, |\Ai'(-\e^{2\pi i/3}z)|\leq C(1+r^\frac{1}{4}) \e^{-\frac{\sqrt{2}}{3} r^{3/2}},\\
&|\Ai'(-z)|, |\Ai'(\e^{2\pi i/3}z)|\leq C(1+r^\frac{1}{4}) \e^{\frac{\sqrt{2}}{3} r^{3/2}}.
\end{split}
\end{align}
\end{lemma}
\begin{proof}
We begin by understanding how to truncate the asymptotic series \eqref{eq:Airy asymptotics}. From \cite{DLMF}, we may truncate them at degree $N$, up to introducing an error of order $|u_N||\zeta|^{-N}$ multiplied by 
\begin{enumerate}
    \item $1$, if $|\arg(z)|<\pi/3$,
    \item $\min\{\csc(\arg(\zeta)),\chi(N+\sigma)+1\}$ if $\pi/3<|\arg(z)|<2\pi/3$,
    \item $\sqrt{2\pi (N+\sigma)}{|\sec(\arg(\zeta))|^{N+\sigma}}+\chi(N+\sigma)+1$ if $2\pi/3<|\arg(z)|<\pi$.
\end{enumerate}
Here we set, 
\[
\chi(x)=\pi^{1/2}\frac{\Gamma\left(\frac{1}{2}x+1\right)}{\Gamma\left(\frac{1}{2}x+\frac{1}{2}\right)},
\]
\begin{align*}
    N\geq 0, \quad \sigma=\frac{1}{6} & \quad \text{for the bounds on } \Ai,\\
    N\geq 1, \quad \sigma=0 & \quad \text{for the bounds on } \Ai'.\\
\end{align*}
It is easy to see that, as long as $\arg(z)$ is not near $\pm \pi$, $\sec(\arg(\zeta))$ is bounded. Hence, we see that 
\begin{align}
\label{eq:Airy bounds}
\begin{split}
&|\Ai(z)| \leq C\left|\frac{e^{-\zeta}}{2\sqrt{\pi}z^\frac{1}{4}}\right|, \quad |\arg(z)|\leq \pi -\delta,\\
&|\Ai'(z)| \leq C\left|\frac{z^\frac{1}{4}e^{-\zeta}}{2\sqrt{\pi}}\right|, \quad |\arg(z)|\leq \pi -\delta.
\end{split}
\end{align}
Furthermore, $\Ai(z), \Ai'(z)$ are always bounded near zero (see \cite{Olver74}), so to deduce \eqref{eq:useful airy bounds} from \eqref{eq:Airy bounds}, it suffices to study the behaviour at infinity. In particular, we will check if the choice of $\sqrt{2}/3$ in the exponential indeed works for all cases. For $r>0$, the principal value of $(r\e^{i\pi/6})^{3/2}$ is given by $r^{3/2} \e^{i \pi/4} $. Thus, in this case the real part of $\e^{i \pi/4}$ is indeed $\sqrt{2}/2$, so 
\[
\Re(\zeta) = \frac{2}{3}\Re(z^{2/3}) = \frac{\sqrt{2}}{3}r.
\]
Similarly, for $-\e^{2 \pi i/3}\e^{i \pi /6}r=\e^{-i \pi /6}r$, we compute $(\e^{-i \pi /6}r)^{3/2}=r^{3/2}\e^{-i \pi /4}$, so once again $\sqrt{2}/2$ is the relevant coefficient of the real part.
Next, we consider $\Ai(-r\e^{i \pi/6})$. Here, we have 
\[
(-r\e^{i \pi/6})^{3/2}=(r\e^{-5\pi i/6})^{3/2}=\e^{3\pi i/4}r^{3/2}.
\]
In particular, we deduce that the exponential has real part $-\sqrt{2}/2$. Finally, we have that 
\[
(\e^{2\pi i/3}\e^{i \pi/6}r)^{3/2}=\e^{-3\pi i/4}r^{3/2},
\]
and so again the choice of $-\sqrt{2}/2$ works. Crucially, as stated in \cite{DLMF}, we may truncate the asymptotic expansions \eqref{eq:Airy asymptotics} at degree $N$, introducing an error of order $u_N\zeta^{-N}$, as long as $|\arg(z)| \leq \pi/3$. Thus, for all $|z|$ sufficiently large, we deduce the bounds  required for our purposes in Section \ref{s:linear-zeroes}.
\end{proof}

\section{The operator around the critical radius}\label{s:around-r0-app}

In this appendix we will show that the approximated operator $\L_2$ around the critical radius $r_0$ has an isolated eigenvalue and an eigenfunction of the form stated in the proof of Theorem \ref{theorem:mainResult}. Since this is a key step in the derivation of the profiles of the growing mode for the operator $\L$, we give more details about the derivation of this result in the following proposition.

\begin{proposition}
    The operator
    \[
    \L_2 = \eps\partial_r^2 - c_2\eps^{-1/3}(r-r_0)^2 + \eps^{1/3}\left(-M^2\left[\frac{1}{r_0^2}+\frac{\Omega_0'}{U_0'}\right] +\sqrt{\frac{-2iM\Omega_0'}{r_0}}\right)
    \]
    has an eigenfunction in $L^\infty(\R,\max\{1,r^2\}w_\eps(r-r_0) \dd r)$ of the form
    \[
    f_\star(r) = \e^{-\frac{1}{2}\eps^{-2/3}c_2^{1/2}(r-r_0)^2},
    \]
    with eigenvalue
    \[
    \lambda_\star = \eps^{1/3}\left[-M^2\left(\frac{1}{r_0^2}+\frac{\Omega_0'}{U_0'}\right) +\sqrt{\frac{-2iM\Omega_0'}{r_0}} - c_2^{1/2}\right],
    \]
    where $c_2^{1/2}$ represents a complex number with positive real part. In fact, this is the spectral value of $\L_2$ with largest real part. Moreover, the eigenvalue $\lambda_\star$ is isolated.
\end{proposition}

\begin{proof}
First of all, a direct computation shows that
\[
f_\star''(r) = \eps^{-2/3}c_2^{1/2}\left( \eps^{-2/3}c_2^{1/2}(r-r_0)^2 - 1 \right) f_\star(r),
\]
so that plugging it into the definition of the operator, we find
\[
\eps f_\star''(r) - c_2\eps^{-1/3}(r-r_0)^2f_\star(r) = - \eps^{1/3}c_2^{1/2}f_\star (r),
\]
and therefore,
\[
\L_2f_\star(r) = \eps^{1/3}\left[-M^2\left(\frac{1}{r_0^2}+\frac{\Omega_0'}{U_0'}\right) +\sqrt{\frac{-2iM\Omega_0'}{r_0}} - c_2^{1/2}\right] f_\star(r).
\]
Next, we prove that the eigenvalue is indeed isolated and the largest one of $\L_2$ in real part. Indeed, most of the work for this has already been done in Section \ref{s:around-r0}. Recall that for $\mu=\lambda_\star+\eps^{1/3}\eta$, $\eta \neq 0$ and $|\eta|$ small enough, there exist two linearly independent solutions to $(\L_2-\mu)f=0$ given by 
\begin{equation*}
v_1(r) = D_{\frac{1}{2}c_2^{-1/2}\eta}\left(\sqrt{2}c_2^{1/4}\eps^{-1/3}(r-r_0)\right),
\end{equation*}
\begin{equation*}
v_2(r) = D_{\frac{1}{2}c_2^{-1/2}\eta}\left(-\sqrt{2}c_2^{1/4}\eps^{-1/3}(r-r_0)\right),
\end{equation*}
where 
\[
D_\nu(z) = 2^{-\nu/2}\e^{-z^2/4} H_\nu\left(\frac{z}{\sqrt{2}}\right).
\]
Therefore, one can construct a Green's function for this equation given by 
\begin{equation*}
(\mathcal{G}_{r_0}f)(r) 
 = \mathfrak{w}(\eta)\eps^{-2/3}v_1(r)\int_0^rv_2(s)f(s)\dd s + \mathfrak{w}(\eta)\eps^{-2/3}v_2(r)\int_r^\infty v_1(s)f(s)\dd s.
\end{equation*}
where $\mathfrak{w}(\eta)$ is the Wronskian. Recall the bounds on the Hermite functions 
\begin{equation*}
|D_\nu(z)|  \lesssim \left|\e^{-z^2/4} \right|\left(1+|z|\right)^{\Re(\nu)}, \quad \text{if } |\arg(z)|<\frac{3\pi}{4},
\end{equation*}
\begin{equation*}
|D_\nu(z)| \lesssim \left|\e^{z^2/4} \right|, \quad \text{if } \arg(z)\in \left(\frac{3\pi}{4},\frac{5\pi}{4}\right),
\end{equation*}
In particular, as long as $\Re(\nu)<1$, we can redo the arguments in Section \ref{s:around-r0} to show that $\G_{r_0}$ is a well defined operator on $L^\infty(\mathbb{R},\max\{1,r^2\}w_\eps(r-r_0)\dd r)$, which acts as a right inverse to $\L_2$. Hence, surjectivity of $\L_2-\lambda$ both for $\lambda$ near $\lambda_\star$, and also for any $\lambda$ with $\Re(\lambda)>\Re(\lambda_\star)$ is proven. Finally, injectivity is an immediate consequence of the asymptotic properties of the functions functions $D_\nu$, which are bounded on $\mathbb{R}$ if and only if $\nu =0,1,2 \dots$. This completes the proof.
\end{proof}

\section{Properties of the weight function}

In this final section, we state a lemma that compiles a set of useful properties of the weight function $w_\eps$, which appears in the definitions of the Banach spaces $X$ and $Y$. These properties play an important role in the proofs of Sections \ref{s:approx-Green's} and \ref{s:linear-zeroes}. For the convenience of the reader---and to avoid repetitive arguments in the course of those longer proofs---we have chosen to collect all relevant facts about $w_\eps$ in one place here.

Recall that the weight function is defined for any $\eps>0$ and any fixed $N\in\N$ as the following polynomial of degree $N$,
\[
w_\eps(s) = 1 + (\eps^{-1/3}|s|)^N.
\]
We now present the following set of properties of the weight function.

\begin{lemma}\label{lemma:properties-weight}
The weight $w_\eps(\cdot)$ satisfies the following properties for any $\eps>0$ sufficiently small.
    \begin{enumerate}
        \item Let $\gamma,\chi>0$ and $\vartheta\in\R$ be fixed. For any $|r-r_0|\geq\eps^\gamma$ there holds the estimate
        \[
        \frac{w_\eps(r-r_0)}{w_\eps(\eps^\gamma)} \frac{(1 + \eps^{-1/3}|r-r_0|)^\vartheta}{(1 + \eps^{-1/3+\gamma})^\vartheta} \leq C \e^{\frac{1}{2}\chi \eps^{-2/3} ((r-r_0)^2 - (\eps^\gamma)^2)},
        \]
        where $C>0$ is a uniform constant that depends continuously on $N$, $\vartheta$, and $\chi$.
        
        \item Fix $a>b\geq 0$, let $r_0<s\leq r$ (analog.\ $r\leq s< r_0$) be such that $|r-s|\leq \eps^a$ and assume that $|s-r_0|\geq \eps^b$. Then,
        \[
        \frac{w_\eps(r-r_0)}{w_\eps(s-r_0)} \leq C
        \]
        where $C>0$ is a uniform constant, independent of all parameters.

        \item Let $a>0$, for any $r\geq \eps^a$ there holds the estimate
        \[
        \frac{w_\eps(r-r_0)}{w_\eps(\eps^a-r_0)}\leq \left (\frac{r}{\eps^a} \right )^N.
        \]
        
        \item For any $\gamma>0$, and $s,r\geq 0$ such that $|s-r_0|\geq \eps^\gamma$ there holds 
        \begin{equation*}
        \frac{w_\eps(r-r_0)}{w_\eps(s-r_0)} \leq C \e^{\eps^{-\gamma}|s-r|}.
        \end{equation*}
        where $C>0$ is a uniform constant that only depends on $N$.
    \end{enumerate}
\end{lemma}

\begin{proof}
    \emph{(1)}
    Let us define $x = \eps^{-1/3}|r-r_0|$ and $y = \eps^{-1/3+\gamma}$, which satisfies $x\geq y>1$ since $\gamma<1/3$ by assumption. With this notation we have
    \[
    \frac{w_\eps(r-r_0)}{w_\eps(\eps^\gamma)} \frac{(1 + \eps^{-1/3}|r-r_0|)^\vartheta}{(1 + \eps^{-1/3+\gamma})^\vartheta} =  \frac{(1+x^N)(1+x)^\vartheta}{(1+y^N)(1+y)^\vartheta}.
    \]
    Now observe that if $\vartheta\geq 0$ we have that $(1+x)^\vartheta\leq (2x)^\vartheta$ and $(1+y)^\vartheta\geq (y)^\vartheta$, whereas if $\vartheta<0$ we have that $(1+x)^\vartheta\leq (x)^\vartheta$ and $(1+y)^\vartheta \geq (2y)^\vartheta$. Therefore, for any $\vartheta\in\R$ we can write
    \[
    \frac{(1+x)^\vartheta}{(1+y)^\vartheta} \leq 2^{|\vartheta|} \left(\frac{x}{y}\right)^{\vartheta},
    \]
    and moreover
    \[
    \frac{(1+x^N)(1+x)^\vartheta}{(1+y^N)(1+y)^\vartheta} \leq 2^{1 + |\vartheta|} \left(\frac{x}{y}\right)^{N+\vartheta}.
    \]
    This in turn implies that there exists a constant $C>0$ depending exclusively on $N\in \N$ and $\vartheta\in\R$ such that the following estimate
    \[
    \frac{(1+x^N)(1+x)^\vartheta}{(1+y^N)(1+y)^\vartheta} \leq C \e^{\frac{1}{2}(x^2y^{-2} - 1)} = C\e^{\frac{1}{2}y^{-2}(x^2-y^2)} \leq C\e^{\frac{1}{2}(x^2-y^2)}
    \]
    holds true uniformly in $x,y\geq 1$, and thus the claim of the lemma follows.

    \emph{(2)} Take $r_0<s\leq r$, since $|r-s|\leq \eps^a$ and $|s-r_0|\geq\eps^b$, we can write via triangle inequality
    \[
    \frac{w_\eps(r-r_0)}{w_\eps(s-r_0)}  \leq 2\frac{(\eps^{-1/3}|r-r_0|)^N}{(\eps^{-1/3}|s-r_0|)^N} \leq 2\left(1 + \frac{|r-s|^N}{|s-r_0|^N}\right) \leq 2(1 + \eps^{N(a-b)}).
    \]
    Thus, the right hand side can be uniformly bounded by $4$ provided that $a>b$.
    The case $r\leq s< r_0$ follows by symmetry.

    \emph{(3)} Take $r\geq \eps^a$ and  write the trivial bound
    \[
    \frac{w_\eps(r-r_0)}{w_\eps(\eps^a-r_0)}\leq \frac{1+(\eps^{-1/3}|r-r_0|)^N}{(\eps^{-1/3}|\eps^a-r_0|)^N}.
    \]
    On the one hand, the numerator can be estimated by
    \[
    1+(\eps^{-1/3}|r-r_0|)^N \lesssim (\eps^{-1/3}r)^N + (\eps^{-1/3}r_0)^N \lesssim \eps^{-\left(\frac{1}{3}-a\right)N} \left(\frac{r}{\eps^a}\right)^N + \eps^{-\frac{1}{3}N}.
    \]
    On the other hand, for the denominator we simply use that $|\eps^a-r_0|\gtrsim r_0$ to write
    \[
    (\eps^{-1/3}|\eps^a-r_0|)^N \gtrsim \eps^{-\frac{1}{3}N},
    \]
    so that combining both estimates we obtain
    \[
    \frac{w_\eps(r-r_0)}{w_\eps(\eps^a-r_0)}\lesssim \eps^{aN} \left(\frac{r}{\eps^a}\right)^N + 1 \lesssim \left(\frac{r}{\eps^a}\right)^N
    \]
    where the last inequality comes from the observation that $r\geq \eps^a$.

    \emph{(4)} We start by estimating
    \[
    \frac{w_\eps(r-r_0)}{w_\eps(s-r_0)} \leq \frac{1 + (\eps^{-1/3}|s-r_0|)^N + (\eps^{-1/3}|s-r|)^N}{1 + (\eps^{-1/3}|s-r_0|)^N} \leq 1 + \frac{(\eps^{-1/3}|s-r|)^N}{1 + (\eps^{-1/3}|s-r_0|)^N}.
    \]
    Now, since $|s-r_0|\geq \eps^\gamma$ we get
    \[
    1 + \frac{(\eps^{-1/3}|s-r|)^N}{1 + (\eps^{-1/3}|s-r_0|)^N} \leq 1 + \frac{|s-r|^N}{|s-r_0|^N} \leq 1 + \eps^{-\gamma N}|s-r|^N \leq C \e^{\eps^{-\gamma}|s-r|}.
    \]
    where for the last inequality we used the fact that there exists a constant $C>0$ only depending on $N$ so that $1+x^N\leq C\e^x$.
\end{proof}

\addtocontents{toc}{\protect\setcounter{tocdepth}{0}}
\section*{Acknowledgements}
The authors thank Michele Coti Zelati, Lucas Ertzbischoff, and Niklas Knobel and Massimo Sorella for insightful discussions. We also thank Michele Coti Zelati and Massimo Sorella for invaluable comments on initial drafts of the manuscript.
The research of VNF is funded by the ERC-EPSRC Horizon Europe Guarantee EP/X020886/1. The research of DV is funded by the Imperial College President's PhD Scholarships.

\addtocontents{toc}{\protect\setcounter{tocdepth}{1}}
\bibliographystyle{abbrv}
\bibliography{dynamo.bib}

\end{document}